\documentclass[12pt]{amsart} 
\usepackage{amsmath,amsfonts,amssymb,amsthm,amsxtra,amscd} 
\usepackage[all]{xy} 
\usepackage{mathrsfs} 
\usepackage{graphicx}
\numberwithin{equation}{section}

\theoremstyle{plain} 
\newtheorem{thm}{Theorem}[section] 
\newtheorem{prop}[thm]{Proposition} 
\newtheorem{lemma}[thm]{Lemma} 
\newtheorem{cor}[thm]{Corollary} 
\newtheorem{conj}[thm]{Conjecture} 

\theoremstyle{definition}

\newtheorem{remark}[thm]{Remark} 
\newtheorem*{notation}{Notation} 
\newtheorem*{note}{Note}

\newcommand{\C}{{\mathbb C}}
 
\newcommand{\p}{{\mathbb P}} 
\newcommand{\pii}{{{\mathbb P}^2}}
\newcommand{\piii}{{{\mathbb P}^3}}
\newcommand{\piv}{{{\mathbb P}^4}}
\newcommand{\pv}{{{\mathbb P}^5}}

\newcommand{\z}{{\mathbb Z}} 
\newcommand{\sca}{\mathscr{A}}
\newcommand{\sce}{\mathscr{E}}
\newcommand{\scf}{\mathscr{F}} 
\newcommand{\scg}{\mathscr{G}}
\newcommand{\sco}{\mathscr{O}} 
\newcommand{\sch}{\mathscr{H}}
\newcommand{\sci}{\mathscr{I}} 

\newcommand{\sck}{\mathscr{K}}
\newcommand{\scl}{\mathscr{L}}
\newcommand{\scq}{\mathscr{Q}}

\newcommand{\fm}{{\mathfrak m}}
\newcommand{\fp}{{\mathfrak p}}
 
\newcommand{\fS}{{\mathfrak S}}
\newcommand{\tH}{\text{H}} 
\newcommand{\h}{\text{h}}

\newcommand{\izo}{\overset{\sim}{\rightarrow}} 
\newcommand{\Izo}{\overset{\sim}{\longrightarrow}} 
\newcommand{\ra}{\rightarrow} 
\newcommand{\lra}{\longrightarrow} 
\newcommand{\xra}{\xrightarrow}
\newcommand{\vb}{\, \vert \, } 
\newcommand{\prim}{{\prime}} 
\newcommand{\secund}{{\prime \prime}}
\newcommand{\Ker}{\text{Ker}\, }
\newcommand{\Cok}{\text{Coker}\, }
\newcommand{\rk}{\text{rk}\, }
 
\newcommand{\e}{\varepsilon}

\begin{document} 

\title[Globally generated vector bundles]{Globally generated vector bundles 
       on $\p^n$ with $c_1 = 4$} 

\author[Anghel,~Coand\u{a}~and~Manolache]{Cristian~Anghel,~Iustin~Coand\u{a}~
and~Nicolae~Manolache}
\address{Institute of Mathematics of the Romanian Academy, P.O. Box 1--764, 
RO--014700, Bucharest, Romania}
\email{Iustin.Coanda@imar.ro~Cristian.Anghel@imar.ro~Nicolae.Manolache@imar.ro}



\begin{abstract}
We provide a complete classification of globally generated vector bundles with 
first Chern class $c_1 \leq 5$ one the projective plane and with $c_1 \leq 4$ 
on the projective $n$-space for $n \geq 3$. This reproves and extends, in a 
systematic manner,  previous results obtained for $c_1 \leq 2$ by Sierra and 
Ugaglia [J. Pure Appl. Algebra 213(2009), 2141--2146], and for $c_1 = 3$ by 
Anghel and Manolache [Math. Nachr. 286(2013), 1407--1423] and, independently, 
by Sierra and Ugaglia [J. Pure Appl. Algebra 218(2014), 174--180]. It turns 
out that the case $c_1 = 4$ is much more involved than the previous cases, 
especially on the projective 3-space. Among the bundles appearing in our 
classification one can find the Sasakura rank 3 vector bundle on the 
projective 4-space (conveniently twisted). We also propose a conjecture 
concerning the classification of globally generated vector bundles with 
$c_1 \leq n - 1$ on the projective $n$-space. We verify the conjecture for 
$n \leq 5$.  
\end{abstract}

\maketitle 
\tableofcontents

\section*{Introduction} 

In the last years there was some interest in the classification of globally 
generated vector bundles on projective spaces. These bundles are related to 
the embeddings of projective spaces into Grassmannians (see, for example, 
the classical papers of Tango \cite{t1}, \cite{t2}, \cite{t3}, or the 
recent papers of Sierra and Ugaglia \cite{su0} and of Huh \cite{h}) 
and to linear spaces of matrices of constant 
rank (see Manivel and Mezzetti \cite{mm}, and Fania and Mezzetti \cite{fm}). 
We are concerned with the problem of the classification of this kind 
of bundles for small values of their first Chern class. 

To be more precise, let $E$ be a globally generated rank $r$ vector bundle 
(= locally free sheaf) on $\p^n$ (= projective $n$-space over an algebraically 
closed field $k$ of characteristic 0), $n \geq 2$, with Chern classes 
$c_i(E) = c_i$, $i = 1, \ldots , n$. The degeneracy locus $D(\phi)$ of a 
general morphism $\phi : (r - i + 1)\sco_\p \ra E$ is a closed subscheme of 
$\p^n$, empty or of pure codimension $i$ and degree $c_i$. It follows that 
$c_i \geq 0$, $i = 1, \ldots , n$. 
In particular, for $i=2$, $r-1$ general global sections of $E$ define an exact 
sequence: 
\begin{equation}
\label{E:oeiy}
0 \lra (r - 1)\sco_\p \overset{\sigma}{\lra} E \lra \sci_Y(c_1) 
\lra 0 
\end{equation}
where $Y$ is a locally Cohen-Macaulay closed subscheme of $\p^n$, of pure 
codimension 2 (or empty), and with $\text{codim}(\text{Sing}\, Y,\p^n) 
\geq 6$ (if $\sigma$ is general then $Y = D(\sigma)$ is nonsigular outside the 
locus where $\sigma$ has rank $\leq r-3$). It follows that $Y$ is nonsingular 
for $n \leq 5$ and reduced and irreducible for $n \geq 4$. One has  
$c_2 = \deg Y \geq 0$ and $c_2 = 0$ iff $Y = \emptyset$ iff 
$E \simeq (r - 1)\sco_\p \oplus \sco_\p(c_1)$. 

Ellia \cite{e} determined the Chern classes $(c_1,c_2)$ of the globally 
generated rank 2 vector bundles on $\pii$, and Chiodera and Ellia 
\cite{ce} determined the Chern classes $(c_1,c_2)$ of the globally generated 
rank 2 vector bundles on $\p^n$ for $c_1 \leq 5$. On the other hand, 
Sierra and Ugaglia \cite{su} classified the globally generated vector bundles 
(of arbitrary rank) on $\p^n$ with $c_1 \leq 2$, while, recently, 
Anghel and Manolache \cite{am} and, independently, Sierra and Ugaglia  
\cite{su2}, did the same thing for $c_1 = 3$. Before stating their results, 
we recall some elementary but useful reductions, already used by Sierra and 
Ugaglia. First of all, when classifying globally generated vector bundles 
one can assume that they fulfil the condition $\tH^i(E^\ast) = 0$, $i = 0, 1$, 
because any globally generated vector bundle can be obtained from a globally 
generated bundle satisfying this additional condition by quotienting a trivial 
subbundle and then adding a trivial summand (see Lemma~\ref{L:h0h1}). 
Secondly, there is a simple construction allowing one to get a new 
globally generated vector bundle from an old one. We were able to trace it 
back, at least, to the paper of Brugui\`{e}res \cite{b}, where it is used, 
in Appendix A, to get a quick proof of Atiyah's classification of vector 
bundles on an elliptic curve. Using the notation from \cite{b}, if $E$ is a 
globally generated vector bundle as above then the dual $P(E)$ of the kernel 
of the evaluation morphism $\text{ev} : \tH^0(E)\otimes_k\sco_\p \ra E$ is a 
globally generated vector bundle. One has, by definition, an exact 
sequence$\, :$  
\[
0 \lra P(E)^\ast \lra \tH^0(E)\otimes_k\sco_\p 
\overset{\text{ev}}{\lra} E \lra 0\, .
\]
One deduces that $c_1(P(E)) = c_1$, $c_2(P(E)) = c_1^2 - c_2$ (and $c_3(P(E)) = 
c_3 + c_1(c_1^2 - 2c_2)$ if $n \geq 3$). Since, as we 
have noticed above, $c_2(P(E)) \geq 0$, one gets that $c_2 \leq c_1^2$. 

Now, the results of Sierra and Ugaglia and Anghel and Manolache are summarized 
in the following statement (see, also, Remark~\ref{R:c1leq3})$\, :$  

\begin{thm}\label{T:suam} 
Let $E$ be an indecomposable globally generated vector bundle on $\p^n$, 
$n \geq 2$, with $1 \leq c_1 \leq 3$, and such that ${\fam0 H}^i(E^\ast) = 0$, 
$i = 0, 1$. Then one of the following holds$\, :$ 
\begin{enumerate} 
\item[(i)] $E \simeq \sco_\p(a)$, $1 \leq a \leq 3\, ;$ 
\item[(ii)] $E \simeq P(\sco_\p(a))$, $1 \leq a \leq 3\, ;$ 
\item[(iii)] $n = 3$ and $E \simeq \Omega_\piii(2)\, ;$ 
\item[(iv)] $n = 4$ and $E \simeq \Omega_\piv(2)\, ;$ 
\item[(v)] $n = 4$ and $E \simeq \Omega_\piv^2(3)$.  
\end{enumerate}
\end{thm}

Notice that $P(\Omega_\p^i(i+1)) \simeq \Omega_\p^{i+1}(i+1)^\ast \simeq 
\Omega_\p^{n-i-1}(n-i)$, $i = 0, \ldots, 
n-1$. In particular, $P(\sco_\p(1)) \simeq \text{T}_\p(-1)$.

In the present paper we go one step further and prove the following$\, :$  

\begin{thm}\label{T:main} 
Let $E$ be an indecomposable globally generated vector bundle on $\p^n$, 
$n \geq 2$, with $c_1 = 4$, and such that ${\fam0 H}^i(E^\ast) = 0$, 
$i = 0, 1$. Then one of the following holds$\, :$ 
\begin{enumerate}
\item[(i)] $E \simeq \sco_\p(4)\, ;$ 
\item[(ii)] $n = 3$, $c_2 = 5$, and $E \simeq N(2)$, where $N$ is a 
nullcorrelation bundle$\, ;$ 
\item[(iii)] $n = 3$, $c_2 = 6$, and $E \simeq F(2)$, where $F$ is a 
$2$-instanton$\, ;$ 
\item[(iv)] $n = 3$, $c_2 = 6$, and, up to a linear change of coordinates, 
$E$ is the kernel of the epimorphism    
$(X_0,X_1,X_2,X_3^2) : 3\sco_\piii(2) \oplus \sco_\piii(1) 
\lra \sco_\piii(3)\, ;$
\item[(v)] $n = 3$, $c_2 = 7$, and $E \simeq F(2)$, where $F$ is a general 
$3$-instanton$\, ;$ 
\item[(vi)] $n = 3$, $c_2 = 7$, and, up to a linear change of coordinates, 
$E$ is the kernel of the epimorphism   
\[
\begin{pmatrix}
X_0 & X_1 & X_2 & X_3 & 0\\
0 & X_0 & X_1 & X_2 & X_3
\end{pmatrix}
: 5\sco_\piii(2) \lra 2\sco_\piii(3)\, ;
\] 
\item[(vii)] $n = 3$, $c_2 = 7$, and, up to a linear change of coordinates, 
$E$ is the kernel of the epimorphism   
$(X_0,X_1,X_2^2,X_2X_3,X_3^2) : 2\sco_\piii(2) \oplus 
3\sco_\piii(1) \lra \sco_\piii(3)\, ;$
\item[(viii)] $n = 5$, $c_2 = 7$, and $E \simeq \Omega_\pv(2)\, ;$ 
\item[(ix)] $E \simeq P(E^\prim)$, where $E^\prim$ is one of the above 
bundles$\, ;$ 
\item[(x)] $n = 3$, $c_2 = 8$, and $E \simeq F(2)$, where $F$ is a general 
$4$-instanton$\, ;$  
\item[(xi)] $n = 3$, $c_2 = 8$, and one has an exact sequence$\, :$  
\[
0 \lra E(-2) \lra 4\sco_\piii \overset{\phi}{\lra}  
\sco_\piii(2) \lra 0 
\]
with $\phi$ an arbitrary epimorphism$\, ;$
\item[(xii)] $n = 3$, $c_2 = 8$, and one has an exact sequence$\,:$ 
\[
0 \lra E(-2) \lra 3\sco_\piii(1) \oplus 3\sco_\piii  
\overset{\phi}{\lra} \Omega_\piii(3) \lra 0 
\]
with $\phi$ a general epimorphism$\, ;$
\item[(xiii)] $n = 3$, $c_2 = 8$, and one has an exact sequence$\, :$ 
\[
0 \lra E(-2) \lra 4\sco_\piii \oplus 2\sco_\piii(-1)  
\overset{\phi}{\lra} 2\sco_\piii(1) \lra 0 
\] 
with $\phi$ a general epimorphism$\, ;$ 
\item[(xiv)] $n = 3$, $c_2 = 8$, and one has an exact sequence$\, :$ 
\[
0 \lra E(-2) \lra \sco_\piii \oplus 5\sco_\piii(-1)  
\overset{\phi}{\lra} \sco_\piii(1) \lra 0
\]
with $\phi$ a general epimorphism$\, ;$
\item[(xv)] $n = 3$, $c_2 = 8$, and, up to a linear change of coordinates, 
$E$ is the cohomology of the monad$\, :$ 
\[
\sco_\piii(-1) \xra{\binom{s}{u}}  
2\sco_\piii(2) \oplus 2\sco_\piii(1) \oplus 4\sco_\piii \xra{(p\, ,\, 0)} 
\sco_\piii(3) 
\]
where $\sco_\piii(-1) \overset{s}{\lra} 2\sco_\piii(2) \oplus 
2\sco_\piii(1) \overset{p}{\lra} \sco_\piii(3)$ is a subcomplex of the 
Koszul complex associated to $X_0, X_1, X_2^2, X_3^2$ and $u : \sco_\piii(-1) 
\ra 4\sco_\piii$ is defined by $X_0, X_1, X_2, X_3\, ;$ 
\item[(xvi)] $n = 4$, $c_2 = 8$, and, up to a linear change of coordinates, 
denoting by $(C_p,\, \delta_p)_{p \geq 0}$ the Koszul complex associated to the 
epimorphism $\delta_1 : 4\sco_\piv(-1) \oplus \sco_\piv(-2) \ra \sco_\piv$ 
defined by $X_0, \ldots ,X_3,X_4^2$, one has exact sequences$\, :$ 
\begin{gather*}
0 \lra \sco_\piv(-2) \xra{\delta_5(4)}
\sco_\piv \oplus 4\sco_\piv(-1) \xra{\delta_4(4)} 
4\sco_\piv(1) \oplus 6\sco_\piv \lra E^\prim \lra 0\, ,\\
0 \lra E \lra E^\prim \overset{\phi}{\lra} \sco_\piv(2) \lra 0\, , 
\end{gather*}  
where $\phi : E^\prim \ra \sco_\piv(2)$ is any morphism with the property that 
${\fam0 H}^0(\phi(-1)) : {\fam0 H}^0(E^\prim(-1)) \ra {\fam0 H}^0(\sco_\piv(1))$ 
is injective $($such morphisms exist and are automatically epimorphisms$)$. 
\end{enumerate}  
\end{thm}  

Let us make some comments about the above statement. The meaning of the term 
``general'' in (v) is explained in Remark~\ref{R:general}, while the meaning 
of this term in (x), (xii), (xiii), and (xiv) is explained in 
Remark~\ref{R:generalc2=8}. It means, in particular, that not all the 
bundles belonging to the respective families are globally generated. For 
example, in case (xii), the degeneracy locus of the component  
$3\sco_\piii(1) \ra \Omega_\piii(3)$ of a ``generic'' epimorphism 
$\phi : 3\sco_\piii(1) \oplus 3\sco_\piii \ra \Omega_\piii(3)$ is a 
smooth quadric surface $Q \subset \piii$ and the cokernel of that component 
is isomorphic to $\sco_Q(1,3)$. The composite epimorphism $3\sco_\piii  
\ra \Omega_\piii(3) \ra \sco_Q(1,3)$ is defined by a base point free, 
3-dimensional vector subspace $\Lambda \subset \tH^0(\sco_Q(1,3))$. It turns 
out that $(\Ker \phi)(2)$ is globally generated if and only if the linear 
system $\vert \, \Lambda \, \vert$ contains no divisor whose components are 
lines contained in $Q$. 

\vskip2mm 

In case (xiii), the cokernel of the component $4\sco_\piii \ra 
2\sco_\piii(1)$ of a ``generic'' epimorphism $\phi : 4\sco_\piii  
\oplus 2\sco_\piii(-1) \ra 2\sco_\piii(1)$ is of the form 
$\sco_Z(1)$, where $Z$ consists of 4 simple points not contained in a plane 
(see Remark~\ref{R:2x4} from Appendix~\ref{A:case4}). Choose a linear 
form $\lambda \in \tH^0(\sco_\piii(1))$ vanishing at none of the points of $Z$. 
The composite epimorphism $2\sco_\piii(-1) \ra 2\sco_\piii(1) \ra 
\sco_Z(1)$ can be written as a composite morphism $2\sco_\piii(-1)  
\ra \sco_Z \overset{\lambda}{\lra} \sco_Z(1)$. Since $\tH^0(\sco_\piii(1)) 
\izo \tH^0(\sco_Z(1))$, the morphism $2\sco_\piii(-1) \ra \sco_Z$ 
factorizes uniquely as $2\sco_\piii(-1) \ra \sco_\piii \ra \sco_Z$. The 
image of the morphism $2\sco_\piii(-1) \ra \sco_\piii$ is the ideal sheaf 
of a line $L \subset \piii$. It turns out that $(\Ker \phi)(2)$ is 
globally generated if and only if the line $L$ intersects none of the edges 
of the tetrahedron with vertices the points of $Z$. 

\vskip2mm 

In case (xiv), the cokernel of the component $\sco_\piii \ra \sco_\piii(1)$ of a 
``generic'' epimorphism $\phi : \sco_\piii \oplus 5\sco_\piii(-1) \ra 
\sco_\piii(1)$ is of the form $\sco_H(1)$, for some plane $H \subset \piii$.    
The composite epimorphism $5\sco_\piii(-1) \ra \sco_\piii(1) \ra \sco_H(1)$ 
is defined by a base point free, 5-dimensional vector subspace 
$V \subset \tH^0(\sco_H(2))$. If $M_V$ is the syzygy bundle associated to $V$ 
(on $H$) then it turns out that $(\Ker \phi)(2)$ is globally generated if and 
only if $M_V$ is stable. For a discussion about the stability of $M_V$, see 
\cite[Example~2.1]{c2}. The bundles $M_V$ were originally considered, in a 
different manner, by Elencwajg~\cite{el}.  

In the cases (xii) and (xiii) there are also some ``degenerate'' 
configurations  which are quite difficult to analyse. We defered this analysis 
to the Appendices~\ref{A:case3} and \ref{A:case4}. 

\vskip2mm 

As a final comment concerning the statement of Theorem~\ref{T:main}$\, :$ the 
bundle $E$ occuring in (xvi) coincides with ${\mathcal G}(1)$, where 
$\mathcal G$ is the bundle occuring in the statement of the main result of 
the paper of Abo, Decker and Sasakura \cite{ads}. The difference of 
description comes from the fact that Abo et al. use the Beilinson monad of 
the bundle while we use its Horrocks monad (due to our method of proof). We 
investigate this bundle in the proof of Prop.~\ref{P:c1=4c2=8n4}.    

\vskip2mm 

As for the proof of Theorem~\ref{T:main}, we classify firstly the globally 
generated vector bundles with $c_1 = 4$ on $\pii$, which is easy (we are even   
able to classify the globally generated vector bundles with $c_1 = 5$ on 
$\pii$, see the second part of Section~\ref{S:n=2c1=4}).  
Then we classify them on $\piii$ (which is quite lengthy and involved), and 
then we try to decide which of the globally generated vector bundles we found 
on $\piii$ extend to higher dimensional vector spaces. We use, as a basic tool, 
Horrocks' method of ``killing cohomology''. We also use information about 
2-codimensional subvarieties of small degree of projective spaces (via the 
exact sequence \eqref{E:oeiy}), and the properties of the \emph{spectrum} 
of stable rank 2 reflexive sheaves on $\piii$, via the following observation:  
if $E$ is a globally generated vector bundle on $\piii$ of rank $r$ then 
$r-2$ general global sections of $E$ define an exact sequence:   
\[
0 \lra (r - 2)\sco_\piii \lra E \lra \sce^\prim \lra 0   
\]
with $\sce^\prim$ a rank 2 reflexive sheaf. Finally, our way of investigating 
the extensions of a globally generated vector bundle on $\piii$ to higher 
dimensional projective spaces is based on the results 
\eqref{L:h1=0}--\eqref{R:lift2} from Section~\ref{S:preliminaries}.   

Looking at the statements of the above theorems one sees that, for fixed 
$c_1$, there are fewer and fewer globally generated vector bundles on 
$\p^n$ while $n$ increases. This suggests that it would be more natural to 
look for a classification of globally generated vector bundles on $\p^n$ 
not for $c_1$ uniformly bounded but for $c_1$ bounded in terms of $n$. 
Therefore, we propose the following$\, :$  

\begin{conj}\label{C:c1leqn-1}
Let $E$ be a globally generated vector bundle on $\p^n$, $n \geq 2$, 
with $1 \leq c_1 \leq n-1$, such that ${\fam0 H}^i(E^\ast) = 0$, $i = 0, 1$. 
Then one of the following holds$\, :$
\begin{enumerate}
\item[(i)] $E \simeq A \oplus P(B)$, where $A$ and $B$ are direct sums of 
line bundles on $\p^n$, with $c_1(A) + c_1(B) \leq n-1$, and such that 
${\fam0 H}^0(A^\ast) = 0$, ${\fam0 H}^0(B^\ast) = 0\, ;$
\item[(ii)] $n \geq 3$ and $E \simeq \Omega_\p(2)\, ;$
\item[(iii)] $n \geq 4$ and $E \simeq \Omega_\p^{n-2}(n-1)$. 
\end{enumerate}
\end{conj} 

\noindent 
The conjecture says, in particular, that the only globally generated vector 
bundles of rank $r$, $2 \leq r < n$, on $\p^n$ with $c_1 \leq n-1$ 
(which are not direct sums of line bundles) are the nullcorrelation bundles 
for $n$ odd and the bundles discovered, independently, by Trautmann-Vetter 
\cite{tr}, \cite{v}, and by Tango \cite{t2}.  The above results show that the 
conjecture is true for $n \leq 5$. 

\begin{notation}
(i) If $X$ is a closed subscheme of $\p^n$, we denote by 
$\sci_X \subset \sco_\p$ its ideal sheaf. If $Y$ is a closed subscheme of 
$X$, we denote by $\sci_{Y,X} \subset \sco_X$ the ideal sheaf defining $Y$ as 
a closed subscheme of $X$. In other words, $\sci_{Y,X} = \sci_Y/\sci_X$. If 
$\scf$ is a coherent sheaf on $\p^n$, we put $\scf_X := \scf 
\otimes_{\sco_\p} \sco_X$ and $\scf \vb X := i^\ast\scf$, where $i : X \ra \p^n$ 
is the inclusion. 

(ii) By a point of a quasi-projective scheme $X$ we always mean a 
\emph{closed point}. If $\scf$ is a coherent sheaf on $X$ and $x \in X$, we 
denote by $\scf(x)$ the \emph{reduced stalk} $\scf_x/\fm_x\scf_x$ of $\scf$ 
at $x$, where $\fm_x$ is the maximal ideal of $\sco_{X,x}$. Notice that 
$\scf \otimes_{\sco_X} \sco_{\{x\}} \simeq \scf(x) \otimes_k \sco_{\{x\}}$.  

(iii) We denote by $S = k[X_0, \ldots ,X_n]$ the projective coordinate ring 
of $\p^n$. If $\scf$ is a coherent sheaf on $\p^n$, we denote by 
$\tH^i_\ast(\scf)$ the graded $S$-module $\bigoplus_{l \in \z}\tH^i(\scf(l))$. 

(iv) If $V$ is a $k$-vector space, we denote by $\p(V)$ the classical 
projective space parametrizing the 1-dimensional vector subspaces of $V$. 

(v) We shall sometimes write ``g.~g.'' instead of ``globally generated''. 
\end{notation}

\vskip2mm 

\noindent 
{\bf Acknowledgements.}\quad 
N. Manolache expresses his thanks to the Institute of Mathematics, Oldenburg 
University, especially to Udo Vetter, for warm hospitality during the 
preparation of this paper.

\section{Preliminaries}\label{S:preliminaries}

We gather, in this section, a number of results that we need in the sequel. 
Most of them are (more or less) well known but we state and (for the benefit 
of the reader) prove them in a form that is appropriate for our purposes. 
First of all, we relate globally generated vector bundles on $\pii$ to 
group of points in the plane satisfying the Cayley-Bacharach condition. Then, 
following Sierra and Ugaglia \cite{su}, we show that, when classifying g.~g. 
vector bundles $E$ on $\p^n$, one can assume that $\tH^i(E^\ast) = 0$, 
$i = 0,\, 1$. (The drawback of this kind of reduction is that, when one 
encounters a concrete vector bundle not satisfying this condition, it might 
not be that easy to decide whether it is globally generated or not just by 
looking at our list. However, without it, a classification would be almost 
impossible). 

After that, we prove some properties of the bundle $P(E)$ defined, in the 
Introduction, as the dual of the kernel of the evaluation morphism of a g.~g.  
vector bundle $E$. An amusing application is that a g.~g. rank 2 vector bundle 
on $\piii$ satisfies $c_2 \leq c_1^2/2$ (compare with Chiodera and Ellia 
\cite[Lemma~2.5]{ce}). Actually, the functor $E \mapsto P(E)$ allows one to 
reduce the classification of g.~g. vector bundles on $\p^n$ (of any rank) to 
the case where $c_2 \leq c_1^2/2$. 

The most substantial part of the section consists of the results 
(\ref{L:h1=0})--(\ref{R:lift2}) investigating the possibility of extending 
g.~g. vector bundles on $\p^n$, $n \geq 3$, to higher dimensional projective 
spaces. Lemma~\ref{L:lift1} and Lemma~\ref{L:lift2} can be considered as the 
main results of this section. 

Finally, we recall some facts concerning the notion of Castelnuovo-Mumford 
regularity and the Theorem of Beilinson.    

\vskip2mm

We use, throughout this section, the notation and conventions stated in the 
Introduction. Dualizing the exact sequence  
\[
0 \lra (r-1)\sco_\p \lra E \lra \sci_Y(c_1) \lra 0
\]
from the Introduction, one gets an exact sequence$\, :$ 
\begin{equation}
\label{E:dual}
0 \lra \sco_\p(-c_1) \lra E^\ast \lra (r - 1)\sco_\p  
\overset{\delta}{\lra} \omega_Y(n+1-c_1) \lra 0
\end{equation}   
where $\omega_Y$ is the dualizing sheaf of $Y$. One deduces that 
$\omega_Y(n+1-c_1)$ is generated by $r-1$ global sections. 

\begin{remark}\label{R:CayleyB} 
In the case $n=2$, where $Y$ consists of $c_2$ simple points, 
the fact that $\omega_Y(3-c_1)$ is globally generated provides no useful 
information. However, if $n=2$ and $r=2$ then, applying 
$\sch om_{\sco_\pii}(-,\omega_\pii)$ to the locally free resolution$\, :$ 
\[
0 \lra \sco_\pii(-c_1) \lra E(-c_1) \lra \sco_\pii \lra \sco_Y \lra 0\, , 
\]  
one gets an exact sequence$\, :$
\[
0 \lra \omega_\pii \lra E^\ast(c_1-3) \lra \sco_\pii(c_1-3) \lra 
\omega_Y \lra 0 
\]
which decomposes into two short exact sequences$\, :$ 
\begin{gather*}
0 \lra \sci_Y(c_1-3) \lra \sco_\pii(c_1-3) \lra \omega_Y \lra 0\, ,\\
0 \lra \omega_\pii \lra E^\ast(c_1-3) \lra \sci_Y(c_1-3) \lra 0\, .
\end{gather*} 
If $\scf$ is a coherent $\sco_Y$-module then, for any locally free sheaf 
$\scl$ on $\pii$, $\text{Ext}_{\sco_\pii}^i(\scf, \scl) = 0$ for 
$i < 2$ (this follows from standard results that can be found in 
Hartshorne's book~\cite{hag}: see the last line of the proof of 
\cite[III,~Lemma~7.4]{hag} and \cite[III,~Prop.~6.7]{hag}). Moreover, 
\cite[III,~Prop.~6.7]{hag} implies that if $\phi : \scl \ra \scl^\prim$ is a 
morphism of locally free sheaves on $\pii$ such that $\phi \vb Y = 0$ then 
the morphism $\text{Ext}_{\sco_\pii}^2(\scf, \scl) \ra 
\text{Ext}_{\sco_\pii}^2(\scf, \scl^\prim)$ is 0. One deduces that the 
composite morphism$\, :$ 
\[
\text{Hom}_{\sco_Y}(\scf, \omega_Y) \lra 
\text{Ext}_{\sco_\pii}^1(\scf, \sci_Y(c_1-3)) \lra 
\text{Ext}_{\sco_\pii}^2(\scf, \omega_\pii) 
\] 
is an isomorphism. This implies easily that the pair: 
\[
(\omega_Y,\,  \tH^0(\omega_Y) \ra \tH^1(\sci_Y(c_1-3)) \ra \tH^2(\omega_\pii) 
\overset{\text{tr}}{\lra} k)
\]
is a \emph{dualizing sheaf} for $Y$. It follows that the composite 
morphism$\, :$ 
\[
\tH^0(\sco_\pii(c_1 - 3)) \lra \tH^0(\omega_Y) 
\overset{\text{tr}}{\lra} k
\] 
is 0, condition which is equivalent to the fact that the points of $Y$ 
satisfy the \emph{Cayley-Bacharach condition} with respect to plane curves 
of degree $c_1 - 3$, which means that any plane curve of degree $c_1 - 3$ 
containing $c_2 - 1$ of the points of $Y$ must contain $Y$. This result goes 
back to Schwarzenberger \cite{sw} who used a geometric argument based on 
blowing-up $\pii$ at the points of $Y$.    
\end{remark} 

\vskip2mm 

We recall, now, some observations that appear in the papers of Sierra and 
Ugaglia. 

\begin{lemma}\label{L:h0h1} 
Let $E$ be a globally generated vector bundle on $\p^n$. 
Then there exists a globally generated vector bundle $F$ on $\p^n$ with 
${\fam0 H}^i(F^\ast) = 0$, $i = 0,1$, and a quotient $G$ of $F$ by a trivial 
subbundle$\, :$  
\[
0 \lra s\sco_\p \lra F \lra G \lra 0
\] 
such that $E \simeq G\oplus t\sco_\p$. Moreover, $F$ and $G$ are 
uniquely determined up to an isomorphism, $t = {\fam0 h}^0(E^\ast)$, and 
$s = {\fam0 h}^1(E^\ast)$. 
\end{lemma} 

\begin{proof}
Let $\e : \tH^0(E^\ast)\otimes_k\sco_\p \ra E^\ast$ be the evaluation 
morphism of $E^\ast$. Consider the composite morphism$\, :$ 
\[
\tH^0(E^\ast)\otimes_k\sco_\p \overset{\e}{\lra} E^\ast 
\overset{\text{ev}^\ast}{\lra} \tH^0(E)^\ast \otimes_k\sco_\p \, .
\]
Since $\tH^0(\e)$ is an isomorphism and $\text{ev}^\ast$ is a 
monomorphism it follows that $\tH^0(\text{ev}^\ast \circ \e)$ is 
injective. This implies that $\text{ev}^\ast \circ \e$ maps 
$\tH^0(E^\ast)\otimes_k\sco_\p$ isomorphically onto a trivial subbundle of 
$\tH^0(E)^\ast\otimes_k\sco_\p$, hence $\e$ has a left inverse. 
Let $E_0 := \Cok \e$. It follows that the sequence$\, :$ 
\[
0 \lra \tH^0(E^\ast)\otimes_k\sco_\p \overset{\e}{\lra} E^\ast 
\lra E_0 \lra 0
\]
is split exact. Moreover, $\tH^0(E_0) = 0$ and $\tH^1(E^\ast) \izo \tH^1(E_0)$.  
Consider, now, the canonical extension$\, :$ 
\[
0 \lra E_0 \lra E_1 \lra \tH^1(E^\ast)\otimes_k\sco_\p \lra 0\, .
\]
One has $\tH^i(E_1) = 0$, $i = 0,\, 1$ and one can take $G := E_0^\ast$ and 
$F := E_1^\ast$. 
\end{proof} 

\begin{note} 
The conclusion of Lemma~\ref{L:h0h1} remains valid, with the same proof, 
under the weaker hypothesis that $E$ is generically globally generated, i.e., 
that the evaluation morphism of $E$ is an epimorphism over a dense open 
subset of $\p^n$.  
\end{note}

\begin{lemma}\label{L:h1e=h1f}
Under the hypothesis and with the notation from Lemma~\ref{L:h0h1}, one has 
$c_i(E) = c_i(F)$, $i = 1, \ldots ,n$, ${\fam0 H}^i_\ast(E) \simeq 
{\fam0 H}^i_\ast(F)$, $1 \leq i \leq n - 2$, and ${\fam0 H}^{n-1}(E(l)) \simeq 
{\fam0 H}^{n-1}(F(l))$ for $l \geq -n$ $($hence ${\fam0 H}^1(E^\ast(l)) \simeq 
{\fam0 H}^1(F^\ast(l))$ for $l \leq -1$$)$.   
\qed 
\end{lemma} 

\begin{lemma}\label{L:osfg} 
Let $F$ be a globally generated vector bundle on $\p^n$ and 
$s \leq {\fam0 rk}\, F$ a positive integer. Then there exists an 
exact sequence$\, :$  
\[
0 \lra s\sco_\p \lra F \lra G \lra 0
\]
with $G$ locally free iff $c_i(F) = 0$ for $i = {\fam0 rk}\, F - s + 1$. 
\end{lemma}

\begin{proof}
Let $i = \text{rk}\, F - s + 1$. Since $F$ is globally generated, 
the degeneracy locus $D(\phi)$ of a general morphism $\phi : s\sco_\p  
\ra F$ is a closed subscheme of $\p^n$, empty or of pure codimension $i$ and  
degree $c_i(F)$. It follows that if $c_i(F) = 0$ then $D(\phi)$ is 
empty, that is, $\phi$ is a locally split monomorphism.    
\end{proof}

\begin{cor}\label{C:osfg}
Let $E$ be a globally generated vector bundle of rank $r$ on $\p^n$, 
$n \geq 4$, and let $H \subset \p^n$ be a hyperplane. 
Consider an integer $r^\prim \leq r$ such that $2 \leq r^\prim < n-1$. 

\emph{(a)} If $c_{r^\prim +1}(E_H) = 0$ then there exists a globally generated 
vector bundle $E^\prim$ of rank $r^\prim$ on $\p^n$ such that $E$ and $E^\prim$ 
have the same Chern classes. 

\emph{(b)} If, moreover, $E_H \simeq G \oplus (r - r^\prim)\sco_H$ with $G$ a 
vector bundle of rank $r^\prim$ on $H$ and if $c_{r^\prim}(G) \neq 0$  
then there exists a globally generated vector bundle $E^\prim$ on $\p^n$ such 
that $E_H^\prim \simeq G$.    
\end{cor}

\begin{proof}
(a) 
Since $r^\prim +1 \leq n-1$, one has $c_{r^\prim +1}(E) = c_{r^\prim +1}(E_H) = 0$. 
Lemma~\ref{L:osfg} implies that there exists an exact sequence$\, :$ 
\[
0 \lra (r - r^\prim)\sco_\p \lra E \lra E^\prim \lra 0
\] 
with $E^\prim$ locally free of rank $r^\prim$. 

(b)
Restricting to $H$ the exact sequence from the proof of (a) one gets an exact 
sequence$\, :$ 
\[
0 \lra (r - r^\prim)\sco_H \overset{u}{\lra} G \oplus 
(r - r^\prim)\sco_H \lra E_H^\prim \lra 0
\]
Since $c_{r^\prim}(G) \neq 0$ there exists no locally split monomorphism 
$\sco_H \ra G$, hence the component $(r - r^\prim)\sco_H \ra 
(r - r^\prim)\sco_H$ of $u$ must be an isomorphism. It follows that 
$E_H^\prim \simeq G$.  
\end{proof}

One can get results about globally generated vector bundles on $\p^n$ by 
using the exact sequences \eqref{E:oeiy} and \eqref{E:dual} and information 
about two codimensional subvarieties of $\p^n$. The previous observations are 
related to this approach via the following two results$\, :$  

\begin{lemma}\label{L:h0h1delta} 
If $n \geq 3$ then ${\fam0 H}^i(E^\ast) = 0$, $i = 0,1$, iff the morphism 
$\delta$ from the exact sequence \eqref{E:dual} can be identified with the 
evaluation morphism$\, :$  
\[
{\fam0 H}^0(\omega_Y(n+1-c_1))\otimes_k\sco_\p \lra \omega_Y(n+1-c_1)\, .  
\qed 
\]  
\end{lemma} 

\begin{lemma}\label{L:uniqueext} 
Let $Y$ be a locally Cohen-Macaulay closed subscheme of $\p^n$ of pure 
codimension $2$. Assume that there are exact sequences$\, :$  
\begin{gather*}
0 \lra (r - 1)\sco_\p \lra E \lra \sci_Y(c_1) \lra 0\, ,\\
0 \lra (s - 1)\sco_\p \lra F \lra \sci_Y(c_1) \lra 0\, ,
\end{gather*} 
with $E$ and $F$ locally free sheaves. If ${\fam0 H}^i(E^\ast)=0$ and 
${\fam0 H}^i(F^\ast) = 0$, $i = 0,1$, then $r=s$ and $E \simeq F$. 
\end{lemma}

\begin{proof}
Since $\text{Ext}^1(E,(s - 1)\sco_\p) \simeq (s - 1)\tH^1(E^\ast) = 0$, 
one deduces the existence of a commutative diagram$\, :$ 
\[
\begin{CD}
0 @>>> (r - 1)\sco_\p @>>> E @>>> \sci_Y(c_1) @>>> 0\\
@. @VVV @VVV @\vert\\
0 @>>> (s - 1)\sco_\p @>>> F @>>> \sci_Y(c_1) @>>> 0
\end{CD}
\]
from which one deduces the existence of an exact sequence$\, :$  
\[
0 \lra (r - 1)\sco_\p \lra E \oplus (s - 1)\sco_\p \lra F \lra 0\, .
\]
Since $\tH^1(F^\ast) = 0$, this exact sequence is split exact, hence$\, :$  
\[
E \oplus (s - 1)\sco_\p \simeq F \oplus (r - 1)\sco_\p \, . 
\] 
But $E$ and $F$ have no trivial direct summand because $\tH^0(E^\ast) = 0$ and 
$\tH^0(F^\ast) = 0$. One deduces that $s = r$ and $E\simeq F$. 
\end{proof}

Next, we state some properties of the bundle $P(E)$ associated, as in the 
Introduction, to a globally generated vector bundle $E$ on $\p^n$. 

\begin{lemma}\label{L:pe}
\emph{(a)} If $E$ is a globally generated vector bundle on $\p^n$ then   
${\fam0 H}^i(P(E)^\ast) = 0$, $i = 0,1$. Moreover, if ${\fam0 H}^i(E^\ast) = 0$, 
$i = 0,1$, then $E \simeq P(P(E))$. 

\emph{(b)} With the notation from the statement of Lemma~\ref{L:h0h1}, 
one has $P(E) \simeq P(F)$.    
\qed 
\end{lemma} 

\begin{remark}\label{R:reduction}
The above lemmata allow one, when studying globally generated vector bundles 
$E$ on $\p^n$, to assume that, moreover, $\tH^i(E^\ast) = 0$, $i=0,1$, and that 
$c_2 \leq c_1^2/2$ (because $c_1(P(E)) = c_1$ and $c_2(P(E)) = c_1^2 - c_2$). 
Actually, this inequality is automatically satisfied by the  
Chern classes of any globally generated rank 2 vector bundle on $\p^n$, 
$n \geq 3$, as the next lemma shows.  
\end{remark}

\begin{lemma}\label{L:eofrk2} 
If $E$ is a globally generated rank $2$ vector bundle on $\piii$ then 
$c_2 \leq c_1^2/2$. 
\end{lemma}

\begin{proof}
One can assume that $c_1 > 0$ (if $c_1 = 0$ then $E \simeq 2\sco_\piii$). 
One has $c_3(P(E)) = c_1(c_1^2 - 2c_2)$ because $c_3 = 0$. 
But $P(E)$ is globally generated, hence $c_3(P(E)) \geq 0$. 
\end{proof} 

\begin{lemma}\label{L:peh} 
Let $E$ be a globally generated vector bundle on $\p^n$ and let 
$\Pi \subset \p^n$ be a linear subspace. Then $P(E) \vb \Pi \simeq G \oplus 
t\sco_\Pi$ with $t = \dim {\fam0 Ker}({\fam0 H}^0(E) \ra 
{\fam0 H}^0(E_\Pi))$ and with $G$ defined by an exact sequence$\, :$ 
\[
0 \lra s\sco_\Pi \lra P(E_\Pi) \lra G \lra 0
\]
where $s = \dim {\fam0 Coker}({\fam0 H}^0(E) \ra {\fam0 H}^0(E_\Pi))$. 
\qed
\end{lemma}

\begin{lemma}\label{L:ggquasib}
Let $E$ be a globally generated vector bundle on $\p^n$ with 
${\fam0 H}^0(E^\ast) = 0$ and $x$ a point of $\p^n$. Consider an epimorphism 
$\e : E \ra \sco_{\{x\}}$. ${\fam0 H}^0(\e) : {\fam0 H}^0(E) \ra 
{\fam0 H}^0(\sco_{\{x\}}) \simeq k$ is a non-zero element of 
${\fam0 H}^0(E)^\ast$ 
which defines a non-zero global section $\sigma$ of $P(E)$. Then$\, :$ 

\emph{(a)} $\sigma$ vanishes at $x$ and if $Z \subset \p^n$ is the zero-scheme  
of $\sigma$ then $\e$ factorizes as$\, :$ 
\[
E \overset{\e^\prim}{\lra} \sco_Z \xra{{\fam0 can}} \sco_{\{x\}} \, .
\] 

\emph{(b)} ${\fam0 Ker}\, \e$ is globally generated if and only if $Z$ 
consists of the simple point $x$. 
\end{lemma} 

\begin{proof}
Consider the commutative diagram$\, :$ 
\[
\begin{CD}
0 @>>> P(E)^\ast @>{u}>> \tH^0(E)\otimes \sco_\p @>{\text{ev}_E}>> E @>>> 0\\
@. @VVV @VV{\tH^0(\e)\otimes \text{id}_\sco}V @VV{\e}V\\
0 @>>> \sci_{\{x\}} @>>> \sco_\p @>>> \sco_{\{x\}} @>>> 0 
\end{CD}
\] 

(a) Considering $\sigma$ as a morphism $\sco_\p \ra P(E)$, one has 
$(\tH^0(\e)\otimes{\text{id}}_\sco) \circ u = \sigma^\ast$. Since, by 
definition, $\text{Im}\, \sigma^\ast = \sci_Z$ it follows that $x \in Z$ and 
that $\e$ factorizes as in the statement. 

(b) The morphism $\Ker \tH^0(\e) \otimes_k \sco_\p \ra \Ker \e$ induced by 
${\text{ev}}_E$ can be identified with the evaluation morphism of 
$\Ker \e$. One applies, now, the Snake Lemma to the above diagram.  
\end{proof}

\begin{cor}\label{C:ggquasib}
Under the hypothesis of Lemma~\ref{L:ggquasib}, assume that 
${\fam0 H}^i(E^\ast) = 0$, $i = 0,\, 1$.  
Then there is a bijection between the set of epimorphisms $\e : E \ra 
\sco_{\{x\}}$ with ${\fam0 Ker}\, \e$ globally generated and the set of global 
sections of $P(E)$ whose zero scheme consists of the simple point $x$. 
\end{cor} 

\begin{proof}
The hypothesis $\tH^i(E^\ast) = 0$, $i = 0,\, 1$, implies that any morphism 
$P(E)^\ast \ra \sco_\p$ can be uniquely extended to a morphism 
$\tH^0(E)\otimes_k \sco_\p \ra \sco_\p$. 
\end{proof}

Another useful technique that can be used in the study of globally generated 
vector bundles on projective spaces is to classify them on low dimensional 
projective spaces and then to try to decide whether they extend, as globally 
generated vector bundles, to higher dimensional projective spaces or not. An 
elementary result one can use for this purpose is the following slight 
generalization of the well known fact asserting that if the restriction of a 
vector bundle on $\p^n$, $n \geq 3$, to a hyperplane splits (as a direct sum 
of line bundles) then the bundle splits, too. 

\begin{lemma}\label{L:h1=0} 
Let $S = k[X_0,\ldots ,X_n]$ be the projective coordinate ring of $\p^n$, 
$n \geq 3$, let 
$H \subset \p^n$ be a hyperplane of equation $h = 0$, $h\in S_1$, and let 
$E$ be a vector bundle on $\p^n$. Consider a minimal graded free resolution 
of the graded $S$-module ${\fam0 H}^0_\ast(E) := 
\bigoplus_{i\in \z}{\fam0 H}^0(E(i))\, :$ 
\[
0 \lra L_{n-1} \lra \cdots \lra L_0 \lra {\fam0 H}^0_\ast(E) \lra 0\, .
\]
Then$\, :$ \emph{(a)} ${\fam0 H}^1_\ast(E) = 0$ iff $L_{n-1} = 0\, ;$   

\emph{(b)} If ${\fam0 H}^1_\ast(E) = 0$ then $L_\bullet /hL_\bullet$ is a 
minimal free resolution of the graded $S/hS$-module ${\fam0 H}^0_\ast(E_H)$. 
\end{lemma} 

\begin{proof}
(a) follows from \cite[Theorem~A4.1~and~Theorem~A4.2]{eca}. A direct proof of 
these results in the case of graded modules over a polynomial ring can be 
found in \cite[Theorem~1.1]{c}.  

(b) Using the exact sequences$\, :$ 
\[
0 \lra \tH^0(E(i-1)) \overset{h}{\lra} \tH^0(E(i)) \lra 
\tH^0(E_H(i)) \lra \tH^1(E(i-1)) = 0
\]
one deduces that $\tH^0_\ast(E_H) \simeq \tH^0_\ast(E)/h\tH^0_\ast(E)$. Since 
$h$ is $\tH^0_\ast(E)$-regular it follows that $\text{Tor}_i^S(\tH^0_\ast(E), 
S/hS) = 0$, $\forall \, i > 0$, hence, by applying $-\otimes_SS/hS$ to the 
exact sequence $L_\bullet \ra \tH^0_\ast(E) \ra 0$ one gets an exact sequence 
$L_\bullet/hL_\bullet \ra \tH^0_\ast(E_H) \ra 0$.  
\end{proof} 

\begin{remark}\label{R:reslifting}
Under the hypothesis of Lemma~\ref{L:h1=0}, assume that $\tH^1_\ast(E) = 0$. 
Let 
\[
0 \lra K_{n-2} \lra \cdots \lra K_0 \lra \tH^0_\ast(E_H) \lra 0
\]
be a graded free resolution of the graded $S/hS$-module $\tH^0_\ast(E_H)$. 
Then there exists a graded free resolution$\, :$ 
\[
0 \lra L_{n-2} \lra \cdots \lra L_0 \lra \tH^0_\ast(E) \lra 0
\]
of the graded $S$-module $\tH^0_\ast(E)$ such that $L_\bullet/hL_\bullet = 
K_\bullet$ (that is, these two complexes are identical, not only isomorphic). 

\emph{Indeed}, let $0 \ra L_{n-2}^\prim \ra \cdots 
\ra L_0^\prim \ra \tH^0_\ast(E) \ra 0$ be a \emph{minimal} graded free 
resolution of $\tH^0_\ast(E)$. By Lemma~\ref{L:h1=0}(b) and by a general fact 
about graded free resolutions over polynomial rings, $K_\bullet$ is isomorphic 
to the direct sum of $L_\bullet^\prim/hL_\bullet^\prim$ and of some complexes of 
the form$\, :$ 
\[
\cdots \lra 0 \lra (S/hS)(a) \overset{c}{\lra} (S/hS)(a) \lra 0 
\lra \cdots  
\]     
where $a$ is an integer and $c$ a nonzero constant. It suffices, now, to notice 
that if $P$ is a graded free $S$-module (of finite rank) then any automorphism 
of the graded $S/hS$-module $P/hP$ lifts to an automorphism of $P$. 
\end{remark} 

\begin{lemma}\label{L:hieh=0}
Let $E$ be a vector bundle on $\p^n$, $n \geq 3$, let $H \subset \p^n$ be a 
hyperplane, and let $i\geq 0$ and $a < b$ be fixed integers. Then$\, :$

\emph{(a)} ${\fam0 h}^i(E(b)) - {\fam0 h}^i(E(a)) \leq 
\sum_{l=a+1}^b{\fam0 h}^i(E_H(l))\, ;$ 

\emph{(b)} ${\fam0 h}^{i+1}(E(a)) - {\fam0 h}^{i+1}(E(b)) \leq 
\sum_{l=a+1}^b\dim {\fam0 Coker}({\fam0 H}^i(E(l)) \ra {\fam0 H}^i(E_H(l)))$.  
\qed 
\end{lemma}

\begin{lemma}\label{L:h1ehast=0}
Let $E$ be a vector bundle on $\p^n$, $n \geq 4$, and let $H \subset \p^n$ be 
a hyperplane. Then$\, :$   

\emph{(a)} ${\fam0 H}^0(E^\ast) = 0$ $\Rightarrow$  
${\fam0 h}^0(E_H^\ast) \leq \sum_{l\leq -1}{\fam0 h}^1(E_H^\ast(l))$.  

\emph{(b)} ${\fam0 H}^1(E^\ast) = 0$ $\Rightarrow$ ${\fam0 h}^1(E_H^\ast) \leq  
\sum_{l>-n}\dim {\fam0 Coker}({\fam0 H}^{n-3}(E(l)) \ra {\fam0 H}^{n-3}(E_H(l)))$. 
\end{lemma} 

\begin{proof}
Using the exact sequence $0 \ra E^\ast(-1) \ra E^\ast \ra E_H^\ast \ra 0$, one 
gets that $\h^0(E_H^\ast) \leq \h^1(E^\ast(-1))$ in case (a) 
and that $\h^1(E_H^\ast) \leq \h^2(E^\ast(-1)) = \h^{n-2}(E(-n))$ in case (b). 
One can apply, now, Lemma~\ref{L:hieh=0}. 
\end{proof}

\begin{lemma}\label{L:lift1}
Let $E$ be a vector bundle on $\p^n$, $n \geq 4$, such that 
${\fam0 H}^i(E^\ast) = 0$, $i = 0, 1$, and let $H \subset \p^n$ be a 
hyperplane. Assume that $E_H \simeq G \oplus t\sco_H$, with $G$ defined 
by an exact sequence$\, :$ 
\[
0 \lra s\sco_H \lra A \oplus P(B) \lra G \lra 0
\]
where $A$ and $B$ are direct sums of line bundles on $H$ such that 
${\fam0 H}^0(A^\ast) = 0$ and ${\fam0 H}^0(B^\ast) = 0$. Then $s = 0$ and 
$E \simeq {\widehat A} \oplus P({\widehat B})$, where ${\widehat A}$ 
and ${\widehat B}$ are direct sums of line bundles on $\p^n$ lifting $A$ and 
$B$, respectively.  
\end{lemma} 

\begin{proof} 
The fact that $s = 0$ follows from Lemma~\ref{L:h1ehast=0}(b) (taking into 
account Lemma~\ref{L:h1e=h1f}).  
Let $S = k[X_0,\ldots ,X_n]$ be the projective coordinate ring of $\p^n$ and 
let $h = 0$ be an equation of $H$. One has an exact sequence$\, :$  
\[
0 \lra B^\ast \lra A \oplus (\tH^0(B)^\ast\otimes_k\sco_H) 
\oplus t\sco_H \lra E_H \lra 0\, .
\] 
Applying $\tH^0_\ast(-)$ to this exact sequence one gets a minimal graded free 
resolution of the $S/hS$-module $\tH^0_\ast(E_H)$. 
It follows, from Lemma~\ref{L:h1=0}, that the above resolution of $E_H$ must 
be the restriction to $H$ of a resolution of $E$ having the same shape, 
namely$\, :$ 
\[
0 \lra {\widehat B}^\ast \overset{u}{\lra}  
{\widehat A} \oplus (\tH^0(B)^\ast\otimes_k\sco_\p) 
\oplus t\sco_\p \lra E \lra 0\, .
\]
The condition $\tH^i(E^\ast) = 0$, $i = 0, 1$, is equivalent to the fact that 
$\tH^0(u^\ast)$ is an isomorphism. Since $\tH^0({\widehat A}^\ast) = 0$, it 
follows that the component $(\tH^0(B)\otimes_k\sco_\p)\oplus 
t\sco_\p \ra {\widehat B}$ of $u^\ast$ can be identified with the 
evaluation morphism $\tH^0({\widehat B})\otimes_k\sco_\p \ra {\widehat B}$. 
Since, for $a > 0$, any morphism $\sco_\p(-a) \ra {\widehat B}$ factorizes 
through the evaluation morphism of ${\widehat B}$, one can assume (up to an 
automorphism of ${\widehat A} \oplus 
(\tH^0({\widehat B})^\ast\otimes_k\sco_\p)$) 
that the component ${\widehat A}^\ast \ra {\widehat B}$ of $u^\ast$ is 0. One 
gets, now, that $E \simeq {\widehat A} \oplus P({\widehat B})$. 
\end{proof}

\begin{lemma}\label{L:lift2}
Let $E$ be a vector bundle on $\p^n$, $n \geq 4$, such that 
${\fam0 H}^i(E^\ast) = 0$, $i = 0,\, 1$, and let $\Pi \subset \p^n$ be a fixed 
$3$-plane. Assume that $E_\Pi \simeq G \oplus t\sco_\Pi$, with $G$ 
defined by an exact sequence$\, :$ 
\[
0 \lra s\sco_\Pi \lra A \oplus P(B) \oplus \Omega_\Pi(2) \lra 
G \lra 0  
\]
where $A$ and $B$ are direct sums of line bundles on $\Pi$ such that 
${\fam0 H}^0(A^\ast) = 0$ and ${\fam0 H}^0(B^\ast) = 0$. Then one of the 
following holds$\, :$ 

\emph{(i)} $A \simeq A_1 \oplus (m - 4)\sco_\Pi(1)$ for some $m$ with 
$4 \leq m \leq n+1$ and $E \simeq {\widehat A}_1 \oplus E^\prim$, where 
${\widehat A}_1$ is a direct sum of line bundles on $\p^n$ lifting $A_1$ and 
$E^\prim$ is defined by an exact sequence$\, :$ 
\[
0 \lra E^\prim \lra m\sco_\p(1) \oplus P(\widehat B) \lra \sco_\p(2) 
\lra 0
\]
such that ${\fam0 Coker}(m\sco_\p(1) \ra \sco_\p(2)) \simeq 
\sco_\Lambda(2)$ for some $m$-codimensional linear subspace $\Lambda \subset 
\p^n$ not intersecting $\Pi$$\, ;$

\emph{(ii)} $B \simeq B_1 \oplus (m - 4)\sco_\Pi(1)$ for some $m$ with 
$4 \leq m \leq n+1$ and $E \simeq P({\widehat B}_1) \oplus E^\secund$, where 
${\widehat B}_1$ is a direct sum of line bundles on $\p^n$ lifting $B_1$ and 
$E^\secund$ admits a minimal resolution of the form$\, :$ 
\[
0 \lra \sco_\p(-2) \overset{d_2}{\lra} p\sco_\p \oplus 
m\sco_\p(-1) \overset{d_1}{\lra} {\widehat A} \oplus 
{\textstyle \binom{m}{2}}\sco_\p \lra E^\secund \lra 0
\] 
for some $p \leq s$, whose linear part $0 \ra \sco_\p(-2) \ra 
m\sco_\p(-1) \ra \binom{m}{2}\sco_\p$ is the dual of the first part of 
the Koszul resolution$\, :$ 
\[
{\textstyle \binom{m}{2}}\sco_\p \overset{\delta_2}{\lra} 
m\sco_\p(1)  \overset{\delta_1}{\lra} 
\sco_\p(2) \overset{\delta_0}{\lra} \sco_\Lambda(2) \lra 0
\] 
of the structure sheaf of an $m$-codimensional linear subspace 
$\Lambda \subset \p^n$ not intersecting $\Pi$, and such that, denoting by 
$d_2^\prim$ the component $\sco_\p(-2) \ra p\sco_\p$ of $d_2$ and by 
$\sigma$ the component $p\sco_\p \ra {\widehat A}$ of $d_1$, 
${\fam0 H}^0(\delta_0 \circ d_2^{\prim \ast})$ injects 
${\fam0 H}^0(p\sco_\p)$ into ${\fam0 H}^0(\sco_\Lambda(2))$ and the 
sequence$\, :$ 
\[
{\widehat A}^\ast \overset{\sigma^\ast}{\lra} p\sco_\p  
\xra{\delta_0 \, \circ \, d_2^{\prim \ast}} \sco_\Lambda(2) \lra 0
\]
is exact. 
\end{lemma} 

\noindent 
Notice that, in case (i) of Lemma~\ref{L:lift2}, if $m = n + 1$ (i.e., if 
$\Lambda = \emptyset$) then $E \simeq {\widehat A}_1 \oplus P(\widehat B) 
\oplus \Omega_\p(2)$ and, in case (ii), if $m = n + 1$ (i.e., if 
$\Lambda = \emptyset$ and $s = 0$) then $E \simeq {\widehat A} \oplus 
P({\widehat B}_1) \oplus \Omega_\p^{n-2}(n-1)$. 

\begin{proof}
Consider a saturated flag $\Pi = \Pi_0 \subset \Pi_1 \subset \ldots \subset 
\Pi_{n-3} = \p^n$ of linear subspaces of $\p^n$ and let $E_i := E \vb \Pi_i$. 
In particular, $E_0 = E_\Pi$. One has $\tH^1_\ast(E_0) \simeq k(2)$. It follows, 
applying inductively Lemma~\ref{L:hieh=0}(a), that $\tH^1(E_i(l)) = 0$ for 
$l \leq -3$, $i = 1, \ldots, n-3$, and that, moreover, $\tH^1(E_1(-2))$ 
injects into $\tH^1(E_0(-2)) \simeq k$. 

\vskip2mm

\noindent
{\bf Case 1.}\quad $\tH^1(E_1(-2)) \neq 0$. 

\vskip2mm

\noindent
In this case, applying Lemma~\ref{L:hieh=0}(b) to $E_1$ one gets 
$\tH^2_\ast(E_1) = 0$. Applying now, inductively, Lemma~\ref{L:hieh=0} one 
gets that $\tH^p_\ast(E_i) = 0$ for $2 \leq p \leq i+1$, $i = 1, \ldots , n-3$. 
Applying Lemma~\ref{L:h1ehast=0}(b), inductively, to $E = E_{n-3}$, 
$E_{n-4}$, $\ldots$, $E_1$ one gets that $\tH^1(E_i^\ast) = 0$, $i = n-4, \ldots, 
0$. In particular, $s = \h^1(E_0^\ast) = 0$. 

Since $\tH^1(E_i(-3)) = 0$, $i = 0, \ldots , n-3$, and $\tH^2(E_i(-3)) = 0$, 
$i = 1, \ldots, n-3$, it follows that $\tH^1(E_i(-2)) \izo \tH^1(E_{i-1}(-2))$, 
$i = 1, \ldots , n-3$. A non-zero element of $\tH^1(E(-2))$ defines an 
extension$\, :$ 
\[
0 \lra E \lra F \lra \sco_\p(2) \lra 0 
\]     
with $F_\Pi \simeq A \oplus P(B) \oplus 4\sco_\Pi(1) \oplus 
t\sco_\Pi$. Since $\tH^i(F^\ast) = 0$, $i = 0,\, 1$, Lemma~\ref{L:lift1} 
implies that $F \simeq {\widehat A} \oplus 4\sco_\p(1) \oplus 
P(\widehat B)$, whence an exact sequence$\, :$ 
\[
0 \lra E \lra {\widehat A} \oplus 4\sco_\p(1) \oplus P(\widehat B) 
\overset{\alpha}{\lra} \sco_\p(2) \lra 0\, .
\] 
Since any automorphism of $A \oplus 4\sco_\Pi(1) \oplus P(B) \oplus 
t\sco_\Pi$ maps $A \oplus 4\sco_\Pi(1)$ isomorphically onto itself 
and since any automorphism of $A \oplus 4\sco_\Pi(1)$ lifts to an 
automorphism of ${\widehat A} \oplus 4\sco_\p(1)$, one can assume that 
the component $A \ra \sco_\Pi(2)$ of $\alpha \vb \Pi$ is 0 and that the 
component $4\sco_\Pi(1) \ra \sco_\Pi(2)$ is an epimorphism. Consequently, 
one can find a decomposition ${\widehat A} \simeq {\widehat A}_1 \oplus 
(m - 4)\sco_\p(1)$ (restricting to a decomposition $A \simeq A_1 \oplus 
(m - 4)\sco_\Pi(1)$) for some $m$ with $4 \leq m \leq n+1$ such that 
$\alpha \vb {\widehat A}_1 = 0$ and the component $(m - 4)\sco_\p(1)   
\oplus 4\sco_\p(1) \ra \sco_\p(2)$ of $\alpha$  
is defined by $m$ linearly independent linear forms on $\p^n$. 
These linear forms define an $m$-codimensional linear subspace $\Lambda$ of 
$\p^n$ not intersecting $\Pi$. It follows that $E \simeq {\widehat A}_1 \oplus 
E^\prim$, where $E^\prim$ is defined by the exact sequence$\, :$ 
\[
0 \lra E^\prim \lra m\sco_\p(1) \oplus P(\widehat B) 
\overset{\alpha^\prim}{\lra} \sco_\p(2) \lra 0\, ,
\] 
$\alpha^\prim$ being the restriction of $\alpha$. 

\vskip2mm

\noindent
{\bf Case 2.}\quad $\tH^1(E_1(-2)) = 0$. 

\vskip2mm 

\noindent 
In this case, applying Lemma~\ref{L:hieh=0}(a) inductively one gets that 
$\tH^1_\ast(E_i) = 0$, $i = 1, \ldots , n-3$. It follows, now, from 
Lemma~\ref{L:h1=0}(b) that $E$ admits a resolution of the form$\, :$ 
\[
0 \ra \sco_\p(-2) \overset{\beta}{\lra} 
s\sco_\p \oplus 4\sco_\p(-1) \oplus {\widehat B}^\ast 
\overset{\alpha}{\lra} 
{\widehat A} \oplus 6\sco_\p \oplus (\tH^0(B)^\ast\otimes_k\sco_\p) 
\oplus t\sco_\p \ra E \ra 0\, .  
\]  
Let us denote this resolution by $\scl_\bullet$. Since the restriction to 
$\Pi$ of the component $\sco_\p(-2) \ra 4\sco_\p(-1)$ of $\beta$ is a 
locally split monomorphism and since the restriction to $\Pi$ of the 
component $\sco_\p(-2) \ra {\widehat B}^\ast$ of $\beta$ is 0, it follows  
that there exists a decomposition 
${\widehat B} \simeq {\widehat B}_1 \oplus (m - 4)\sco_\p(1)$ for some $m$ 
with $4 \leq m \leq n+1$ (restricting to a decomposition $B \simeq B_1 \oplus 
(m - 4)\sco_\Pi$) such that the component $\sco_\p(-2) \ra 
{\widehat B}_1^\ast$ of $\beta$ is 0 and the component $\sco_\p(-2) 
\ra 4\sco_\p(-1) \oplus (m - 4)\sco_\p(-1)$ of $\beta$ is defined by $m$ 
linearly independent linear forms $h_1, \ldots , h_m$ on $\p^n$. 
These linear forms define an $m$-codimensional linear subspace $\Lambda$ of 
$\p^n$ not intersecting $\Pi$. 

Now, cancelling some direct summands isomorphic to $\sco_\p$ from $\scl_1$ 
and $\scl_0$, one gets a \emph{minimal} resolution for $E$ of the form$\, :$ 
\[
0 \ra \sco_\p(-2) \overset{\beta^\prim}{\lra} 
p\sco_\p \oplus m\sco_\p(-1) \oplus {\widehat B}_1^\ast 
\overset{\alpha^\prim}{\lra} {\widehat A} \oplus q\sco_\p  
\ra E \ra 0 
\]   
for some $p \leq s$ and some $q$. Of course, the component $\sco_\p(-2) \ra 
m\sco_\p(-1)$ of $\beta^\prim$ is defined by the linear forms 
$h_1, \ldots , h_m$ considered above, the component 
$\sco_\p(-2) \ra {\widehat B}_1^\ast$ of $\beta^\prim$ is 0, and the component 
$p\sco_\p \ra q\sco_\p$ of $\alpha^\prim$ is 0. 

The hypothesis $\tH^i(E^\ast) = 0$, $i = 0,\, 1$, becomes, now, equivalent to 
the exactness of the sequence$\, :$ 
\[
0 \lra \tH^0(q\sco_\p) \xra{\tH^0(\alpha^{\prim \ast})} 
\tH^0(p\sco_\p) \oplus \tH^0(m\sco_\p(1)) \oplus 
\tH^0({\widehat B}_1) \xra{\tH^0(\beta^{\prim \ast})} 
\tH^0(\sco_\p(2))\, .
\] 
Since $\Ker \tH^0(\beta^{\prim \ast}) = \text{Im}\, \tH^0(\alpha^{\prim \ast}) 
\subseteq \tH^0(m\sco_\p(1)) \oplus \tH^0({\widehat B}_1)$ it follows 
that$\, :$ 
\[
\Ker \tH^0(\beta^{\prim \ast}) = \Ker(\tH^0(m\sco_\p(1)) \lra 
\tH^0(\sco_\p(2))) \oplus \tH^0({\widehat B}_1)\, .
\] 
One deduces, in particular, that the composite map$\, :$ 
\[
\tH^0(p\sco_\p) \xra{\tH^0(\beta^{\prim \ast})} \tH^0(\sco_\p(2)) 
\lra \Cok(\tH^0(m\sco_\p(1)) \lra \tH^0(\sco_\p(2))) \simeq 
\tH^0(\sco_\Lambda(2)) 
\]
is \emph{injective}. But $\Ker(\tH^0(m\sco_\p(1)) \ra \tH^0(\sco_\p(2)))$ 
consists of the Koszul relations between the linear forms $h_1, \ldots , h_m$. 
Since $\tH^0(\alpha^{\prim \ast})$ maps $\tH^0(q\sco_\p)$ isomorphically 
onto $\Ker \tH^0(\beta^{\prim \ast})$, there exists a decomposition$\, :$ 
\[
q\sco_\p \simeq {\textstyle \binom{m}{2}}\sco_\p \oplus 
(\tH^0({\widehat B}_1)\otimes_k\sco_\p)
\]
such that $\tH^0(\alpha^{\prim \ast})$ maps $\tH^0(\binom{m}{2}\sco_\p)$ 
isomorphically onto the kernel of 
$\tH^0(m\sco_\p(1)) \ra \tH^0(\sco_\p(2))$ and 
such that the restriction of 
$\alpha^{\prim \ast}$ to $\tH^0({\widehat B}_1)\otimes \sco_\p$ 
coincides with the evaluation morphism $\tH^0({\widehat B}_1)\otimes \sco_\p 
\ra {\widehat B}_1$. It follows that $E \simeq P({\widehat B}_1) \oplus 
E^\secund$, with $E^\secund$ admitting a minimal resolution of the form$\, :$ 
\[
0 \lra \sco_\p(-2) 
\xra{\left(\begin{smallmatrix} d_2^\prim\\ \delta_1^\ast 
\end{smallmatrix}\right)} 
p\sco_\p \oplus m\sco_\p(-1)  
\xra{\left(\begin{smallmatrix} \sigma & d_1^\prim\\ 0 & \delta_2^\ast 
\end{smallmatrix}\right)} 
{\widehat A} \oplus {\textstyle \binom{m}{2}}\sco_\p \lra E^\secund \lra 0\, ,
\]
where $\binom{m}{2}\sco_\p \overset{\delta_2}{\lra} m\sco_\p(1)  
\overset{\delta_1}{\lra} \sco_\p(2) \overset{\delta_0}{\lra} 
\sco_\Lambda(2) \ra 0$ is the Koszul complex associated to $h_1, \ldots , h_m$. 
Let us denote this resolution by $\scl_\bullet^\secund$. Its dual 
$\scl_\bullet^{\secund \ast}$ is the mapping cone of the morphism of 
complexes$\, :$ 
\[
\begin{CD}
{\textstyle \binom{m}{2}}\sco_\p @>{\delta_2}>> m\sco_\p(1) @>{\delta_1}>> 
\sco_\p(2)\\
@AAA @A{-d_1^{\prim \ast}}AA @AA{d_2^{\prim \ast}}A\\
0 @>>> {\widehat A}^\ast @>{\sigma^\ast}>> p\sco_\p  
\end{CD}
\]
It follows that the exactness of the sequence $\scl_0^{\secund \ast} \ra 
\scl_1^{\secund \ast} \ra \scl_2^{\secund \ast} \ra 0$ is equivalent to the 
exactness of the sequence$\, :$
\[
{\widehat A}^\ast \overset{\sigma^\ast}{\lra} p\sco_\p  
\xra{\delta_0 \, \circ \, d_2^{\prim \ast}} \sco_\Lambda(2) 
\lra 0\, .
\qedhere
\]
\end{proof}

\begin{remark}\label{R:lift2}
(a) If $F$ is a 0-regular vector bundle on $\p^n$ then $P(F)^\ast$ is 
1-regular (use the exact sequence $0 \ra P(F)^\ast \ra \tH^0(F)\otimes_k\sco_\p 
\ra F \ra 0$). It follows that if $a \geq 1$ is an integer and if $\Lambda$ 
is a linear subspace of $\p^n$ then any morphism $P(F) \ra \sco_\Lambda(a)$ 
lifts to a morphism $P(F) \ra \sco_\p(a)$. 

\vskip2mm 

(b) Let $\delta_1 : m\sco_\p(1) \ra \sco_\p(2)$ and $\psi : P(\widehat B) 
\ra \sco_\p(2)$ be the components of the morphism $\phi$ 
whose kernel is the bundle 
$E^\prim$ from the statement of Lemma~\ref{L:lift2}(i). The fact that $\phi$ 
is an epimorphism is equivalent to the fact that the composite morphism 
$P(\widehat B) \overset{\psi}{\lra} \sco_\p(2) \overset{\delta_0}{\lra} 
\sco_\Lambda(2)$ is an epimorphism. Since $\Ker \delta_1$ is globally generated 
and since $\tH^1(\Ker \delta_1) = 0$, one sees easily that $E^\prim$ is globally 
generated if and only if $\Ker(\delta_0 \circ \psi)$ is globally generated. 
Let $W\subseteq \tH^0(\sco_\Lambda(2))$ be the image of $\tH^0(\delta_0 \circ 
\psi) : \tH^0(P(\widehat B)) \ra \tH^0(\sco_\Lambda(2))$. Then, as in the proof  
of Lemma~\ref{L:ggquasib}(b), $\Ker(\delta_0 \circ \psi)$ is globally 
generated if and only if the bottom row of the diagram$\, :$ 
\[
\begin{CD}
{\widehat B}^\ast @>>> \tH^0(P(\widehat B))\otimes_k\sco_\p @>>> 
P(\widehat B) @>>> 0\\
@\vert @VVV @VV{\delta_0 \, \circ \, \psi}V\\
{\widehat B}^\ast @>>> W\otimes_k\sco_\p @>>> \sco_\Lambda(2) @>>> 0
\end{CD}
\]   
is exact. 

Conversely, since any morphism $P(\widehat B) \ra \sco_\Lambda(2)$ lifts to a 
morphism $P(\widehat B) \ra \sco_\p(2)$, to any exact sequence 
${\widehat B}^\ast \ra p\sco_\p \overset{\eta}{\lra} \sco_\Lambda(2) \ra 0$ 
with $\tH^0(\eta) : \tH^0(p\sco_\p) \ra \tH^0(\sco_\Lambda(2))$ injective 
one can associate an epimorphism $\phi : m\sco_\p(1) \oplus 
P(\widehat B) \ra \sco_\p(2)$ with $\Ker \phi$ globally generated. 

\vskip2mm

(c) Let $\Lambda$ be an $m$-codimensional linear subspace of $\p^n$ and 
consider a presentation ${\widehat B}^\ast \ra p\sco_\p \ra 
\sco_\Lambda(2) \ra 0$, with $\widehat B$ a direct sum of line bundles on 
$\p^n$ such that $\tH^0({\widehat B}^\ast) = 0$. Assume that $\Lambda \neq 
\emptyset$, i.e., that $m \leq n$. Then $p \geq n - m + 1$ and 
$\text{rk}\, {\widehat B}^\ast - p + 1 \geq m$, hence $\text{rk}\, 
{\widehat B}^\ast \geq n$, hence $c_1(\widehat B) \geq n$. It follows that if 
the bundle $E^\prim$ from the statement of Lemma~\ref{L:lift2}(i) is globally 
generated and $m \leq n$ then $c_1(E^\prim) \geq n + m - 2 \geq n + 2$. 

The same kind of argument shows that if $m \leq n$ in the statement of 
Lemma~\ref{L:lift2}(ii) then $c_1(E^\secund) \geq n + 2$. 

\vskip2mm 

(d) The functor $P(-)$ induces a certain bijection between the globally 
generated vector bundles described in Lemma~\ref{L:lift2}(i) and the 
bundles described in Lemma~\ref{L:lift2}(ii) (with $A$ and $B$ interchanged). 

\emph{Indeed}, let $E$ be a vector bundle on $\p^n$ as in 
Lemma~\ref{L:lift2}(i) and assume that, moreover, $E$ is globally generated 
(see item (b) of this remark). Then, by Lemma \ref{L:peh} and 
Lemma~\ref{L:pe}(b), $P(E)_\Pi \simeq G^\prim \oplus t^\prim \sco_\Pi$, with 
$G^\prim$ defined by an exact sequence$\, :$ 
\[
0 \lra s^\prim \sco_\Pi \lra B \oplus P(A) \oplus \Omega_\Pi(2) 
\lra G^\prim \lra 0\, .
\]
Using the description of $E$ from the statement of Lemma~\ref{L:lift2}(i), 
one sees easily that $\tH^1_\ast(P(E)) \simeq \tH^2_\ast(E^\ast) = 0$. The 
proof of Lemma~\ref{L:lift2} shows, now, that $P(E)$ is as in the 
statement of Lemma~\ref{L:lift2}(ii), with $A$ and $B$ interchanged. 
\end{remark}

We shall also need the following well known fact, which is a slight 
generalization of the Lemma of Castelnuovo-Mumford$\, :$ 

\begin{lemma}\label{L:cm} 
Let $\scf$ be a coherent sheaf on $\p^n$, let $S = k[X_0, \ldots ,X_n]$ be the 
projective coordinate ring of $\p^n$, and let $0 \leq i < n$ and $m$ be 
integers. Assume that ${\fam0 H}^j(\scf(m-j)) = 0$, for $i < j \leq n$. 
Then$\, :$  

\emph{(a)} ${\fam0 H}^j(\scf(m+1-j)) = 0$, for $i < j \leq n$$\, ;$  

\emph{(b)} The graded $S$-module ${\fam0 H}_\ast^i(\scf)$ is generated in 
degrees $\leq m-i$. 
\end{lemma}

The following easy exercise concerning the notion of Castelnuovo-Mumford 
regularity is sometimes called the ``Lemma of Le Potier''. 

\begin{lemma}\label{L:LePotier}
Let $\scg$ be a coherent sheaf on $\p^n$, $n \geq 2$, $H \subset \p^n$ a 
hyperplane of equation $h = 0$, and $m$ an integer. Assume that the 
multiplication by $h : \scg(-1) \ra \scg$ is a monomorphism. If $\scg_H$ is 
$m$-regular then ${\fam0 H}^i(\scg(l)) = 0$ for $i \geq 2$ and $l \geq m-i$  
and, for every $l \geq m$, ${\fam0 h}^1(\scg(l-1)) \geq {\fam0 h}^1(\scg(l))$, 
with equality iff ${\fam0 H}^1(\scg(l-1)) = 0$. 
\end{lemma} 

We also state, for the reader's convenience, the well known theorem of 
Beilinson \cite{be} (a different proof of this theorem can be found in the 
paper of Eisenbud et al. \cite[(6.1)]{efs}). 

\begin{thm}\label{T:beilinson}
If $\scf$ is a coherent sheaf on $\p^n$ then there exists a complex 
$C^\bullet$ of locally free sheaves on $\p^n$ with $C^p = \bigoplus_{j \geq p}
{\fam0 H}^j(\scf(p-j))\otimes_k\Omega_\p^{j-p}(j-p)$, $p \in \z$, 
such that $\sch^p(C^\bullet) = 0$ for $p \neq 0$ and $\sch^0(C^\bullet) \simeq 
\scf$.  
\end{thm}  

\begin{remark}\label{R:contraction}
Let us recall, in connection with Beilinson's theorem, the interpretation of 
the morphisms between sheaves of the form $\Omega_\p^i(i)$ in terms of the 
exterior algebra. 

\vskip2mm 

(a) Let $V$ denote the $k$-vector space $k^{n+1}$ such that $\p^n = \p(V) = 
(V\setminus \{0\})/k^\ast$. Using the perfect pairing$\, :$ 
\[
-\cdot - : \overset{p}{\textstyle \bigwedge} V^\ast \times 
\overset{p}{\textstyle \bigwedge} V \lra k\, ,\  
(f_1\wedge \ldots \wedge f_p,\, v_1\wedge \ldots \wedge v_p) \mapsto 
\text{det}(f_i(v_j))\, ,  
\]
the dual of the exterior multiplication $-\wedge - : 
\overset{p}{\bigwedge} V \otimes \overset{q}{\bigwedge} V \ra 
\overset{p+q}{\bigwedge} V$ can be identified with the map  
$\overset{p+q}{\bigwedge} V^\ast \ra 
\overset{p}{\bigwedge} V^\ast \otimes 
\overset{q}{\bigwedge} V^\ast$ given by$\, :$
\[    
f_1 \wedge \ldots \wedge f_{p+q} \mapsto 
\sum_{\substack{\sigma \in \fS_{p+q}\\ \sigma(1) < \cdots < \sigma(p)\\ 
\sigma(p+1) < \cdots < \sigma(p+q)}} \e(\sigma) f_{\sigma(1)} \wedge \ldots 
\wedge f_{\sigma(p)} \otimes f_{\sigma(p+1)} \wedge \ldots \wedge f_{\sigma(p+q)} 
\, .
\]
It follows that, for $\omega \in \overset{p}{\bigwedge} V$, the dual of 
$\omega \wedge - : \overset{q}{\bigwedge} V \ra \overset{p+q}{\bigwedge} V$ 
can be identified with the map $\overset{p+q}{\bigwedge} V^\ast \ra 
\overset{q}{\bigwedge} V^\ast$ given by$\, :$ 
\[    
f_1 \wedge \ldots \wedge f_{p+q} \mapsto 
\sum_{\substack{\sigma \in \fS_{p+q}\\ \sigma(1) < \cdots < \sigma(p)\\ 
\sigma(p+1) < \cdots < \sigma(p+q)}} \e(\sigma) 
(f_{\sigma(1)} \wedge \ldots \wedge f_{\sigma(p)} \cdot \omega) 
f_{\sigma(p+1)} \wedge \ldots \wedge f_{\sigma(p+q)} 
\, .
\]
We \emph{denote} the image of any $\varphi \in \overset{p+q}{\bigwedge} V^\ast$ 
by $\varphi \cdot \omega$. 

For $\varphi \in \overset{p+q+r}{\bigwedge} V^\ast$, $\omega \in 
\overset{p}{\bigwedge} V$ and $\eta \in \overset{q}{\bigwedge} V$ 
one has $(\varphi \cdot \omega)\cdot \eta = \varphi \cdot (\omega \wedge 
\eta)$ (because $(\eta \wedge -)^\ast \circ (\omega \wedge -)^\ast = 
((\omega \wedge -)\circ (\eta \wedge -))^\ast = 
(\omega \wedge \eta \wedge -)^\ast$). 

Moreover, using the perfect pairing $-\wedge - : 
\overset{n+1-p}{\bigwedge} V \times \overset{p}{\bigwedge} V \ra 
\overset{n+1}{\bigwedge} V \simeq k$, the map $- \cdot \omega : 
\overset{p+q}{\bigwedge} V^\ast \ra \overset{q}{\bigwedge} V^\ast$ can be 
identified with $- \wedge \omega : \overset{n+1-p-q}{\bigwedge} V \ra 
\overset{n+1-q}{\bigwedge} V$ (because, for $\eta \in 
\overset{n+1-p-q}{\bigwedge} V$ and $\xi \in \overset{q}{\bigwedge} V$, one 
has $\eta \wedge (\omega \wedge \xi) = (\eta \wedge \omega) \wedge \xi$). 

\vskip2mm

(b) Let $K_\bullet$ denote the tautological Koszul complex on 
$\p^n = \p(V)$$\, :$ 
\[
0 \lra \overset{n+1}{\textstyle \bigwedge} V^\ast \otimes \sco_\p(-n-1) 
\xra{d_{n+1}} \overset{n}{\textstyle \bigwedge} V^\ast \otimes 
\sco_\p(-n) \lra \cdots \lra V^\ast \otimes \sco_\p(-1) 
\overset{d_1}{\lra} \sco_\p \lra 0 
\]  
associated to the canonical epimorphism $d_1 : V^\ast \otimes \sco_\p(-1) 
\ra \sco_\p$. The reduced stalk of $d_{i+1}(i)$ at a point $[v] \in \p(V)$ is 
the map $\overset{i+1}{\bigwedge} V^\ast \otimes kv \ra 
\overset{i}{\bigwedge} V^\ast$, $\xi \otimes v \mapsto \xi \cdot v$. It 
follows that, for any $\omega \in \overset{p}{\bigwedge} V$, the 
diagrams$\, :$ 
\[
\begin{CD}
\overset{p+q+1}{\textstyle \bigwedge} V^\ast \otimes \sco_\p(-1) 
@>{d_{p+q+1}(p+q)}>> \overset{p+q}{\textstyle \bigwedge} V^\ast \otimes \sco_\p\\
@V{\displaystyle -\cdot \omega}VV @VV{\displaystyle -\cdot \omega}V\\
\overset{q+1}{\textstyle \bigwedge} V^\ast \otimes \sco_\p(-1) 
@>{(-1)^pd_{q+1}(q)}>> \overset{q}{\textstyle \bigwedge} V^\ast \otimes \sco_\p
\end{CD}
\] 
are commutative. Since $\Omega_\p^i(i) \simeq \text{Im}\, d_{i+1}(i) = 
\Ker d_i(i)$, one gets a linear map$\, :$ 
\[
\overset{p}{\textstyle \bigwedge} V \lra 
\text{Hom}_{\sco_\p}(\Omega_\p^{p+q}(p+q),\, \Omega_\p^q(q))
\]
which turns out to be an \emph{isomorphism}. In this way, the composition 
map$\, :$ 
\begin{gather*}
\text{Hom}_{\sco_\p}(\Omega_\p^{q+r}(q+r),\, \Omega_\p^r(r)) \times 
\text{Hom}_{\sco_\p}(\Omega_\p^{p+q+r}(p+q+r),\, \Omega_\p^{q+r}(q+r)) \lra\\
\text{Hom}_{\sco_\p}(\Omega_\p^{p+q+r}(p+q+r),\, \Omega_\p^r(r))
\end{gather*}
is identified with the map $(-1)^{pq}(-\wedge -) : 
\overset{q}{\bigwedge} V \times \overset{p}{\bigwedge} V \ra 
\overset{p+q}{\bigwedge} V$. 

\vskip2mm 

(c) Consider, now, a decomposition $V = V^\prim \oplus V^\secund$ with 
$\dim V^\prim = n^\prim + 1$, $\dim V^\secund = n^\secund$, $n^\prim + n^\secund = n$, 
and the dual decomposition $V^\ast = V^{\prim \ast} \oplus V^{\secund \ast}$. The 
exterior multiplication induces an isomorphism$\, :$ 
\[
{\textstyle \bigwedge}(V^{\prim \ast}) \otimes_k 
{\textstyle \bigwedge}(V^{\secund \ast}) \Izo 
{\textstyle \bigwedge}(V^\ast)\, . 
\]
Consider the complexes $K^\prim_\bullet$ and $K^\secund_\bullet$ on 
$\p^\prim := \p(V^\prim) \simeq \p^{n^\prim}$$\, :$ 
\begin{gather*}
\overset{n^\prim+1}{\textstyle \bigwedge} 
V^{\prim \ast} \otimes \sco_{\p^\prim}(-n^\prim-1) 
\xra{d^\prim_{n^\prim +1}} \overset{n^\prim}{\textstyle \bigwedge} 
V^{\prim \ast} \otimes \sco_{\p^\prim}(-n^\prim) 
\ra \cdots \ra V^{\prim \ast} \otimes \sco_{\p^\prim}(-1) 
\overset{d^\prim_1}{\lra} \sco_{\p^\prim}\, ,\\
\overset{n^\secund}{\textstyle \bigwedge} 
V^{\secund \ast} \otimes \sco_{\p^\prim}(-n^\secund) 
\overset{0}{\ra} \overset{n^\secund -1}{\textstyle \bigwedge} 
V^{\secund \ast} \otimes \sco_{\p^\prim}(-n^\secund +1) 
\overset{0}{\ra} \cdots \overset{0}{\ra}  
V^{\secund \ast} \otimes \sco_{\p^\prim}(-1) 
\overset{0}{\ra} \sco_{\p^\prim} \, ,
\end{gather*} 
$K^\prim_\bullet$ being the tautological Koszul complex on $\p(V^\prim)$. 
Then $K_\bullet \vb \p(V^\prim) \simeq K^\prim_\bullet 
\otimes_{\sco_{\p(V^\prim)}} K^\secund_\bullet$, hence 
$\Omega_\p^p(p) \vb \p(V^\prim) \simeq 
\bigoplus_{i+j = p}\Omega^i_{\p(V^\prim)}(i)\otimes_k 
\overset{j}{\textstyle \bigwedge} V^{\secund \ast}$. 
\end{remark} 

\begin{remark}\label{R:linbeilinson}
Let $\scf$ be a coherent sheaf on $\p^n$ and let $C^\bullet$ be the Beilinson 
monad of $\scf$ from Thm.~\ref{T:beilinson}. It turns out that the 
differential $d_C^p : C^p \ra C^{p+1}$ maps 
$\tH^j(\scf(p-j))\otimes_k\Omega_\p^{j-p}(j-p)$ into 
$\bigoplus_{l \leq j}\tH^l(\scf(p+1-l))\otimes_k\Omega_\p^{l-p-1}(l-p-1)$. 
Moreover, according to the results of Eisenbud, Fl\o ystad and Schreyer 
\cite{efs}, the linear part $\text{lin}\, C^\bullet$ of the complex 
$C^\bullet$ obtained by turning into 0 the components 
$\tH^j(\scf(p-j))\otimes_k\Omega_\p^{j-p}(j-p) \ra 
\tH^l(\scf(p+1-l))\otimes_k\Omega_\p^{l-p-1}(l-p-1)$ with $l < j$ of 
$d_C^p$, $p \in \z$, is the total complex of a double complex 
$X^{\bullet \bullet}$ with terms $X^{ij} = \tH^j(\scf(i)) \otimes_k 
\Omega_\p^{-i}(-i)$ concentrated in the square $-n \leq i \leq 0$, 
$0 \leq j \leq n$, with $d_{II}^{-i,j} : X^{-i,j} \ra X^{-i,j+1}$ equal to 0 and 
with $d_I^{-i,j} : X^{-i,j} \ra X^{-i+1,j}$ equal to the composite morphism$\, :$ 
\[
\tH^j(\scf(-i))\otimes \Omega^i(i) \ra 
\tH^j(\scf(-i))\otimes V^\ast \otimes \Omega^{i-1}(i-1) \ra 
\tH^j(\scf(-i+1))\otimes \Omega^{i-1}(i-1) 
\]
where $\Omega_\p^i(i) \ra V^\ast\otimes_k\Omega_\p^{i-1}(i-1)$ is induced by 
$\overset{i}{\bigwedge}V^\ast \otimes_k\sco_\p \ra V^\ast \otimes_k 
\overset{i-1}{\bigwedge}V^\ast \otimes_k \sco_\p$ and 
$\tH^j(\scf(-i))\otimes_kV^\ast \ra \tH^j(\scf(-i+1))$ is the multiplication 
map of the graded $S$-module $\tH_\ast^j(\scf)$. 
\end{remark}

\section{Some general results}\label{S:c2leq4}

We prove, in this section, several general classification results for 
globally generated vector bundles on projective spaces. General as they are, 
they suffice, in particular, to reobtain the classification of g.~g. 
vector bundles with $c_1 \leq 3$ (see Remark~\ref{R:c1leq3}). 

We classify, firstly, the g.~g. vector bundles $E$ for which the scheme 
$Y$ appearing in the exact sequence \eqref{E:oeiy} from the Introduction is a 
complete intersection. Then, following an idea of Ellia \cite{e} based on 
the use of the Cayley-Bacharach property, we classify the globally generated 
vector bundles with $c_2 \leq c_1$. 

The main result of this section asserts that if $E$ is a g.~g. vector bundle 
with $c_1 \geq 3$ and such that $\tH^i(E^\ast) = 0$, $i = 0,\, 1$, and 
$\tH^0(E(-c_1+2)) \neq 0$ then $E \simeq \sco_\p(c_1-a) \oplus F$, where   
$0 \leq a \leq 2$ and $F$ is a g.~g. vector bundle with $c_1(F) = a$. The 
proof is based on a classification of the locally Cohen-Macaulay curves 
$Z \subset \piii$ with $\sci_Z(2)$ globally generated. 

\vskip2mm 

We continue to use the notation and conventions from the Introduction.      

\begin{lemma}\label{L:yci} 
Assume that the scheme $Y$ in the exact sequence \eqref{E:oeiy} is a complete 
intersection of type $(a,b)$, $a \leq b$, and that ${\fam0 H}^i(E^\ast) = 0$, 
$i = 0,1$. Then one of the following holds$\, :$  
\begin{enumerate}   
\item[(i)] $b < c_1 \leq a+b$ and $E \simeq \sco_\p(c_1-b) \oplus 
\sco_\p(c_1 -a) \oplus P(\sco_\p(a+b-c_1))$$\, ;$  
\item[(ii)] $a < b = c_1$ and $E \simeq \sco_\p(c_1-a) \oplus 
P(\sco_\p(a+b-c_1))$$\, ;$ 
\item[(iii)] $a=b=c_1$ and $E \simeq P(\sco_\p(a+b-c_1))$.  
\end{enumerate}  
\end{lemma}

\begin{proof}
Since $\text{Ext}^1(\sco_\p(c_1-a) \oplus \sco_\p(c_1-b),(r - 1)\sco_\p) 
= 0$, one gets a commutative diagram$\, :$ 
\[
\begin{CD}  
0 @>>> \sco_\p(c_1 \, \text{--} \, a \, \text{--} \, b) 
@>>> \sco_\p(c_1 \, \text{--} \, a) \oplus \sco_\p(c_1 \, \text{--} \, b) 
@>>> \sci_Y(c_1) @>>> 0\\
@. @VVV @VVV @\vert\\
0 @>>> (r \, \text{--} \, 1)\sco_\p @>>> E @>>> \sci_Y(c_1) @>>> 0
\end{CD}
\]
from which one deduces the existence of an exact sequence$\, :$   
\[
0 \lra \sco_\p(c_1-a-b) \lra \sco_\p(c_1-a) \oplus \sco_\p(c_1-b) \oplus 
(r - 1)\sco_\p \lra E \lra 0\, .
\] 
Taking duals, one gets an exact sequence$\, :$ 
\[
0 \lra E^\ast \lra \sco_\p(a-c_1)\oplus \sco_\p(b-c_1) \oplus (r - 1)\sco_\p  
\lra \sco_\p(a+b-c_1) \lra 0\, .
\]
The hypothesis on $E$ implies that, by taking global sections, one gets an 
isomorphism 
\[
\tH^0(\sco_\p(a-c_1)\oplus \sco_\p(b-c_1) \oplus (r - 1)\sco_\p) \Izo 
\tH^0(\sco_\p(a+b-c_1))\, .
\]
In particular, since $r \geq 2$, it follows that $c_1 \leq a+b$. 

Assume that $b < c_1 \leq a+b$. One deduces that $r-1 = 
\text{h}^0(\sco_\p(a+b-c_1))$ and the component $(r - 1)\sco_\p \ra 
\sco_\p(a+b-c_1)$ of the morphism 
\[
\sco_\p(a-c_1)\oplus \sco_\p(b-c_1) \oplus (r - 1)\sco_\p  
\lra \sco_\p(a+b-c_1) 
\]
can be identified with the evaluation morphism of $\sco_\p(a+b-c_1)$. Since, 
for any $t > 0$, every morphism $\sco_\p(-t) \ra \sco_\p(a+b-c_1)$ factorizes 
through this evaluation morphism, one gets that$\, :$ 
\[
E^\ast \simeq \sco_\p(a-c_1) \oplus \sco_\p(b-c_1) \oplus 
P(\sco_\p(a+b-c_1))^\ast \, . 
\]
The other two cases are treated similarly. 
\end{proof}

As we saw at the beginning of the Introduction, if $c_2 = 0$ then 
$E \simeq \sco_\p(c_1) \oplus (r - 1)\sco_\p$. One deduces that if 
$c_1 = 0$ then $E \simeq r\sco_\p$, and if $c_1 = 1$ then either 
$c_2 = 0$ and $E \simeq \sco_\p(1) \oplus (r - 1)\sco_\p$ or $c_2 = 1$ and 
$E \simeq P(\sco_\p(1)) \oplus t\sco_\p \simeq \text{T}_\p(-1) \oplus 
t\sco_\p$ (because $c_n(\text{T}_\p(-1)) > 0$).  

We prove, now, two results that are based on an idea of 
Ellia~\cite[Lemma~2.4]{e}. 

\begin{prop}\label{P:c2=c1-1} 
Let $E$ be a globally generated vector bundle on $\p^n$ with $c_1 \geq 2$. 
If $c_2 > 0$ then $c_2 \geq c_1 - 1$ and if $c_2 = c_1 - 1$ then 
$E \simeq \sco_\p(1) \oplus \sco_\p(c_1-1) \oplus (r - 2)\sco_\p$. 
\end{prop}

\begin{proof}
Restricting to a $\pii \subseteq \p^n$ one may assume that $n=2$. 
In this case, the exact sequence \eqref{E:oeiy} can be obtained in two 
steps$\, :$  
\begin{gather*}
0 \lra (r - 2)\sco_\pii \lra E \lra F \lra 0\, ,\\
0 \lra \sco_\pii \lra F \lra \sci_Y(c_1) \lra 0\, , 
\end{gather*} 
where $F$ is a globally generated rank 2 vector bundle on $\p^2$ and $Y$ 
consists of $c_2$ simple points. It follows, from Remark~\ref{R:CayleyB}, 
that the points of $Y$ satisfy the Cayley-Bacharach condition with respect to 
plane curves of degree $c_1 - 3$. We split the proof into several cases 
according to the values of $c_1$. 

\vskip2mm 

\noindent 
$\bullet$ If $c_1=2$ and $c_2 = 1$ then $Y$ is a complete intersection of 
type $(1,1)$ and the assertion follows from Lemma~\ref{L:yci}(i).    

\vskip2mm

\noindent 
$\bullet$ If $c_1=3$ then one cannot have $c_2 = 1$ because, if this were the 
case, $Y$ would be a complete intersection of type $(1,1)$ which is not 
possible by Lemma~\ref{L:yci}. The assertion about the case $c_2=2$ follows 
from Lemma~\ref{L:yci}(i), again, because in this case $Y$ is a complete 
intersection of type $(1,2)$. 

\vskip2mm

\noindent 
$\bullet$ If $c_1 \geq 4$ the fact that $c_2 \geq c_1-1$ follows from the fact 
that $Y$ satisfies the Cayley-Bacharach condition with respect to plane 
curves of degree $c_1 - 3$ (see \cite[Lemma~2.4]{e} where one uses suitable 
unions of lines to get a contradiction if $c_2 < c_1-1$). The same condition  
implies, also, that if $c_2 = c_1-1$ then the points of $Y$ are colinear, 
i.e., $Y$ is a complete intersection of type $(1,c_1-1)$. The proof can be 
concluded, now, using Lemma~\ref{L:yci}(i).  
\end{proof} 

\begin{prop}\label{P:c2=c1}
Let $E$ be a globally generated vector bundle on $\p^n$ with $c_2 = c_1 \geq 2$ 
and such that ${\fam0 H}^i(E^\ast) = 0$, $i = 0,1$. Then one of the following 
holds$\, :$ 
\begin{enumerate}
\item[(i)] $E \simeq \sco_\p(c_1-1) \oplus {\fam0 T}_\p(-1)\, ;$ 
\item[(ii)] $c_1 = 2$, $n = 3$ and $E \simeq \Omega_\piii(2)\, ;$ 
\item[(iii)] $c_1 = 3$ and $E \simeq 3\sco_\p(1)\, ;$
\item[(iv)] $c_1 = 4$ and $E \simeq 2\sco_\p(2)$.  
\end{enumerate}
\end{prop}

\begin{proof} 
We split the proof into several cases because the small values of $c_1$ 
need to be treated separatedly. 

\vskip2mm 

\noindent
$\bullet$ If $c_1 = 2$ and $n = 2$ or $n \geq 4$ then $Y$ is a complete 
intersection of type $(1,2)$ and one can apply Lemma~\ref{L:yci}(ii). If 
$n = 3$ then either $Y$ is a complete intersection of type $(1,2)$ or 
$Y$ is the union of two disjoint lines. In the latter case, there is an exact 
sequence$\, :$ 
\[
0 \lra 2\sco_\piii \lra \Omega_\piii(2) \lra \sci_Y(2) \lra 0
\] 
and Lemma~\ref{L:uniqueext} implies that $E \simeq \Omega_\piii(2)$. 

\vskip2mm

\noindent 
$\bullet$ If $c_1 = 3$ and $n = 2$ then $Y$ consists of three simple points. 
If these points are colinear then $Y$ is a complete intersection of type 
$(1,3)$ and one can apply Lemma~\ref{L:yci}(ii). If they are not colinear 
then $\sci_Y$ admits a resolution of the form$\, :$  
\[
0 \lra 2\sco_\pii \lra 3\sco_\pii(1) \lra \sci_Y(3) \lra 0  
\]  
and Lemma~\ref{L:uniqueext} implies that $E \simeq 3\sco_\pii(1)$. 

If $n = 3$ then the curve $Y$ must be connected (because $\omega_Y(1)$ is 
globally generated) hence either $Y$ is a complete intersection of type 
$(1,3)$ or it is a twisted cubic curve. In the latter case, $\sci_Y$ 
admits a resolution of the form$\, :$  
\[
0 \lra 2\sco_\piii \lra 3\sco_\piii(1) \lra \sci_Y(3) \lra 0  
\]
hence, by Lemma~\ref{L:uniqueext}, $E \simeq 3\sco_\piii(1)$. 

If $n \geq 4$ then, using Lemma~\ref{L:lift1} and the case $n = 3$ already 
settled, one shows easily that either 
$E \simeq \sco_\p(c_1-1) \oplus \text{T}_\p(-1)$ or 
$E \simeq 3\sco_\p(1)$.   

\vskip2mm 

\noindent
$\bullet$ If $c_1 \geq 4$ then, 
taking into account Lemma~\ref{L:yci}, it suffices to prove that if $E$ is a 
globally generated vector bundle on $\p^n$ with $c_2 = c_1 \geq 4$ then 
either $Y$ is a complete intersection of type $(1,c_1)$ or $c_1 = 4$ and 
$E \simeq 2\sco_\p(2) \oplus (r - 2)\sco_\p$. 
The proof of this assertion can be obviously reduced to the case $n=2$ and 
$r = 2$. Now, in this case, $Y$ consists of $c_1$ simple points satisfying 
the Cayley-Bacharach condition with respect to plane curves of degree $c_1-3$, 
i.e., any curve of degree $c_1-3$ containing $c_1-1$ of the points of $Y$ 
must contain $Y$. One deduces easily, using suitable unions of lines, that 
either the points of $Y$ are colinear or any 3 of them are non-colinear. 
Since a set of $2d+1$ points of $\pii$, any 3 of them non-colinear, impose 
independent conditions on forms of degree $d$, one deduces from the 
Cayley-Bacharach condition, once more, that the points of $Y$ must be colinear 
if $c_1 \geq 5$. Consequently, either the points of $Y$ are colinear or 
$c_1 = 4$ and $Y$ is a complete intersection of type $(2,2)$. In the later 
case, $E \simeq 2\sco_{\p^2}(2)$.     
\end{proof} 

Next, we discuss an idea that appears in the papers of Sierra and Ugaglia. 
It follows from the exact sequence \eqref{E:oeiy}  
that $\tH^0(E(-c_1-1)) = 0$ and that $\tH^0(E(-c_1)) \neq 0$ iff $Y = 
\emptyset$ iff $E \simeq \sco_\p(c_1) \oplus (r - 1)\sco_\p$ (this is a 
particular case of a result of Sierra~\cite[Prop.~1]{s}). The next proposition 
and its corollary extend a result of Sierra and Ugaglia~\cite[Prop.~2.2]{su2}. 

\begin{prop}\label{P:h0e-c1+1} 
Let $E$ be a globally generated vector bundle of rank $r \geq 2$ on $\p^n$ 
with $c_1 \geq 2$ and such that ${\fam0 H}^i(E^\ast) = 0$, $i = 0, 1$. 
If ${\fam0 H}^0(E(-c_1+1))\neq  0$ and ${\fam0 H}^0(E(-c_1)) = 0$ 
then $E \simeq \sco_\p(c_1-1) \oplus F$, where $F= \sco_\p(1)$ or 
$F = {\fam0 T}_\p(-1)$. 
\end{prop}  

\begin{proof}
The hypotheses are equivalent to $Y \neq \emptyset$ 
(hence $c_2 > 0$) and $\tH^0(\sci_Y(1)) \neq 0$. 
It follows that $Y$ is a complete intersection of type $(1,c_2)$. 
$\sci_Y(c_1)$ globally generated implies that $c_1 \geq c_2$. One deduces, 
from Prop.~\ref{P:c2=c1-1}, that $c_2 = c_1-1$ or $c_2 = c_1$.    
The conclusion of the proposition follows, now, from Lemma~\ref{L:yci}. 
\end{proof}

\begin{cor}\label{C:h0e-c1+1}
Let $E$ be a globally generated vector bundle on $\p^n$ with $c_1 \geq 2$. 
If ${\fam0 H}^0(E(-c_1+1)) = 0$ then, for every linear subspace  
$\Pi \subseteq \p^n$ of dimension $\geq 2$, one has  
${\fam0 H}^0(E_\Pi(-c_1+1)) = 0$.  
\end{cor} 

\begin{proof}
Assume that there exists a linear subspace $\Pi \subseteq \p^n$ of dimension 
$\geq 2$ such that $\tH^0(E_\Pi(-c_1+1)) \neq 0$. If $\tH^0(E_\Pi(-c_1)) \neq 0$ 
then $E_\Pi \simeq \sco_\Pi(c_1) \oplus (r - 1)\sco_\Pi$ hence $E \simeq 
\sco_\p(c_1) \oplus (r - 1)\sco_\p$, which contradicts the hypothesis 
$\tH^0(E(-c_1+1)) = 0$. 
If $\tH^0(E_\Pi(-c_1)) = 0$ then, by the proof of Prop.~\ref{P:h0e-c1+1}, either 
$E_\Pi \simeq \sco_\Pi(1) \oplus \sco_\Pi(c_1-1) \oplus (r - 2)\sco_\Pi$ or 
$c_2 = c_1$. In the former case, 
$E \simeq \sco_\p(1) \oplus \sco_\p(c_1-1) \oplus (r - 2)\sco_\p$, which  
contradicts the hypothesis. In the latter case one gets a contradiction by 
using Prop.~\ref{P:c2=c1} (taking into account Lemma~\ref{L:h0h1}). 
\end{proof} 

We study, now, the globally generated vector bundles $E$ on $\p^n$ with 
$\tH^0(E(-c_1+2)) \neq 0$. We shall treat separatedly the cases $n = 2$, 
$n = 3$ and $n \geq 4$ because we use different kind of arguments for each of 
them. We shall need some general facts about \emph{liaison in codimension} 2  
which we recall, for the reader's convenience, refering to the classical paper 
of Peskine and Szpiro \cite{ps} for proofs. 

\begin{remark}\label{R:liaison}  
Consider closed subschemes $Y \subset X \subset \p^n$ of $\p^n$, $n \geq 2$,  
with $Y$ locally Cohen-Macaulay of pure codimension 2 and with $X$ a complete 
intersection of type $(a,b)$. Let $\sci_{Y,X} := \sci_Y/\sci_X$ denote the 
ideal sheaf of $\sco_X$ defining $Y$ as a closed subscheme of $X$. Define a 
closed subscheme $Y^\prim$ of $X$ by$\, :$  
\[
\sci_{Y^\prim,X} := \sca nn_{\sco_X}(\sci_{Y,X}) \simeq \sch om_{\sco_X}(\sco_Y,
\sco_X)\, .
\]  
One knows that $Y^\prim$ is locally Cohen-Macaulay of pure codimension 2 in 
$\p^n$ and that $\sci_{Y,X} \simeq \sca nn_{\sco_X}(\sci_{Y^\prim,X})$. Since 
$\omega_X \simeq \sco_X(a+b-n-1)$ one gets that$\, :$  
\[
\sci_{Y^\prim,X} \simeq \sch om_{\sco_X}(\sco_Y,\sco_X) \simeq 
\omega_Y(-a-b+n+1) 
\] 
whence the so called \emph{fundamental exact sequence of liaison}$\, :$  
\[
0 \lra \omega_Y(-a-b+n+1) \lra \sco_X \lra \sco_{Y^\prim} \lra 0  
\]
with its variant$\, :$ 
\[
0 \lra \sco_\p(-a-b) \lra \sco_\p(-a) \oplus \sco_\p(-b) \lra 
\sci_{Y^\prim} \lra \omega_Y(-a-b+n+1) \lra 0\, .
\]
Of course, there are similar exact sequences with the roles of $Y$ and 
$Y^\prim$ interchanged. Using Hibert polynomials, one deduces that  
$\text{deg}\, Y + \text{deg}\, Y^\prim = ab$.\footnote{and 
$\chi(\sco_{Y^\prim}) - \chi(\sco_Y) = \frac{1}{2}(a + b - 4)(\text{deg}\, Y - 
\text{deg}\, Y^\prim)$ if $n = 3$.} 

\vskip2mm

If one has a resolution $0 \ra L \ra F \ra \sci_Y \ra 0$ 
with $L$ a direct sum of line bundles and $F$ a vector bundle then a result 
attributed in \cite[Prop.~2.5]{ps} to D. Ferrand asserts that the 
diagram$\, :$  
\[
\xymatrix{{} & {} & {} & \sco_\p(-a)\oplus \sco_\p(-b) \ar[d] 
\ar @{-->}[dl]\\ 
0 \ar[r] & L \ar[r] & F \ar[r] & \sci_Y \ar[r] & 0}
\] 
induces a resolution$\, :$  
\[
0 \lra F^\ast \lra L^\ast \oplus \sco_\p(a) \oplus \sco_\p(b) \lra 
\sci_{Y^\prim}(a+b) \lra 0\, .
\] 
\end{remark} 

\begin{lemma}\label{L:iz(2)n2}
Let $Z$ be a $0$-dimensional subscheme of $\pii$ with $\sci_Z(2)$ globally 
generated. Then one of the following holds$\, :$ 
\begin{enumerate}
\item[(i)] $Z$ is a complete intersection of type $(1,1)$, $(1,2)$, or 
$(2,2)\, ;$ 
\item[(ii)] $\sci_Z(2)$ admits a resolution of the form$\, :$ 
\[
0 \lra 2\sco_\pii(-1) \lra 3\sco_\pii \lra \sci_Z(2) \lra 0\, .
\] 
\end{enumerate}
\end{lemma}

\begin{proof}
One has $\text{deg}\, Z \leq 4$. If $\text{deg}\, Z = 1$ then $Z$ is a 
complete intersection of type $(1,1)$. If $\text{deg}\, Z = 4$ then $Z$ is a 
complete intersection of type $(2,2)$. If $\text{deg}\, Z = 2$ then, using the 
exact sequence$\, :$ 
\[
0 \lra \sci_Z(1) \lra \sco_\pii(1) \lra \sco_Z(1) \lra 0 
\] 
one derives that $\tH^0(\sci_Z(1)) \neq 0$ hence $Z$ is a complete 
intersection of type $(1,2)$. 

Finally, if $\text{deg}\, Z = 3$ then $Z$ is linked by a complete intersection 
of type $(2,2)$ to a simple point hence, using Ferrand's result recalled in 
Remark~\ref{R:liaison}, $\sci_Z(2)$ admits a resolution of the form 
stated in item (ii). 
\end{proof}

\begin{prop}\label{P:h0e-c1+2n2} 
Let $E$ be a globally generated vector bundle of rank $r \geq 2$  
on $\pii$ with $c_1 \geq 3$ and such that ${\fam0 H}^i(E^\ast) = 0$, $i = 0, 1$. 
If ${\fam0 H}^0(E(-c_1+1)) = 0$ and ${\fam0 H}^0(E(-c_1+2)) \neq 0$ 
then $E \simeq \sco_\pii(c_1-2) \oplus F$, 
where $F$ is one of the bundles $\sco_\pii(2)$ $($for $c_1 \geq 4)$, 
$2\sco_\pii(1)$, $\sco_\pii(1) \oplus {\fam0 T}_\pii(-1)$, 
$2{\fam0 T}_\pii(-1)$, and $P(\sco_\pii(2))$.  
\end{prop} 

\begin{proof}
Considering $r-2$ general global sections of $E$ one gets an exact 
sequence$\, :$  
\[
0 \lra (r - 2)\sco_\pii \lra E \lra E^\prim \lra 0
\]
with $E^\prim$ a rank 2 vector bundle. Since 
$\tH^0(E^\prim(-c_1+2)) \neq 0$ and $\tH^0(E^\prim(-c_1+1)) = 0$ it follows that 
any non-zero global section of $E^\prim(-c_1+2)$ defines an exact 
sequence$\, :$  
\[
0 \lra \sco_\pii(c_1-2) \lra E^\prim \lra \sci_Z(2) \lra 0
\] 
with $Z$ a closed subscheme of $\pii$ which is either empty or 
0-dimensional. One deduces an exact sequence$\, :$  
\[
0 \lra \sco_\pii(c_1-2) \oplus (r - 2)\sco_\pii \lra E \lra \sci_Z(2) 
\lra 0\, .
\]
If $Z = \emptyset$ then $E \simeq \sco_\pii(c_1-2) \oplus \sco_\pii(2)$ (recall  
that $\tH^0(E^\ast) = 0$). If $Z$ is 0-dimensional then one can use 
Lemma~\ref{L:iz(2)n2}. Assume, firstly, that $Z$ is a complete intersection 
of type $(a,b)$, with $(a,b) \in \{(1,1), (1,2), (2,2)\}$. In this case, 
as at the beginning of the proof of Lemma~\ref{L:yci}, $E$ admits a 
resolution of the form$\, :$ 
\[
0 \ra \sco_\pii(-a-b+2) \lra \sco_\pii(c_1-2) \oplus \sco_\pii(-a+2) \oplus 
\sco_\pii(-b+2) \oplus (r - 2)\sco_\pii \lra E \ra 0 
\] 
Dualizing this resolution and taking into account that $\tH^i(E^\ast) = 0$, 
$i = 0, 1$, one deduces, as in the proof of Lemma~\ref{L:yci}, that 
$E \simeq \sco_\pii(c_1-2) \oplus F$ where $F = 2\sco_\pii(1)$ if 
$(a,b) = (1,1)$, $F = \sco_\pii(1) \oplus \text{T}_\pii(-1)$ if $(a,b) = 
(1,2)$, and $F = P(\sco_\pii(2))$ if $(a,b) = (2,2)$. 

Assume, finally, that $\sci_Z(2)$ admits a resolution as in 
Lemma~\ref{L:iz(2)n2}(ii). Then $E$ admits a resolution of the form$\, :$  
\[
0 \lra 2\sco_\pii(-1) \lra \sco_\pii(c_1-2) \oplus (r + 1)\sco_\pii  
\lra E \lra 0\, . 
\]
Dualizing this resolution and using the fact that $\tH^i(E^\ast) = 0$, 
$i = 0, 1$, one gets that $E \simeq \sco_\pii(c_1-2) \oplus 
2\text{T}_\pii(-1)$. 
\end{proof} 

\begin{lemma}\label{L:iz(2)n3}
Let $Z$ be a locally Cohen-Macaulay subscheme of $\piii$ of pure codimension 
$2$. If $\sci_Z(2)$ is globally generated then one of the following 
holds$\, :$ 
\begin{enumerate}
\item[(i)] $Z$ is a complete intersection of type $(1,1)$, $(1,2)$, or 
$(2,2)\, ;$
\item[(ii)] $\sci_Z(2)$ admits a resolution of the form$\, :$ 
\[
0 \lra 2\sco_\piii(-1) \lra 3\sco_\piii \lra \sci_Z(2) \lra 0\, ;
\]
\item[(iii)] $\sci_Z(2)$ admits a resolution of the form$\, :$
\[
0 \lra 2\sco_\piii \lra \Omega_\piii(2) \lra \sci_Z(2) \lra 0\, .
\]
\end{enumerate}
\end{lemma}

\begin{proof}
One has $\text{deg}\, Z \leq 4$. If $\text{deg}\, Z = 1$ then $Z$ is a 
complete intersection of type $(1,1)$. If $\text{deg}\, Z = 4$ then $Z$ is a 
complete intersection of type $(2,2)$. If $\text{deg}\, Z = 3$ then $Z$ is 
linked to a line by a complete intersection of type $(2,2)$ hence 
$\sci_Z(2)$ admits a resolution as in item (ii) of the statement (see, again, 
Remark~\ref{R:liaison}). 

If $\text{deg}\, Z = 2$ then there are three possibilities: $Z$ is a complete 
intersection of type $(1,2)$, $Z = L_1 \cup L_2$ with $L_1$ and $L_2$ 
disjoint lines, or $Z$ is a multiple structure supported by a line $L$. 

If $Z = L_1 \cup L_2$ then, using the fact that $\sco_Y = \sco_{L_1} \oplus 
\sco_{L_2}$ and the exact sequences$\, :$ 
\begin{gather*}
0 \lra \sco_\piii(-2) \lra 2\sco_\piii(-1) \lra \sco_\piii \lra 
\sco_{L_i} \lra 0,\  i = 0, 1,\\
0 \lra \sco_\piii \xra{\binom{\text{id}}{\text{id}}}  
\sco_\piii \oplus \sco_\piii \xra{(-\text{id}\, ,\, \text{id})} \sco_\piii \lra 
0\, ,  
\end{gather*} 
one gets that $\sci_Z$ is the middle cohomology of a monad$\, :$  
\[
0 \lra 2\sco_\piii(-2) \lra 4\sco_\piii(-1) \lra \sco_\piii 
\lra 0\, .
\]
Since the kernel of any epimorphism $4\sco_\piii(-1) \ra \sco_\piii$ is 
isomorphic to $\Omega_\piii$, one gets a resolution for $\sci_Z(2)$ as in 
item (iii) of the statement. 

Finally, it remains to analyse the case where $\text{deg}\, Z = 2$ and 
$Z_{\text{red}} = L$, with $L$ a line in $\piii$. In this case one uses 
standard techniques from the theory of multiple structures supported by 
smooth curves, cf. B\u{a}nic\u{a} and Forster~\cite{bf} and 
Manolache~\cite{m1}. One starts by showing that $\sci_L^2 \subseteq \sci_Z$. 
For that one needs the following \emph{algebraic observation}: 
there is a classical trick asserting that is $\scf$ is a 
torsion free sheaf on an integral scheme $X$ with $\text{Supp}\, \scf \neq X$ 
then $\scf = 0$. This can be generalized to a locally Cohen-Macaulay scheme 
$X$ of pure dimension $n \geq 1$ as it follows: if $\scf$ is a coherent sheaf 
on $X$ which can be locally embedded into $r\sco_X$, for some $r \geq 1$, 
and if $X \setminus \text{Supp}\, \scf$ is (topologically) dense in $X$ then 
$\scf = 0$ (if $A$ is a local Cohen-Macaulay ring, if $M$ is an $A$-submodule 
of $rA$ and if $M_\fp = (0)$, $\forall \, \fp$ minimal prime ideal of $A$,  
then $M = (0)$ because the prime ideals of $A$ that are minimal over 
$\text{Ann}_A(M)$ belong to $\text{Ass}_A(M) \subseteq \text{Ass}_A(rA) = 
\text{Ass}_A(A)$ and $\text{Ass}_A(A)$ consists of the minimal prime ideals 
of $A$).    

Now, if $\sci_{L,Z}/\sci_{L,Z}^2$ would have finite support then it would follow, 
from the Lemma of Nakayama, that $\sci_{L,Z}$ has finite support hence, 
from the above algebraic observation, $\sci_{L,Z}$ would be 0, which is not the 
case. It remains that the $\sco_L$-module $\sci_{L,Z}/\sci_{L,Z}^2$ has rank 
$\geq 1$. Using Hilbert polynomials, one deduces that the subscheme of $Z$ 
defined by $\sci_{L,Z}^2$ has degree $\geq 2$, hence it has, actually, degree 
2. Using Hilbert polynomials again, one gets that $\sci_{L,Z}^2$ has finite 
support hence, by the above algebraic observation, it is 0 as we have asserted. 

Since $\sci_Z(2)$ is globally generated and since the subscheme $L^{(1)}$ of 
$\piii$ defined by $\sci_L^2$ has degree 3, it follows that $\tH^0(\sci_Z(2))$ 
is strictly larger that $\tH^0(\sci_L^2(2))$. Assuming that $L$ is the line 
of equations $X_2 = X_3 = 0$, this means that $\tH^0(\sci_Z(2))$ contains a 
non-zero element of the form $fX_2 + gX_3$, with $f,\, g \in k[X_0,X_1]_1$. 
If $f$ and $g$ are linearly dependent then $\tH^0(\sci_Z(1)) \neq 0$ and $Z$ 
is a complete intersection of type $(1,2)$. 

If $f$ and $g$ are linearly independent, consider the morphism$\, :$  
\[
\phi = 
\begin{pmatrix}
f & g & X_2 & X_3\\
-X_3 & X_2 & 0 & 0
\end{pmatrix} 
: 4\sco_\piii \lra 2\sco_\piii(1)\, .
\] 
The degeneracy scheme $D(\phi)$ of $\phi$ is exactly $Z$. Let $\sck$ be the 
kernel of $\phi$. $\sck$ is a rank 2 reflexive sheaf with $c_1(\sck(1)) = 0$ 
and $\tH^0(\sck) = 0$. Moreover, since $\tH^0(\sck(1))$ contains the 
elements $(0,0,-X_3,X_2)$ and $(X_2,X_3,-f,-g)$, one has $\text{h}^0(\sck(1)) 
\geq 2$. One deduces that $\sck(1) \simeq 2\sco_\piii$. Now, applying the 
Snake Lemma to the diagram$\, :$ 
\[
\begin{CD}
0 @>>> 0 @>>> 4\sco_\piii @>>> 4\sco_\piii @>>> 0\\
@. @VVV @VV{\phi}V @VVV\\
0 @>>> \sco_\piii(1) @>>> 2\sco_\piii(1) @>>> \sco_\piii(1) @>>> 0
\end{CD}
\]
one gets an exact sequence$\, :$ 
\[
0 \lra 2\sco_\piii(-1) \lra \Omega_\piii(1) \lra \sco_\piii(1) \lra 
\Cok \phi \lra 0\, .
\] 
Using the defining property of Fitting ideals, it follows that $\Cok \phi 
\simeq \sco_{D(\phi)}(1) \simeq \sco_Z(1)$ hence $\sci_Z(2)$ admits a resolution 
as in item (iii) of the statement.   
\end{proof} 

\begin{prop}\label{P:h0e-c1+2n3} 
Let $E$ be a globally generated vector bundle on $\piii$ with $c_1 \geq 3$ and 
such that ${\fam0 H}^i(E^\ast) = 0$, $i = 0, 1$. If ${\fam0 H}^0(E(-c_1+1)) = 
0$ and ${\fam0 H}^0(E(-c_1+2)) \neq 0$ then $E \simeq \sco_\piii(c_1-2) \oplus 
F$, where $F$ is one of the bundles $\sco_\piii(2)$ $($for $c_1 \geq 4)$,  
$2\sco_\piii(1)$, $\sco_\piii(1) \oplus {\fam0 T}_\piii(-1)$, 
$\Omega_\piii(2)$, $2{\fam0 T}_\piii(-1)$, and $P(\sco_\piii(2))$. 
\end{prop}

\begin{proof}
We use the same kind of argument as in the proof of Prop.~\ref{P:h0e-c1+2n2}. 
Considering $r-2$ general global sections of $E$ one gets an exact 
sequence$\, :$  
\[
0 \lra (r - 2)\sco_\piii \lra E \lra \sce^\prim \lra 0
\]
with $\sce^\prim$ a rank 2 reflexive sheaf. Since 
$\tH^0(\sce^\prim(-c_1+2)) \neq 0$ and $\tH^0(\sce^\prim(-c_1+1)) = 0$ it follows 
that any non-zero global section of $\sce^\prim(-c_1+2)$ defines an exact 
sequence$\, :$ 
\[
0 \lra \sco_\piii(c_1-2) \lra \sce^\prim \lra \sci_Z(2) \lra 0
\] 
with $Z$ a closed subscheme of $\piii$ which is either empty or locally 
Cohen-Macaulay of pure codimension 2 (and locally complete intersection except 
at finitely many points). One deduces an exact sequence$\, :$  
\[
0 \lra \sco_\piii(c_1-2) \oplus (r - 2)\sco_\piii \lra E \lra \sci_Z(2) 
\lra 0\, .
\]
If $Z = \emptyset$ then $E \simeq \sco_\piii(c_1-2) \oplus \sco_\piii(2)$ 
(recall that $\tH^0(E^\ast) = 0$). If $Z$ is not empty then we can use 
Lemma~\ref{L:iz(2)n3} in the same way we used Lemma~\ref{L:iz(2)n2} in the 
proof of Prop.~\ref{P:h0e-c1+2n2}. There is only one case which has no 
analogue in the proof of Prop.~\ref{P:h0e-c1+2n2}, namely the case where 
$\sci_Z(2)$ admits a resolution as in Lemma~\ref{L:iz(2)n3}(iii). In this case, 
since $\text{Ext}^1(\Omega_\piii(2),\, \sco_\piii(c_1-2) \oplus 
(r - 2)\sco_\piii) = 0$ one deduces, as in the proof of 
Lemma~\ref{L:yci}, that $E$ admits a resolution of the form$\, :$  
\[
0 \lra 2\sco_\piii \lra \sco_\piii(c_1-2) \oplus (r - 2)\sco_\piii 
\oplus \Omega_\piii(2) \lra E \lra 0\, .
\] 
Dualizing this exact sequence and taking into account that $\tH^i(E^\ast) = 0$, 
$i = 0, 1$, one gets that $E \simeq \sco_\piii(c_1-2) \oplus 
\Omega_\piii(2)$. 
\end{proof} 

\begin{prop}\label{P:h0e-c1+2n4}
Let $E$ be a globally generated vector bundle on $\p^n$, $n\geq 4$, with 
$c_1 \geq 3$ and such that ${\fam0 H}^i(E^\ast) = 0$, $i = 0, 1$. Let   
$\Pi \subset \p^n$ be a fixed $3$-plane. If ${\fam0 H}^0(E(-c_1+1)) = 0$ 
and ${\fam0 H}^0(E_\Pi(-c_1+2)) \neq 0$ then one of the following holds$\, :$  
\begin{enumerate} 
\item[(i)] $E \simeq \sco_\p(c_1-2) \oplus F$, where $F$ is one of the bundles 
$\sco_\p(2)$ $($for $c_1 \geq 4$$)$, $2\sco_\p(1)$, 
$\sco_\p(1) \oplus {\fam0 T}_\p(-1)$, $2{\fam0 T}_\p(-1)$, and 
$P(\sco_\p(2))\, ;$ 
\item[(ii)] $n = 4$, $c_1 = 3$, and $E \simeq \Omega_\piv(2)$.  
\end{enumerate}  
\end{prop} 

\begin{proof}
It follows, from Corollary~\ref{C:h0e-c1+1}, that $\tH^0(E_\Pi(-c_1+1)) = 0$. 
Prop.~\ref{P:h0e-c1+2n3} and Lemma~\ref{L:h0h1} imply, now, that 
$E_\Pi \simeq G \oplus t\sco_\Pi$, with $G$ defined by an exact sequence$\, :$  
\[
0 \lra s\sco_\Pi \lra \sco_\Pi(c_1-2) \oplus F^\prim \lra G \lra 0
\]  
where $F^\prim$ is one of the bundles on $\Pi \simeq \piii$ appearing in the 
last part of the statement of Prop.~\ref{P:h0e-c1+2n3}. One can conclude the 
proof by induction on $n \geq 4$, using Lemma~\ref{L:lift1} and 
Lemma~\ref{L:lift2}. 
\end{proof} 

\begin{remark}\label{R:c1leq3}
The results of this section can be used to settle completely the case 
$c_1 \leq 3$. \emph{Indeed}, let $E$ be a globally generated vector bundle on 
$\p^n$ with $c_1 \leq 3$, $c_2 \leq c_1^2/2$ and such that $\tH^i(E^\ast) = 0$, 
$i = 0, 1$. As we noticed before the statement of Prop.~\ref{P:c2=c1-1}, 
the cases $c_1 = 0$ and $c_1 = 1$ are easy. Prop.~\ref{P:c2=c1-1} and 
Prop.~\ref{P:c2=c1} settle the case $c_1 = 2$. They also settle the case 
$c_1 = 3$, $c_2 \leq 3$. It remains to consider only the case $c_1 = 3$ and 
$c_2 = 4$. This case can be split into several subcases according to the 
values of $n$. 

\vskip2mm 

\noindent
$\bullet$\quad If $n = 2$ then $Y$ consists of four simple points, hence 
$\tH^0(\sci_Y(2)) \neq 0$, which implies that $\tH^0(E(-1)) \neq 0$. 
Moreover, since 
$\sci_Y(3)$ is globally generated, it follows that $\tH^0(\sci_Y(1)) = 0$, 
hence $\tH^0(E(-2)) = 0$. Prop.~\ref{P:h0e-c1+2n2} implies, now, that 
$E \simeq 2\sco_\pii(1) \oplus \text{T}_\pii(-1)$. 

\vskip2mm

\noindent
$\bullet$\quad If $n = 3$ then $Y$ is a nonsingular curve of degree 4 with 
$\omega_Y(1)$ globally generated. Since $Y$ cannot be the union of two 
disjoint conics (because, in that case, the intersection line of the planes 
supporting the conics would be a 4-secant of $Y$ and this would contradict 
the fact that $\sci_Y(3)$ is globally generated), $Y$ must be connected, 
hence $Y$ is a complete intersection of type $(2,2)$ or a rational curve of 
degree 4. In the former case, by Lemma~\ref{L:yci}, $E \simeq 
2\sco_\piii(1) \oplus \text{T}_\piii(-1)$. In the latter case, one has 
$\tH^0(\sci_Y(2)) \neq 0$ and $\tH^0(\sci_Y(1)) = 0$, hence $\tH^0(E(-1)) \neq 
0$ and $\tH^0(E(-2)) = 0$. Prop.~\ref{P:h0e-c1+2n3} implies, now, that 
$E \simeq \sco_\piii(1) \oplus \Omega_\piii(2)$. 

\vskip2mm 

\noindent 
$\bullet$\quad For $n \geq 4$ one can use induction on $n$ based on 
Lemma~\ref{L:lift1} and Lemma~\ref{L:lift2}. It follows that either 
$E \simeq 2\sco_\p(1) \oplus \text{T}_\p(-1)$ or $n = 4$ and 
$E \simeq \Omega_\piv(2)$.    
\end{remark}

\section{The cases $c_1=4$ and $c_1 = 5$ on $\pii$}
\label{S:n=2c1=4} 

This section has two parts. In the first part we complete the classification 
of globally generated vector bundles with $c_1 = 4$ on $\pii$ by considering 
the cases $5 \leq c_2 \leq 8$. Besides its own interest, this part was 
meant to expose, in a simpler situation, the method we shall use to 
classify g.~g. vector bundles on $\piii$. More precisely, if 
$E$ is a g.~g. vector bundle on $\pii$ with $c_1 = 4$ then 
$r-2$ general global sections of $E$ define an exact squence$\, :$ 
\[
0 \lra (r-2)\sco_\pii \lra E \lra E^\prim \lra 0 
\]
where $E^\prim$ is a g.~g. rank 2 vector bundle with the same 
Chern classes as $E$. Consider the normalized bundle $F= E^\prim (-2)$ which 
has $c_1(F) = 0$. The bundles with $\tH^0(E(-2)) \neq 0$ were classified in 
Prop.~\ref{P:h0e-c1+2n2} ($-2 = -c_1+2$). If $\tH^0(E(-2)) = 0$ then 
$\tH^0(F) = 0$ hence $F$ is \emph{stable} and one might take advantage of the 
properties of stable rank 2 vector bundles on $\pii$. 

This was the initial plan but we realized, meanwhile, that it is easy to 
classify the 0-dimensional subschemes $Z$ of $\pii$ with $\sci_Z(3)$ globally 
generated and this leads to the classification of the g.~g. 
vector bundles on $\pii$ with $\tH^0(E(-c_1+3)) \neq 0$. The latter 
classification settles immediately the case $c_1 = 4$. 

\vskip2mm 

Stable rank 2 vector bundles on $\pii$ (this time with $c_1(F) = -1$) are, 
however, necessary for the classification of g.~g. vector bundles 
on $\pii$ with $c_1 = 5$, which we accomplish in the second part of the 
section. It turns out that the classification of g.~g. vector 
bundles $E$ on $\pii$ with $c_1 \leq 5$ and such that $\tH^i(E^\ast) = 0$, 
$i = 0,\, 1$, is discrete. They are direct sums of line bundles and of 
bundles appearing in the following list$\, :$ $\text{T}_\pii(-1)$, 
$P(\sco_\pii(a))$, $2 \leq a \leq 5$, $\text{T}_\pii(1)$, and 
$P(\text{T}_\pii(1))$. $\text{T}_\pii$ does not appear in this list because 
$\tH^1(\text{T}_\pii^\ast) \neq 0$; actually, one has an exact sequence 
$0 \ra \sco_\pii \ra 3\sco_\pii(1) \ra \text{T}_\pii \ra 0$. 

\vskip2mm   
 
We begin, now, the classification of globally generated vector bundles on 
$\pii$ with $c_1 = 4$ and $5 \leq c_2 \leq 5$. 
The results stated in Lemma~\ref{L:iz(2)n2} and Prop.~\ref{P:h0e-c1+2n2} can 
be easily pushed one step further. 

\begin{lemma}\label{L:iz(3)n2}
Let $Z$ be a $0$-dimensional subscheme of $\pii$ with $\sci_Z(3)$ globally 
generated. Then one of the following holds$\, :$ 
\begin{enumerate}
\item[(i)] $Z$ is a complete intersection of type $(a, b)$, 
$1 \leq a \leq b \leq 3\, ;$ 
\item[(ii)] $\sci_Z(3)$ admits a resolution of the form$\, :$ 
\[
0 \lra 2\sco_\pii \lra 3\sco_\pii(1) \lra \sci_Z(3) \lra 0\, ;
\]
\item[(iii)] $\sci_Z(3)$ admits a minimal resolution of the form$\, :$ 
\[
0 \lra \sco_\pii \oplus \sco_\pii(-1) \lra 2\sco_\pii(1) \oplus \sco_\pii 
\lra \sci_Z(3) \lra 0\, ;
\]
\item[(iv)] $\sci_Z(3)$ admits a resolution of the form$\, :$ 
\[
0 \lra 2\sco_\pii(-1) \lra \sco_\pii(1) \oplus 2\sco_\pii \lra 
\sci_Z(3) \lra 0\, ;
\]
\item[(v)] $\sci_Z(3)$ admits a resolution of the form$\, :$ 
\[
0 \lra 3\sco_\pii(-1) \lra 4\sco_\pii \lra \sci_Z(3) \lra 0\, ;
\] 
\item[(vi)] $\sci_Z(3)$ admits a resolution of the form$\, :$ 
\[
0 \lra \sco_\pii(-1) \oplus \sco_\pii(-2) \lra 3\sco_\pii \lra 
\sci_Z(3) \lra 0\, .
\]
\end{enumerate} 
\end{lemma}

\begin{proof}
One has $\text{deg}\, Z \leq 9$. If $\text{deg}\, Z = 1$, 2, or 9 then $Z$ is 
a complete intersection (for $\text{deg}\, Z = 2$ see the proof of 
Lemma~\ref{L:iz(2)n2}). We analyse the remaining degrees case by case.  
Consider the exact sequence$\, :$ 
\begin{equation}\label{E:izopiioz} 
0 \lra \sci_Z \lra \sco_\pii \lra \sco_Z \lra 0\, .
\end{equation} 

\noindent 
$\bullet$ If $\text{deg}\, Z = 3$ and $\tH^0(\sci_Z(1)) \neq 0$ then $Z$ is a 
complete intersection of type $(1,3)$. If $\tH^0(\sci_Z(1)) = 0$ then, using 
the exact sequence \eqref{E:izopiioz}, one gets that $\tH^1(\sci_Z(1)) = 0$, 
hence $\sci_Z$ is 2-regular. Using the exact sequence \eqref{E:izopiioz} and 
the fact that $\tH^1(\sci_Z(2)) = 0$ (by the Lemma of Castelnuovo-Mumford) one 
deduces that $\text{h}^0(\sci_Z(2)) = 3$ hence $\sci_Z(2)$ admits a linear  
resolution of the form$\, :$  
\[
0 \lra 2\sco_\pii(-1) \lra 3\sco_\pii \lra \sci_Z(2) \lra 0\, .
\]

\noindent
$\bullet$ If $\text{deg}\, Z = 4$ then one gets, using the exact sequence 
\eqref{E:izopiioz}, that $\tH^0(\sci_Z(2)) \neq 0$ hence $Z$ is linked, by a 
complete intersection of type $(2,3)$, to a 0-dimensional scheme $Z^\prim$ 
with $\text{deg}\, Z^\prim = 2$. $Z^\prim$ must be a complete intersection of 
type $(1,2)$ hence, using Ferrand's result recalled in Remark~\ref{R:liaison}, 
one deduces that $\sci_Z(5)$ admits a resolution of the form$\, :$  
\[
0 \lra \sco_\pii(1) \oplus \sco_\pii(2) \lra \sco_\pii(2) \oplus 
2\sco_\pii(3) \lra \sci_Z(5) \lra 0\, .
\] 
If this resolution is not minimal then $Z$ is a complete intersection of type 
$(2,2)$. 

\vskip2mm 

\noindent
$\bullet$ If $\text{deg}\, Z = 5$ then, as in the previous case, $Z$ is 
linked, by a complete intersection of type $(2,3)$, to a simple point hence 
$\sci_Z(5)$ admits a resolution of the form$\, :$ 
\[
0 \lra 2\sco_\pii(1) \lra 2\sco_\pii(2) \oplus \sco_\pii(3) \lra 
\sci_Z(5) \lra 0\, .
\] 

\noindent
$\bullet$ If $\text{deg}\, Z = 6$ and $\tH^0(\sci_Z(2)) \neq 0$ then $Z$ is a 
complete intersection of type $(2,3)$. If $\tH^0(\sci_Z(2)) = 0$ then, using 
the exact sequence \eqref{E:izopiioz}, one gets that $\tH^1(\sci_Z(2)) = 0$, 
hence $Z$ is 3-regular. Using the exact sequence \eqref{E:izopiioz} and the 
fact that $\tH^1(\sci_Z(3)) = 0$, one deduces that $\text{h}^0(\sci_Z(3)) = 4$ 
hence $\sci_Z(3)$ admits a linear resolution of the form$\, :$  
\[
0 \lra 3\sco_\pii(-1) \lra 4\sco_\pii \lra \sci_Z(3) \lra 0\, .
\] 

\noindent
$\bullet$ If $\text{deg}\, Z = 7$ then $Z$ is linked, by a complete 
intersection of type $(3,3)$, to a 0-dimensional scheme $Z^\prim$ of degree 2. 
One deduces, as in the case $\text{deg}\, Z = 4$, that $\sci_Z(6)$ admits a 
resolution of the form$\, :$  
\[
0 \lra \sco_\pii(1) \oplus \sco_\pii(2) \lra 3\sco_\pii(3) \lra 
\sci_Z(6) \lra 0\, .
\] 

\noindent 
$\bullet$ If $\text{deg}\, Z = 8$ then $Z$ is linked, by a complete 
intersection of type $(3,3)$, to a simple point hence $\sci_Z(6)$ admits a 
resolution of the form$\, :$  
\[
0 \lra 2\sco_\pii(1) \lra \sco_\pii(2) \oplus 2\sco_\pii(3)  
\lra \sci_Z(6) \lra 0\, .
\]
But this implies that $\sci_Z(3)$ is not globally generated, hence 
\emph{this case cannot occur}. 
\end{proof} 

\begin{prop}\label{P:h0e-c1+3n2} 
Let $E$ be a globally generated vector bundle of rank $r \geq 2$ on $\pii$ 
with $c_1 \geq 4$ and such that ${\fam0 H}^i(E^\ast) = 0$, $i = 0, 1$. 
If ${\fam0 H}^0(E(-c_1+2)) = 0$ and ${\fam0 H}^0(E(-c_1+3)) \neq 0$ then one 
of the following holds$\, :$ 
\begin{enumerate}
\item[(i)] $E \simeq \sco_\pii(c_1-3) \oplus F$, where $F$ is one of the 
bundles $\sco_\pii(3)$ $($for $c_1 \geq 6)$, $\sco_\pii(2) \oplus \sco_\pii(1)$  
$($for $c_1 \geq 5)$, $\sco_\pii(2) \oplus {\fam0 T}_\pii(-1)$ 
$($for $c_1 \geq 5)$, 
$3\sco_\pii(1)$, $2\sco_\pii(1) \oplus {\fam0 T}_\pii(-1)$, 
$\sco_\pii(1) \oplus 2{\fam0 T}_\pii(-1)$, $3{\fam0 T}_\pii(-1)$, 
$\sco_\pii(1) \oplus P(\sco_\pii(2))$, ${\fam0 T}_\pii(-1) \oplus 
P(\sco_\pii(2))$, and $P(\sco_\pii(3))\, ;$  
\item[(ii)] $c_1 \geq 5$ and $E$ admits a resolution of the form$\, :$ 
\[
0 \lra \sco_\pii(1) \lra \sco_\pii(c_1-3) \oplus 2\sco_\pii(2) 
\lra E \lra 0\, .
\] 
\end{enumerate}
\end{prop} 

\begin{proof} 
As at the beginning of the proof of Prop.~\ref{P:h0e-c1+2n2}, one deduces an 
exact sequence$\, :$ 
\[
0 \lra \sco_\pii(c_1-3) \oplus (r - 2)\sco_\pii \lra E \lra \sci_Z(3) 
\lra 0\, ,
\]
with $Z$ a closed subscheme of $\pii$ which is either empty or 0-dimensional. 
If $Z = \emptyset$ then $E \simeq \sco_\pii(c_1-3) \oplus \sco_\pii(3)$. If 
$Z \neq \emptyset$, one splits the proof into several cases, according to 
the items in the statement of Lemma~\ref{L:iz(3)n2}.   

\vskip2mm 

\noindent 
$\bullet$ If $Z$ is a complete intersection of type $(a,b)$, $1 \leq a \leq b 
\leq 3$ then, as at the beginning of the proof of Lemma~\ref{L:yci}, $E$ 
admits a resolution of the form$\, :$ 
\[
0 \ra \sco_\pii(-a-b+3) \ra \sco_\pii(c_1-3) \oplus \sco_\pii(-a+3) \oplus 
\sco_\pii(-b+3) \oplus (r - 2)\sco_\pii \ra E \ra 0\, .
\]
Dualizing this resolution and taking into account that $\tH^i(E^\ast) = 0$, 
$i = 0, 1$, one deduces, as in the proof of Lemma~\ref{L:yci}, that, except 
the case where $(a,b) = (1,1)$,  
$E \simeq \sco_\pii(c_1-3) \oplus F$ where $F = \sco_\pii(2) \oplus \sco_\pii(1)$ 
if $(a,b) = (1,2)$, $F = \sco_\pii(2) \oplus \text{T}_\pii(-1)$ if $(a,b) = 
(1,3)$, $F = 2\sco_\pii(1) \oplus \text{T}_\pii(-1)$ if $(a,b) = (2,2)$, 
$F = \sco_\pii(1) \oplus P(\sco_\pii(2))$ if $(a,b) = (2,3)$, and 
$F = P(\sco_\pii(3))$ if $(a,b) = (3,3)$. 

If $(a,b) = (1,1)$ then, since $E$ has no trivial direct summand, either 
$c_1 = 4$ and $E \simeq 2\sco_\pii(2)$, but this contradicts the fact that 
$\tH^0(E(-c_1+2)) = 0$, or $c_1 \geq 5$ and $E$ admits a resolution of the 
form$\, :$ 
\[
0 \lra \sco_\pii(1) \lra \sco_\pii(c_1-3)\oplus 2\sco_\pii(2) \lra 
E \lra 0\, .
\] 

\vskip2mm

\noindent
$\bullet$ If $Z$ is as in Lemma~\ref{L:iz(3)n2}(ii) then $E$ admits a 
resolution of the form$\, :$  
\[
0 \lra 2\sco_\pii \lra \sco_\pii(c_1-3) \oplus 3\sco_\pii(1) \oplus 
(r - 2)\sco_\pii \lra E \lra 0 
\]  
from which one deduces that $E \simeq \sco_\pii(c_1-3) \oplus 
3\sco_\pii(1)$. 

\vskip2mm 

\noindent
$\bullet$ If $Z$ is as in Lemma~\ref{L:iz(3)n2}(iii) then $E$ admits a 
resolution of the form$\, :$  
\[
0 \lra \sco_\pii \oplus \sco_\pii(-1) \lra \sco_\pii(c_1-3) \oplus 
2\sco_\pii(1) \oplus (r - 1)\sco_\pii \lra E \lra 0 
\]  
from which one deduces that $E \simeq \sco_\pii(c_1-3) \oplus 
2\sco_\pii(1) \oplus \text{T}_\pii(-1)$. 

\vskip2mm

\noindent
$\bullet$ If $Z$ is as in Lemma~\ref{L:iz(3)n2}(iv) then $E$ admits a 
resolution of the form$\, :$ 
\[
0 \lra 2\sco_\pii(-1) \lra \sco_\pii(c_1-3) \oplus 
\sco_\pii(1) \oplus r\sco_\pii \lra E \lra 0 
\]  
from which one deduces that $E \simeq \sco_\pii(c_1-3) \oplus 
\sco_\pii(1) \oplus 2\text{T}_\pii(-1)$. 

\vskip2mm 

\noindent
$\bullet$ If $Z$ is as in Lemma~\ref{L:iz(3)n2}(v) then $E$ admits a 
resolution of the form$\, :$ 
\[
0 \lra 3\sco_\pii(-1) \lra \sco_\pii(c_1-3) \oplus (r + 2)\sco_\pii  
\lra E \lra 0 
\]  
from which one deduces that $E \simeq \sco_\pii(c_1-3) \oplus 
3\text{T}_\pii(-1)$. 

\vskip2mm

\noindent
$\bullet$ If $Z$ is as in Lemma~\ref{L:iz(3)n2}(vi) then $E$ admits a 
resolution of the form$\, :$  
\[
0 \lra \sco_\pii(-1) \oplus \sco_\pii(-2) \lra \sco_\pii(c_1-3) \oplus 
(r + 1)\sco_\pii \lra E \lra 0 
\]  
from which one deduces that $E \simeq \sco_\pii(c_1-3) \oplus \text{T}_\pii(-1)  
\oplus P(\sco_\pii(2))$.
\end{proof} 

\begin{note} 
The result from Prop.~\ref{P:h0e-c1+3n2} has been extended to globally 
generated vector bundles on $\p^n$, $n \geq 3$, in the recent paper of the 
authors ``Locally Cohen-Macaulay space curves defined by cubic equations and 
globally generated vector bundles'', \texttt{arXiv:1502.05553}. 
\end{note}

\begin{prop}\label{P:n=2c1=4} 
Let $E$ be a globally generated vector bundle on $\pii$ with $c_1 = 4$, 
$5 \leq c_2 \leq 8$, and such that ${\fam0 H}^i(E^\ast) = 0$, $i = 0, 1$. 
Then one of the following holds$\, :$ 
\begin{enumerate}
\item[(i)] $c_2 = 5$ and $E \simeq 2\sco_\pii(1) \oplus 
\sco_\pii(2)$$\, ;$  
\item[(ii)] $c_2 = 6$ and $E \simeq \sco_\pii(1) \oplus \sco_\pii(2) \oplus 
{\fam0 T}_\pii(-1)$$\, ;$  
\item[(iii)] $c_2 = 6$ and $E \simeq 4\sco_\pii(1)$$\, ;$  
\item[(iv)] $c_2 = 7$ and $E \simeq \sco_\pii(2) \oplus 
2{\fam0 T}_\pii(-1)$$\, ;$  
\item[(v)] $c_2 = 7$ and $E \simeq 3\sco_\pii(1) \oplus 
{\fam0 T}_\pii(-1)$$\, ;$  
\item[(vi)] $c_2 = 8$ and $E \simeq \sco_\pii(2) \oplus P(\sco_\pii(2))$$\, ;$  
\item[(vii)] $c_2 = 8$ and $E \simeq 2\sco_\pii(1) \oplus 
2{\fam0 T}_\pii(-1)$. 
\end{enumerate} 
\end{prop} 

\begin{proof}
One has an exact sequence $0 \ra (r - 1)\sco_\pii \ra E \ra \sci_Y(4) \ra 0$, 
where $Y$ is a closed subscheme of $\pii$ consisting of $c_2$ simple points. 
Using the exact sequence$\, :$ 
\[
0 \lra \sci_Y(3) \lra \sco_\pii(3) \lra \sco_Y(3) \lra 0 
\]
one deduces that $\text{h}^0(\sci_Y(3)) \geq 10 - c_2 \geq 2$. It follows that 
$\tH^0(E(-1)) \neq 0$. One can, now, use Prop.~\ref{P:h0e-c1+1}, 
Prop.~\ref{P:h0e-c1+2n2}, and Prop.~\ref{P:h0e-c1+3n2}. 
\end{proof} 

\vskip2mm

We begin the classification of globally generated vector bundles with 
$c_1 = 5$ on $\pii$ by recalling some facts about rank 2 vector bundles on the 
projective plane $\pii$. Let $F$ be such a bundle, with Chern classes 
$c_1(F) = c$ and $c_2(F) = d$. We assume that $F$ is 
\emph{normalized} so that $c \in \{0,-1\}$. Since $\text{rk}\, F = 2$, 
$F^\ast \simeq F(-c)$. The \emph{Riemann-Roch formula} says, 
in this particular case, that$\, :$
\[
\chi(F(l)) = \chi((\sco_\pii \oplus \sco_\pii(c))(l)) - d,\  \forall \, 
l \in \z\, .
\]

If $c=0$ then $F$ is semistable iff $\tH^0(F(-1)) = 0$ and $F$ is stable iff 
$\tH^0(F) = 0$. If $c = -1$ then $F$ is stable iff $\tH^0(F) = 0$. We shall 
need the following fact which should be well known$\, :$ 

\begin{lemma}\label{L:h1f} 
Under the above hypotheses, let $S = k[X_0,X_1,X_2]$ be the projective 
coordinate ring of $\pii$ and consider the graded $S$-module 
${\fam0 H}^1_\ast(F)$.  

\emph{(a)} If $c=0$ and $F$ is semistable then ${\fam0 H}^1_\ast(F)$ 
is generated in degrees $\leq -2$. 

\emph{(b)} If $c = -1$ and $F$ is stable then  
${\fam0 H}^1_\ast(F)$ is generated in degrees $\leq -2$, except when $F \simeq 
{\fam0 T}_\pii(-2) \simeq \Omega_\pii(1)$. 
\end{lemma} 

\begin{proof}
(a) Tensorizing by $F(i)$, $i \geq -1$, the exact sequence$\, :$  
\[
0 \lra \Omega_\pii \lra S_1\otimes_k\sco_\pii(-1) \lra \sco_\pii \lra 0 
\] 
one gets an exact sequence$\, :$  
\[
S_1\otimes_k\tH^1(F(i-1)) \lra \tH^1(F(i)) \lra \tH^2(\Omega_\pii \otimes F(i)) 
\, .
\]
Taking into account that $\Omega_\pii \simeq \Omega_\pii^\ast \otimes 
\omega_\pii$, one derives, from Serre Duality, that 
\[
\tH^2(\Omega_\pii \otimes F(i)) \simeq \tH^0(\Omega_\pii \otimes F^\ast(-i))^\ast 
\, .
\]
Now, $F^\ast \simeq F(-c) = F$ and one has an exact sequence$\, :$  
\[
0 \lra \tH^0(\Omega_\pii \otimes F(-i)) \lra S_1\otimes_k\tH^0(F(-i-1)) \lra 
\tH^0(F(-i))\, .
\]
Since $F$ is semistable, $\tH^0(F(-i-1)) = 0$ for $i\geq 0$ and, except the 
case $F \simeq 2\sco_\pii$, one has $\text{h}^0(F) \leq 1$. One 
deduces that $\tH^0(\Omega_\pii \otimes F(-i)) = 0$ for $i \geq -1$. 

(b) Using analogous arguments, one sees that the assertion from the statement 
is true if $\text{h}^0(F(1)) \leq 1$. Assume, now, that $\text{h}^0(F(1)) 
\geq 2$. The scheme of zeroes of a non-zero section $s \in \tH^0(F(1))$ is a 
0-dimensional subscheme $Z$ of $\pii$, of degree $c_2(F) = d$, and one has an 
exact sequence$\, :$  
\[
0 \lra \sco_\pii(-1) \overset{s}{\lra} F \lra \sci_Z \lra 0\, .
\]   
The condition $\text{h}^0(F(1)) \geq 2$ implies that $\tH^0(\sci_Z(1)) \neq 0$ 
hence $Z$ is a complete intersection of type $(1,d)$. Using the same kind of 
arguments as at the beginning of the proof of Lemma~\ref{L:yci}, one derives 
that $F$ admits a resolution of the form$\, :$  
\[
0 \lra \sco_\pii(-d-1) \lra 2\sco_\pii(-1) \oplus \sco_\pii(-d) \lra 
F \lra 0\, .
\]
Dualizing this resolution and taking into account that $F^\ast \simeq F(1)$, 
one gets an exact sequence$\, :$ 
\[
0 \lra F(1) \lra 2\sco_\pii(1) \oplus \sco_\pii(d) \lra \sco_\pii(d+1) 
\lra 0
\]
from which one deduces that $\tH^1_\ast(F)$ is generated by $\tH^1(F(-d))$. 
It follows that $\tH^1_\ast(F)$ is generated in degrees $\leq -2$, except when 
$d = 1$. But if $d = 1$ then $F \simeq \text{T}_\pii(-2)$. 
\end{proof} 

Now, coming to the subject of this second part of the section, 
let $E$ be a globally generated rank $r$ vector bundle on $\pii$, with 
$c_1 = 5$, $6 \leq c_2 \leq 12$, and such that $\tH^i(E^\ast) = 0$, $i = 0, 1$. 
One has exact sequences$\, :$  
\begin{gather*}
0 \lra (r - 2)\sco_\pii \lra E \lra E^\prim \lra 0\, ,\\
0 \lra \sco_\pii \lra E^\prim \lra \sci_Y(5) \lra 0\, , 
\end{gather*} 
where $E^\prim$ is a globally generated rank 2 vector bundle and $Y$ consists 
of $c_2$ simple points. Consider the ``normalized'' vector bundle $F := 
E^\prim(-3)$. One has $c_1(F) = - 1$ and $c_2(F) = c_2 - 6$. 
\emph{Assume} that, moreover, $F$ is stable, i.e., $\tH^0(F) = 0$, and that it 
is not isomorphic to $\text{T}_\pii(-2)$.  
Dualizing the first of the above exact sequences and taking into account 
that $F^\ast \simeq F(1)$, one gets an exact sequence$\, :$  
\[
0 \lra F(-2) \lra E^\ast \lra (r - 2)\sco_\pii \lra 0
\]
from which one deduces that the graded $S := k[X_0,X_1,X_2]$-module 
$\tH^1_\ast(E^\ast)$ can be obtained from $\tH^1_\ast(F)(-2)$ by erasing the 
homogeneous components of degree $\geq 0$ (see Lemma~\ref{L:h1f}(b)). 
Moreover, since $\tH^i(E^\ast) = 0$, $i = 0, 1$, one has 
$r-2 = \text{h}^1(F(-2))$.  
By Serre duality, 
$\tH^2(F(-2)) \simeq \tH^0(F^\ast(-1))^\ast \simeq \tH^0(F)^\ast = 0$ hence, 
by the Riemann-Roch formula, $\text{h}^1(F(-2)) = c_2(F) - 1 = c_2 -7$.  
It follows that, under the above additional assumption (that $F$ is stable 
and not isomorphic to $\text{T}_\pii(-2)$),   
\begin{equation}
\label{E:r}
r = c_2 - 5\, .
\end{equation} 

\begin{lemma}\label{L:h1epriml} 
Let $E^\prim$ be a globally generated vector bundle on $\pii$ and let 
$l \geq -1$ be an integer. If ${\fam0 H}^1(E^\prim(l)) \neq 0$ then 
${\fam0 h}^1(E^\prim(l)) \leq {\fam0 h}^1(E^\prim(l-1)) - 2$. 
\end{lemma}
 
\begin{proof} 
Consider an arbitrary linear form $0 \neq \ell \in \tH^0(\sco_\pii(1))$ and let 
$L \subset \pii$ be the line of equation $\ell = 0$. From Grothendieck's 
theorem and from the fact that $E^\prim_L$ is globally generated it follows 
that $\tH^1(E^\prim_L(l)) = 0$. Using the exact sequence$\, :$ 
\[
\tH^1(E^\prim(l-1)) \overset{\ell}{\lra} \tH^1(E^\prim(l)) \lra 
\tH^1(E^\prim_L(l)) = 0\, , 
\]
one deduces that the multiplication by $\ell : \tH^1(E^\prim(l-1)) \lra 
\tH^1(E^\prim(l))$ is surjective. Applying the Bilinear map lemma 
(see Hartshorne~\cite[Lemma~5.1]{ha}) to the bilinear map: 
\[
\tH^0(\sco_\pii(1)) \times \tH^1(E^\prim(l))^\ast \lra 
\tH^1(E^\prim(l-1))^\ast 
\] 
deduced from the multiplication map $\tH^0(\sco_\pii(1)) \times 
\tH^1(E^\prim(l-1)) \ra \tH^1(E^\prim(l))$, one gets the desired inequality. 
\end{proof} 

\begin{prop}\label{P:c1=5n2} 
Let $E$ be an indecomposable globally generated vector bundle on $\pii$ with 
$c_1 = 5$ and such that ${\fam0 H}^i(E^\ast) = 0$, $i = 0,\, 1$. Then $E$ is 
isomorphic to one of the bundles $\sco_\pii(5)$, ${\fam0 T}_\pii(1)$, 
$P({\fam0 T}_\pii(1))$, or $P(\sco_\pii(5))$. 
\end{prop} 

\begin{proof}
We shall, actually, enumerate all the globally generated vector bundles $E$ 
on $\pii$ with $c_1 = 5$ and such that $\tH^i(E^\ast) = 0$, $i = 0,\, 1$. 
Let $E$ be such a bundle. Assume, moreover, that $c_2 \leq 12$ (see Remark 
\ref{R:reduction}). If $c_2 = 0$ then $E \simeq \sco_\pii(5)$. 
If $c_2 > 0$ then, by Prop.~\ref{P:c2=c1-1}, $c_2 \geq 4$ and if $c_2 = 4$ 
then $E \simeq \sco_\pii(4) \oplus \sco_\pii(1)$. If $c_2 = 5$ then, by 
Prop.~\ref{P:c2=c1}, $E \simeq \sco_\pii(4) \oplus \text{T}_\pii(-1)$. 
Assume, now, that $c_2 \geq 6$. Prop.~\ref{P:h0e-c1+1} implies that, 
in this case, $\tH^0(E(-4)) = 0$. 

\vskip2mm 

If $\tH^0(E(-3)) \neq 0$ then, by Prop.~\ref{P:h0e-c1+2n2}, one of the 
following holds$\, :$ 
\begin{enumerate}
\item[(i)] $c_2 = 6$ and $E \simeq \sco_\pii(3) \oplus \sco_\pii(2)$$\, ;$  
\item[(ii)] $c_2 = 7$ and $E \simeq \sco_\pii(3) \oplus 2\sco_\pii(1)$$\, ;$  
\item[(iii)] $c_2 = 8$ and $E \simeq \sco_\pii(3) \oplus \sco_\pii(1) 
\oplus \text{T}_\pii(-1)$$\, ;$ 
\item[(iv)] $c_2 = 9$ and $E \simeq \sco_\pii(3) \oplus 
2\text{T}_\pii(-1)$$\, ;$  
\item[(v)] $c_2 = 10$ and $E \simeq \sco_\pii(3) \oplus P(\sco_\pii(2))$. 
\end{enumerate} 

\vskip2mm 

Assume, from now on, that $\tH^0(E(-3)) = 0$. If $\tH^0(E(-2)) \neq 0$ then, 
by Prop.~\ref{P:h0e-c1+3n2}, one of the following holds$\, :$  
\begin{enumerate}
\item[(vi)] $c_2 = 7$ and $E \simeq \text{T}_\pii(1)$ (cf. 
Prop.~\ref{P:h0e-c1+3n2}(ii))$\, ;$ 
\item[(vii)] $c_2 = 8$ and $E \simeq 2\sco_\pii(2) \oplus \sco_\pii(1)$$\, ;$  
\item[(viii)] $c_2 = 9$ and $E \simeq 2\sco_\pii(2) \oplus 
\text{T}_\pii(-1)$$\, ;$ 
\item[(ix)] $c_2 = 9$ and $E \simeq \sco_\pii(2) \oplus 3\sco_\pii(1)$$\, ;$  
\item[(x)] $c_2 = 10$ and $E \simeq \sco_\pii(2) \oplus 2\sco_\pii(1)  
\oplus \text{T}_\pii(-1)$$\, ;$ 
\item[(xi)] $c_2 = 11$ and $E \simeq \sco_\pii(2) \oplus \sco_\pii(1) \oplus 
2\text{T}_\pii(-1)$$\, ;$  
\item[(xii)] $c_2 = 12$ and $E \simeq \sco_\pii(2) \oplus 
3\text{T}_\pii(-1)$$\, ;$  
\item[(xiii)] $c_2 = 12$ and $E \simeq \sco_\pii(2) \oplus \sco_\pii(1) \oplus 
P(\sco_\pii(2))$. 
\end{enumerate} 

\vskip2mm 

Finally, let us assume that $\tH^0(E(-2)) = 0$. Consider the rank 2 vector 
bundles $E^\prim$ and $F$ defined after the proof of Lemma~\ref{L:h1f}. One has 
$\tH^0(F(1)) \simeq \tH^0(E(-2)) = 0$. Riemann-Roch implies that$\, :$  
\[
\text{h}^1(F(1)) = c_2(F) - 4 = c_2 - 10\, .
\]
In particular, $c_2 \geq 10$. Since we assumed that $c_2 \leq 12$, one gets 
that $\text{h}^1(F(1)) \leq 2$. One deduces, now, from 
Lemma~\ref{L:h1epriml}, that $\tH^1(F(l)) = 0$ for $l \geq 2$ hence, by Serre 
duality, $\tH^1(F(l)) \simeq \tH^1(F^\ast(-l-3))^\ast \simeq 
\tH^1(F(-l-2))^\ast = 0$ for $l \leq -4$. Moreover, $\text{h}^1(F(-3)) = 
\text{h}^1(F(1)) = c_2 - 10$. The discussion preceeding 
Lemma~\ref{L:h1epriml} shows, now, that $\tH^1_\ast(E^\ast) = 0$ if $c_2 = 10$, 
$\tH^1_\ast(E^\ast) \simeq k(1)$ if $c_2 = 11$, and that $\tH^1_\ast(E^\ast) 
\simeq 2k(1)$ if $c_2 = 12$ hence, by Serre duality, 
$\tH^1_\ast(E) = 0$ if $c_2 = 10$, $\tH^1_\ast(E) \simeq k(2)$ if $c_2 = 11$, 
and $\tH^1_\ast(E) \simeq 2k(2)$ if $c_2 = 12$. Recalling the relation 
\eqref{E:r}, it follows that one of the following holds$\, :$  
\begin{enumerate}
\item[(xiv)] $c_2 = 10$ and $E \simeq 5\sco_\pii(1)$$\, ;$  
\item[(xv)] $c_2 = 11$ and $E \simeq 4\sco_\pii(1) \oplus 
\text{T}_\pii(-1)$$\, ;$ 
\item[(xvi)] $c_2 = 12$ and $E \simeq 3\sco_\pii(1) \oplus 
2\text{T}_\pii(-1)$.     
\qedhere 
\end{enumerate}
\end{proof}

\section{The case $c_1 = 4$, $c_2 = 5, 6$ on $\piii$}
\label{S:c1=4c2=5,6n3} 

In this section we begin the classification of globally generated vector 
bundles $E$ on $\piii$ with $c_1 = 4$, $5 \leq c_2 \leq 8$, 
$\tH^i(E^\ast) = 0$, $i = 0, 1$, and with $\tH^0(E(-2)) = 0$ by considering the 
easiest cases $c_2 = 5$ and $c_2 = 6$ in order to make our strategy clear. 
Actually, this strategy was explained in the introduction to 
Section~\ref{S:n=2c1=4}. The technical difference between the cases $\pii$ and 
$\piii$ is that on $\piii$ one has to use stable rank 2 \emph{reflexive 
sheaves} instead of stable rank 2 vector bundles. We take advantage, in our 
analysis, of the properties of the \emph{spectrum} of such a sheaf which we 
recall, in Thm.~\ref{T:spectrum}, quoting the classical paper of Hartshorne 
\cite{ha}.   

\vskip2mm

Now, if $r$ is the rank of $E$ (as above) then $r-1$ general global sections 
of $E$ define an exact sequence$\, :$  
\begin{equation}
\label{E:oeiy4}
0 \lra (r - 1)\sco_\piii \lra E \lra \sci_Y(4) \lra 0
\end{equation}
with $Y$ a nonsingular (but, maybe, reducible) curve in $\piii$ of degree 
$d = c_2$. Dualizing \eqref{E:oeiy4}, one gets an exact sequence$\, :$ 
\begin{equation}
\label{E:dualoeiy4}
0 \lra \sco_\piii(-4) \lra E^\ast \lra (r - 1)\sco_\piii \lra \omega_Y 
\lra 0\, .
\end{equation} 
The following result was proven, under the additional assumption $\omega_Y 
\simeq \sco_Y$, by Chiodera and Ellia \cite[Prop.~2.6]{ce}.  

\begin{lemma}\label{L:yconnected} 
The nonsingular curve $Y\subset \piii$ is connected, except when $c_2 = 8$ 
and $Y$ is a disjoint union of two elliptic quartics $($in which case $r = 3$ 
and $c_3 = 0$$)$. 
\end{lemma} 

\begin{proof}
Since $\omega_Y$ is globally generated, each irreducible (= connected) 
component of $Y$ must have genus $\geq 1$, hence it must have degree $\geq 3$ 
and if it has degree 3 then it must be a plane cubic. Assume that $Y = C_0 
\cup C_1$, where $C_0$ is a plane cubic. Let $H_0\subset \piii$ be the plane 
containing $C_0$. There exists a line $L \subset H_0$ such that 
$\deg (L\cap C_1) \geq 2$. It follows that $\deg (L \cap Y) \geq 5$, which 
contradicts the fact that $\sci_Y(4)$ is globally generated. 

Consequently, if $Y$ is not connected then $Y = C_0 \cup C_1$, with 
$\deg C_i \geq 4$, $i = 0, 1$, hence with $\deg C_i = 4$, $i = 0, 1$. The same 
kind of argument as the one used above shows that none of the curves $C_0$, 
$C_1$ can be a plane quartic, hence they must be both elliptic quartics, i.e., 
complete intersections of type $(2,2)$. 
\end{proof} 

\begin{remark}\label{R:reflexive}
With the notation from the beginning of this section, 
one can take the quotient of $E$ by $r-1$ general global sections in two 
steps and one obtains two exact sequences$\, :$ 
\begin{gather*}
0 \lra (r - 2)\sco_\piii \lra E \lra \sce^\prim \lra 0\, ,\\ 
0 \lra \sco_\piii \lra \sce^\prim \lra \sci_Y(4) \lra 0\, ,  
\end{gather*}   
where $\sce^\prim$ is a rank 2 \emph{reflexive sheaf} on $\piii$. 
Let $\scf := \sce^\prim(-2)$ be the \emph{normalized}  
reflexive sheaf associated to $\sce^\prim$. One has 
$c_1(\scf) = 0$, $c_2(\scf) = c_2 - 4$, and, by \cite[Cor.~2.2]{ha}, 
$c_3(\scf) = c_3(\sce^\prim) = c_3(E) =: c_3$. Dualizing the last exact 
sequence, one gets an exact sequence$\, :$  
\[
0 \lra \sco_\piii(-4) \lra \sce^{\prim \ast} \lra \sco_\piii \lra \omega_Y \lra 
\sce xt^1(\sce^\prim ,\sco_\piii) \lra 0 
\] 
hence, by \cite[Prop.~2.6]{ha}$\, :$ 
\[
c_3(\scf) = c_3(\sce^\prim) = \text{h}^0(\sce xt^1(\sce^\prim ,\sco_\piii)) = 
\deg \omega_Y = 2p_a(Y) - 2\, . 
\]
It is convenient to write the exact sequences relating $E$ to $\sce^\prim$ 
and $\sce^\prim$ to $\sci_Y$ in the form$\, :$  
\begin{gather*}
0 \lra (r - 2)\sco_\piii \lra E \lra \scf(2) \lra 0\, ,\\ 
0 \lra \sco_\piii(-2) \lra \scf \lra \sci_Y(2) \lra 0\, .  
\end{gather*}  

Now, the condition $\tH^0(E(-2)) = 0$ is equivalent to $\tH^0(\scf) = 0$, 
which is equivalent to $\scf$ being \emph{stable}. According to Barth's 
restriction theorem \cite{ba}, if $H \subset \piii$ is a general plane 
avoiding the singular points of $\scf$, then $\scf \vb H$ is stable (i.e., 
$\tH^0(\scf \vb H) = 0$) except when $\scf$ is a nullcorrelation bundle, i.e.,  
when $\scf$ is a rank 2 vector bundle defined by an exact sequence 
$0 \ra \sco_\piii(-1) \ra \Omega_\piii(1) \ra \scf \ra 0$ (see, also, 
\cite[Thm.~3.3]{ehv}). Consequently, if $c_2 \geq 6$ then, for the general 
plane $H \subset \piii$, $\tH^0(E_H(-2)) = 0$.  

\vskip2mm

Assume, now, that $Y$ \emph{is connected}. Then $\tH^2(E^\ast) \simeq 
\tH^1(E(-4))^\ast \simeq \tH^1(\sci_Y)^\ast = 0$. Since $\tH^3(E^\ast) \simeq 
\tH^0(E(-4))^\ast = 0$, it follows, from the Riemann-Roch formula recalled 
below, that$\, :$ 
\[
0 = \chi(E^\ast) = (r-1) + \chi(\sco_\piii(-4)) - 2c_2 + 
\frac{1}{2}(-c_3 + 4c_2) = r -2 - \frac{1}{2}c_3\  \text{hence} \  
r = \frac{1}{2}c_3 + 2\, .  
\] 

Moreover, using the exact sequence of the hyperplane section and 
Serre duality one deduces that, for every plane $H \subset \piii$$\, :$  
\begin{gather*}
\tH^0(E^\ast_H) \simeq \tH^1(E^\ast(-1)) \simeq \tH^2(E(-3))^\ast \simeq 
\tH^2(\scf(-1))^\ast \, ,\\ 
\tH^1(E^\ast_H) \simeq \tH^2(E^\ast(-1)) \simeq \tH^1(E(-3))^\ast \simeq 
\tH^1(\scf(-1))^\ast \, . 
\end{gather*}      
\end{remark}

\begin{lemma}\label{L:east1}
Under the hypotheses and with the notation from Remark~\ref{R:reflexive}$\, :$ 

\emph{(a)} If ${\fam0 H}^1(\scf(-1)) = 0$ then $E^\ast$ is $1$-regular. 

\emph{(b)} If $c_2 \geq 6$ and ${\fam0 h}^1(\scf(-1)) = 1$ then 
${\fam0 H}^1(E^\ast(l)) = 0$ for $l \geq 0$ hence ${\fam0 H}^2(E(l)) = 0$ 
for $l \leq -4$. 

\emph{(c)} If ${\fam0 H}^2(\scf(-1)) = 0$ then the graded $S$-module 
${\fam0 H}^1_\ast(E)$ is generated in degrees $\leq -2$. 
\end{lemma} 

\begin{proof}
(a) One has $\tH^1(E^\ast) = 0$, $\tH^2(E^\ast(-1)) \simeq 
\tH^1(E(-3))^\ast \simeq \tH^1(\scf(-1))^\ast = 0$ and $\tH^3(E^\ast(-2)) 
\simeq \tH^0(E(-2))^\ast = 0$. 

\vskip2mm

(b) As we saw in Remark~\ref{R:reflexive}, for a general plane $H \subset 
\piii$ one has $\tH^0(E_H(-2)) = 0$. According to Lemma~\ref{L:h0h1} one 
has $E_H \simeq G \oplus t\sco_H$, with $G$ defined by an exact 
sequence $0 \ra s\sco_H \ra F \ra G \ra 0$ where $F$ is a vector bundle 
on $H$ with $\tH^i(F^\ast) = 0$, $i = 0,\, 1$. Moreover, 
$s = \h^1(E_H^\ast) \leq \h^2(E^\ast(-1)) = \h^1(E(-3)) = \h^1(\scf(-1)) 
= 1$. 

\vskip2mm

\noindent
{\bf Claim.}\quad $\tH^1(E_H^\ast(l)) = 0$ \emph{for} $l \geq 1$. 

\vskip2mm 

\noindent 
\emph{Indeed}, since $\tH^2(F^\ast(-1)) \simeq \tH^0(F(-2))^\ast = 0$, 
$F^\ast$ is 1-regular. In particular, $\tH^1(F^\ast(l)) = 0$ for $l \geq 0$. 

If $s = 0$ then $G = F$, hence $\tH^1(E_H^\ast(l)) = 0$ for $l \geq 0$. 

If $s = 1$  then consider the dual exact sequence$\, :$ 
\[
0 \lra G^\ast \lra F^\ast \overset{\phi}{\lra} \sco_H \lra 0\, .
\]
$F^\ast$ being 1-regular, $F^\ast(1)$ is globally generated. But for any 
epimorphism $\e : m\sco_H \ra \sco_H(1)$ one has that $\tH^0(\e(l))$ is 
surjective, $\forall \, l \geq 0$. One deduces that $\tH^0(\phi(l))$ is 
surjective $\forall \, l \geq 1$ hence $\tH^1(G^\ast(l)) = 0$, $\forall \, 
l \geq 1$, whence the claim. 

\vskip2mm 

\noindent
Now, using the claim, the fact that $\tH^1(E^\ast) = 0$ and 
Lemma~\ref{L:hieh=0}(a) one gets that $\tH^1(E^\ast(l)) = 0$ for $l \geq 0$. 

\vskip2mm

(c) $\tH^2(E(-3)) \simeq \tH^2(\scf(-1)) = 0$ and $\tH^3(E(-4)) \simeq 
\tH^0(E^\ast)^\ast = 0$. One can apply, now, Lemma~\ref{L:cm}.  
\end{proof}

Let us recall now, from \cite[\S 7]{ha}, the definition and properties of the 
\emph{spectrum} of a stable rank 2 reflexive sheaf $\scf$ on $\piii$$\, :$  

\begin{thm}\label{T:spectrum}
Let $\scf$ be a stable rank $2$ reflexive sheaf on $\piii$ with Chern 
classes $c_1(\scf) = 0$ and $c_2(\scf) = c \geq 1$.  
Then there exists a unique nonincreasing 
sequence of integers $(k_i)_{1 \leq i \leq c}$, called the spectrum of $\scf$, 
with the following properties$\, :$  
\begin{enumerate}
\item[(i)] ${\fam0 h}^1(\piii ,\scf(l)) = {\fam0 h}^0(\p^1, 
\bigoplus_{i=1}^c\sco_{\p^1}(k_i+l+1))$, for $l \leq -1$$\, ;$  
\item[(ii)] ${\fam0 h}^2(\piii ,\scf(l)) = {\fam0 h}^1(\p^1, 
\bigoplus_{i=1}^c\sco_{\p^1}(k_i+l+1))$, for $l \geq -3$$\, ;$ 
\item[(iii)] $c_3 = - 2\sum_{i=1}^ck_i$$\, ;$ 
\item[(iv)] If $k > 0$ occurs in the spectrum then $0,1,\ldots ,k$ also 
occur$\, ;$  
\item[(v)] If $k < 0$ occurs in the spectrum then $-1,-2, \ldots , k$ also 
occur$\, ;$ 
\item[(vi)] If $0$ doesn't occur in the spectrum then $-1$ occurs twice$\, ;$  
\item[(vii)] If $\scf$ is locally free then $(-k_{c-i+1})_{1 \leq i \leq c} = 
(k_i)_{1 \leq i \leq c}$. 
\end{enumerate} 
\end{thm}

We also recall from \cite[Thm.~2.3]{ha}, for the reader's convenience, the 
Riemann-Roch formula on $\piii$. The reader is also advised to recall, from 
\cite[Lemma~2.1]{ha}, the formula for calculating the Chern classes of the 
twists of a coherent sheaf on $\p^n$ in terms of the Chern classes of the 
sheaf. 

\begin{thm}[Riemann-Roch]\label{T:rronp3}
If $\scf$ is a rank $r$ coherent sheaf on $\piii$, with Chern classes 
$c_1,\, c_2,\, c_3$ then, $\forall \, l \in \z$$\, :$ 
\[
\chi(\scf(l)) = \chi\left(((r - 1)\sco_\piii \oplus \sco_\piii(c_1))(l)\right) 
- (l+2)c_2 + \frac{1}{2}(c_3 - c_1c_2)\, .
\]
In particular, $c_3 \equiv c_1c_2 \pmod{2}$. 
\qed
\end{thm} 

\begin{prop}\label{P:c1=4c2=5n3}
Let $E$ be a globally generated vector bundle on $\piii$, with 
$c_1 = 4$ and $c_2 = 5$. Assume that ${\fam0 H}^i(E^\ast) = 0$, $i = 0, 1$, 
and that ${\fam0 H}^0(E(-2)) = 0$. Then $E \simeq N(2)$, where 
$N$ is a nullcorrelation bundle. 
\end{prop} 

\begin{proof}
With the notation from Remark~\ref{R:reflexive}, $c_1(\scf) = 0$, 
$c_2(\scf) = 1$ and $\scf$ is stable. If $F$ is a stable rank 2 vector bundle 
on $\pii$ with $c_1(F) = 0$ then, from Riemann-Roch, $2 - c_2(F) = \chi(F) = 
-\h^1(F) \leq 0$, hence $c_2(F) \geq 2$. The restriction theorem of Barth 
(recalled in Remark~\ref{R:reflexive}) implies, now, that $\scf \simeq N$, 
for some nullcorrelation bundle $N$. 
Since $\tH^i(N^\ast(-2)) \simeq \tH^i(N(-2)) = 0$, 
$i = 0, 1$, one deduces that $r = 2$ and $E \simeq N(2)$. 
\end{proof} 

\begin{remark}\label{R:instantons}
Let $m \geq 1$ be an integer. A \emph{mathematical instanton bundle of charge} 
$m$ on $\piii$ ($m$-\emph{instanton}, for short) is a rank 2 vector bundle $F$ 
on $\piii$ with $c_1(F) = 0$, $c_2(F) = m$, $\tH^0(F) = 0$ and $\tH^1(F(-2)) 
= 0$. The 1-instantons are exactly the nullcorrelation bundles introduced and 
studied by Barth~\cite{ba}. Any stable rank 2 vector bundle $F$ on $\piii$ 
with $c_1(F) = 0$ and $c_2(F) = 2$ is a 2-instanton; they were studied by 
Hartshorne~\cite{ha0}. Any $m$-instanton is the cohomology of an 
anti-selfdual monad$\, :$ 
\[
m\sco_\piii(-1) \overset{\beta}{\lra} (2m + 2)\sco_\piii  
\overset{\alpha}{\lra} m\sco_\piii(1)\, , 
\] 
and the moduli space $\text{MI}(m)$ of $m$-instantons is smooth and 
irreducible, of dimension $8m-3$, at least for $m \leq 5$. 

It is well known that if $F$ is an $m$-instanton then $\tH^1(F(l)) = 0$ for 
$l \geq m-1$. Let us recall the argument: using the monad, one gets that the 
graded $S:= k[X_0,X_1,X_2,X_3]$-module $\tH^1_\ast(F)$ has a presentation of 
the form$\, ;$  
\[
(2m + 2)S \overset{A}{\lra} mS(1) \lra \tH^1_\ast(F) \lra 0\, ,
\] 
where $A = \tH^0_\ast(\alpha)$. The $m\times m$ minors of the matrix $A$ 
annihilate $\tH^1_\ast(F)$. Let $\scl_0 := (2m + 2)\sco_\piii$ and 
$\scl_1 := m\sco_\piii(1)$. Using the Eagon-Northcott complex 
associated to the epimorphism $\alpha$: 
\[
\overset{m+3}{\textstyle \bigwedge}\scl_0\otimes (S^3\scl_1)^\ast \lra 
\overset{m+2}{\textstyle \bigwedge}\scl_0\otimes (S^2\scl_1)^\ast \lra 
\overset{m+1}{\textstyle \bigwedge}\scl_0\otimes \scl_1^\ast \lra 
\overset{m}{\textstyle \bigwedge}\scl_0 \lra 
\overset{m}{\textstyle \bigwedge}\scl_1 \lra 0 
\] 
one derives that these minors generate $S_m$ as a $k$-vector space. 

Since $\tH^2(F(l)) = 0$ for $l \geq -2$ (the spectrum of $F$ is 
$(0,\ldots ,0)$) and $\tH^3(F(l)) = 0$ for $l \geq -4$ (by Serre duality), 
one deduces that $F$ is $m$-regular. In particular, $F(m)$ is globally 
generated. 

It might be true that $F(m-1)$ is globally generated unless $F$ has a 
jumping line of maximal order $m$, but we are not aware of an argument 
(or reference) for this assertion. It is, however, true for $m = 3$, by 
Gruson and Skiti~\cite[Prop.~1.1.1]{gs}.   
\end{remark} 

\begin{prop}\label{P:c1=4c2=6n3}
Let $E$ be a globally generated vector bundle on $\piii$, with 
$c_1 = 4$ and $c_2 = 6$. Assume that ${\fam0 H}^i(E^\ast) = 0$, $i = 0, 1$, 
and that ${\fam0 H}^0(E(-2)) = 0$. Then one of the following holds$\, :$  
\begin{enumerate} 
\item[(i)] $E \simeq F(2)$, where $F$ is a $2$-instanton$\, ;$ 
\item[(ii)] $c_3 = 2$ and, up to a linear change of coordinates, 
$E$ is the kernel of the epimorphism     
\[
(X_0,X_1,X_2,X_3^2) : 3\sco_\piii(2) \oplus \sco_\piii(1) 
\lra \sco_\piii(3)\, ;
\]
\item[(iii)] $c_3 = 4$ and $E \simeq 4\sco_\piii(1)$.
\end{enumerate}
\end{prop}

\begin{proof}
Let  $Y$ be as at the beginning of this section and let $\scf$ be as in 
Remark~\ref{R:reflexive}. 
One has $c_1(\scf) = 0$, $c_2(\scf) = 2$ and the possible 
spectra of $\scf$ are $(0,0)$, $(0,-1)$ and $(-1,-1)$. 

\vskip2mm 

\noindent
{\bf Case 1.}\quad \textit{The spectrum of $\scf$ is $(0,0)$}.  

\vskip2mm

\noindent
In this case $\scf$ is a 2-instanton.  
Since $\tH^i(\sce^{\prim \ast}) \simeq \tH^i(\scf(-2)) = 0$, $i = 0, 1$, 
it follows that $E \simeq \sce^\prim = \scf(2)$.  

\vskip2mm

\noindent
{\bf Case 2.}\quad \textit{The spectrum of $\scf$ is $(0,-1)$}.  

\vskip2mm

\noindent
In this case, $c_3(\scf) = 2$, $p_a(Y) = 2$ and $r = 3$.  

\vskip2mm

\noindent
{\bf Claim 2.1.}\quad $\tH^2_\ast(E) = 0$

\vskip2mm 

\noindent
\emph{Indeed}, $\h^1(\scf(-1)) = 1$ and Lemma~\ref{L:east1}(b) implies that 
$\tH^2(E(l)) = 0$ for $l \leq -4$. 
On the other hand, $\tH^2(E(l)) \simeq \tH^2(\scf(l+2)) = 0$ for $l \geq -3$. 

\vskip2mm

\noindent
Now, $\tH^1(E(l)) \simeq \tH^1(\scf(l+2)) = 0$ for $l \leq -4$, 
$\h^1(E(-3)) = \h^1(\scf(-1)) =1$ and $\h^1(E(-2)) = \h^1(\scf) = 
- \chi(\scf) = 1$ (we used the fact that $\tH^i(\scf) = 0$ for $i \neq 1$).  
Moreover, $\tH^2(\scf(-1)) = 0$ hence, by Lemma~\ref{L:east1}(c), 
the graded module $\tH^1_\ast(E)$ is generated in degrees 
$\leq -2$. A non-zero element of $\tH^1(E(-2))$ defines an extension$\, :$ 
\[
0 \lra E(-2) \lra E^\prim \lra \sco_\piii \lra 0\, .
\]
The intermediate cohomology modules of $E^\prim$ are $\tH^1_\ast(E^\prim) = 
k(1)$ and $\tH^2_\ast(E^\prim) = 0$. It follows that $E^\prim \simeq 
\Omega_\piii(1) \oplus \scl$, where $\scl$ is a direct sum of line bundles. 
Since $\text{rk}\, E^\prim = 4$ and $c_1(E^\prim) = -2$ it follows that 
$\scl \simeq \sco_\piii(-1)$. One thus gets an exact sequence$\, :$ 
\[
0 \lra E(-2) \lra \Omega_\piii(1) \oplus \sco_\piii(-1) 
\overset{\psi}{\lra} \sco_\piii \lra 0\, .
\] 
Recalling the exact sequence $0 \ra \Omega_\piii(1) \ra 4\sco_\piii  
\ra \sco_\piii(1) \ra 0$, the component $\psi_0 : \Omega_\piii(1) \ra \sco_\piii$ 
of $\psi$ factorizes as$\, :$ 
\[
\Omega_\piii(1) \lra 4\sco_\piii \overset{g_0}{\lra} \sco_\piii\, .
\]
Since the component $\psi_1 : \sco_\piii(-1) \ra \sco_\piii$ of $\psi$ cannot be 
an epimorphism it follows that $g_0 \neq 0$. Applying, now, the Snake Lemma 
to the diagram$\, :$
\[
\begin{CD}
0 @>>> \Omega_\piii(1) \oplus \sco_\piii(-1) @>>> 4\sco_\piii  
\oplus \sco_\piii(-1) @>>> \sco_\piii(1) @>>> 0\\
@. @V{\psi}VV @V{(g_0, \psi_1)}VV @VVV\\
0 @>>> \sco_\piii @>>> \sco_\piii @>>> 0 @>>> 0 
\end{CD}
\]
one derives an exact sequence$\, :$ 
\[
0 \lra E(-2) \lra 3\sco_\piii \oplus \sco_\piii(-1) 
\overset{\phi}{\lra} \sco_\piii(1) \lra 0\, .
\]
Since $\tH^0(E(-2)) = 0$ the component $\phi_0 : 3\sco_\piii \ra 
\sco_\piii(1)$ of $\phi$ is defined by three linearly independent linear forms 
$h_0,\, h_1,\, h_2$. Complete this sequence to a basis $h_0, \ldots , h_3$ of 
$\tH^0(\sco_\piii(1))$. Then, up to an automorphism of $3\sco_\piii  
\oplus \sco_\piii(-1)$ invariating $3\sco_\piii$, one can assume that 
$\phi$ is defined by $h_0,\, h_1,\, h_2,\, h_3^2$. 
 
\vskip2mm 

\noindent
Finally, using the Koszul complex associated to $X_0,\, X_1,\, X_2,\, X_3^2$, 
one sees easily that the bundle from item (ii) of the statement is globally 
generated. 

\vskip2mm

\noindent
{\bf Case 3.}\quad \textit{The spectrum of $\scf$ is $(-1,-1)$}.

\vskip2mm 

\noindent
In this case, $c_3(\scf) = 4$, $p_a(Y) = 3$ and $r = 4$. One has $\h^1(E(-2)) = 
\h^1(\scf) = - \chi(\scf) = 0$, $\tH^2(E(-3)) \simeq \tH^2(\scf(-1)) = 0$ and 
$\tH^3(E(-4)) \simeq \tH^0(E^\ast)^\ast = 0$ hence $E$ is $(-1)$-regular. In 
particular, $E(-1)$ is globally generated. Since $c_1(E(-1)) = 0$ it follows 
that $E(-1) \simeq 4\sco_\piii$. 
\end{proof}

\vskip1cm

\section{The case $c_1 = 4$, $c_2 = 7$ on $\piii$}
\label{S:c1=4c2=7n3} 
  
In this section we classify the globally generated vector bundles on $\piii$ 
with $c_1 = 4$ and $c_2 = 7$, using the strategy explained in the 
introductions to Sections~\ref{S:n=2c1=4} and~\ref{S:c1=4c2=5,6n3}. This is a 
case of medium difficulty. Among the bundles appearing in our 
classification one should mention the bundles of the form $F(2)$, where $F$ 
is a general 3-instanton. The meaning of the term ``general'' is explained 
in Remark~\ref{R:general}, using the results of Gruson and Skiti \cite{gs}.  

\vskip2mm

Let $E$ is a globally generated vector bundle on $\piii$ with 
$c_1 = 4$, $c_2 = 7$, $\tH^i(E^\ast) = 0$, $i = 0,\, 1$, and $\tH^0(E(-2)) = 
0$. We shall frequently use the following two exact sequences defined in 
Remark~\ref{R:reflexive}$\, :$  
\begin{gather*}
0 \lra (r - 2)\sco_\piii \lra E \lra \scf(2) \lra 0\, ,\\ 
0 \lra \sco_\piii(-2) \lra \scf \lra \sci_Y(2) \lra 0\, ,  
\end{gather*}
where $\scf$ a stable rank 2 reflexive sheaf with $c_1(\scf) = 0$, 
$c_2(\scf) = 3$, and $Y$ is a nonsingular curve of degree 7. According to 
Lemma~\ref{L:yconnected}, $Y$ is connected. 

\begin{prop}\label{P:c1=4c2=7n3} 
Let $E$ be a globally generated vector bundle on $\piii$ with $c_1 = 4$, 
$c_2 = 7$, ${\fam0 H}^i(E^\ast) = 0$, $i = 0, 1$, and with ${\fam0 H}^0(E(-2)) 
= 0$. Then one of the following holds$\, :$  
\begin{enumerate}
\item[(i)] $E \simeq F(2)$, where $F$ is a general $3$-instanton$\, ;$ 
\item[(ii)] $c_3 = 2$ and, up to a linear change of coordinates, 
$E$ is the kernel of the epimorphism   
\[
\begin{pmatrix}
X_0 & X_1 & X_2 & X_3 & 0\\
0 & X_0 & X_1 & X_2 & X_3
\end{pmatrix}
: 5\sco_\piii(2) \lra 2\sco_\piii(3)\, ;
\]
\item[(iii)] $c_3 = 4$ and, up to a linear change of coordinates, 
$E$ is the kernel of the epimorphism   
\[
(X_0,X_1,X_2^2,X_2X_3,X_3^2) : 2\sco_\piii(2) \oplus 
3\sco_\piii(1) \lra \sco_\piii(3)\, ; 
\]
\item[(iv)] $c_3 = 6$ and $E \simeq 2\sco_\piii(1) \oplus 
\Omega_\piii(2)\, ;$ 
\item[(v)] $c_3 = 8$ and $E \simeq 3\sco_\piii(1) \oplus 
{\fam0 T}_\piii(-1)$. 
\end{enumerate}
\end{prop} 

\begin{remark}\label{R:general}
The meaning of the term ``general'' in the statement of 
Prop.~\ref{P:c1=4c2=7n3}(i) can be explicitated as follows: according to a 
result of Gruson and Skiti \cite[Prop.~1.1.1]{gs}, if $F$ is a 3-instanton 
then $F(2)$ is globally generated iff F has no jumping line of maximal order 
3. The moduli space $\text{MI}(3)$ of 3-instantons is smooth and irreducible 
of dimension 21 and, according to results of Rao~\cite{r} and 
Skiti~\cite{sk}, the 3-instantons admitting a jumping line of order 3 form 
a subvariety of $\text{MI}(3)$ of dimension 20. 
\end{remark}

\begin{proof}[Proof of Proposition~\ref{P:c1=4c2=7n3}] 
We split the proof into several cases according to the spectrum of $\scf$. 
Firstly, we show that some spectra, allowed by the general theory, cannot 
occur in our context.  

\vskip2mm 

\noindent
$\bullet$ The case where $\scf$ has spectrum $(1,0,-1)$ cannot occur because, 
as we have seen at the beginning of this section, $Y$ is connected hence 
$\tH^1(\scf(-2)) \simeq \tH^1(\sci_Y) = 0$.

\vskip2mm 

\noindent
$\bullet$ If the spectrum of $\scf$ is $(0,-1,-2)$ then $c_3(\scf) = 6$ and 
$p_a(Y) = 4$. $\sco_Y(1)$ is a line bundle of degree $7 = 2\times 4 - 1$ on 
$Y$ hence $\tH^1(\sco_Y(1)) = 0$ hence $\tH^2(\sci_Y(1)) = 0$. But this 
contradicts the fact that, by the definition of the spectrum, 
$\text{h}^2(\scf(-1)) = 1$. Consequently, the case where the spectrum of 
$\scf$ is $(0,-1,-2)$ \emph{cannot occur}. 

\vskip2mm 

Next, we analyse, case by case, the remaining spectra. 

\vskip2mm

\noindent
{\bf Case 1.}\quad \textit{The spectrum of $\scf$ is $(0,0,0)$}.  

\vskip2mm

\noindent
In this case, $\scf$ is a 3-instanton bundle and $E \simeq \scf(2)$. 

\vskip2mm 

\noindent
{\bf Case 2.}\quad \textit{The spectrum of $\scf$ is $(0,0,-1)$}.  

\vskip2mm

\noindent
In this case, $c_3(\scf) = 2$ and $p_a(Y) = 2$ hence $r = 3$. 

\vskip2mm

\noindent
{\bf Claim 2.1.}\quad  $\tH^0(E^\ast(1)) = 0$. 

\vskip2mm

\noindent 
\emph{Indeed}, recall the exact sequence$\, :$ 
\[
0 \lra \sco_\piii(-4) \lra E^\ast \lra 2\sco_\piii \lra \omega_Y 
\lra 0 \, .
\]  
Choose a basis $s_1,s_2$ of $\tH^0(\omega_Y)$ such that each of the 
divisors $(s_i)_0$, $i = 1, 2$, consists of 2 simple points and such that 
$(s_1)_0 \cap (s_2)_0 = \emptyset$. If $\tH^0(E^\ast(1)) \neq 0$ then there 
exists a non-zero pair of linear forms $(h_1,h_2) \in 
\tH^0(2\sco_\piii(1))$ such that $h_1s_1 + h_2s_2 = 0$. Since $h_1$ and 
$h_2$ cannot be proportional, they define distinct planes $H_1,H_2 \subset 
\piii$. Let $L := H_1 \cap H_2$. We will show that $L$ is a 5-secant of 
$Y$ and this \emph{will contradict} the fact that $\sci_Y(4)$ is globally 
generated. Since $h_1s_1 = -h_2s_2$ one has the following equality of 
divisors on $Y$$\, :$ 
\[
(H_1 \cap Y) + (s_1)_0 = (H_2 \cap Y) + (s_2)_0 \, .
\] 
It follows that there exists an effective divisor $D$ on $Y$ of degree 5 
such that$\, :$ 
\[
H_1 \cap Y = D + (s_2)_0,\  H_2 \cap Y = D + (s_1)_0\, .
\]
One has $D \subseteq L \cap Y$ hence $L$ is a secant of $Y$ of order 
$\geq 5$. It remains that $\tH^0(E^\ast(1)) = 0$. 

\vskip2mm 

\noindent
{\bf Claim 2.2.}\quad $E$ \emph{is the kernel of an epimorphism} 
$\phi : 5\sco_\piii(2) \ra 2\sco_\piii(3)$. 

\vskip2mm

\noindent
\emph{Indeed}, by Claim 2.1, $E^\ast(1)$ is a stable rank 3 vector bundle 
with Chern classes $c_1(E^\ast(1)) = -1$, $c_2(E^\ast(1)) = 2$ and 
$c_3(E^\ast(1)) = 2$. The claim follows, now, from Okonek and Spindler 
\cite[Lemma~1.12]{oks}.

\vskip2mm 

\noindent
{\bf Claim 2.3.}\quad \emph{Up to a linear change of coordinates, the 
matrix defining} $\phi$ \emph{is that in item} (ii) \emph{of the statement}. 

\vskip2mm

\noindent
\emph{Indeed}, let $A$ be a 5-dimensional 
$k$-vector space, $B$ a 2-dimensional one, and $W := \tH^0(\sco_\piii(1))$. 
$\phi(-2)$ is defined by a $k$-linear map $f : A \ra B\otimes_kW$. The fact 
that $\phi(-2)$ is an epimorphism is equivalent to the fact that, $\forall \, 
v \in W^\ast$, $v \neq 0$, the composite map$\, :$  
\[
A \overset{f}{\lra} B\otimes_kW \overset{\text{id}\otimes v}{\lra} B 
\]
is surjective. This is equivalent to the fact that $f^\ast : B^\ast \otimes_k
W^\ast \ra A^\ast$ is \emph{nondegenerate} in the sense of the Bilinear map 
lemma from Hartshorne~\cite[Lemma~5.1]{ha}, i.e., $f^\ast(\beta \otimes v) 
\neq 0$, $\forall \, \beta \in B^\ast \setminus \{0\}$, $\forall \, v \in 
W^\ast \setminus \{0\}$. $f$ defines also a morphism of vector bundles 
on $\p(B) \simeq \p^1$, $\psi : A\otimes_k\sco_{\p(B)} \ra \sco_{\p(B)}(1) 
\otimes_kW$, which is an epimorphism, too, because $f^\ast$ is nondegenerate. 
The kernel of $\psi$ has rank 1, hence it must be $\sco_{\p(B)}(-4)$. 
Consider the exact sequence$\, :$ 
\[
0 \lra \sco_{\p(B)}(-4) \overset{\mu}{\lra} A\otimes_k\sco_{\p(B)} 
\overset{\psi}{\lra} \sco_{\p(B)}(1)\otimes_kW \lra 0\, .
\] 
Let $t_0,t_1$ be a $k$-basis of $B$. There exists a $k$-basis of $A$ such 
that $\mu$ is defined by $(t_1^4,-t_1^3t_0,t_1^2t_0^2,-t_1t_0^3,t_0^4)$. One 
deduces that there exists a $k$-basis $h_0,h_1,h_2,h_3$ of $W$ such that 
$\psi$ is defined by the matrix$\, :$ 
\[
\begin{pmatrix} 
t_0 & t_1 & 0 & 0 & 0\\
0 & t_0 & t_1 & 0 & 0\\
0 & 0 & t_0 & t_1 & 0\\
0 & 0 & 0 & t_0 & t_1
\end{pmatrix}
\]  
With respect to the above bases of $A$ and $B$, the matrix defining 
$\phi(-2)$ is$\, :$ 
\[
\begin{pmatrix}
h_0 & h_1 & h_2 & h_3 & 0\\
0 & h_0 & h_1 & h_2 & h_3
\end{pmatrix}
\]

\vskip2mm 

\noindent
{\bf Claim 2.4.}\quad \emph{The bundle} $E$ \emph{from item} (ii) 
\emph{of the statement is globally generated}. 

\vskip2mm

\noindent
We will show something more, namely that $E$ is 
0-regular. The only thing to verify is that $\tH^1(E(-1)) = 0$. But this 
follows by noticing that the graded $S := k[X_0,X_1,X_2,X_3]$-module 
$\tH^1_\ast(E)$ is generated by $\tH^1(E(-3))$ and is annihilated by the 
$2\times 2$ minors of the matrix from the statement and that these minors 
generate $S_2$ as a $k$-vector space. 

\vskip2mm

\noindent
{\bf Case 3.}\quad \textit{The spectrum of $\scf$ is $(0,-1,-1)$}.  

\vskip2mm

\noindent
In this case, $c_3(\scf) = 4$, $p_a(Y) = 3$, hence $r = 4$. 

\vskip2mm 

\noindent
{\bf Claim 3.1.}\quad $\tH^2_\ast(E) = 0$.  

\vskip2mm

\noindent 
\emph{Indeed}, $\tH^2(E(l)) \simeq \tH^2(\scf(l+2)) = 0$ for $l \geq -3$. 
On the other hand, since $\h^1(\scf(-1)) = 1$, Lemma~\ref{L:east1}(b) 
implies that $\tH^2(E(l)) = 0$ for $l \leq -4$.   

\vskip2mm

\noindent
{\bf Claim 3.2.}\quad $E$ \emph{is the kernel of an epimorphism} 
$\phi : 2\sco_\piii(2) \oplus 3\sco_\piii(1) \ra \sco_\piii(3)$. 

\vskip2mm

\noindent
\emph{Indeed}, $\tH^1(E(l)) \simeq \tH^1(\scf(l+2)) = 0$ for $l \leq -4$,  
$\h^1(E(-3)) = \h^1(\scf(-1)) = 1$ and $\h^1(E(-2)) = \h^1(\scf) = 
- \chi(\scf) = 2$. Moreover, $\tH^2(\scf(-1)) = 0$ hence, 
by Lemma~\ref{L:east1}(c), the graded module $\tH^1_\ast(E)$ is generated 
in degrees $\leq -2$. 
Choosing a $k$-basis of $\tH^1(E(-2))$ one gets an extension$\, :$ 
\[
0 \lra E(-2) \lra E^\prim \lra 2\sco_\piii \lra 0\, .
\] 
The intermediate cohomology modules of the vector bundle $E^\prim$ are 
$\tH^1_\ast(E^\prim) = k(1)$ and $\tH^2_\ast(E^\prim) = 0$ (by Claim 3.1). It 
follows that $E^\prim \simeq \Omega_\piii(1) \oplus \scl$, where $\scl$ is a 
direct sum of line bundles. Since $\text{rk}\, E^\prim = 6$, $c_1(E^\prim ) 
= -4$ and $\tH^0(E^\prim ) = 0$ it follows that $\scl \simeq 
3\sco_\piii(-1)$. One thus gets an exact sequence$\, :$ 
\[
0 \lra E(-2) \lra \Omega_\piii(1) \oplus 3\sco_\piii(-1)  
\overset{\psi}{\lra} 2\sco_\piii \lra 0\, .
\] 
Let $\psi_0 : \Omega_\piii(1) \ra 2\sco_\piii$ and 
$\psi_1 : 3\sco_\piii(-1) \ra 2\sco_\piii$ be the components of 
$\psi$. Recalling the exact sequence $0 \ra \Omega_\piii(1) \ra 
4\sco_\piii \ra \sco_\piii(1) \ra 0$, the morphism $\psi_0$ factorizes 
as$\, :$ 
\[
\Omega_\piii(1) \lra 4\sco_\piii \overset{g_0}{\lra} 2\sco_\piii\, .
\]
Since there is no epimorphism $3\sco_\piii(-1) \ra \sco_\piii$, the 
morphism $g_0$ must have rank 2. Applying, now, the Snake Lemma to the 
diagram$\, :$ 
\[
\begin{CD}
0 @>>> \Omega_\p(1) \oplus 3\sco_\p(-1) @>>> 4\sco_\p \oplus 
3\sco_\p(-1) @>>> \sco_\p(1) @>>> 0\\
@. @V{\psi}VV @V{(g_0,\psi_1)}VV @VVV\\
0 @>>> 2\sco_\p @>>> 2\sco_\p @>>> 0 @>>>0 
\end{CD}
\]
one gets an exact sequence$\, :$ 
\[
0 \lra E(-2) \lra 2\sco_\piii \oplus 3\sco_\piii(-1) \lra 
\sco_\piii(1) \lra 0\, .
\]

\noindent
{\bf Claim 3.3.}\quad \emph{Up to a linear change of coordinates, the matrix 
defining} $\phi$ \emph{is that in item} (iii) \emph{of the statement}. 

\vskip2mm 

\noindent
\emph{Indeed}, since $\tH^0(\phi(-2))$ is injective (because $\tH^0(E(-2)) 
= 0$) it follows that the component 
$\phi_0 : 2\sco_\piii(2) \ra \sco_\piii(3)$ of $\phi$ 
is defined by two linearly independent linear forms $h_0$ and $h_1$. One 
has $\Cok \phi_0 \simeq \sco_L(3)$, where $L$ is the line of equations 
$h_0 = h_1 = 0$. 

Let $\phi_1 : 3\sco_\piii(1) \ra \sco_\piii(3)$ be the other component of 
$\phi$. We assert that the composite morphism$\, :$ 
\[
\begin{CD}
3\sco_\piii @>{\phi_1(-1)}>> \sco_\piii(2) \lra \sco_L(2)
\end{CD}
\]   
induces an isomorphism $\tH^0(3\sco_\piii) \izo \tH^0(\sco_L(2))$. 

\emph{Indeed}, otherwise one would have $E \simeq K \oplus \sco_\piii(1)$ with 
$K$ the kernel of an epimorphism $2\sco_\piii(2) \oplus 
2\sco_\piii(1) \ra \sco_\piii(3)$. But $K_L \simeq 2\sco_L(2)  
\oplus \sco_L(-1)$ which would \emph{contradict} the fact that $E$ is globally 
generated. 

Completing $h_0,\, h_1$ to a basis $h_0, \ldots , h_3$ of $\tH^0(\sco_\piii(1))$ 
one deduces, now, that, up to an automorphism of $2\sco_\piii(2)  
\oplus 3\sco_\piii(1)$ invariating $2\sco_\piii(2)$, one can assume 
that $\phi_1$ is defined by $h_2^2,\, h_2h_3,\, h_3^2$. 
 
\vskip2mm

\noindent
{\bf Claim 3.4.}\quad \emph{The vector bundle from item} (iii)  
\emph{of the statement is globally generated}. 

\vskip2mm

\noindent
\emph{Indeed}, tensorizing the exact sequences$\, :$  
\begin{gather*}
0 \lra S(-2) \lra 2S(-1) \lra S \lra S/(X_0,X_1) \lra 0\, ,\\ 
0 \lra 2S(-3) \lra 3S(-2) \lra S \lra 
S/(X_2^2,X_2X_3,X_3^2) \lra 0\, , 
\end{gather*}
one gets a resolution of $S/(X_0,X_1,X_2^2,X_2X_3,X_3^2)$ from which one 
gets a resolution$\, :$ 
\[
0 \lra 2\sco_\piii(-2) \lra 7\sco_\piii(-1) \lra 
\sco_\piii(1) \oplus 8\sco_\piii \lra E \lra 0\, . 
\]

\vskip2mm 

\noindent
{\bf Case 4.}\quad \textit{The spectrum of $\scf$ is $(-1,-1,-1)$}.  

\vskip2mm

\noindent
In this case, $c_3(\scf) = 6$, 
$p_a(Y) = 4$ hence $r = 5$. Moreover, $\tH^1(\scf(-1)) = 0$ hence, 
by Lemma~\ref{L:east1}(a), $F := E^\ast(1)$ is a globally generated vector 
bundle with Chern classes $c_1(F) = 1$, $c_2(F) = 1$. 
Taking into account that $r = 5$ one deduces that 
$F \simeq 2\sco_\piii \oplus \text{T}_\piii(-1)$ hence $E \simeq 
2\sco_\piii(1) \oplus \Omega_\piii(2)$. 

\vskip2mm

\noindent
{\bf Case 5.}\quad \textit{The spectrum of $\scf$ is $(-1,-1,-2)$}.  

\vskip2mm 

\noindent 
In this case, $c_3(\scf) = 8$, 
$p_a(Y) = 5$ hence $r = 6$. Moreover, $\tH^1(\scf(-1)) = 0$ hence, by 
Lemma~\ref{L:east1}(a), $F := E^\ast(1)$ is a globally generated vector bundle 
with Chern classes $c_1(F) = 2$, $c_2(F) = 2$. Taking into account that 
$r = 6$ and that $\tH^0(F(-1)) = \tH^0(E^\ast) = 0$ one deduces, from 
Prop.~\ref{P:c2=c1}, that either $F \simeq 3\sco_\piii \oplus 
\Omega_\piii(2)$ or $F \simeq 4\sco_\piii \oplus N(1)$, where $N$ is 
a nullcorrelation bundle. In the former case $E \simeq 3\sco_\piii(1)  
\oplus \text{T}_\piii(-1)$. In the latter case, $E \simeq 
4\sco_\piii(1) \oplus N$. This case cannot, however, occur because 
$N$ is not globally generated.       
\end{proof}

\section{The case $c_1 = 4$, $c_2 = 8$ on $\piii$}
\label{S:c1=4c2=8n3} 
  
We classify, in this section, the globally generated vector bundles $E$ on 
$\piii$ with $c_1 = 4$ and $c_2 = 8$. The difficult cases of this 
classification are those in which the globally generated vector bundles we 
want to classify have small $c_3$. It is not hard to show that these bundles 
belong to certain families of vector bundles. It is, however, rather 
difficult to show that the general members of those families are globally 
generated and, even more difficult, to distinguish the globally generated 
members from those that are not. 

\vskip2mm

Actually, the characterization of 
the 4-instantons $F$ with $F(2)$ globally generated remains a, probably 
difficult, open problem. What we can show, in a rather tricky way (see 
Remark~\ref{R:generalc2=8}), is that if $F(2)$ is globally generated then 
$\tH^1(F(2)) = 0$. Another necessary condition is that $F$ should have no 
jumping line of order $\geq 3$. We do not know whether these conditions are 
also sufficient.  

\vskip2mm
  
Now, let $E$ is a globally generated vector bundle on $\piii$ with 
$c_1 = 4$, $c_2 = 8$, $\tH^i(E^\ast) = 0$, $i = 0,\, 1$, and 
$\tH^0(E(-2)) = 0$.  
Consider the exact sequences (defined in Remark~\ref{R:reflexive})$\, :$  
\begin{gather*}
0 \lra (r - 2)\sco_\piii \lra E \lra \scf(2) \lra 0\, ,\\ 
0 \lra \sco_\piii(-2) \lra \scf \lra \sci_Y(2) \lra 0\, ,  
\end{gather*}
where $\scf$ a stable rank 2 reflexive sheaf with $c_1(\scf) = 0$, 
$c_2(\scf) = 4$, and $Y$ is a smooth curve of degree 8. According to 
Lemma~\ref{L:yconnected}, $Y$ is connected or it is the disjoint union of two 
elliptic quartics.  

We shall need the following result of Hartshorne~\cite[Prop.~5.1]{ha2}, 
stating a finer property of the spectrum of a stable rank 2 reflexive 
sheaf$\, :$ 

\begin{prop}\label{P:spectrum2}
Let $\scf$ be a stable rank $2$ reflexive sheaf on $\piii$ with $c_1(\scf) = 
0$, $c_2(\scf) = c$, and let $(k_i)_{1\leq i \leq c}$ be its spectrum. If, 
for some $2\leq i \leq c-1$, $0 \geq k_{i-1} > k_i > k_{i+1}$ then 
$k_{i+1} > k_{i+2} > \cdots > k_c$. 
\qed
\end{prop}

For example, if $\scf$ is a stable rank 2 reflexive sheaf on $\piii$ with 
$c_1(\scf) = 0$, $c_2(\scf) = 4$ then $\scf$ cannot have spectrum 
$(0,-1,-2,-2)$. 

\begin{lemma}\label{L:h0fast} 
Let $\scg$ be a rank $3$ reflexive sheaf on $\piii$.  
\begin{enumerate}
\item[(a)] If $c_1(\scg) = 0$, ${\fam0 H}^0(\scg) = 0$ and 
${\fam0 H}^0(\scg^\ast(-1)) = 0$ then ${\fam0 h}^0(\scg^\ast) \leq 1$.  
\item[(b)] If $c_1(\scg) = -1$, ${\fam0 H}^0(\scg(-1)) = 0$ and 
${\fam0 H}^0(\scg^\ast(-2)) = 0$ then ${\fam0 h}^0(\scg^\ast(-1)) \leq 1$. 
\item[(c)] If $c_1(\scg) = -2$, ${\fam0 H}^0(\scg) = 0$ and 
${\fam0 H}^0(\scg^\ast(-2)) = 0$ then ${\fam0 h}^0(\scg^\ast(-1)) \leq 1$. 
\end{enumerate}
\end{lemma} 

\begin{proof}
(a)  
A non-zero element of $\tH^0(\scg^\ast)$ defines a morphism $\phi : \scg \ra 
\sco_\piii$. Since $\tH^0(\scg^\ast(-1)) = 0$, the image of $\phi$ is the ideal 
sheaf $\sci_Z$ of a closed subscheme $Z$ of $\piii$, of codimension $\geq 2$. 
One gets an exact sequence$\, :$  
\[
0 \lra \scf^\prim \lra \scg \lra \sci_Z \lra 0 
\] 
where $\scf^\prim$ is a rank 2 reflexive sheaf on $\piii$ with 
$c_1(\scf^\prim) = 0$. Dualizing this exact sequence one obtains an exact 
sequence$\, :$ 
\[
0 \lra \sco_\piii \lra \scg^\ast \lra \scf^{\prim \ast}\, .
\]
Since $\scf^{\prim \ast} \simeq \scf^\prim$ and since $\tH^0(\scf^\prim) 
\subseteq \tH^0(\scg) = 0$, one deduces that $\text{h}^0(\scg^\ast) = 1$. 

(b) and (c) are analogous.  
\end{proof} 

\begin{prop}\label{P:c1=4c2=8n3} 
Let $E$ be a globally generated vector bundle on $\piii$ with $c_1 = 4$, 
$c_2 = 8$, ${\fam0 H}^i(E^\ast) = 0$, $i = 0, 1$, and with ${\fam0 H}^0(E(-2)) 
= 0$. Then one of the following holds$\, :$  
\begin{enumerate}
\item[(i)] $E \simeq F(2)$, where $F$ is a general $4$-instanton$\, ;$  
\item[(ii)] $c_3 = 0$ and $E(-2)$ is the kernel of an arbitrary 
epimorphism$\, :$  
\[
4\sco_\piii \lra \sco_\piii(2)\, ;
\]
\item[(iii)] $c_3 = 2$ and $E(-2)$ is the kernel of a general 
epimorphism$\, :$  
\[
3\sco_\piii(1) \oplus 3\sco_\piii \lra \Omega_\piii(3)\, ; 
\]
\item[(iv)] $c_3 = 4$ and $E(-2)$ is the kernel of a general 
epimorphism$\, :$  
\[
4\sco_\piii \oplus 2\sco_\piii(-1) \lra 2\sco_\piii(1)\, ;  
\]
\item[(v)] $c_3 = 6$ and $E(-2)$ is the kernel of a general 
epimorphism$\, :$  
\[
\sco_\piii \oplus 5\sco_\piii(-1) \lra \sco_\piii(1)\, ; 
\]
\item[(vi)] $c_3 = 8$ and, up to a linear change of coordinates, 
$E$ is the cohomology of the monad$\, :$ 
\[ 
\sco_\piii(-1) \xra{\binom{s}{u}} 2\sco_\piii(2) \oplus 
2\sco_\piii(1) \oplus 4\sco_\piii \xra{(p\, ,\, 0)} \sco_\piii(3) 
\]
where $\sco_\piii(-1) \overset{s}{\lra} 2\sco_\piii(2) \oplus 
2\sco_\piii(1) \overset{p}{\lra} \sco_\piii(3)$ is a subcomplex of the 
Koszul complex associated to $X_0, X_1, X_2^2, X_3^2$ and $u : \sco_\piii(-1) 
\ra 4\sco_\piii$ is defined by $X_0, X_1, X_2, X_3$$\, ;$ 
\item[(vii)] $c_3 = 8$ and $E \simeq 2\Omega_\piii(2)\, ;$ 
\item[(viii)] $c_3 = 10$ and $E \simeq \sco_\piii(1) \oplus {\fam0 T}_\piii(-1) 
\oplus \Omega_\piii(2)\, ;$
\item[(ix)] $c_3 = 12$ and $E \simeq 2\sco_\piii(1) \oplus 
2{\fam0 T}_\piii(-1)\, .$
\end{enumerate}
\end{prop} 

\begin{remark}\label{R:generalc2=8} 
We want to explain the meaning of the term "general" in the statement of 
Prop.~\ref{P:c1=4c2=8n3}. The vector bundles appearing in 
Prop.~\ref{P:c1=4c2=8n3}(i),(iii),(iv), and (v) belong to some  
families. In case (i) the elements of the family are "parametrized" by the 
moduli space MI(4) of 4-instantons, while in the other three cases they are 
parametrized by spaces of epimorphisms between two fixed vector bundles. 
It is known that the property of being "globally generated" is not open in 
flat families (as the example of the family $(\sco_C(P-P_0))_{P\in C}$ of line 
bundles of degree 0 on an elliptic curve $C$ shows; the problem is that 
$\text{h}^0(E)$ is not necessarily locally constant in flat families). We 
will show in this remark that, in our cases, the points of the parametrizing 
spaces corresponding to globally generated vector bundles form, however,  
open subsets of these spaces. To be more precise, it is easy to see that 
if $E = F(2)$ where $F$ is a 4-instanton or if $E(-2)$ is the kernel of any 
epimorphism of the kind appearing in Prop.~\ref{P:c1=4c2=8n3}(iii)-(v) then 
$\tH^2(E) = 0$ and $\tH^3(E) = 0$. We will show that if $E$ is globally 
generated then $\tH^1(E) = 0$ (and $\tH^0(E(-1)) = 0$). 

We start by noticing that if $E$ is a globally generated vector bundle on 
$\piii$, with Chern classes $c_1,\, c_2,\, c_3$, then $c_1(P(E)) = c_1$, 
$c_2(P(E)) = c_1^2 - c_2$, $c_3(P(E)) = c_3 + c_1(c_1^2 - 2c_2)$. Assume, from 
now on, that $c_1 = 4$, $c_2 = 8$, $\tH^i(E^\ast) = 0$, $i = 0, 1$, and that 
$\tH^0(E(-2)) = 0$. One has $c_1(P(E)) = 4$, $c_2(P(E)) = 8$ and 
$c_3(P(E)) = c_3$. Moreover, using the notation from the beginning of this 
section$\, :$ 
\[
\tH^0(P(E)(-2)) \simeq \tH^1(E^\ast(-2)) \simeq \tH^2(E(-2))^\ast 
\simeq \tH^2(\scf)^\ast = 0
\]  
because, as we shall see at the beginning of the proof of 
Prop.~\ref{P:c1=4c2=8n3}, $\scf$ can have neither the spectrum 
$(0,-1,-2,-3)$ nor the spectrum $(-1,-1,-2,-3)$. Consequently, if $E$ is one 
of the bundles listed in the statement of Prop.~\ref{P:c1=4c2=8n3} then 
$P(E)$ belongs to the same list. 

Now, $c_3 = 0$ in the cases (i) and (ii), $c_3 = 2$ in case (iii), 
$c_3 = 4$ in case (iv), $c_3 = 6$ in case (v), and $c_3 \geq 8$ in the rest of 
the cases. Moreover, if $E$ is one of the bundles occuring in case (ii) then, 
using the Koszul complex associated to the epimorphism $4\sco_\piii \ra 
\sco_\piii(2)$, one sees easily that $P(E) \simeq E$. Consequently, if $E$ is 
one of the bundles occuring in case (i) (resp., (iii), (iv), (v)) then $P(E)$ 
is one of the bundles occuring in the same case. 

Finally, looking at the beginning of each of the cases occuring in the proof 
of Prop.~\ref{P:c1=4c2=8n3}, one sees that$\, :$ 
\[
\tH^2(E^\ast) \simeq \tH^1(E(-4))^\ast \simeq \tH^1(\scf(-2))^\ast = 0
\]  
in all cases except (ii), and that$\, :$ 
\[
\tH^1(E^\ast(-1)) \simeq \tH^2(E(-3))^\ast \simeq \tH^2(\scf(-1))^\ast = 0
\]
in all cases except (vi), (viii) and (ix). One deduces that if $E$ is one of 
the bundles occuring in one of the cases (i), (iii), (iv), or (v) then 
$\tH^1(E) \simeq \tH^2(P(E)^\ast) = 0$ and $\tH^0(E(-1)) \simeq 
\tH^1(P(E)^\ast(-1)) = 0$. 

We would also like to mention that we will explicitate geometrically the 
meaning of the term ``general'' for the cases (iii)-(v) during the proof of 
Prop.~\ref{P:c1=4c2=8n3}. We were not able to do the same thing for the case 
(i) of 4-instantons (compare with Remark~\ref{R:general}). 
\end{remark} 

\begin{proof}[Proof of Proposition~\ref{P:c1=4c2=8n3}] 
We split the proof into several cases according to the spectrum of $\scf$. 
Firstly, we show that some spectra, allowed by the general theory, cannot 
occur in our context.  

\vskip2mm

\noindent 
$\bullet$ If the spectrum of $\scf$ is $(1,0,-1,-1)$ then $c_3(\scf) = 2$ 
and $p_a(Y) = 2$. According to Lemma~\ref{L:yconnected}, $Y$ is connected, 
hence $\tH^1(\sci_Y) = 0$, hence $\tH^1(\scf(-2)) = 0$. On the other hand, 
by the definition of the spectrum, $\text{h}^1(\scf(-2)) = 1$, a 
\emph{contradiction}. Consequently, the spectrum $(1,0,-1,-1)$ \emph{cannot 
occur}. 

\vskip2mm

\noindent 
$\bullet$ The spectrum $(1,0,-1,-2)$ can be eliminated analogously. 

\vskip2mm

\noindent 
$\bullet$ If the spectrum of $\scf$ is $(0,0,-1,-2)$ then $c_3(\scf) = 6$ and 
$p_a(Y) = 4$. One has $\text{deg}\, \sco_Y(1) = 8 \geq 2p_a(Y) - 1$ hence 
$\tH^1(\sco_Y(1)) = 0$ hence $\tH^2(\sci_Y(1)) = 0$ and $\tH^2(\scf(-1)) = 
0$. On the other hand, by the definition of the spectrum, 
$\text{h}^2(\scf(-1)) = 1$, a \emph{contradiction}. Consequently, the 
spectrum $(0,0,-1,-2)$ \emph{cannot occur}. 

\vskip2mm 

\noindent 
$\bullet$ If the spectrum of $\scf$ is $(0,-1,-2,-3)$ then $c_3(\scf) = 12$ 
and $p_a(Y) = 7$. By the definition of the spectrum, $\text{h}^2(\scf) = 1$ 
hence $\text{h}^2(\sci_Y(2)) = 1$ hence $\text{h}^1(\sco_Y(2)) = 1$. On 
the other hand, $\text{deg}\, \sco_Y(2) = 16 \geq 2p_a(Y) -1$ which implies 
that $\tH^1(\sco_Y(2)) = 0$, a \emph{contradiction}. Consequently, 
the spectrum $(0,-1,-2,-3)$ \emph{cannot occur}. 

\vskip2mm 

\noindent 
$\bullet$ If the spectrum of $\scf$ is $(-1,-1,-2,-3)$ then $c_3(\scf) = 14$ 
and $p_a(Y) = 8$. By the definition of the spectrum, $\text{h}^2(\scf) = 1$ 
hence $\text{h}^2(\sci_Y(2)) = 1$ hence $\text{h}^1(\sco_Y(2)) = 1$. On 
the other hand, $\text{deg}\, \sco_Y(2) = 16 \geq 2p_a(Y) -1$ which implies 
that $\tH^1(\sco_Y(2)) = 0$, a \emph{contradiction}. Consequently, 
the spectrum $(-1,-1,-2,-3)$ \emph{cannot occur}. 

\vskip2mm 

Next, we analyse, case by case, the remaining spectra. Before starting, we 
notice that if $E$ is a vector bundle on $\piii$ satisfying the hypothesis 
of the proposition then $P(E)$ satisfies this hypothesis, too. 

\emph{Indeed}, $P(E)$ is a globally generated vector bundle with 
$\tH^i(P(E)^\ast) = 0$, $i = 0,\, 1$, and has the same Chern classes as $E$. 
Moreover, $\tH^0(P(E)(-2)) \simeq \tH^1(E^\ast(-2)) \simeq \tH^2(E(-2))^\ast 
\simeq \tH^2(\scf)^\ast$. But, as we saw above, the entries of the spectrum of 
$\scf$ must be all $\geq -2$, hence $\tH^2(\scf) = 0$.  

\vskip2mm 

\noindent
{\bf Case 1.}\quad  \textit{The spectrum of $\scf$ is $(0,0,0,0)$}. 

\vskip2mm 

\noindent
In this case $c_3(\scf) = 0$ hence 
$\scf$ is a stable rank 2 \emph{vector bundle} on $\piii$. Actually, it is 
a 4-instanton. Since $\text{h}^i(\scf^\ast(-2)) = \text{h}^i(\scf(-2)) = 0$, 
$i = 0, 1$, it follows that $E \simeq \scf(2)$. It is known, conversely, 
that there exist 4-instantons $F$ such that $E := F(2)$ is globally 
generated (see Chiodera and Ellia~\cite[Lemma~1.10]{ce}\footnote{another 
argument for the existence of such bundles can be found in a recent paper of 
the authors ``Four generated 4-instantons'', \texttt{arXiv:1604.01970}.}; 
this result is also 
stated in D'Almeida~\cite{da}, however the proof there seems to be incomplete, 
according to Chiodera and Ellia). By the way$\, :$ the idea of the 
``duality'' $E \mapsto P(E)$ stated above, which will be very useful in the 
proof below, comes from the paper of D'Almeida. 

\vskip2mm  

\noindent
{\bf Case 2.}\quad  \textit{The spectrum of $\scf$ is $(1,0,0,-1)$}.  

\vskip2mm 

\noindent
In this case $c_3(\scf) = 0$ hence 
$\scf$ is a stable rank 2 \emph{vector bundle} on $\piii$. Since 
$\text{h}^1(\scf^\ast(-2)) = \text{h}^1(\scf(-2)) = 1$, $E$ must be realised 
as a non-trivial extension$\, :$ 
\[
0 \lra \sco_\piii \lra E \lra \scf(2) \lra 0\, .
\]   
It follows, from Corollary~\ref{C:h0e-c1+1}, that $\tH^0(\scf_H(-1)) = 0$, for 
every plane $H \subset \piii$, hence $\scf$ has no unstable plane. Now, a 
result of Chang~\cite[Prop.~1.5]{ch1} implies that $\scf$ is the cohomology 
of a selfdual monad$\, :$ 
\[
0 \lra \sco_\piii(-2) \lra 4\sco_\piii \lra \sco_\piii(2) \lra 0\, . 
\]     
One deduces that $E(-2)$ must be isomorphic to the kernel of the second 
morphism of the monad. 
Conversely, if $K$ is the kernel of an epimorphism 
$4\sco_\piii \ra \sco_\piii(2)$ then, using the Koszul complex associated 
to the epimorphism, it follows that $E := K(2)$ is globally generated. 

\vskip2mm 

\noindent
{\bf Case 3.}\quad  \textit{The spectrum of $\scf$ is $(0,0,0,-1)$}.  

\vskip2mm

\noindent
In this case $c_3(\scf) = 2$, $p_a(Y) = 2$, hence $r = 3$. 

\vskip2mm 

\noindent
{\bf Claim 3.1.}\quad $E(-2)$ \emph{is the cohomology of a monad} 
$\sco_\piii(-1) \ra 7\sco_\piii \ra 3\sco_\piii(1)$. 

\vskip2mm  

\noindent
\emph{Indeed}, we shall apply Beilinson's theorem~\ref{T:beilinson} to 
$E(-2)$. We have to compute $\h^i(E(l))$ for $0 \leq i \leq 3$ and 
$-5 \leq l \leq -2$. Since $\tH^0(E(-2)) = 0$, it follows that $\tH^0(E(l)) 
= 0$ for $l \leq -2$. One shows, as in Case 2 of the proof of 
Prop.~\ref{P:c1=4c2=7n3}, that $\tH^0(E^\ast(1)) = 0$. One deduces, using Serre 
duality, that $\tH^3(E(l)) = 0$ for $l \geq -5$. 

Using the exact sequence $0 \ra \sco_\piii(-2) \ra E(-2) \ra \scf \ra 0$ and 
the fact that $\scf$ has spectrum $(0,0,0,-1)$, one derives that$\, :$ 
\begin{gather*}
\tH^1(E(l)) \simeq \tH^1(\scf(l+2)) = 0\  \text{for} \  l \leq -4,\\  
\h^1(E(-3)) = \h^1(\scf(-1)) = 3,\  \h^1(E(-2)) = \h^1(\scf) = -\chi(\scf) = 
5,\\
\tH^2(E(l)) \simeq \tH^2(\scf(l+2)) = 0\  \text{for}\  l\geq -3,\\  
\h^2(E(-4)) = \h^2(\scf(-2)) - 1 = 0,\  \h^2(E(-5)) = \h^2(\scf(-3)) - 4 = 
1\, .
\end{gather*} 
(We used the fact that $\tH^i(\scf) = 0$ for $i \neq 1$).  
Now, the Theorem of Beilinson (Thm.~\ref{T:beilinson}) implies that $E(-2)$ 
is the middle cohomology of a monad$\, :$ 
\[
\Omega_\piii^3(3) \lra 3\Omega_\piii^1(1) \lra 5\sco_\piii\, .
\] 
Applying the Snake Lemma to the commutative diagram$\, :$  
\[
\begin{CD}
0 @>>> 3\Omega_\piii^1(1) @>>> 12\sco_\piii @>>> 3\sco_\piii(1) @>>> 0\\
@. @VVV @VVV @VVV\\
0 @>>> 5\sco_\piii @>>> 5\sco_\piii @>>> 0 @>>> 0  
\end{CD}
\] 
one derives that $E(-2)$ is the middle cohomology of a monad$\, :$  
\[
\sco_\piii(-1) \lra 7\sco_\piii \lra 3\sco_\piii(1)\, .
\]

\vskip2mm 

\noindent
{\bf Claim 3.2.}\quad $E(-2)$ \emph{is the kernel of an epimorphism} 
$3\sco_\piii(1) \oplus 3\sco_\piii \ra \Omega_\piii(3)$. 

\vskip2mm

\noindent
\emph{Indeed}, 
since the morphism $\sco_\piii(-1) \ra 7\sco_\piii$ is the dual of an 
epimorphism $7\sco_\piii \ra \sco_\piii(1)$ one derives that, actually, 
one has an exact sequence$\, :$  
\[
0 \lra E(-2) \lra \text{T}_\piii(-1) \oplus 3\sco_\piii   
\overset{\eta}{\lra} 3\sco_\piii(1) \lra 0\, .
\]
We assert that the component $\eta_0 : \text{T}_\piii(-1) \lra 
3\sco_\piii(1)$ of $\eta$ is a monomorphism. \emph{Indeed}, if $\eta_0$ 
has, generically, rank 2 then $\Ker \eta_0 \simeq \sco_\piii(a)$ with $a \leq 
0$. If $a = 0$ then $\sco_\piii \subset E(-2)$ which contradicts the fact that 
$\tH^0(E(-2)) = 0$. If $a < 0$ then $c_1(\text{Im}\, \eta_0) = 1 - a$ hence 
$\Cok \eta_0/(\Cok \eta_0)_{\text{tors}} \simeq \sci_T(b)$ with $T$ a subscheme 
of codimension $\geq 2$ of $\piii$ and $b \leq 3 - (1-a) = a+2 \leq 1$. Since 
there is an epimorphism $3\sco_\piii(1) \ra \sci_T(b)$ one deduces that 
$T = \emptyset$ and $b = 1$. As there is no epimorphism $3\sco_\piii  
\ra \sco_\piii(1)$, one gets, again, a contradiction. 

If $\eta_0$ has, generically, rank 1 then $\text{Im}\, \eta_0 \simeq 
\sco_\piii(1)$ or $\text{Im}\, \eta_0 \simeq \sci_L(1)$ for some line $L 
\subset \piii$. In both cases $\Cok \eta_0$ has a quotient of the form 
$2\sco_\piii(1)$ and one gets, again, a contradiction. 

It thus remains that $\eta_0$ is a monomorphism. Let $S_1 := 
\tH^0(\sco_\piii(1))$. Recall that $\text{T}_\piii(-1) \simeq \Omega_\piii^2(3)$. 
Applying the Snake Lemma to the diagram$\, :$  
\[
\begin{CD} 
0 @>>> \Omega_\piii^2(3) @>>> \overset{2}{\bigwedge}S_1\otimes_k\sco_\piii(1) 
@>>> \Omega_\piii(3) @>>> 0\\
@. @V{\eta_0}VV @VVV @VVV\\
0 @>>> 3\sco_\piii(1) @>>> 3\sco_\piii(1) @>>> 0 @>>> 0
\end{CD}
\]
one derives that there exists a monomorphism $\phi_0 : 3\sco_\piii(1) \ra 
\Omega_\piii(3)$ such that one has $\Cok \eta_0 \simeq \Cok \phi_0$. 
One deduces an exact sequence$\, :$ 
\begin{equation}
\label{E:e-2kerphi}
0 \lra E(-2) \lra 3\sco_\piii(1) \oplus 3\sco_\piii   
\overset{\phi}{\lra} \Omega_\piii(3) \lra 0\, .
\end{equation} 

\noindent
{\bf Characterization of the epimorphisms} $\phi : 3\sco_\piii(1) \oplus 
3\sco_\piii \ra \Omega_\piii(3)$ {\bf with} $(\Ker \phi)(2)$ 
{\bf globally generated}. 

\vskip2mm

\noindent 
Let us denote $(\Ker \phi)(2)$ by $E$. 
As we saw in the proof of Claim 3.2 above, the component $\phi_0 : 
3\sco_\piii(1) \ra \Omega_\piii(3)$ of $\phi$ must be a monomorphism. Let 
$\psi := \phi_0(-1) : 3\sco_\piii \ra \Omega_\piii(2)$.  
Using the resolution$\, :$ 
\begin{equation}\label{E:resOmega(2)} 
0 \lra \sco_\piii(-2) \lra \overset{3}{\textstyle \bigwedge}S_1 
\otimes \sco_\piii(-1) 
\lra \overset{2}{\textstyle \bigwedge}S_1 \otimes \sco_\piii \lra 
\Omega_\piii(2) \lra 0 
\end{equation}
one sees that $\psi$ is defined by a 3-dimensional vector subspace $W$ of 
$\overset{2}{\bigwedge}S_1$, that is, $\psi$ is the composite morphism 
$W \otimes \sco_\piii \hookrightarrow \overset{2}{\bigwedge}S_1 \otimes 
\sco_\piii \ra \Omega_\piii(2)$. Moreover, one has an exact sequence$\, :$ 
\[
0 \lra E(-2) \lra 3\sco_\piii \lra (\Cok \psi)(1) \lra 0\, . 
\] 
Let us recall, now, the classification of $3$-dimensional vector subspaces $W$ 
of $\overset{2}{\bigwedge}S_1$. 
Let $\Delta \subset \overset{2}{\bigwedge}S_1$ be the image of the wedge 
product $S_1 \times S_1 \ra \overset{2}{\bigwedge}S_1$, i.e., the set of 
decomposable elements of $\overset{2}{\bigwedge}S_1$. $\Delta$ is the affine 
cone over the Pl\"{u}cker hyperquadric $\mathbb{G} \subset 
\p(\overset{2}{\bigwedge}S_1) \simeq \pv$.  
If $W \subset \Delta$ then one of the following holds$\, :$
\begin{enumerate}
\item[(I)] $W = \overset{2}{\bigwedge}V^\prim$ for some 1-codimensional 
subspace $V^\prim$ of $S_1$$\, ;$ 
\item[(II)] $W = h\wedge S_1$ for some non-zero $h \in S_1$. 
\end{enumerate}

\noindent
If $W$ is not contained in $\Delta$ then three more other cases can 
occur$\, :$ 
\begin{enumerate} 
\item[(III)] $\p(W) \cap \mathbb{G}$ is a nonsingular conic$\, ;$ 
\item[(IV)] $\p(W) \cap \mathbb{G}$ is the union of two intersecting 
lines$\, ;$ 
\item[(V)] $\p(W) \cap \mathbb{G}$ is a double line. 
\end{enumerate}

\noindent
We describe, now, the cokernel of $\psi : W \otimes \sco_\piii \ra 
\Omega_\piii(2)$ in each of these cases. 

\vskip2mm 

\noindent 
$\bullet$\quad In case (I), $\psi$ drops rank at every point of $\piii$. 
\emph{Indeed}, if $x$ is the point of $\piii$ with $\tH^0(\sci_{\{x\}}(1)) = 
V^\prim$ then one has an exact sequence$\, :$ 
\[
0 \lra \overset{3}{\textstyle \bigwedge}V^\prim \otimes \sco_\piii(-1) 
\lra \overset{2}{\textstyle \bigwedge}V^\prim \otimes \sco_\piii 
\overset{\psi}{\lra} \Omega_\piii(2) \overset{\theta}{\lra} 
\sci_{\{x\}}(1) \lra 0  
\] 
whose first part is a subcomplex of the resolution \eqref{E:resOmega(2)} and 
with $\theta$ the composite morphism $\Omega_\piii(2) \ra S_1 \otimes 
\sco_\piii(1) \overset{\lambda}{\lra} \sco_\piii(1)$, where $\lambda : S_1 \ra k$ 
is the linear function with $\Ker \lambda = V^\prim$. 
Consequently, this case \emph{cannot occur} in our context.   

\vskip2mm 

\noindent 
$\bullet$\quad In case (II), let $H\subset \piii$ be the plane of equation 
$h = 0$. Then one has an exact sequence: 
\[
0 \lra 3\sco_\piii \overset{\psi}{\lra} \Omega_\piii(2) \lra 
\Omega_H(2) \lra 0 
\] 
because $\psi$ drops rank by 2 at every point of $H$. This case \emph{cannot 
occur} in our context, either. 

\emph{Indeed}, we know that, in our context, there exists an epimorphism 
$3\sco_\piii \ra (\Cok \psi)(1)$. One would deduce the existence of an 
epimorphism $3\sco_H \ra \Omega_H(3)$. But such an epimorphism \emph{cannot 
exist} because its kernel would be $\sco_H(-3)$ and $\tH^1(\Omega_H(3)) = 0$. 

\vskip2mm 

\noindent 
$\bullet$\quad In case (III), there exists, according to 
Lemma~\ref{L:wsubsetwedge2v}(a) from Appendix~\ref{A:case3}, a $k$-basis 
$h_0, \ldots ,h_3$ of $S_1$ such that $W$ is generated by $h_0 \wedge h_1$, 
$h_2 \wedge h_3$, $(h_0 + h_2) \wedge (h_1 + h_3)$. Let $W^\prim$ be the 
2-dimensional subspace of $W$ generated by $h_0 \wedge h_1$ and 
$h_2 \wedge h_3$, let $\alpha : S_1 \ra S_1$ be the linear automorphism 
defined by $\alpha(h_i) = -h_i$, $i = 0,\, 1$, $\alpha(h_i) = h_i$, $i = 2,\, 
3$, and let $\pi$ be the composite morphism$\, :$ 
\[
\Omega_\piii(1) \lra S_1 \otimes \sco_\piii \overset{\alpha}{\lra} 
S_1 \otimes \sco_\piii \overset{\text{ev}}{\lra} \sco_\piii(1)\, . 
\]      
Then the image of $\pi(1)$ is $\sci_{L_1 \cup L_2}(2)$, where $L_1$ (resp., 
$L_2$) is the line of equations $h_0 = h_1 = 0$ (resp., $h_2 = h_3 = 0$) 
(consider the composite morphism $\overset{2}{\bigwedge} S_1 \otimes \sco_\piii 
\ra \Omega_\piii(2) \overset{\pi(1)}{\lra} \sco_\piii(2)$). Moreover, one has 
an exact sequence$\, :$ 
\[
0 \lra W^\prim \otimes \sco_\piii \xra{\psi \vb W^\prim} \Omega_\piii(2) 
\overset{\pi(1)}{\lra} \sci_{L_1 \cup L_2}(2) \lra 0\, . 
\]  
Since $\pi(1)((h_0 + h_2) \wedge (h_1 + h_3)) = 2h_0h_3 - 2h_1h_2$ it follows 
that 
\[
\Cok \psi \simeq \sci_{L_1 \cup L_2 , Q}(2) 
\]
where $Q$ is the nonsingular quadric surface of equation $h_0h_3 - h_1h_2 = 0$. 

Consequently, in this case one has an exact sequence$\, :$ 
\[
0 \lra E(-2) \lra 3\sco_\piii \lra \sci_{L_1 \cup L_2,Q}(3) \lra 0 
\]
which decomposes into two exact sequences$\, :$  
\begin{gather*} 
0 \lra \scg \lra 3\sco_Q \overset{\e}{\lra}  
\sci_{L_1 \cup L_2,Q}(3) \lra 0\, ,\\
0 \lra 3\sco_\piii(-2) \lra E(-2) \lra \scg \lra 0\, , 
\end{gather*}
where $\scg$ is the kernel of the right morphism in the first exact 
sequence. It follows that $E$ is globally generated if and only if $\scg(2)$  
is globally generated.  

\vskip2mm 

\noindent
$\bullet$\quad In case (IV), there exists, according to 
Lemma~\ref{L:wsubsetwedge2v}(b) from Appendix~\ref{A:case3}, a $k$-basis 
$h_0, \ldots ,h_3$ of $S_1$ such that $W$ is generated by $h_0 \wedge h_1$, 
 $h_1 \wedge h_2$, $h_2 \wedge h_3$. Let $W^\prim \subset W$, $\alpha : S_1 
\ra S_1$ and $\pi : \Omega_\piii(1) \ra \sco_\piii(1)$ be as in case (III) 
above. Since $\pi(1)(h_1 \wedge h_2) = 2h_1h_2$ it follows that, in this 
case, 
\[
\Cok \psi \simeq \sci_{L_1 \cup L_2, H_1 \cup H_2}(2) 
\]
where $L_1$ (resp., $L_2$) is the line of equations $h_0 = h_1 = 0$ (resp., 
$h_2 = h_3 = 0$) and $H_i \supset L_i$ is the plane of equation $h_i = 0$, 
$i = 1,\, 2$. 

Consequently, in this case one has an exact sequence$\, :$ 
\[
0 \lra E(-2) \lra 3\sco_\piii \lra \sci_{L_1 \cup L_2,H_1 \cup H_2}(3) \lra 0 
\]
which decomposes into two exact sequences$\, :$  
\begin{gather*} 
0 \lra \scg \lra 3\sco_{H_1 \cup H_2} \overset{\e}{\lra}  
\sci_{L_1 \cup L_2,H_1 \cup H_2}(3) \lra 0\, ,\\
0 \lra 3\sco_\piii(-2) \lra E(-2) \lra \scg \lra 0\, . 
\end{gather*}
It follows that $E$ is globally generated if and only if $\scg(2)$ is globally 
generated. 

\vskip2mm 

\noindent 
$\bullet$\quad In case (V), there exists, according to 
Lemma~\ref{L:wsubsetwedge2v}(c) from Appendix~\ref{A:case3}, a $k$-basis 
$h_0, \ldots ,h_3$ of $S_1$ such that $W$ is generated by $h_0 \wedge h_1$, 
$h_0 \wedge h_2$, $h_0 \wedge h_3 - h_1 \wedge h_2$. Let $W^\secund$ be the 
subspace of $W$ generated by $h_0 \wedge h_1$ and $h_0 \wedge h_3 - 
h_1 \wedge h_2$, let $\beta :S_1 \ra S_1$ be the linear endomorphism defined 
by $\beta(h_i) = 0$, $i = 0,\, 1$, $\beta(h_i) = h_{i-2}$, $i = 2,\, 3$, and 
let $\rho$ be the composite morphism$\, :$   
\[
\Omega_\piii(1) \lra S_1 \otimes \sco_\piii \overset{\beta}{\lra} 
S_1 \otimes \sco_\piii \overset{\text{ev}}{\lra} \sco_\piii(1)\, . 
\]
Denoting by $L$ the line of equations $h_0 = h_1 = 0$ one checks easily, using 
the composite morphism $\overset{2}{\bigwedge} S_1 \otimes \sco_\piii 
\ra \Omega_\piii(2) \overset{\rho(1)}{\lra} \sco_\piii(2)$, that the image of 
$\rho(1)$ is $\sci_Z(2)$, where $Z$ is the divisor $2L$ on the nonsingular 
quadric surface $Q$ of equation $h_0h_3 - h_1h_2 = 0$. Moreover, one has 
an exact sequence$\, :$ 
\[
0 \lra W^\secund \otimes \sco_\piii \xra{\psi \vb W^\secund} \Omega_\piii(2) 
\overset{\rho(1)}{\lra} \sci_Z(2) \lra 0\, . 
\]   
Since $\rho(1)(h_0 \wedge h_2) = h_0^2$ it follows that, in this case$\, :$ 
\[
\Cok \psi \simeq \sci_{Z,H^{(1)}}(2) 
\]
where $H$ is the plane of equation $h_0 = 0$ and $H^{(1)}$ is its first 
infinitesimal neighbourhood in $\piii$, of equation $h_0^2 = 0$. 

Consequently one has, in case (V), an exact sequence$\, :$ 
\[
0 \lra E(-2) \lra 3\sco_\piii \lra \sci_{Z,H^{(1)}}(3) \lra 0 
\]
which decomposes into two exact sequences$\, :$  
\begin{gather*} 
0 \lra \scg \lra 3\sco_{H^{(1)}} \overset{\e}{\lra}  
\sci_{Z,H^{(1)}}(3) \lra 0\, ,\\
0 \lra 3\sco_\piii(-2) \lra E(-2) \lra \scg \lra 0\, . 
\end{gather*}
It follows that $E$ is globally generated if and only if $\scg(2)$ is globally 
generated. 

\vskip2mm 

\noindent 
{\bf Claim 3.3.}\quad \emph{Let $L_1$, $L_2$ be two disjoint lines contained 
in a nonsingular quadric surface $Q \subset \piii$ and consider an epimorphism 
$\e : 3\sco_Q \ra \sci_{L_1 \cup L_2,Q}(3)$ defined by a $3$-dimensional vector 
subspace $\Lambda$ of} $\tH^0(\sci_{L_1 \cup L_2,Q}(3))$. 
\emph{Then} $\Ker \e(2)$ \emph{is 
globally generated if and only if the linear subsystem $\vb \Lambda \vb$ of} 
$\vb \sco_Q(3) \vb$ \emph{contains no divisor supported on a union of 
lines}. 
   
\vskip2mm

\noindent
\emph{Indeed}, let us fix an isomorphism $\p^1\times \p^1 \izo Q$ such that 
$L_1$ and $L_2$ belong to the linear system $\vb \sco_Q(1,0) \vb$. 
Choose $\lambda_i \in \tH^0(\sco_Q(1,0))$ vanishing on $L_i$, $i = 1, 2$.  
Multiplication by $\lambda_1\lambda_2$ induces an isomorphism  
$\sco_Q(1,3) \izo \sci_{L_1 \cup L_2,Q}(3)$ and there exists a 3-dimensional 
vector subspace $\Lambda^\prim$ of $\tH^0(\sco_Q(1,3))$ such that $\Lambda = 
\lambda_1\lambda_2\Lambda^\prim$.  
The kernel $\scg$ of $\e$ coincides with the kernel of the epimorphism 
$\e^\prim : 3\sco_Q \ra \sco_Q(1,3)$ defined by $\Lambda^\prim$ hence it 
is a locally free $\sco_Q$-module of rank 2. 
We have to show $\scg(2,2)$ is globally generated if and only if the linear 
system $\vb \Lambda^\prim \vb \subset \vb \sco_Q(1,3) \vb$ contains no 
divisor supported on a union of lines.  

Since $\scg$ has rank 2 and $\overset{2}{\bigwedge}\scg \simeq \sco_Q(-1,-3)$ 
one deduces that $\scg^\ast \simeq \scg(1,3)$. Dualizing the exact 
sequence$\, :$  
\begin{equation}\label{E:g00} 
0 \lra \scg \lra 3\sco_Q \lra \sco_Q(1,3) \lra 0 
\end{equation}
one thus gets an exact sequence$\, :$  
\begin{equation}
\label{E:g13}
0 \lra \sco_Q(-1,-3) \lra 3\sco_Q \lra \scg(1,3) \lra 0 
\end{equation} 
which tensorized by $\sco_Q(1,-1)$ produces an exact sequence$\, :$  
\begin{equation}
\label{E:g22} 
0 \lra \sco_Q(0,-4) \lra 3\sco_Q(1,-1) \lra \scg(2,2) \lra 0\, .
\end{equation}
It follows that $\text{h}^0(\scg(2,2)) = \text{h}^1(\sco_Q(0,-4)) = 3$. If 
$L \subset Q$ is a line belonging to the linear system $\vb \sco_Q(1,0) \vb$ 
then restricting \eqref{E:g00} to $L$ one gets an exact sequence$\, :$  
\[
0 \lra \scg_L \lra 3\sco_L \lra \sco_L(3) \lra 0
\]
from which one deduces that $\scg_L \simeq \sco_L(-1) \oplus \sco_L(-2)$ or 
$\scg_L \simeq \sco_L \oplus \sco_L(-3)$. Notice that the latter case occurs 
iff there is a nonzero element $t \in \Lambda^\prim \subset \tH^0(\sco_Q(1,3))$ 
whose associated divisor $(t)_0$ is of the form $L + D$ with 
$D \in \vb \sco_Q(0,3) \vb$. 

Now, $\scg(2,2)$ globally generated implies that $\scg_L(2)$ is globally 
generated for every line $L \in \vb \sco_Q(1,0) \vb$ hence 
$\scg_L \simeq \sco_L(-1) \oplus \sco_L(-2)$ 
for any such line. We will show that, conversely, if $\scg_L \simeq 
\sco_L(-1) \oplus \sco_L(-2)$ for every line $L \in \vb \sco_Q(1,0) \vb$ 
then $\scg(2,2)$ is globally generated. \emph{Indeed}, let $L$ be an 
arbitrary line in the linear system $\vb \sco_Q(1,0) \vb$. 
Tensorizing \eqref{E:g13} by 
$\sco_Q(0,-1)$ one gets an exact sequence$\, :$  
\[
0 \lra \sco_Q(-1,-4) \lra 3\sco_Q(0,-1) \lra \scg(1,2) \lra 0 
\] 
from which one deduces that $\tH^0(\scg(1,2)) = 0$. From the exact 
sequence$\, :$  
\[
0 \lra \tH^0(\scg(1,2)) \lra \tH^0(\scg(2,2)) \lra \tH^0(\scg_L(2)) 
\]
and from the fact that $\text{h}^0(\scg(2,2)) = 3 = \text{h}^0(\scg_L(2))$ one 
deduces that $\tH^0(\scg(2,2)) \izo \tH^0(\scg_L(2))$. Since we assumed that 
$\scg_L(2)$ is globally generated for every line $L$ in the linear system  
$\vb \sco_Q(1,0) \vb$ it follows that $\scg(2,2)$ is globally generated. 
Claim 3.3 is proven. 

\vskip2mm  

In order to show that there exists epimorphisms $\e : 3\sco_Q \ra 
\sci_{L_1 \cup L_2,Q}(3)$ with $\Ker \e(2)$ globally generated,   
we notice that, using the notation from the above proof of Claim 3.3, 
the linear system $\vb \Lambda^\prim \vb \subset \vb \sco_Q(1,3) \vb$ contains 
no divisor supported on a union of lines iff the 3-dimensional vector subspace 
$\Lambda^\prim$ of $\tH^0(\sco_Q(1,3))$ intersects the image of the bilinear map 
$\tH^0(\sco_Q(1,0)) \times \tH^0(\sco_Q(0,3)) \ra \tH^0(\sco_Q(1,3))$ only 
in 0.     

\vskip2mm 

For the characterization of the epimorphisms $\e : 3\sco_{H_1 \cup H_2} \ra 
\sci_{L_1 \cup L_2,H_1 \cup H_2}(3)$ (resp., $\e : 3\sco_{H^{(1)}} \ra 
\sci_{Z,H^{(1)}}(3)$) with $\Ker \e(2)$ globally generated we refer the reader to 
Prop.~\ref{P:case3IV} (resp., Prop.~\ref{P:case3V}) in Appendix~\ref{A:case3}.  

\vskip2mm 

\noindent
{\bf Case 4.}\quad  \textit{The spectrum of $\scf$ is $(0,0,-1,-1)$}.  

\vskip2mm 

\noindent
In this case $c_3(\scf) = 4$, $p_a(Y) = 3$ hence $r = 4$. Using the exact 
sequence $0 \ra \sco_\piii(-2) \ra E(-2) \ra \scf \ra 0$ and the fact that 
$\scf$ has spectrum $(0,0,-1,-1)$ one gets that$\, :$
\begin{gather*} 
\tH^1(E(l)) \simeq \tH^1(\scf(l+2)) = 0\  \text{for}\  l \leq -4,\\ 
\h^1(E(-3)) = \h^1(\scf(-1)) = 2,\  \h^1(E(-2)) = \h^1(\scf) = -\chi(\scf) 
= 4\, .  
\end{gather*}
(We used the fact that $\tH^i(\scf) = 0$ for $i \neq 1$). 
Moreover, $\tH^2(\scf(l)) = 0$ for $l \geq -1$ hence, by Serre duality$\, :$ 
\[
\tH^1(E^\ast(l)) \simeq \tH^2(E(-l-4))^\ast \simeq \tH^2(\scf(-l-2))^\ast = 0   
\  \text{for}\   l \leq -1. 
\]    

\noindent
{\bf Claim 4.1.}\quad $\tH^0(E(-1)) = 0$. 

\vskip2mm 

\noindent 
\emph{Indeed}, as we saw at the beginning of the proof of the proposition, 
$P(E)$ satisfies the hypothesis of the proposition and has the same Chern 
classes as $E$.  Consequently, one has an exact sequence$\, :$ 
\[
0 \lra (\rho - 2)\sco_\piii \lra P(E) \lra \scf^\prim \lra 0
\]  
with $\rho = \text{rk}\, P(E)$ and where $\scf^\prim$ is a stable rank 2 
reflexive sheaf on $\piii$ with the same Chern classes as $\scf$. It follows   
$\scf^\prim$ has spectrum $(0,0,-1,-1)$ (there is no other spectrum allowed by 
these Chern classes). One deduces, from what has been proven at the beginning 
of this case, that $\tH^1(P(E)^\ast(-1)) = 0$. But $\tH^0(E(-1)) \simeq 
\tH^1(P(E)^\ast(-1))$.  

\vskip2mm

\noindent
{\bf Claim 4.2.}\quad \emph{Any non-zero global section of} $E^\ast(1)$ 
\emph{vanishes in at most one point}. 

\vskip2mm 

\noindent
\emph{Indeed}, consider such a section and let $\phi : E \ra \sco_\piii(1)$ 
be the corresponding morphism. Since $\tH^0(E^\ast) = 0$, it follows that 
$\text{Im}\, \phi = \sci_Z(1)$, where $Z$ is a closed subscheme of $\piii$ of 
codimension $\geq 2$. Since $E$ is globally generated, $\sci_Z(1)$ is 
globally generated hence $Z$ is a line, a point or the empty set. Assume that 
$Z$ is a line and put $F := \Ker \phi$. $F$ is a rank 3 vector bundle.  
It follows, from Claim 4.1 and from the fact that $P(E)$ 
shares with $E$ the same cohomological properties, that $\tH^0(P(E)(-1)) = 0$. 
Applying, now, the Snake Lemma to the diagram$\, :$ 
\[
\begin{CD} 
0 @>>> \tH^0(F)\otimes \sco_\p @>>> \tH^0(E)\otimes \sco_\p @>>> 
2\sco_\p @>>> 0\\
@. @VVV @VVV @VVV\\
0 @>>> F @>>> E @>>> \sci_Z(1) @>>> 0
\end{CD}
\] 
one deduces an exact sequence$\, :$ 
\[
0 \lra P(E)^\ast \lra \tH^0(F)\otimes \sco_\p \lra F \lra \sco_\p(-1) \lra 0
\]
which decomposes into two short exact sequences$\, :$ 
\begin{gather*}
0 \lra K \lra F \lra \sco_\p(-1) \lra 0\, ,\\
0 \lra P(E)^\ast \lra \tH^0(F) \otimes \sco_\p \lra K \lra 0\, .
\end{gather*} 
One deduces, from the first short exact sequence, that $K$ is a rank 2 vector 
bundle, and from the second one that $K$ has the same Chern classes as $E$. 
But \emph{this is not possible} because $c_3(E) = 4 \neq 0$. 

\vskip2mm

\noindent
{\bf Claim 4.3.}\quad $\tH^1_\ast(E^\ast) = 0$ \emph{hence} $\tH^2_\ast(E) = 0$. 

\vskip2mm 

\noindent
\emph{Indeed}, we saw at the beginning of this case, that $\tH^1(E^\ast(l)) 
= 0$ for $l \leq -1$. By hypothesis, $\tH^1(E^\ast) = 0$. Since 
$Y$ is connected, $\tH^2(E^\ast) = 0$. Moreover, $\tH^3(E^\ast(-1)) \simeq 
\tH^0(E(-3)) = 0$ hence, by Lemma~\ref{L:cm}, the graded module 
$\tH^1_\ast(E^\ast)$ is generated in degrees $\leq 1$. Consequently, it remains 
to show that $\tH^1(E^\ast(1)) = 0$.   

It follows, from Riemann-Roch, that $\h^0(E^\ast(1)) - \h^1(E^\ast(1)) = 
\chi(E^\ast(1)) = 2$. Take a non-zero global section of $E^\ast(1)$ and let 
$\phi : E \ra \sco_\p(1)$ be the corresponding morphism. As we saw in the 
proof Claim 4.2, $\text{Im}\, \phi = \sci_Z(1)$, where $Z$ consists of a 
point or is the empty set. $\Ker \phi$ is a rank 3 reflexive sheaf with 
$c_1(\Ker \phi) = 3$. Put $\scg := (\Ker \phi)(-1)$. It has $c_1(\scg) = 0$. 
Dualizing the exact sequence$\, :$ 
\[
0 \lra \scg(1) \lra E \lra \sci_Z(1) \lra 0
\] 
one gets an exact sequence $0 \ra \sco_\p(-1) \ra E^\ast \ra \scg^\ast(-1) 
\ra 0$. From Claim 4.1, $\tH^0(\scg) = 0$. Moreover, $\tH^0(\scg^\ast(-1)) = 0$ 
because $\tH^0(E^\ast) = 0$. Lemma~\ref{L:h0fast}(a) implies, now that 
$\h^0(\scg^\ast) \leq 1$, hence $\h^0(E^\ast(1)) \leq 2$ hence $\h^1(E^\ast(1)) 
= 0$, and Claim 4.3 is proven.  

\vskip2mm

\noindent
{\bf Claim 4.4.}\quad $E(-2)$ \emph{is the kernel of an epimorphism} 
$4\sco_\piii \oplus 2\sco_\piii(-1) \ra 2\sco_\piii(1)$. 

\vskip2mm

\noindent
\emph{Indeed}, we showed, at the beginning of Case 4, 
that $\tH^1(E(l)) = 0$ for $l \leq -4$, $\text{h}^1(E(-3)) = 2$ and  
$\text{h}^1(E(-2)) = 4$. Moreover, $\tH^2(E(-3)) \simeq \tH^2(\scf(-1)) = 0$ 
and $\tH^3(E(-4)) \simeq \tH^0(E^\ast)^\ast = 0$ hence, by 
Lemma~\ref{L:east1}(c), 
the graded module $\tH^1_\ast(E)$ is generated in degrees $\leq -2$. 
Considering a basis of $\tH^1(E(-2))$ one gets an extension$\, :$ 
\[
0 \lra E(-2) \lra E^\prim \lra 4\sco_\piii \lra 0\, .
\]    
One has $\tH^1_\ast(E^\prim) \simeq 2k(1)$ and $\tH^2_\ast(E^\prim) \simeq 
\tH^2_\ast(E)(-2) = 0$ (by Claim 4.3). One deduces, using Horrocks 
correspondence, that $E^\prim \simeq 2\Omega_\piii(1) \oplus \sco_\piii(a) 
\oplus \sco_\piii(b)$. Since $c_1(E^\prim) = c_1(E(-2)) = -4$ and since 
$\tH^0(E^\prim) = 0$ it follows that $ a = b = -1$. 
Consequently, one has an exact sequence$\, :$  
\[
0 \lra E(-2) \lra 2\Omega_\piii(1) \oplus 2\sco_\piii(-1)   
\overset{\e}{\lra} 4\sco_\piii \lra 0\, .
\]
Let $\eta : 2\sco_\piii(-1) \ra 4\sco_\piii$, 
$\e_i : \Omega_\piii(1) \ra 4\sco_\piii$, $i = 1, 2$, be the 
components of $\e$. For $i = 1, 2$, $\e_i$ factorizes as$\, :$  
\[
\Omega_\piii(1) \overset{u}{\lra} 4\sco_\piii  
\overset{g_i}{\lra} 4\sco_\piii\, .
\]
Since there is no epimorphism $2\sco_\piii(-1) \ra \sco_\piii$, it 
follows that $\text{Im}\, g_1 + \text{Im}\, g_2 = 4\sco_\piii$. 
Applying the Snake Lemma to the diagram$\, :$ 
\[
\begin{CD} 
0 @>>> 2\Omega_\p(1) \oplus 2\sco_\p(-1)   
@>{2u \, \oplus \, \text{id}}>>  
8\sco_\p \oplus 2\sco_\p(-1)   
@>>> 2\sco_\p(1) @>>> 0\\
@. @V{\e}VV @VV{(g_1\, ,\, g_2\, ,\, \eta)}V @VVV\\
0 @>>> 4\sco_\p @>>> 4\sco_\p @>>> 0 @>>> 0 
\end{CD}
\] 
one derives an exact sequence$\, :$  
\[
0 \lra E(-2) \lra 4\sco_\piii \oplus 2\sco_\piii(-1)   
\overset{\rho}{\lra} 2\sco_\piii(1) \lra 0 
\]
and Claim 4.4 is proven. 

\vskip2mm 

\noindent
{\bf Characterization of the epimorphisms} $\rho : 
4\sco_\piii \oplus 2\sco_\piii(-1) \ra 2\sco_\piii(1)$ {\bf for which} 
$(\Ker \rho)(2)$ {\bf is globally generated}. 

\vskip2mm 

\noindent
Put $E := (\Ker \rho)(2)$. 
Let us denote by $\phi$ (resp., $\psi$) 
the component $: 4\sco_\piii \ra 2\sco_\piii(1)$ 
(resp., $2\sco_\piii(-1) \ra 2\sco_\piii(1)$) of $\rho$. 
One deduces an exact sequence$\, :$ 
\begin{equation}
\label{E:phie(-2)}
0 \lra \Ker \phi \lra E(-2) \lra 2\sco_\piii(-1)  
\overset{\overline{\psi}}{\lra} \Cok \phi \lra 0 
\end{equation}
where $\overline{\psi}$ is the composite morphism$\, :$  
\[
2\sco_\piii(-1) \overset{\psi}{\lra} 2\sco_\piii(1) \lra \Cok \phi \, .
\]

\vskip2mm 

\noindent
{\bf Claim 4.5.}\quad \emph{The degeneracy scheme of $\phi$ is 
$0$-dimensional}. 

\vskip2mm

\noindent 
\emph{Indeed}, $\phi$ corresponds to a $k$-linear map $A : k^4 \ra k^2\otimes 
S_1$ ($S = k[X_0,\ldots ,X_3]$). Since $\tH^0(E(-2)) = 0$ it follows that 
$\tH^0(\Ker \phi) = 0$ hence $A$ is injective. If $A$ is not stable then, 
as we notice in the first part of Remark~\ref{R:2x4} 
from Appendix~\ref{A:case4}, there exists a line 
$L \subset \piii$ such that $\phi$ degenerates along $L$. One cannot have 
$\phi_L = 0$ because there is no epimorphism $2\sco_L(-1) \ra 
2\sco_L(1)$. Consequently, $\Cok \phi_L \simeq \sco_L(a) \oplus 
\scq$ with $a \geq 1$ and $\scq$ a torsion sheaf on $L$. One gets, from the 
exact sequence$\, :$  
\[
E_L(-2) \lra 2\sco_L(-1) \lra \Cok \phi_L \lra 0\, , 
\] 
an epimorphism $E_L(-2) \twoheadrightarrow \sco_L(b)$ with $b \leq -3$, 
which contradicts the fact that $E$ is globally generated. It thus remains 
that $A$ \emph{is stable}. Now, according to Remark~\ref{R:2x4}, there are 
only two cases in which the degeneracy locus of $\phi$ has positive 
dimension: in one of the cases there exists a line $L^\prim \subset \piii$ 
such that $\Cok \phi_{L^\prim} \simeq \sco_{L^\prim}(2)$ and this contradicts, as 
we have just seen, the fact that $E$ is globally generated. In the other case, 
there exists a (smooth) conic $C \subset \piii$ such that, identifying $C$ 
with $\p^1$, $\Cok \phi_C \simeq \sco_{\p^1}(3)$. Using the exact 
sequence$\, :$  
\[
E_C(-2) \lra 2\sco_C(-2) \lra \Cok \phi_C \lra 0
\]  
one contradicts, again, the fact that $E$ is globally generated. It remains 
that the degeneracy scheme of $\phi$ is $0$-dimensional and Claim 4.5 is 
proven. 

\vskip2mm 

\noindent
Let us denote by $Z$ the degeneracy scheme of $\phi$. Since $Z$ has dimension 
0, the sheaf $\scg := \Cok \phi^\ast(-1)$ is a rank 2 \emph{reflexive} sheaf 
with Chern classes $c_1(\scg) = 0$, $c_2(\scg) = 2$, $c_3(\scg) = 4$. 
Dualizing the exact sequence$\, :$  
\begin{equation}
\label{E:cokphiast-1} 
0 \lra 2\sco_\piii(-2) \xra{\phi^\ast(-1)} 
4\sco_\piii(-1) \lra \scg \lra 0
\end{equation}  
and taking into account that $\scg^\ast \simeq \scg$ one gets an exact 
sequence$\, :$ 
\begin{equation}
\label{E:kerphi}
0 \lra \scg(-1) \lra 4\sco_\piii \overset{\phi}{\lra} 
2\sco_\piii(1) \lra \Cok \phi \lra 0\, .
\end{equation} 
Using the fact that $\tH^1(\scg) = 0$ one deduces 
that for any morphism $\tau : 2\sco_\piii(-1) \ra 
2\sco_\piii(1)$ such that the composite morphism$\, :$  
\[
2\sco_\piii(-1) \overset{\tau}{\lra} 2\sco_\piii(1) \lra \Cok \phi 
\]
is 0 there exists a morphism $\tau^\prim : 2\sco_\piii(-1) \ra 
4\sco_\piii$ such that $\tau = \phi \circ \tau^\prim$. It follows 
that if $\rho^\prim : 2\sco_\piii(-1) \oplus 4\sco_\piii  
\ra 2\sco_\piii(1)$ is another morphism with $\phi^\prim = \phi$ and 
${\overline{\psi}}^{\, \prim} = \overline{\psi}$ then $\rho^\prim$ differs 
from $\rho$ by an automorphism of $2\sco_\piii(-1) \oplus 
4\sco_\piii$ invariating $4\sco_\piii$. 

Now, fix a linear form $\lambda \in S_1$ vanishing at no point of $Z$. Since 
the multiplication by $\lambda : \Cok \phi(-1) \ra \Cok \phi$ is an 
isomorphism and since any morphism $2\sco_\piii(-1) \ra \Cok \phi(-1)$ 
can be lifted to a morphism $2\sco_\piii(-1) \ra 2\sco_\piii$ 
(because $\tH^2(\scg(-1)) = 0$),   
\emph{we can assume}, using the above observation, 
that $\psi$ is a composite morphism of the form$\, :$ 
\[
2\sco_\piii(-1) \overset{{\widehat \psi}}{\lra} 2\sco_\piii  
\overset{\lambda}{\lra} 2\sco_\piii(1)\, . 
\]   
Denoting by $\widetilde \psi$ the composite morphism$\, :$  
\[
2\sco_\piii(-1) \overset{{\widehat \psi}}{\lra} 2\sco_\piii  
\lra \Cok \phi(-1)  
\]
one can derive from the exact sequence \eqref{E:phie(-2)} an exact 
sequence$\, :$  
\begin{equation} 
\label{E:g(-1)e(-2)o(-1)} 
0 \lra \scg(-1) \lra E(-2) \lra 2\sco_\piii(-1)  
\overset{{\widetilde \psi}}{\lra} \Cok \phi(-1) \lra 0\, .
\end{equation}

Now, according to Remark~\ref{R:2x4} from Appendix~\ref{A:case4}, 
$Z$ is either a locally complete intersection of length 4 or a fat point. 

\vskip2mm

\noindent 
{\bf Subcase 4.6.}\quad $Z$ \emph{is a locally complete intersection}. 

\vskip2mm

\noindent
In this case, by Remark~\ref{R:2x4} and the beginning of the proof of 
Lemma~\ref{L:ilcupz2}, $\phi$ is defined, up to a linear change of coordinates 
in $\piii$, by a matrix of the form$\, :$ 
\[
\begin{pmatrix} 
X_0 & X_1 & X_2 & X_3\\
h_0 & h_1 & h_2 & h_3
\end{pmatrix}
\] 
and if one denotes by $\eta$ the composite morphism$\, :$  
\[
\begin{CD}
\sco_\piii @>{\binom{0}{\text{id}}}>> 2\sco_\piii \lra \Cok \phi(-1) 
\end{CD}
\]
then one has an exact sequence$\, :$ 
\[
\Omega_\piii \overset{t^\ast}{\lra} \sco_\piii \overset{\eta}{\lra} 
\Cok \phi(-1) \lra 0 
\]
for some section $t \in \tH^0(\text{T}_\piii)$ whose scheme of zeroes is 
exactly $Z$. In particular, it follows that $\Cok \phi(-1) \simeq \sco_Z$. 
Using the Koszul 
complex associated to $t^\ast : \Omega_\piii \ra \sco_\piii$, one derives easily 
that $\tH^i(\sci_Z(1)) = 0$, $i = 0, 1$. It follows that any morphism 
$2\sco_\piii(-1) \ra \Cok \phi(-1)$ lifts to a morphism 
$2\sco_\piii(-1) \ra \sco_\piii$. This morphism is defined by two linear 
forms $\lambda_0,\,  \lambda_1 \in S_1$. Thus, we may assume that the matrix 
of $\rho$ has the form$\, :$  
\[
\begin{pmatrix}
X_0 & X_1 & X_2 & X_3 & 0 & 0\\
h_0 & h_1 & h_2 & h_3 & \lambda \lambda_0 & \lambda \lambda_1
\end{pmatrix}
\]
where $\lambda \in S_1$ is a fixed linear form vanishing at no point of $Z$. 
The exact sequence \eqref{E:g(-1)e(-2)o(-1)} becomes, in this case$\, :$  
\begin{equation} 
\label{E:g(-1)e(-2)o(-1)zlci}
0 \lra \scg(-1) \lra E(-2) \lra 2\sco_\piii(-1)  
\overset{{\widetilde \psi}}{\lra} \sco_Z \lra 0
\end{equation}  
where the morphism $\widetilde \psi$ is equal to the composite morphism$\, :$  
\[ 
2\sco_\piii(-1) \xra{(\lambda_0\, ,\, \lambda_1)} \sco_\piii \lra \sco_Z\, . 
\]
Let us put $\sck := \Ker {\widetilde \psi}$. 
We notice, at this point, that $\lambda_0$ and $\lambda_1$ must be linearly 
independent. \emph{Indeed}, otherwise $\sck$ would be isomorphic to 
$\sco_\piii(-1) \oplus \sci_Z(-1)$ and this would contradict the fact that $E$ 
is globally generated. Let, now, $L \subset \piii$ be the line of equations 
$\lambda_0 = \lambda_1 = 0$. Since $\widetilde \psi$ is an epimorphism, one 
must have $L \cap Z = \emptyset$. Applying the Snake Lemma to the 
diagram$\, :$  
\[
\begin{CD}
0 @>>> \sco_\piii(-2) @>>> 2\sco_\piii(-1) @>>> \sci_L @>>> 0\\
@. @VVV @VVV @VVV\\
0 @>>> \sck @>>> 2\sco_\piii(-1) @>{{\widetilde \psi}}>> \sco_Z @>>> 0
\end{CD}
\]  
one gets an exact sequence$\, :$ 
\begin{equation}
\label{E:o(-2)kilcupz} 
0 \lra \sco_\piii(-2) \lra \sck \lra \sci_{L \cup Z} \lra 0\, .
\end{equation}
One derives that $E$ is globally generated iff $\sci_{L \cup Z}(2)$ is globally 
generated. Now, if $L \cap Z = \emptyset$ then $\tH^1(\sci_{L \cup Z}(2)) = 0$. 
\emph{Indeed}, consider the exact sequence$\, :$  
\[
0 \lra \sci_{L \cup Z} \lra \sci_Z \lra \sco_L \lra 0 
\]
and a plane $H \supset L$ such that $H \cap Z = \emptyset$. Since 
$\tH^i(\sci_Z(1)) = 0$, $i = 0, 1$, one deduces, from the exact sequence$\, :$  
\[
0 \lra \sci_Z(-1) \lra \sci_Z \lra \sco_H \lra 0\, ,
\]
that $\tH^0(\sci_Z(2)) \izo \tH^0(\sco_H(2))$. Since $\tH^0(\sco_H(2)) \ra 
\tH^0(\sco_L(2))$ is surjective it follows that $\tH^0(\sci_Z(2)) \ra 
\tH^0(\sco_L(2))$ is surjective hence $\tH^1(\sci_{L \cup Z}(2)) = 0$ and 
$\text{h}^0(\sci_{L \cup Z}(2)) = 3$. 

\vskip2mm 

If $\sci_{L \cup Z}(2)$ is globally generated then, for every plane $H \supset 
L$, $\text{length}(H \cap Z) \leq 1$. \emph{Indeed}, one has$\, :$ 
\[
\sci_{L \cup Z}\cdot \sco_H = \sci_{L \cup (H \cap Z),H} \simeq \sci_{H \cap Z,H}(-1) 
\]
and $\sci_{H \cap Z,H}(1)$ is globally generated iff $\text{length}(H \cap Z) 
\leq 1$. 

\vskip2mm

\emph{Conversely}, assume that for every plane $H \supset L$, 
$\text{length}(H \cap Z) \leq 1$. Then $\sci_{L \cup Z}(2)$ is globally 
generated. This is shown in Lemma~\ref{L:ilcupz2} from Appendix~\ref{A:case4}, 
but in the case where $Z$ consists of 4 simple points $P_0, \ldots , P_3$  
(not contained in a plane) there is a simple argument. \emph{Indeed}, 
in this case the above condition is equivalent to the fact that the line 
$L$ intersects none of the six edges of the tetrahedron with vertices 
$P_0, \ldots , P_3$. Let $h_i$ be an equation of the plane determined by 
$\{P_0, \ldots ,P_3\} \setminus \{P_i\}$ and let $h_i^\prim$ be an equation 
of the plane determined by $L$ and $P_i$, $i = 0, \ldots ,3$. The degenerate  
quadratic forms $h_ih_i^\prim$, $i = 0, \ldots , 3$, vanish on $L\cup Z$. 
They are not linearly independent (because $\text{h}^0(\sci_{L\cup Z}(2)) = 3$) 
but it is easy to see that they generate $\sci_{L\cup Z}(2)$. 

\vskip2mm 

\noindent
{\bf Subcase 4.7.}\quad $Z$ \emph{is a fat point}. 

\vskip2mm

\noindent 
In this case, $E$ is always globally generated. We defer the proof of this 
fact to Lemma~\ref{L:z=fatpoint} in Appendix~\ref{A:case4}. 
   
\vskip2mm

\noindent
{\bf Case 5.}\quad  \textit{The spectrum of $\scf$ is $(0,-1,-1,-1)$}.  

\vskip2mm

\noindent
In this case $c_3(\scf) = 6$, $p_a(Y) = 4$ hence $r = 5$. 

\vskip2mm 

\noindent
{\bf Claim 5.1.}\quad $\tH^2_\ast(E) = 0$. 

\vskip2mm

\noindent
\emph{Indeed}, $\tH^2(E(l)) \simeq \tH^2(\scf(l+2)) = 0$ for $l \geq -3$. 
On the other hand, since $\h^1(\scf(-1)) = 1$, Lemma~\ref{L:east1}(b) 
implies that $\tH^2(E(l)) = 0$ for $l \leq -4$.   

\vskip2mm 

\noindent
{\bf Claim 5.2.}\quad $\tH^0(E(-1)) = 0$. 

\vskip2mm

\noindent
\emph{Indeed}, as in the proof of Claim 4.1 above, $P(E)$ and $E$ have the 
same Chern classes and the same cohomological properties. In particular, 
by Claim 5.1 above, $\tH^1_\ast(P(E)^\ast) = 0$. But $\tH^0(E(-1)) \simeq 
\tH^1(P(E)^\ast(-1)) = 0$. 

\vskip2mm 

\noindent
{\bf Claim 5.3.}\quad $E(-2)$ \emph{is the kernel of an epimorphism} 
$5\sco_\piii(-1) \oplus \Omega_\piii(1) \ra 3\sco_\piii$.  

\vskip2mm

\noindent
\emph{Indeed}, we have to deal with the graded module $\tH^1_\ast(E)$. 
Using the exact sequence $0 \ra \sco_\piii(-2) \ra E(-2) \ra \scf \ra 0$ and 
the fact that $\scf$ has spectrum $(0,-1,-1,-1)$ one gets that$\, :$ 
\begin{gather*}
\tH^1(E(l)) \simeq \tH^1(\scf(l+2)) = 0\  \text{for}\  l \leq -4,\\
\h^1(E(-3)) = \h^1(\scf(-1)) = 1,\  \h^1(E(-2)) = \h^1(\scf) = -\chi(\scf) 
= 3\, . 
\end{gather*} 
(We used the fact that $\tH^i(\scf) = 0$ for $i \neq 1$). Moreover, 
$\tH^2(E(-3)) \simeq \tH^2(\scf(-1)) = 0$ and $\tH^3(E(-4)) \simeq 
\tH^0(E^\ast)^\ast = 0$  hence, by Lemma~\ref{L:east1}(c), the graded module 
$\tH^1_\ast(E)$ is generated in degrees $\leq -2$.   
Consider the extension$\, :$  
\[
0 \lra E(-2) \lra E^\prim \lra 3\sco_\piii \lra 0
\] 
defined by choosing a basis of the 3-dimensional vector space $\tH^1(E(-2))$. 
One has $\tH^2_\ast(E^\prim) = 0$ and $\tH^1_\ast(E^\prim) = k(1)$. 
It follows that $E^\prim \simeq \scl \oplus \Omega_\piii(1)$, where 
$\scl$ is a direct sum of line bundles. One has $\text{rk}\, E^\prim = 8$, 
$c_1(E^\prim) = -6$, hence $\text{rk}\, \scl = 5$ and $c_1(\scl) = -5$. 
Since $\tH^0(E^\prim) = 0$ one gets that $\scl \simeq 5\sco_\piii(-1)$. 
One thus obtains an exact sequence$\, :$ 
\[
0 \lra E(-2) \lra 5\sco_\piii(-1) \oplus \Omega_\piii(1) 
\overset{\alpha}{\lra} 3\sco_\piii \lra 0
\]
and Claim 5.3 is proven. 

\vskip2mm 

\noindent
Recalling the exact sequence  
$0 \ra \Omega_\piii(1) \ra 4\sco_\piii \ra \sco_\piii(1) \ra 0$, 
the component $\alpha_2 : \Omega_\piii(1) \ra 3\sco_\piii$ of the 
epimorphism $\alpha$ above factorizes as$\, :$ 
\[
\Omega_\piii(1) \lra 4\sco_\piii \overset{\phi}{\lra} 
3\sco_\piii\, .
\]

\noindent
{\bf Claim 5.4.}\quad $\text{rk}\, \phi = 3$. 

\vskip2mm 

\noindent
\emph{Indeed}, if $\text{rk}\, \phi = 2$ then applying the Snake Lemma 
to the diagram$\, :$  
\[
\begin{CD}
0 @>>> \Omega_\piii(1) @>>> 4\sco_\piii @>>> \sco_\piii(1) @>>> 0\\
@. @VV{\alpha_2}V @VV{\phi}V @VVV\\
0 @>>> 2\sco_\piii @>>> 2\sco_\piii @>>> 0 @>>> 0
\end{CD}
\]
one gets an exact sequence$\, :$ 
\[
0 \lra \sco_\piii(-1) \lra \Omega_\piii(1) \overset{\alpha_2}{\lra} 
2\sco_\piii \lra \sco_L(1) \lra 0
\]
for some line $L \subset \piii$. It follows that $\sco_\piii(-1) \subset 
\Ker \alpha = E(-2)$ which \emph{contradicts} the fact that $\tH^0(E(-1)) = 
0$ (see Claim 5.2). 

If $\text{rk}\, \phi = 1$ then one has an exact sequence$\, :$  
\[
0 \lra \sck \lra \Omega_\piii(1) \overset{\alpha_2}{\lra}  
\sco_\piii \lra \sco_{\{x\}} \lra 0
\] 
for some point $x \in \piii$. It follows that $\sck \subset E(-2)$. 
But $\h^0(\sck(1)) \geq \h^0(\Omega_\piii(2)) - 
\h^0(\sci_{\{x\}}(1)) = 6 - 3 = 3$ and this \emph{contradicts} the fact that 
$\tH^0(E(-1)) = 0$. 

Finally, if $\phi = 0$ then $\Omega_\piii(1) \subset E(-2)$ and this 
\emph{contradicts} again the fact that $\tH^0(E(-1)) = 0$. Claim 5.4 is 
proven.  

\vskip2mm 

\noindent 
{\bf Claim 5.5.}\quad $E(-2)$ \emph{is the kernel of an epimorphism} 
$\sco_\piii \oplus 5\sco_\piii(-1) \ra \sco_\piii(1)$.  

\vskip2mm 

\noindent 
\emph{Indeed}, applying the Snake Lemma to the diagram$\, :$
\[
\begin{CD}
0 @>>> \Omega_\piii(1) @>>> 4\sco_\piii @>>> \sco_\piii(1) @>>> 0\\
@. @VV{\alpha_2}V @VV{\phi}V @VVV\\
0 @>>> 3\sco_\piii @>>> 3\sco_\piii @>>> 0 @>>> 0
\end{CD}
\]
one gets an exact sequence$\, :$ 
\[
0 \lra \Omega_\piii(1) \overset{\alpha_2}{\lra}  
3\sco_\piii \lra \sco_H(1) \lra 0
\]
for some plane $H \subset \piii$. Applying the Snake Lemma to the 
diagram$\, :$  
\[
\begin{CD}
0 @>>> 0 @>>> \Omega_\piii(1) @>>> \Omega_\piii(1) @>>> 0\\
@. @VVV @VVV @VVV\\
0 @>>> E(-2) @>>> 5\sco_\piii(-1) \oplus \Omega_\piii(1) @>>> 
3\sco_\piii @>>> 0
\end{CD}
\]
one gets an exact sequence$\, :$  
\[
0 \lra E(-2) \lra 5\sco_\piii(-1) \lra \sco_H(1) \lra 0 
\]
from which one gets an exact sequence$\, :$ 
\[
0 \lra E(-2) \lra \sco_\piii \oplus 5\sco_\piii(-1) \lra \sco_\piii(1) \lra 0 
\]
and Claim 5.5 is proven. 

\vskip2mm 

\noindent
{\bf Characterization of the epimorphisms} $\rho : \sco_\piii \oplus 
5\sco_\piii(-1) \ra \sco_\piii(1)$ {\bf with} $(\Ker \rho)(2)$ 
{\bf globally generated}. 

\vskip2mm 

\noindent
The cokernel of the component $\sco_\piii \ra \sco_\piii(1)$ of such an 
epimorphism $\rho$ is of the form $\sco_H(1)$ for some plane $H \subset 
\piii$. One deduces an exact sequence$\, :$ 
\[
0 \lra \Ker \rho \lra 5\sco_\piii(-1)  
\overset{\overline \rho}{\lra} \sco_H(1) \lra 0\, . 
\]
${\overline \rho} : 5\sco_\piii(-1) \ra \sco_H(1)$ is defined by an 
epimorphism $\e : 5\sco_H(-1) \ra \sco_H(1)$. Let $M$ be 
the rank 4 vector bundle on $H \simeq \pii$ defined by the exact 
sequence$\, :$ 
\begin{equation}
\label{E:mohoh}
0 \lra M \lra 5\sco_H \overset{\e (1)}{\lra} \sco_H(2) \lra 0\, .
\end{equation} 

\vskip2mm 

\noindent
{\bf Claim 5.6.}\quad $(\Ker \rho)(2)$ \emph{is globally generated if and 
only if} $M(1)$ \emph{is globally generated}. 

\vskip2mm

\noindent
\emph{Indeed}, applying the Snake Lemma to the diagram$\, :$  
\[
\begin{CD}
0 @>>> \Ker \rho @>>> 5\sco_\piii(-1) @>>> \sco_H(1) @>>> 0\\
@. @VV{\beta}V @VVV @\vert\\
0 @>>> M(-1) @>>> 5\sco_H(-1) @>>> \sco_H(1) @>>> 0
\end{CD}
\]
one gets an exact sequence$\, :$ 
\[
0 \lra 5\sco_\piii(-2) \lra \Ker \rho \overset{\beta}{\lra} M(-1) \lra 0
\]
from which Claim 5.6 follows.  

\vskip2mm 

\noindent
{\bf Claim 5.7.}\quad \emph{If $M(1)$ is globally generated then} 
$\tH^0(M) = 0$ \emph{which means that the epimorphism}  
$\e (1) : 5\sco_H \ra \sco_H(2)$  
\emph{is defined by some base point free, 
$5$-dimensional, vector subspace} $V \subset \tH^0(\sco_H(2))$.  

\vskip2mm

\noindent
\emph{Indeed}, if $\text{h}^0(M) = 1$ then $M \simeq \sco_H \oplus G$, where 
$G$ is defined by an exact sequence$\, :$  
\[
0 \lra G \lra 4\sco_H \lra \sco_H(2) \lra 0\, .
\] 
In this case, $G(1)$ is a globally generated rank 3 vector bundle with  
$c_1(G(1)) = 1$ and $c_2(G(1)) = 3$, which is not possible because 
of the inequality $c_2 \leq c_1^2$ for the Chern classes of a globally 
generated vector bundle. 

If $\text{h}^0(M) = 2$ then $M \simeq 2\sco_H \oplus G^\prim$, 
where $G^\prim$ is defined by an exact sequence$\, :$  
\[
0 \lra G^\prim \lra 3\sco_H \lra \sco_H(2) \lra 0\, .
\]  
In this case, $G^\prim(1)$ is a globally generated rank 2 vector bundle with 
$c_1(G^\prim(1)) = 0$ and $c_2(G^\prim(1)) = 3$ which, again, is not possible. 
Consequently, $\tH^0(M) = 0$.  

\vskip2mm

\noindent
{\bf Claim 5.8.}\quad $M(1)$ \emph{is globally generated if and only if 
$V$ contains no subspace of the form} $\tH^0(\sco_H(1))\cdot \ell$, 
\emph{with} $0 \neq \ell \in \tH^0(\sco_H(1))$. 

\vskip2mm 

\noindent
\emph{Indeed}, if there exists $0 \neq \ell \in \tH^0(\sco_H(1))$ such that 
$\tH^0(\sco_H(1))\cdot \ell \subset V$ then, restricting the exact sequence 
\eqref{E:mohoh} to the line $L \subset H$ of equation $\ell = 0$ one gets 
$M_L \simeq 3\sco_L \oplus \sco_L(-2)$, hence $M(1)$ is not globally 
generated. 

\emph{Conversely}, assume that $V$ contains no subspace of the above form.  
Then, for every line $L \subset H$, $M_L \simeq 
2\sco_L \oplus 2\sco_L(-1)$. One has $\text{h}^1(M(-2)) = 
1$, $\text{h}^1(M(-1)) = 3$, $\text{h}^1(M) = 1$, and $\tH^1(M(l)) = 0$ for 
$l \notin \{-2, -1, 0\}$ (for $l \geq 1$ this follows from 
Lemma~\ref{L:LePotier}, taking into account that $\text{h}^1(M) = 1$ and that 
there are lines $L \subset H$ such that $M_L \simeq 2\sco_L \oplus 
2\sco_L(-1)$ hence $\tH^1(M_L) = 0$). 
Moreover, for every $0 \neq \ell \in \tH^0(\sco_H(1))$, the 
multiplication by $\ell : \tH^1(M(-2)) \ra \tH^1(M(-1))$ is injective and 
the multiplication by $\ell : \tH^1(M(-1)) \ra \tH^1(M)$ is surjective. 
One deduces that the graded $S$-module $\tH^1_\ast(M^\ast) \simeq 
\tH^1_\ast(M(-3))^\ast$ is generated by $\tH^1(M^\ast(-3)) \simeq 
\tH^1(M)^\ast$, which is 1-dimensional. A non-zero element of 
$\tH^1(M^\ast(-3))$ defines an extension$\, :$  
\[
0 \lra M^\ast(-3) \lra A \lra \sco_\pii \lra 0
\]
with $\tH^1_\ast(A) = 0$, i.e., with $A$ a direct sum of line bundles. As 
$\text{rk}\, A = 5$, $c_1(A) = -10$, and $\tH^0(A(1)) = 0$ one deduces that 
$A \simeq 5\sco_\pii(-2)$. It follows that $M(1)$ is globally 
generated. 

We notice that a general 5-dimensional subspace $V$ of $\tH^0(\sco_H(2))$ 
satisfies the condition from Claim 5.8 (see \cite[Example~2.1(c)]{c2}).  

\vskip2mm 

\noindent
{\bf Case 6.}\quad  \textit{The spectrum of $\scf$ is $(0,-1,-1,-2)$}.  

\vskip2mm 

\noindent
In this case $c_3(\scf) = 8$, 
$p_a(Y) = 5$ hence $r = 6$. One has $\h^2(\scf(-1)) = 1$, hence 
$\h^2(\sci_Y(1)) = 1$, hence $\h^1(\sco_Y(1)) = 1$. Since 
$\text{deg}\, \sco_Y(1) = 8 = 2p_a(Y) - 2$ it follows that $\sco_Y(1) \simeq 
\omega_Y$, i.e., $Y$ is a \emph{canonical curve} of genus 5 (and degree 8). 
A nonzero global section of $\omega_Y(-1) \simeq \sco_Y$ defines a Serre 
extension$\, :$ 
\begin{equation}
\label{E:f3}
0 \lra \sco_\piii \lra F(3) \lra \sci_Y(5) \lra 0 
\end{equation}
where $F$ is a rank 2 vector bundle with Chern classes $c_1(F) = -1$, 
$c_2(F) = 2$. $F$ is stable because $\tH^0(F) \simeq \tH^0(\sci_Y(2)) \simeq 
\tH^0(E(-2)) = 0$. These bundles were studied by Hartshorne and 
Sols~\cite{hs} and, independently, by Manolache~\cite{ma}. 
The zero scheme of a general global section of $F(2)$ is the disjoint union of 
two (nonsingular) conics $C_1$ and $C_2$ (see, for example, 
\cite[Prop.~2.3]{hs}). In other words, $F(2)$ can be realized as an 
extension$\, :$ 
\[
0 \lra \sco_\piii \lra F(2) \lra \sci_{C_1\cup C_2}(3) \lra 0\, .
\]    

\vskip2mm 

\noindent
{\bf Claim 6.1.}\quad $\h^0(E(-1)) = 1$. 

\vskip2mm

\noindent
\emph{Indeed}, $\h^0(E(-1)) = \h^0(\sci_Y(3)) = \h^0(F(1)) = 
\h^0(\sci_{C_1\cup C_2}(2)) = 1$. 

\vskip2mm

\noindent
$\bullet$\quad \emph{The monad of} $\sci_{C_1\cup C_2}(3)$. 

\vskip2mm

\noindent
Let $S := k[X_0,X_1,X_2,X_3]$ be the projective coordinate ring of $\piii$ and 
$S(C_i) = S/I(C_i)$ the projective coordinate ring of $C_i$, $i = 1, 2$. 
$I(C_i)$ is generated by a linear form $h_i$ and by a quadratic form $q_i$, 
$i = 0, 1$. Using the exact sequences$\, :$  
\begin{gather*} 
0 \lra S(-3) \xra{\binom{-q_i}{h_i}} S(-1)\oplus S(-2) 
\xra{(h_i\, ,\, q_i)} S \lra S(C_i) \lra 0,\  i = 1,\, 2,\\
0 \lra \tH^0_\ast(\sci_{C_1\cup C_2}) \lra S \lra S(C_1)\oplus S(C_2) \lra 
\tH^1_\ast(\sci_{C_1\cup C_2}) \lra 0\, ,\\ 
0 \lra S \xra{\binom{\text{id}}{\text{id}}} S\oplus S 
\xra{(-\text{id}\, ,\, \text{id})} S \lra 0\, ,  
\end{gather*} 
one deduces a complex$\, :$ 
\[
0 \lra 2S(-3) \lra  2S(-1) \oplus 2S(-2) \lra S \lra 0\, , 
\]
with the differentials given by the matrices$\, :$  
\[
A^{\text{t}} = 
\begin{pmatrix}
-q_1 & 0 & h_1 & 0\\
0 & -q_2 & 0 & h_2
\end{pmatrix}
\, , \  B = (-h_1,\, h_2,\, -q_1,\, q_2)\, , 
\] 
whose cohomology is $\tH^i_\ast(\sci_{C_1 \cup C_2})$, $i = 0, 1$. 
It follows that $\sci_{C_1\cup C_2}(3)$ is the cohomology of the monad  
$0 \ra 2\sco_\piii \overset{\alpha}{\lra} 2\sco_\piii(2) \oplus 
2\sco_\piii(1) \overset{\beta}{\lra} \sco_\piii(3) \ra 0$, 
with $\alpha$ and $\beta$ defined by the matrices $A$ and $B$, respectively. 

\vskip2mm

\noindent
$\bullet$\quad \emph{A presentation of} $\sci_Y(4)$. 

\vskip2mm 

\noindent
Consider the bundle $K$ defined by the exact sequence$\, :$  
\[
0 \lra K \lra 2\sco_\piii(2) \oplus 2\sco_\piii(1)  
\overset{\beta}{\lra} \sco_\piii(3) \lra 0\, .
\] 
From the above complex one gets an exact sequence$\, :$  
\[
0 \lra 2\sco_\piii \overset{\alpha}{\lra} K \lra \sci_{C_1 \cup C_2}(3) 
\lra 0\, .
\] 
Since $\tH^1_\ast(K^\ast) = 0$, one gets a commutative diagram$\, :$  
\[
\begin{CD}
0 @>>> 2\sco_\piii @>>> K @>>> \sci_{C_1 \cup C_2}(3) @>>> 0\\
@. @VVV @VVV @\vert\\
0 @>>> \sco_\piii @>>> F(2) @>>> \sci_{C_1 \cup C_2}(3) @>>> 0
\end{CD}
\]
Since the bottom line of this diagram does not split, the left vertical 
morphism must be non-zero. One deduces an exact sequence$\, :$  
\[
0 \lra \sco_\piii \overset{t}{\lra} K \lra F(2) \lra 0\, .
\] 
(Since the cokernel of $t$ is the locally free sheaf $F(2)$, the  
morphism $t$  must be defined by $(-a_1q_1,\, -a_2q_2,\, a_1h_1,\, a_2h_2)$, 
with $a_1,\, a_2 \in k\setminus \{0\}$). From the extension 
\eqref{E:f3} one gets, now, an exact sequence$\, :$ 
\[
0 \lra \sco_\piii \oplus \sco_\piii(-1) \lra K \lra \sci_Y(4) \lra 0\, . 
\]  

\noindent
$\bullet$\quad \emph{A presentation of} $E$. 

\vskip2mm 

\noindent
Using the fact that $\tH^1_\ast(K^\ast) = 0$ one deduces the existence of a 
commutative diagram$\, :$ 
\[
\begin{CD}
0 @>>> \sco_\piii \oplus \sco_\piii(-1) @>>> K @>>> \sci_Y(4) @>>> 0\\
@. @V{\phi}VV @VVV @\vert\\
0 @>>> 5\sco_\piii @>>> E @>>> \sci_Y(4) @>>> 0
\end{CD}
\]
which, by dualization, gives a commutative diagram$\, :$  
\[
\begin{CD}
0 @>>> \sco_\piii(-4) @>>> E^\ast @>>> \sco_\piii \oplus 4\sco_\piii   
@>>> \omega_Y @>>> 0\\
@. @\vert @VVV @VV{\phi^\ast}V @\vert\\
0 @>>> \sco_\piii(-4) @>>> K^\ast @>>> \sco_\piii \oplus \sco_\piii(1) 
@>>> \omega_Y @>>> 0
\end{CD}
\]
One deduces that $\tH^0(\phi^\ast)$ is an isomorphism. Now, one gets an 
exact sequence$\, :$  
\[ 
0 \lra \sco_\piii \oplus \sco_\piii(-1) 
\xra{\binom{\sigma}{\phi}} K \oplus 5\sco_\piii \lra 
E \lra 0\, . 
\] 
But, since $\tH^0(\phi^\ast)$ is an isomorphism, one can cancel a direct 
summand $\sco_\piii$ from the first and the second term of the above exact 
sequence and one gets a new exact sequence of the form$\, :$  
\[
0 \lra \sco_\piii(-1) \xra{\binom{s^\prim}{u}} K \oplus 4\sco_\piii \lra 
E \lra 0\, . 
\]

\noindent
{\bf Claim 6.2.}\quad $E$ \emph{is as in item} (vi) \emph{of the statement}. 

\vskip2mm

\noindent
\emph{Indeed}, up to a linear change of coordinates and up to an automorphism 
of $2\sco_\piii(2) \oplus 2\sco_\piii(1)$, one can assume that the 
morphism $\beta$ whose kernel is $K$ coincides with the morphism 
$p : 2\sco_\piii(2) \oplus 2\sco_\piii(1) \ra \sco_\piii(3)$ defined 
by $X_0,\, X_1,\, X_2^2,\, X_3^2$. 

The condition $\tH^i(E^\ast) = 0$, $i = 0,\, 1$, is equivalent to the fact that 
$\tH^0(u^\ast) : \tH^0(4\sco_\piii) \ra \tH^0(\sco_\piii(1))$ is an 
isomorphism hence, up to an automorphism of $4\sco_\piii$, one can 
assume that $u$ is defined by $X_0, \dots ,X_3$. 

Let $s : \sco_\piii(-1) \ra K$ be the morphism induced by 
$(0,0,-X_3^2,X_2^2) : \sco_\pii(-1) \ra 2\sco_\piii(2) \oplus 
2\sco_\piii(1)$. Using the resolution of $K$ deduced from the Koszul 
complex associated to $X_0,\, X_1,\, X_2^2,\, X_3^2$, one sees easily that 
there exist a constant $c \in k$ and a morphism $\tau : 4\sco_\piii  
\ra K$ such that $s^\prim = cs + \tau \circ u$. It follows that, up to the  
automorphism of $K \oplus 4\sco_\piii$  defined by the matrix 
$\left(\begin{smallmatrix} \text{id} & -\tau\\ 0 & \text{id} 
\end{smallmatrix}\right)$, one can assume that $s^\prim = cs$. 

If $c = 0$ then one would get $E \simeq K \oplus \text{T}_\piii(-1)$. But this 
is not possible because $K$ is not globally generated (because 
$\sci_{C_1 \cup C_2}(3)$ is not globally generated because $C_1 \cup C_2$ has a 
4-secant). It remains that $c \neq 0$ and one can, actually, assume that 
$c = 1$. 

\vskip2mm

\noindent
Although Case 6 is concluded, we make one more observation. 

\vskip2mm

\noindent
{\bf Claim 6.3.}\quad $P(E) \simeq E$. 

\vskip2mm

\noindent 
\emph{Indeed}, by definition, $P(E)^\ast$ is the kernel of the evaluation 
morphism $\tH^0(E)\otimes\sco_\piii \ra E$ hence it is isomorphic to the 
kernel of the morphism$\, :$ 
\[
\begin{pmatrix} \text{ev}_K & 0 & s\\ 0 & \text{id} & u \end{pmatrix} : 
(\tH^0(K)\otimes \sco_\piii)\oplus 4\sco_\piii \oplus \sco_\piii(-1) 
\lra K \oplus 4\sco_\piii  
\]   
and this is isomorphic to the kernel of the morphism$\, :$ 
\[
(\text{ev}_K,\, s) : (\tH^0(K)\otimes \sco_\piii) \oplus \sco_\piii(-1) 
\lra K\, . 
\] 
But from the Koszul resolution of $K$ one deduces an exact sequence$\, :$ 
\[
0 \lra K^\ast \lra \sco_\piii(1) \oplus 4\sco_\piii \oplus \sco_\piii(-1) 
\lra K \lra 0\, , 
\]
from which one gets, now, an exact sequence$\, :$ 
\[
0 \lra P(E)^\ast \lra K^\ast \oplus (\tH^0(\sco_\piii(1))\otimes \sco_\piii) 
\lra \sco_\piii(1) \lra 0\, .
\]
Dualizing this exact sequence and using the arguments from the proof of 
Claim 6.2 one obtains that $P(E) \simeq E$.  

\vskip2mm 

\noindent
{\bf Case 7.}\quad  \textit{The spectrum of $\scf$ is $(-1,-1,-1,-1)$}.

\vskip2mm 

\noindent
In this case $c_3(\scf) = 8$, $p_a(Y) = 5$ hence $r = 6$. 

\vskip2mm

\noindent
{\bf Claim 7.1.}\quad $\tH^2_\ast(E) = 0$. 

\vskip2mm

\noindent
\emph{Indeed}, $\tH^1(\scf(-1)) = 0$ hence, by 
Lemma~\ref{L:east1}(a), $E^\ast$ is 1-regular. In particular, one has 
$\tH^1(E^\ast(l)) = 0$ for $l \geq 0$ hence, by Serre duality, 
$\tH^2(E(l)) = 0$ for $l \leq -4$. 
On the other hand, $\tH^2(E(l)) \simeq \tH^2(\scf(l+2)) = 0$ for $l \geq -3$. 

\vskip2mm

\noindent
{\bf Claim 7.2.}\quad $E$ \emph{is $0$-regular}. 

\vskip2mm 

\noindent
\emph{Indeed}, $\tH^2(E(-2)) \simeq \tH^2(\scf) = 0$ and $\tH^3(E(-3)) \simeq 
\tH^0(E^\ast(-1))^\ast = 0$. On the other hand, $\tH^1(E(-1)) \simeq 
\tH^2(P(E)^\ast(-1)) \simeq \tH^1(P(E)(-3))^\ast$. As we saw at the beginning 
of the proof of the proposition, $P(E)$ satisfies the hypothesis of the 
proposition and has the same Chern classes as $E$. In particular, one has an 
exact sequence$\, :$ 
\[
0 \lra (\rho - 2)\sco_\piii \lra P(E) \lra \scf^\prim \lra 0
\] 
where $\rho = \text{rk}\, P(E)$ and $\scf^\prim$ is a stable rank 2 reflexive 
sheaf with the same Chern classes as $\scf$. It follows that the spectrum 
of $\scf^\prim$ is either $(0,-1,-1,-2)$ or $(-1,-1,-1,-1)$. But 
$\tH^0(P(E)(-1)) \simeq \tH^1(E^\ast(-1)) \simeq \tH^2(E(-3))^\ast = 0$ and 
Claim 6.1 above implies that the spectrum of $\scf^\prim$ must be the latter 
one. One deduces, now, that $\tH^1(P(E)(-3)) \simeq \tH^1(\scf^\prim (-1)) = 
0$ hence $\tH^1(E(-1)) = 0$ hence $E$ is 0-regular.   

\vskip2mm 

\noindent
{\bf Claim 7.3.}\quad $\tH^1_\ast(E) \simeq 2k(2)$. 

\vskip2mm

\noindent
\emph{Indeed}, $\tH^1(E(l)) \simeq \tH^1(\scf(l+2)) = 0$ for $l \leq -3$ and 
$\h^1(E(-2)) \simeq \h^1(\scf) = - \chi(\scf) = 2$. 
On the other hand, since $E$ is 0-regular, one has $\tH^1(E(l)) = 0$ for 
$l \geq -1$. 

\vskip2mm

\noindent
Since $\text{rk}\, E = 6$, Claim 7.1 and Claim 7.3 above imply that 
$E \simeq 2\Omega_\piii(2)$.     

\vskip2mm

\noindent
{\bf Case 8.}\quad  \textit{The spectrum of $\scf$ is $(-1,-1,-1,-2)$}.  

\vskip2mm 

\noindent
In this case $c_3(\scf) = 10$, $p_a(Y) = 6$ hence $r = 7$. 

\vskip2mm

\noindent
{\bf Claim 8.1.}\quad $\tH^2_\ast(E) \simeq k(3)$. 

\vskip2mm

\noindent
One uses an argument similar to that used in the proof of Claim 7.1 above. 

\vskip2mm

\noindent
{\bf Claim 8.2.}\quad $E$ \emph{is $0$-regular}. 

\vskip2mm 

\noindent
\emph{Indeed}, $\tH^2(E(-2)) \simeq \tH^2(\scf) = 0$ and $\tH^3(E(-3)) 
\simeq \tH^0(E^\ast(-1))^\ast = 0$. On the other hand, $\tH^1(E(-1)) \simeq 
\tH^2(P(E)^\ast(-1)) \simeq \tH^1(P(E)(-3))^\ast$. As in the proof of 
Claim 7.2 above, one has an exact sequence$\, :$
\[
0 \lra (\rho - 2)\sco_\piii \lra P(E) \lra \scf^\prim \lra 0
\]
with $\scf^\prim$ a stable rank 2 reflexive sheaf having the same Chern 
classes as $\scf$. It follows that $\scf^\prim$ has spectrum 
$(-1,-1,-1,-2)$ (this is the only spectrum allowed by these Chern classes). 
In particular, $\tH^1(P(E)(-3)) \simeq \tH^1(\scf^\prim(-1)) = 0$ hence 
$\tH^1(E(-1)) = 0$ hence $E$ is 0-regular. 

\vskip2mm

\noindent
{\bf Claim 8.3.}\quad $\tH^1_\ast(E) \simeq k(2)$. 

\vskip2mm 

\noindent
One uses the same kind of argument as in the proof of Claim 7.3 above. 

\vskip2mm

\noindent
Now, taking into account Claim 8.1 and Claim 8.3, Beilinson's 
theorem~\ref{T:beilinson} implies that$\, :$ 
\[
E(-1) \simeq (\tH^0(E(-1))\otimes_k \sco_\piii) \oplus 
(\tH^1(E(-2)) \otimes_k \Omega_\piii^1(1)) \oplus 
(\tH^2(E(-3)) \otimes_k \Omega_\piii^2(2))\, .
\]
Since $\text{rk}\, E = 7$ and since $\Omega_\piii^2(3) \simeq 
\text{T}_\piii(-1)$ it follows that $E \simeq \sco_\piii(1) \oplus 
\text{T}_\piii(-1) \oplus \Omega_\piii(2)$.   

\vskip2mm

\noindent
{\bf Case 9.}\quad  \textit{The spectrum of $\scf$ is $(-1,-1,-2,-2)$}. 

\vskip2mm 

\noindent
In this case $c_3(\scf) = 12$, $p_a(Y) = 7$ hence $r = 8$. 
One proves that $\tH^2_\ast(E) \simeq 2k(3)$ (as in Claim 7.1), that 
$E$ is 0-regular (as in Claim 8.2) and that $\tH^1_\ast(E) = 0$ (as in 
Claim 7.3). Since $\text{rk}\, E = 8$ and since $c_1(E) = 4$ it follows 
that $E \simeq 2\sco_\piii(1) \oplus 2\text{T}_\piii(-1)$. 
\end{proof}

\newpage

\section{The case $c_1 = 4$, $5 \leq c_2 \leq 8$ on 
$\p^n$, $n \geq 4$} 
\label{S:c1=4n4}

In this section we look at the list of globally generated vector bundles 
$E$ on $\piii$ with $c_1 = 4$, $5 \leq c_2 \leq 8$ and $\tH^0(E(-2)) = 0$ 
established in the previous sections and try to decide which of them extend 
to higher dimensional projective spaces. Our main tools are the results 
\eqref{L:h1=0}--\eqref{R:lift2} from Section~\ref{S:preliminaries}. We also 
use a natural correspondence, provided by Serre's construction, between 
a class of vector bundles $E$ on $\piv$ with $c_1 = 4$ and $\tH^i(E^\ast) = 0$, 
$i = 0,\, 1$, (including all g.~g. vector bundles satisfying these 
conditions) and nonsingular surfaces $Y$ in $\piv$. A very useful fact 
furnished by this correspondence, allowing one to show that many g.~g. vector 
bundles with $c_1 = 4$ on $\piii$ do not extend to $\piv$, is a result of 
Lanteri \cite{l} and Okonek \cite{ok3} providing a lower bound for the 
sectional genus of $Y$.  

The most interesting bundle appearing in this classification is a rank 5 
vector bundle on $\piv$ having a trivial subbundle of rank 2. The quotient 
bundle is the Sasakura rank 3 vector bundle on $\piv$ (conveniently twisted) 
(cf. Abo, Decker and Sasakura \cite{ads}). We investigate thoroughly this 
bundle in Case 6 of the proof of Prop.~\ref{P:c1=4c2=8n4}. We also show, in 
Case 8 of the same proof, that it does not extend to $\pv$.  

\vskip2mm

We need to recall some facts concerning the geometry of surfaces in $\piv$. 
Most of these facts are extracted from the papers of Okonek \cite{ok1}, 
\cite{ok2}, \cite{ok3}. Let $Y \subset \piv$ be a nonsingular surface of 
degree $d$, \emph{geometric genus} $p_g := \h^0(\omega_Y) = \h^2(\sco_Y)$ and 
\emph{irregularity} $q := \h^1(\sco_Y) = \h^1(\omega_Y)$. We assume that 
$Y$ is nondegenerate (i.e., contained in no hyperplane) and that 
$\omega_Y(1)$ is \emph{globally generated}. [The latter condition is not 
really restrictive$\, :$ by Adjunction Theory (cf. Sommese \cite{so} and 
Van de Ven \cite{vdv}), there are only three types of nondegenerate surfaces 
in $\piv$ not satisfying this condition, namely the rational normal scrolls 
(with $d = 3$), the Veronese surfaces (with $d = 4$), and the elliptic 
scrolls with $d = 5$. These elliptic scrolls are described, for example, 
in Hartshorne's book \cite[V,~Example~2.11.6~and~Ex.~2.12,~2.13]{hag}.] 

\vskip2mm 

Let $C := H\cap Y$ be a nonsingular hyperplane section of $Y$. The genus 
$\pi$ of $C$ is called the \emph{sectional genus} of $Y$. 
By the \emph{Adjunction Formula}, $\omega_Y \vb C \simeq \omega_C(-1)$. One 
has $\tH^1(\sco_Y(l)) = 0$ for $l < 0$ (hence $\tH^1(\omega_Y(l)) = 0$ for 
$l > 0$) by \emph{Kodaira Vanishing}. Moreover, since $Y$ is not a Veronese 
surface, \emph{Severi's Theorem} implies that $Y$ is \emph{linearly normal}, 
i.e., $\h^0(\sco_Y(1)) = 5$ (hence $\tH^1(\sci_Y(1)) = 0$). 

\vskip2mm 

Now, let $r := \h^0(\omega_Y(1)) + 1$. By Serre's method of extensions, 
there exists an exact sequence$\, :$ 
\begin{equation}\label{E:oeiy(4)} 
0 \lra (r - 1)\sco_\piv \lra E \lra \sci_Y(4) \lra 0
\end{equation} 
with $E$ locally free of rank $r$, with $c_1(E) = 4$, and such that, 
in the dual sequence$\, :$ 
\begin{equation}\label{E:oeiy(4)dual}
0 \lra \sco_\piv(-4) \lra E^\ast \lra (r - 1)\sco_\piv  
\overset{\delta}{\lra} \omega_Y(1) \lra 0\, ,
\end{equation}
the connecting map $\delta$ is (identified with) the evaluation 
map $\tH^0(\omega_Y(1))\otimes_k\sco_Y \ra \omega_Y(1)$. 

\newpage

\noindent 
It follows 
that $\tH^i(E^\ast) = 0$, $i = 0,\, 1$. Morever, using Kodaira Vanishing$\, :$ 
\[
\tH^2(E^\ast) \simeq \tH^2(E(-5))^\ast \simeq \tH^2(\sci_Y(-1))^\ast 
\simeq \tH^1(\sco_Y(-1))^\ast = 0\, .
\] 

\begin{remark}\label{R:s=qt=pg} 
Of course, the vector bundle $E$ associated above to a surface $Y \subset 
\piv$ is globally generated iff $\sci_Y(4)$ is.  
\emph{Conversely}, by Lemma~\ref{L:h0h1delta}, any globally 
generated vector bundle $E$ on $\piv$ with $c_1(E) = 4$ and such that 
$\tH^i(E^\ast) = 0$, $i = 0,\, 1$, can be obtained in this way. It is amusing 
to notice that, in this case, for any hyperplane $H \subset \piv$, 
one has$\, :$
\begin{gather*}
\tH^0(E_H^\ast) \simeq \tH^1(E^\ast(-1)) \simeq \tH^3(E(-4))^\ast \simeq 
\tH^3(\sci_Y)^\ast \simeq \tH^2(\sco_Y)^\ast \simeq \tH^0(\omega_Y)\, ,\\
\tH^1(E_H^\ast) \simeq \tH^2(E^\ast(-1)) \simeq \tH^2(E(-4))^\ast \simeq 
\tH^2(\sci_Y)^\ast \simeq \tH^1(\sco_Y)^\ast \simeq \tH^1(\omega_Y)\, .
\end{gather*}
It follows, from Lemma~\ref{L:h0h1}, that $E_H \simeq G \oplus t\sco_H$ 
with $G$ defined by an exact sequence $0 \ra s\sco_H \ra F \ra G \ra 0$, 
where $F$ is a globally generated vector bundle on $H$ such that 
$\tH^i(F^\ast) = 0$, $i = 0,\, 1$, $t = \h^0(E_H^\ast) = p_g$, and 
$s = \h^1(E_H^\ast) = q$.     
\end{remark} 

\begin{lemma}\label{L:rciofe}
The rank $r$ and the Chern classes $c_i = c_i(E)$, $i = 1,\ldots ,4$, of the 
vector bundle $E$ defined by the extension \eqref{E:oeiy(4)} are related to the 
invariants of the surface $Y$ by the following formulae$\, :$ 
\[
r = 1 + \pi - q + p_g\, ,\  c_1 = 4\, ,\  c_2 = d\, ,\  c_3 = 2\pi - 2\, ,\  
c_4 = (C+K)^2\, ,
\]
where $K$ is a canonical divisor on $Y$. 
\end{lemma}

\begin{proof}
One deduces, from the Adjunction Formula, an exact sequence$\, :$ 
\[
0 \lra \omega_Y \lra \omega_Y(1) \lra \omega_C \lra 0\, . 
\]
Since $\tH^1(\omega_Y(1)) = 0$ (by Kodaira Vanishing), it follows that  
$\h^0(\omega_Y(1)) = \pi - q + p_g$ whence the formula for $r$. 

The formulae for $c_1$ and $c_2$ are clear. Using the exact sequence$\, :$
\[
0 \lra (r - 1)\sco_H \lra E_H \lra \sci_{C,H}(4) \lra 0\, ,
\]
one gets $c_3(E) = c_3(E_H) = c_3(\sci_{C,H}(4)) = 2\pi - 2$ (see 
Remark~\ref{R:reflexive}). 

As for $c_4$, choosing a convenient basis of $(r - 1)\sco_\piv$, one can 
assume that the components $\delta_1,\ldots ,\delta_{r-1} \in 
\tH^0(\omega_Y(1))$ of $\delta$ 
have the property that the scheme $(\delta_2)_0 \cap (\delta_3)_0$ 
is 0-dimensional and that $(\delta_1)_0 \cap (\delta_2)_0 \cap 
(\delta_3)_0 = \emptyset$ (here $(\delta_i)_0$ denotes the divisor of zeroes 
of $\delta_i$). To the epimorphism $\delta^\prim : 3\sco_\piv \ra 
\omega_Y(1)$ defined by $\delta_1,\, \delta_2,\, \delta_3$ it corresponds a 
Serre extension$\, :$ 
\[
0 \lra 3\sco_\piv \overset{\sigma}{\lra} E^\prim \lra \sci_Y(4) \lra 0
\] 
with $E^\prim$ locally free. Dualizing, one gets an exact sequence$\, :$ 
\[
0 \lra \sco_\piv(-4) \lra E^{\prim \ast} \overset{\sigma^\ast}{\lra} 
3\sco_\piv \overset{\delta^\prim}{\lra} \omega_Y(1) \lra 0\, .
\]
Let $\sigma_1,\, \sigma_2,\, \sigma_3 \in \tH^0(E^\prim)$ be the components of 
$\sigma$. Using the short exact sequence of (vertical) complexes$\, :$ 
\[
\begin{CD}
\omega_Y(1) @>{\text{id}}>> \omega_Y(1) @>>> 0\\
@A{(\delta_2\, ,\, \delta_3)}AA @AA{\delta^\prim}A @AAA\\
2\sco_\piv @>>> 3\sco_\piv @>{\text{pr}_1}>> \sco_\piv\\
@AAA @AA{\sigma^\ast}A @AA{\sigma_1^\ast}A\\
0 @>>> E^{\prim \ast} @>{\text{id}}>> E^{\prim \ast}
\end{CD}
\]
one sees that $\Cok \sigma_1^\ast \simeq \Cok (\delta_2\, ,\, \delta_3)$. It 
follows that $c_4(E) = c_4(\sci_Y(4)) = c_4(E^\prim) = (C+K)^2$. 
\end{proof} 

\begin{thm}[Riemann-Roch]\label{T:rronp4}
Let $E$ be a rank $r$ vector bundle on $\piv$ with Chern classes $c_1,\, c_2,
\, c_3,\, c_4$. Then, $\forall \, l \in \z$$\, :$ 
\begin{gather*}
\chi(E(l)) = \chi(((r-1)\sco_\piv \oplus \sco_\piv(c_1))(l)) - 
\frac{1}{2}(l+2)(l+3)c_2 + \frac{1}{2}(l+2)(c_3 - c_1c_2) \, +\\
+ \, \frac{(2c_1 + 3)(c_3 - c_1c_2) + c_2^2 + c_2 - 2c_4}{12}\, .
\end{gather*}
\end{thm}

\begin{proof}
One can use the elementary approach suggested in the proof of 
\cite[Thm.~2.3]{ha}. 
\end{proof}

\begin{cor}[Schwarzenberger]\label{C:schw} 
The Chern classes of a vector bundle on $\piv$ satisfy the relation   
$(2c_1 + 3)(c_3 - c_1c_2) + c_2^2 + c_2 \equiv 2c_4\  ({\fam0 mod}\  12)$.
\qed 
\end{cor} 

\begin{cor}[Double Point Formula]\label{C:dpf}
Let $Y$ be a nondegenerate, nonsingular surface in $\piv$ with $\omega_Y(1)$ 
globally generated. Then, using the notation introduced at the beginning of 
this section,  
\[
(C + K)^2 = \frac{(d-3)(d-4)}{2} + 1 - \pi -6q + 6p_g\, .
\] 
\end{cor}

\begin{proof}
Let $E$ be the vector bundle defined by the extension \eqref{E:oeiy(4)}. As 
we have already noticed, one has $\tH^i(E^\ast) = 0$, $i = 0,\, 1,\, 2$. 
Moreover$\, :$ 
\begin{gather*}
\tH^3(E^\ast) \simeq \tH^1(E(-5))^\ast \simeq \tH^1(\sci_Y(-1))^\ast = 0\, ,\\
\tH^4(E^\ast) \simeq \tH^0(E(-5))^\ast \simeq \tH^0(\sci_Y(-1))^\ast = 0\, .
\end{gather*}
Consequently, $\chi(E^\ast) = 0$. One applies, now, the Riemann-Roch formula 
to $E^\ast$ and one takes into account Lemma~\ref{L:rciofe}. 
\end{proof} 

\begin{note} The Double Point Formula, as stated in Hartshorne's book 
\cite[Appendix~A,~Example~4.1.3]{hag}, asserts that$\, :$ 
\[
d^2 - 10d - 5\, C\cdot K - 2K^2 + 12\chi(\sco_Y) = 0\, .
\] 
Since $C^2 = d$ and $C\cdot K = 2\pi - 2 - d$ (by the Adjunction Formula), 
it turns out that this formula is equivalent to the one from Cor.~\ref{C:dpf}. 
\end{note}

The following lemma is a weak version of results due to Lanteri~\cite{l}, in 
the case $d = 7$, and to Okonek~\cite[Lemma~2.1]{ok3}, in the case $d = 8$.  
It will be very 
helpful in showing that many globally generated vector bundles on $\piii$ do 
not extend to $\piv$ as globally generated vector bundles. We include, for 
completeness, an argument extracted from the paper of Okonek. 

\begin{lemma}\label{L:lanok}
Let $Y$ be a nonsingular surface of degree $d$ in $\piv$ and let 
$C := H \cap Y$ be a nonsingular hyperplane section of $Y$. If $5 \leq d 
\leq 8$ and if $\omega_Y(1)$ is globally generated then the genus $\pi$ of $C$ 
satisfies the inequality $\pi \geq d-3$. 
\end{lemma}

\begin{proof}
Recall, from the proof of Lemma~\ref{L:rciofe}, the formula$\, :$ 
\begin{equation}
\label{E:h0omegay1}
\text{h}^0(\omega_Y(1)) = \pi - q + p_g\, .
\end{equation} 
Assume, now, that $d \geq 5$ and that $\tH^0(\omega_Y(-1)) = 0$. By 
Serre duality, $\tH^2(\sco_Y(1)) = 0$. By Severi's theorem, 
$\text{h}^0(\sco_Y(1)) = 5$. Applying $\chi(-)$ to the exact sequence$\, :$  
\[
0 \lra \sco_Y \lra \sco_Y(1) \lra \sco_C(1) \lra 0
\]
one derives that$\, :$ 
\begin{equation}
\label{E:pi-d+3}
\pi - d + 3 = \text{h}^1(\sco_Y(1)) - q + p_g\, .
\end{equation} 
Using these relations, the lemma can be proven as it follows. Assume that 
$5 \leq d \leq 8$ and that $\pi < d - 3$. It follows that $d > 2\pi - 2$ 
hence $\tH^1(\sco_C(1)) = 0$, hence $\tH^0(\omega_C(-1)) = 0$. Using the 
exact sequence deduced from the Adjunction Formula$\, :$
\begin{equation}\label{E:adjunction}
0 \lra \omega_Y(-1) \lra \omega_Y \lra \omega_C(-1) \lra 0
\end{equation}
(tensorized by $\sco_Y(l)$, for 
$l \leq 0$) one derives that $\tH^0(\omega_Y) = 0$, i.e., $p_g = 0$. Relation 
\eqref{E:pi-d+3} implies that $q > 0$. Since $\omega_Y(1)$ is globally 
generated, relation \eqref{E:h0omegay1} implies, now, that $\pi \geq 2$. 
This eliminates, already, the case $d = 5$. $\omega_Y(1)$ being globally 
generated, $(C + K)^2 \geq 0$. One deduces, now, from Cor.~\ref{C:dpf}, 
that $6q \leq 4 - \pi$ for $d = 6$, that $6q \leq 7 - 
\pi$ for $d = 7$, and that $6q \leq 11 - \pi$ for $d = 8$. In the cases 
$d = 6$ and $d = 7$ one gets that $q = 0$ (recall that $\pi \geq 2$), a 
\emph{contradiction}. It remains that $d = 8$  and $q = 1$ hence, from 
\eqref{E:pi-d+3}, $\pi = 4$. \eqref{E:h0omegay1} implies, now, that 
$\text{h}^0(\omega_Y(1)) = 3$.  
The base point free linear system $\vert \, C + K\, \vert$
defines a morphism $\varphi : Y \ra \pii$ which must be birational 
because, by Cor.~\ref{C:dpf}, $(C + K)^2 = 1$. But this \emph{contradicts} 
the fact that $\text{h}^1(\sco_Y) = q > 0$.    
\end{proof} 

\begin{cor}\label{C:lanok}
Let $E$ be a globally generated vector bundle on $\piv$ with $c_1 = 4$ and 
$5 \leq c_2 \leq 8$. Then $c_3 \geq 2c_2 - 8$. 
\qed
\end{cor}

The following lemma is also due to Okonek~\cite{ok3}. 

\begin{lemma}\label{L:d=8pi=5} 
Let $Y \subset \piv$ be a nonsingular surface of degree $d = 8$, with 
$\omega_Y(1)$ globally generated, and let $C = H \cap Y$ be a nonsingular 
hyperplane section of $Y$. Assume that $C$ has genus $\pi = 5$. Then 
$p_g = 0$ and $q \leq 1$. 
\end{lemma} 

\begin{proof}
We show, firstly, that one cannot have $\omega_Y \simeq \sco_Y$. 
\emph{Indeed}, if $\omega_Y \simeq \sco_Y$ then there would exist a rank 2 
vector bundle $F$ on $\piv$ given by an extension$\, :$  
\[
0 \lra \sco_\piv \lra F \lra \sci_Y(5) \lra 0\, .
\]
This bundle would have Chern classes $c_1(F) = 5$ and $c_2(F) = 8$ which would 
\emph{contradict} Schwarzenberger's condition Cor.~\ref{C:schw}. 

Since $\omega_Y$ is not isomorphic to $\sco_Y$ and since, by the 
Adjunction Formula, $C \cdot K = 2\pi - 2 - d = 0$, it follows that 
$\tH^0(\omega_Y) = 0$, i.e., $p_g = 0$. Cor.~\ref{C:dpf} and 
the fact that $\omega_Y(1)$ is globally generated implies, now$\, :$ 
\[
6 - 6q = (C+K)^2 \geq 0
\] 
hence $q \leq 1$. Notice, also, that relation \eqref{E:pi-d+3} implies 
that $\text{h}^1(\sco_Y(1)) = q$. 
\end{proof} 

It is often quite difficult to decide whether a concrete vector bundle 
(or coherent sheaf) is globally generated or not. The next lemmata, which 
illustrate this fact, will be used in the proof of Prop.~\ref{P:c1=4c2=8n4} 
below.  

\begin{lemma}\label{L:omegaquasibdl}
Let $\Pi \subset \p^n$ be a linear subspace with $0 \leq \dim \Pi < n$ and 
let $\scf$ be the kernel of an epimorphism $\e : \Omega_\p(1) \ra 
\sco_\Pi(1)$. Then $\scf(1)$ is not globally generated. 
\end{lemma}

\begin{proof}
We use induction on $n$. If $n = 1$ then one must have $\Pi = \{x\}$ for 
some point $x \in \p^1$. As $\Omega_{\p^1}(1) \simeq \sco_{\p^1}(-1)$ it follows 
that $\scf \simeq \sco_{\p^1}(-2)$ hence $\scf(1)$ is not globally generated. 

Assume, now, that $n \geq 2$ and that the assertion has been proven for all 
the lower dimensional projective spaces. Choose a hyperplane $H \supseteq 
\Pi$. If $\scf(1)$ is globally generated then so is $\scg(1)$, where $\scg$ 
is the kernel of $\e_H : \Omega_\p(1) \vb H \ra \sco_\Pi(1)$. One 
has $\Omega_\p(1) \vb H \simeq \sco_H \oplus \Omega_H(1)$. The cokernel of 
the component $\e_H^\prim : \sco_H \ra \sco_\Pi(1)$ of $\e_H$ 
is of the form $\sco_{\Pi^\prim}(1)$, where $\Pi^\prim = \Pi$ if 
$\e_H^\prim = 0$ and where $\Pi^\prim$ is a 1-codimensional linear 
subspace of $\Pi$ if $\e_H^\prim \neq 0$. Let $\e_H^\secund$ 
denote the composite morphism $\Omega_H(1) \ra \sco_\Pi(1) \ra 
\sco_{\Pi^\prim}(1)$. $\e_H^\secund$ is an epimorphism and there is an 
epimorphism  $\scg \ra \Ker \e_H^\secund$. 
We split the proof of the \emph{induction step} into three cases, 
according to the codimension of $\Pi$ in $\p^n$. 

\vskip2mm 

\noindent
{\bf Case 1.}\quad $\Pi$ \emph{is a hyperplane}. 

\vskip2mm

\noindent 
In this case, $H = \Pi$. Since there is no 
epimorphism $\Omega_H(1) \ra \sco_H(1)$ (any global section of 
$\text{T}_H$ vanishes in at least one point of $H$) it follows that  
$\e_H^\prim \neq 0$ hence $\Pi^\prim$ is a hyperplane in $H \simeq 
\p^{n-1}$. By the induction hypothesis, $(\Ker \e_H^\secund)(1)$ 
is not globally generated hence $\scg(1)$ is not globally generated.  

\vskip2mm

\noindent
{\bf Case 2.}\quad  $\Pi$ \emph{is a point} $\{x\}$. 

\vskip2mm 

\noindent
In this case, Corollary~\ref{C:ggquasib} implies that $\scf(1)$ is not globally 
generated. Indeed, recall that $P(\Omega_\p(2)) \simeq \Omega_\p^{n-2}(n-1)$.  
Using the coKoszul resolution of $\Omega_\p^{n-2}(n-1)$$\, :$ 
\[
0 \lra \sco_\p(-2) \lra W\otimes_k\sco_\p(-1) \lra 
\overset{2}{\textstyle \bigwedge}W\otimes_k\sco_\p 
\lra \Omega_\p^{n-2}(n-1) \lra 0\, , 
\] 
with $W = \tH^0(\sco_\p(1))^\ast$, one derives easily that the zero locus of 
any non-zero global section of $\Omega_\p^{n-2}(n-1)$ is either empty or a 
line.   

\vskip2mm

\noindent
{\bf Case 3.}\quad $0 < \dim \Pi < n-1$. 

\vskip2mm

\noindent
In this case, $0 \leq \dim \Pi^\prim < n-1$ and, by the induction hypothesis, 
$(\Ker \e_H^\secund)(1)$ is not globally generated hence $\scg(1)$ is 
not globally generated. 
\end{proof} 

\begin{lemma}\label{L:2x2n+1}
If $F$ is the kernel of an epimorphism $(2n + 1)\sco_{\p^n} \ra 2\sco_{\p^n}(1)$ 
then $F(1)$ is not globally generated. 
\end{lemma}

\begin{proof}
The epimorphism is defined by a $2 \times (2n+1)$ matrix $A$ of linear forms 
on $\p^n$. Since the composite map $(2n + 1)\sco_\p \ra 2\sco_\p(1)  
\overset{\text{pr}_1}{\lra} \sco_\p(1)$ is an epimorphism it follows that 
the linear forms on the first line of $A$ generate $\tH^0(\sco_\p(1))$. 
Modulo an automorphism of $(2n + 1)\sco_\p$, we can assume that the matrix 
$A$ is of the form$\, :$ 
\[
A = 
\begin{pmatrix}
X_0 & \ldots & X_n & 0 & \ldots & 0\\
* & \ldots & * & h_0 & \ldots & h_{n-1}
\end{pmatrix}
\]
Since $A$ defines an epimorphism, at least one of the linear forms 
$h_0, \ldots , h_{n-1}$ must be non-zero. Let $\Pi$ be the linear subspace of 
$\p^n$ of equations $h_0 = \cdots = h_{n-1} = 0$. Applying the Snake Lemma to 
the diagram$\, :$ 
\[
\begin{CD}
0 @>>> F @>>> (2n + 1)\sco_\p @>>> 2\sco_\p(1) @>>> 0\\
@. @VVV @VV{\text{pr}_1}V @VV{\text{pr}_1}V\\
0 @>>> \Omega_\p(1) @>>> (n + 1)\sco_\p @>>> \sco_\p(1) @>>> 0
\end{CD}
\] 
one gets an exact sequence $F \ra \Omega_\p(1) \ra \sco_\Pi(1) \ra 0$.  
Lemma~\ref{L:omegaquasibdl} implies, now, that $F(1)$ is not globally 
generated. 
\end{proof} 

\begin{prop}\label{P:c1=4c2=8n4}
Let $E$ be a globally generated vector bundle on $\p^n$, $n \geq 4$,  
with $c_1 = 4$, $5 \leq c_2 \leq 8$, ${\fam0 H}^i(E^\ast) = 0$, $i = 0, 1$, 
and with ${\fam0 H}^0(E_\Pi(-2)) = 0$ for every $3$-plane $\Pi \subset \p^n$. 
Then one of the following holds$\, :$ 
\begin{enumerate}
\item[(i)] $c_2 = 6$ and $E \simeq 4\sco_\p(1)\, ;$
\item[(ii)] $c_2 = 7$ and $E \simeq 3\sco_\p(1) \oplus 
{\fam0 T}_\p(-1)\, ;$
\item[(iii)] $c_2 = 7$, $n = 4$, and $E \simeq \sco_\piv(1) \oplus 
\Omega_\piv(2)\, ;$ 
\item[(iv)] $c_2 = 7$, $n = 5$, and $E \simeq \Omega_\pv(2)\, ;$
\item[(v)] $c_2 = 8$ and $E \simeq 2\sco_\p(1) \oplus 
2{\fam0 T}_\p(-1)\, ;$ 
\item[(vi)] $c_2 = 8$, $n = 4$, and $E \simeq {\fam0 T}_\piv(-1) \oplus 
\Omega_\piv(2)\, ;$
\item[(vii)] $c_2 = 8$, $n = 4$, and $E \simeq \sco_\piv(1) \oplus 
\Omega_\piv^2(3)\, ;$ 
\item[(viii)] $c_2 = 8$, $n = 4$, and, up to a linear change of coordinates, 
denoting by $(C_p,\, \delta_p)_{p \geq 0}$ the Koszul complex associated to the 
epimorphism $\delta_1 : 4\sco_\piv(-1) \oplus \sco_\piv(-2) \ra \sco_\piv$ 
defined by $X_0, \ldots ,X_3,X_4^2$, one has exact sequences$\, :$ 
\begin{gather*}
0 \lra \sco_\piv(-2) \xra{\delta_5(4)}  
\sco_\piv \oplus 4\sco_\piv(-1) \xra{\delta_4(4)}  
4\sco_\piv(1) \oplus 6\sco_\piv \lra E^\prim \lra 0\, ,\\
0 \lra E \lra E^\prim \overset{\phi}{\lra} \sco_\piv(2) \lra 0\, , 
\end{gather*}   
where $\phi : E^\prim \ra \sco_\piv(2)$ is any morphism with the property that 
${\fam0 H}^0(\phi(-1)) : {\fam0 H}^0(E^\prim(-1)) \ra {\fam0 H}^0(\sco_\piv(1))$ 
is injective $($such morphisms exist and are automatically epimorphisms$)$. 
\end{enumerate}
\end{prop} 

\begin{proof}  
If $\Pi \subset \p^n$ is a 3-plane then, by Lemma~\ref{L:h0h1},  
one has $E_\Pi \simeq G \oplus t\sco_\Pi$ for some bundle $G$ on $\Pi$ 
defined by an exact sequence 
$0 \ra s\sco_\Pi \ra F \lra G \ra 0$,   
where $F$ is a globally generated vector bundle on $\Pi$ with $\tH^i(F^\ast) 
= 0$, $i = 0,\, 1$. Notice that $\tH^0(F(-2)) = 0$ (because $\tH^0(E_\Pi(-2)) 
= 0$) hence $F$ is one of the bundles described in Prop.~\ref{P:c1=4c2=5n3}, 
Prop.~\ref{P:c1=4c2=6n3}, Prop.~\ref{P:c1=4c2=7n3} and 
Prop.~\ref{P:c1=4c2=8n3}. 

Now, Cor.~\ref{C:lanok} implies that $c_3 \geq 2c_2 - 8$.  
On the other hand, $c_3 = 0$ if $c_2 = 5$, by Prop.~\ref{P:c1=4c2=5n3} 
(hence the case $c_2 = 5$ cannot occur), 
$c_3 \leq 4$ if $c_2 = 6$, by Prop.~\ref{P:c1=4c2=6n3}, 
$c_3 \leq 8$ if $c_2 = 7$, by Prop.~\ref{P:c1=4c2=7n3}, 
and $c_3 \leq 12$ if $c_2 = 8$, by Prop.~\ref{P:c1=4c2=8n3}. 
Recall, also, that $c_3$ is even (see the last part of 
Remark~\ref{R:reflexive}). We split the proof into several cases according 
to the 
values of $c_2$ and $c_3$. 

\vskip2mm 

\noindent
{\bf Case 1.}\quad $c_2 = 6$ \emph{and} $c_3 = 4$. 

\vskip2mm

\noindent
In this case, $F$ is as in Prop.~\ref{P:c1=4c2=6n3}(iii), i.e., 
$F \simeq 4\sco_\Pi(1)$. Lemma~\ref{L:lift1} implies that 
$E \simeq 4\sco_\p(1)$. 

\vskip2mm

\noindent
{\bf Case 2.}\quad $c_2 = 7$ \emph{and} $c_3 = 8$. 

\vskip2mm

\noindent
In this case, $F$ is as in Prop.~\ref{P:c1=4c2=7n3}(v), i.e., 
$F \simeq 3\sco_\Pi(1) \oplus \text{T}_\Pi(-1)$. Lemma~\ref{L:lift1} 
implies that $E \simeq 3\sco_\p(1) \oplus \text{T}_\p(-1)$. 

\vskip2mm

\noindent
{\bf Case 3.}\quad $c_2 = 7$ \emph{and} $c_3 = 6$.

\vskip2mm 

\noindent
In this case, $F$ is as in Prop.~\ref{P:c1=4c2=7n3}(iv), i.e., 
$F \simeq 2\sco_\Pi(1) \oplus \Omega_\Pi(2)$. Lemma~\ref{L:lift2} (and 
Remark~\ref{R:lift2}) implies that either $n = 4$ and  
$E \simeq \sco_\piv(1) \oplus \Omega_\piv(2)$, or $n = 5$ and $E \simeq 
\Omega_\pv(2)$. 

\vskip2mm

\noindent
{\bf Case 4.}\quad $c_2 = 8$ \emph{and} $c_3 = 12$.  

\vskip2mm

\noindent
In this case, $F$ is as in Prop.~\ref{P:c1=4c2=8n3}(ix), i.e., 
$F \simeq 2\sco_\Pi(1) \oplus 2\text{T}_\Pi(-1)$.  
Lemma~\ref{L:lift1} implies that $E \simeq 2\sco_\p(1)  
\oplus 2\text{T}_\p(-1)$. 

\vskip2mm

\noindent
{\bf Case 5.}\quad $c_2 = 8$ \emph{and} $c_3 = 10$.  

\vskip2mm 

\noindent
In this case, $F$ is as in Prop.~\ref{P:c1=4c2=8n3}(viii), i.e., 
$F \simeq \sco_\Pi(1) \oplus \text{T}_\Pi(-1) \oplus \Omega_\Pi(2)$. 
Lemma~\ref{L:lift2} (and Remark~\ref{R:lift2}) implies that $n = 4$ and 
either $E \simeq \text{T}_\piv(-1) \oplus \Omega_\piv(2)$ or 
$E \simeq \sco_\piv(1) \oplus \Omega_\piv^2(3)$.      

\vskip2mm

\noindent
{\bf Case 6.}\quad $c_2 = 8$, $c_3 = 8$, $n = 4$ \emph{and} $F\simeq 
2\Omega_\Pi(2)$.   

\vskip2mm

\noindent
We will show that, in this case,  $E$ is as in item (viii) of the statement.  
Let $h = 0$ be the equation of the hyperplane $\Pi \subset \piv$.  
Lemma~\ref{L:hieh=0}(a) implies that $\tH^1(E(l)) = 0$ for $l \leq -3$ and 
Lemma~\ref{L:h1ehast=0} implies that $t = 0$ and that  
$s \leq \dim \Cok(\tH^1(E(-2)) \ra \tH^1(E_\Pi(-2))) = 2 - \h^1(E(-2))$.  

\vskip2mm

\noindent
{\bf Claim 6.1.}\quad $s \neq 0$. 

\vskip2mm 

\noindent 
\emph{Indeed}, if $s = 0$ then $\tH^1_\ast(E_\Pi^\ast) = 0$ hence, by 
Lemma~\ref{L:hieh=0}(a), $\tH^1_\ast(E^\ast) = 0$. Lemma~\ref{L:h1=0} implies, 
now, that $E^\ast$ admits a resolution of the form  
$
0 \ra 2\sco_\piv(-2) \ra 8\sco_\piv(-1) \ra E^\ast \ra 0
$
whence an exact sequence $0 \ra E \ra 8\sco_\piv(1) \ra 
2\sco_\piv(2) \ra 0$. It follows, from Lemma \ref{L:2x2n+1}, that $E$ is 
not globally generated, a \emph{contradiction}. 

\vskip2mm

\noindent
{\bf Claim 6.2.}\quad $\tH^1(E(-2)) \neq 0$. 

\vskip2mm 

\noindent
\emph{Indeed}, if $\tH^1(E(-2)) = 0$ then, by Lemma~\ref{L:hieh=0}(a), 
$\tH^1_\ast(E) = 0$ hence, by Lemma~\ref{L:h1=0}, $E$ admits a resolution of 
the form$\, :$ 
\[
0 \lra 2\sco_\piv(-2) \lra 8\sco_\piv(-1) \lra 
(12 - s)\sco_\piv \lra E \lra 0\, .
\]  
Dualizing this resolution and applying, again, Lemma~\ref{L:2x2n+1} one gets 
a \emph{contradiction}. 

\vskip2mm

The above two claims imply that $s = 1$ and $\h^1(E(-2)) = 1$. One thus has 
an exact sequence$\, :$ 
\[
0 \lra \sco_\Pi \lra 2\Omega_\Pi(2) \lra E_\Pi \lra 0\, .
\]

\noindent
{\bf Claim 6.3.}\quad $\tH^3_\ast(E) = 0$. 

\vskip2mm 

\noindent 
\emph{Indeed}, dualizing the above exact sequence, one gets an exact 
sequence$\, :$ 
\[
0 \lra E_\Pi^\ast \lra 2\text{T}_\Pi(-2) \overset{\e}{\lra} 
\sco_\Pi \lra 0\, .
\]
Composing $\e(1)$ with the epimorphism $8\sco_\Pi \ra 
2\text{T}_\Pi(-1)$ one gets an epimorphism $\e^\prim : 8\sco_\Pi \ra 
\sco_\Pi(1)$. For such an epimorphism, $\tH^0(\e^\prim)$ is surjective, 
hence $\tH^0(\e(1))$ is surjective, too. One deduces that 
$\tH^1_\ast(E_\Pi^\ast) \simeq k$. Since, by hypothesis, $\tH^1(E^\ast) = 0$, 
Lemma~\ref{L:hieh=0}(a) implies that $\tH^1_\ast(E^\ast) = 0$ hence, by Serre 
duality, $\tH^3_\ast(E) = 0$. 

\vskip2mm   

\noindent 
{\bf The intermediate cohomology of} $E$.\quad  
One has $\tH^1(E(l)) = 0$ for $l \leq -3$ by Lemma \ref{L:hieh=0}(a). 
Using the exact sequence$\, :$ 
\[
0 = \tH^0(E_\Pi(-1)) \lra \tH^1(E(-2)) \overset{h}{\lra} \tH^1(E(-1)) 
\lra \tH^1(E_\Pi(-1)) = 0
\] 
one gets that $\h^1(E(-1)) = \h^1(E(-2)) = 1$. Since $E_\Pi$ is 0-regular, 
Lemma~\ref{L:LePotier} implies that $\tH^1(E(l)) = 0$ for $l \geq 0$. 

On the other hand, we saw in the proof of Claim 6.3 that 
$\tH^1_\ast(E_\Pi^\ast) \simeq k$ hence, by Serre duality, $\tH^2_\ast(E_\Pi) 
\simeq k(4)$. Lemma~\ref{L:hieh=0}(a) implies that $\tH^2(E(l)) = 0$ for 
$l \leq -5$, while Lemma~\ref{L:hieh=0}(b) implies that $\tH^2(E(l)) = 0$ for 
$l \geq -2$. Using the exact sequences$\, :$ 
\begin{gather*}
0 = \tH^2(E(-5)) \lra \tH^2(E(-4)) \lra \tH^2(E_\Pi(-4)) \lra 
\tH^3(E(-5)) = 0\, ,\\
0 = \tH^1(E_\Pi(-3)) \lra \tH^2(E(-4)) \overset{h}{\lra} \tH^2(E(-3)) \lra 
\tH^2(E_\Pi(-3)) = 0\, ,
\end{gather*}
one derives that $\h^2(E(-3)) = \h^2(E(-4)) = 1$. 

One can also notice that $\tH^0(E(l)) = 0$ for $l\leq -1$ because 
$\tH^0(E_\Pi(-1)) = 0$ and that $\tH^4(E(l)) \simeq \tH^0(E^\ast(-l-5))^\ast = 0$ 
for $l \geq -5$. 

\vskip2mm 

\noindent
{\bf Claim 6.4.}\quad $E$ \emph{is as in item} (viii) \emph{of the statement}.  

\vskip2mm

\noindent
\emph{Indeed}, a non-zero element of $\tH^1(E(-2))$ defines an extension$\, :$  
\[
0 \lra E \lra E^\prim \lra \sco_\piv(2) \lra 0\, .
\]
Since the graded $S$-module $\tH^1_\ast(E)$ is generated by $\tH^1(E(-2)) 
\simeq k$, it follows that $\tH^1_\ast(E^\prim) = 0$. Moreover, 
$\tH^3_\ast(E^\prim) \simeq \tH^3_\ast(E) = 0$ and 
$\tH^2_\ast(E^\prim) \simeq \tH^2_\ast(E)$. The subspace $W$ of 
$\tH^0(\sco_\piv(1))$ consisting of the linear forms $h^\prim$ for which 
the multiplication map $h^\prim \cdot - : \tH^2(E(-4)) \ra 
\tH^2(E(-3))$ is 0 has dimension 4. If $h_0,\ldots ,\, h_3$ is a $k$-basis of 
$W$ then $h_0,\ldots ,\, h_3,\, h$ is a $k$-basis of $\tH^0(\sco_\piv(1))$. 
$h_0,\ldots ,\, h_3$ vanish simultaneously in only one point $x \in \piv 
\setminus \Pi$. One has$\, :$ 
\[
\tH^2_\ast(E^\prim) \simeq \tH^2_\ast(E) \simeq 
(S/(h_0,\ldots, h_3, h^2))(4)\, .
\]  
Consequently, $E^\prim$ is what Horrocks \cite{ho} calls an Eilenberg-Maclane 
bundle (see, for example, \cite[Example~6.3]{ct}).  
It follows that, denoting by $(C_p,\, \delta_p)_{p \geq 0}$ the Koszul complex 
associated to the epimorphism 
$\delta_1 : 4\sco_\piv(-1) \oplus \sco_\piv(-2) \ra \sco_\piv$ 
defined by $h_0, \ldots ,h_3,h^2$, one has exact sequences$\, :$ 
\begin{gather*}
0 \lra \sco_\piv(-2) \xra{\delta_5(4)} \sco_\piv \oplus 
4\sco_\piv(-1) \xra{\delta_4(4)} 4\sco_\piv(1) \oplus 
6\sco_\piv \lra E^\prim \lra 0\, ,\\ 
0 \lra E^\prim \lra 6\sco_\piv(2) \oplus 4\sco_\piv(1)  
\xra{\delta_2(4)} 4\sco_\piv(3) \oplus \sco_\piv(2) 
\xra{\delta_1(4)} \sco_\piv(4) \lra 0\, .
\end{gather*}
Since $\tH^0(E(-1)) = 0$, the map $\tH^0(E^\prim(-1)) \ra \tH^0(\sco_\piv(1))$ 
is \emph{injective}. 

\vskip2mm 

\noindent
{\bf Claim 6.5.}\quad 
\emph{Conversely, let} $E^\prim$ 
\emph{be the vector bundle on} $\piv$ \emph{associated to 
a Koszul complex as above and let} $\phi : E^\prim \ra \sco_\piv(2)$ 
\emph{be a morphism with the property that}  
$\tH^0(\phi(-1)) : \tH^0(E^\prim(-1)) \ra \tH^0(\sco_\piv(1))$ 
\emph{is injective. Then such morphisms exist and are, 
automatically, epimorphisms}.  

\vskip2mm

\noindent 
\emph{Indeed}, let $(K_p,\kappa_p)_{p \geq 0}$ denote the Koszul complex 
associated to the morphism $\kappa_1 : 4\sco_\piv(-1) \ra \sco_\piv$ 
defined by $h_0, \ldots ,h_3$. The differentials of the Koszul complex 
$C_\bullet$ considered above are described by the following matrices$\, :$  
\[
\delta_p = 
\begin{pmatrix}
\kappa_p & (-1)^{p-1}(h^2\cdot -)\\
0 & \kappa_{p-1}(-2)
\end{pmatrix}
\, .
\] Let $\phi : E^\prim \ra \sco_\piv(2)$ be any morphism.  
The composite morphism $4\sco_\piv(1) \oplus 6\sco_\piv \ra E^\prim 
\overset{\phi}{\lra} \sco_\piv(2)$ has two components: $f : 
4\sco_\piv(1) \ra \sco_\piv(2)$ and $g : 6\sco_\piv \ra \sco_\piv(2)$. 
They satisfy the relations$\, :$ 
\[
f \circ \kappa_4(4) = 0 \  \text{and} \  g \circ \kappa_3(2) + 
f \circ (h^2\cdot -) = 0\, . 
\]
On the other hand, $\phi$ can be also represented as a composite 
morphism  
$E^\prim \ra 6\sco_\piv(2) \oplus 4\sco_\piv(1)  
\overset{\widehat \phi}{\lra} \sco_\piv(2)$. The components 
${\widehat f} : 6\sco_\piv(2) \ra \sco_\piv(2)$ and ${\widehat g} : 
4\sco_\piv(1) \ra \sco_\piv(2)$ of $\widehat \phi$ are related to 
$f$ and $g$ by the relations$\, :$  
\[
f = {\widehat f}\circ \kappa_3(4)\  \text{and} \  g = {\widehat g} \circ 
\kappa_2(2) - {\widehat f} \circ (h^2\cdot -)\, .
\] 

Now, for any morphism $f : 4\sco_\piv(1) \ra \sco_\piv(2)$ such 
that $f \circ \kappa_4(4) = 0$ there exists a unique morphism ${\widehat f} : 
6\sco_\piv(2) \ra \sco_\piv(2)$ such that $f = {\widehat f} \circ 
\kappa_3(4)$. The last relation also shows that the linear forms  
$h_0^\prim,\ldots ,h_3^\prim$ defining $f$ belong to the subspace of 
$\tH^0(\sco_\piv(1))$ generated by $h_0,\ldots ,h_3$. 
The condition $\tH^0(\phi(-1)) : \tH^0(E^\prim(-1)) \ra 
\tH^0(\sco_\piv(1))$ injective is equivalent to the fact that $\tH^0(f(-1)) : 
\tH^0(4\sco_\piv) \ra \tH^0(\sco_\piv(1))$ is injective hence to the fact 
that $h_0^\prim,\ldots ,h_3^\prim$ are linearly independent. Since $\kappa_4(4)$ 
is defined by $h_0,\, -h_1,\, h_2,\, -h_3$, such morphisms $f$ obviously exist: 
one can take, for example, the morphism defined by 
$h_0^\prim = h_1$, $h_1^\prim = h_0$, $h_2^\prim = h_3$, $h_3^\prim = h_2$. 

Assume, from now on, that $f : 4\sco_\piv(1) \ra \sco_\piv(2)$ is a 
morphism defined by four linearly independent linear forms 
$h_0^\prim,\ldots ,h_3^\prim$ such that $f \circ \kappa_4(4) = 0$ and let 
$\widehat f$ be as above. Consider any morphism ${\widehat g} : 
4\sco_\piv(1) \ra \sco_\piv(2)$ and define a morphism $g : 
6\sco_\piv \ra \sco_\piv(2)$ by the formula$\, :$ 
\[
g = {\widehat g} \circ \kappa_2(2) - {\widehat f} \circ (h^2\cdot -)\, .
\]
The morphism $(f,g) : 4\sco_\piv(1) \oplus 6\sco_\piv \ra 
\sco_\piv(2)$ has the property that $(f,g) \circ \delta_4(4) = 0$ hence it 
induces a morphism $\phi : E^\prim \ra \sco_\piv(2)$. Let $x$ be the point 
considered above where $h_0,\ldots ,h_3$ vanish simultaneously. Since 
$\kappa_2(2)$ vanish at $x$ and since $h$ doesn't, the defining formula of 
$g$ shows that $g(x) \neq 0$. Since the image of $\tH^0(f)$ is 
$S_1h_0^\prim + \cdots + S_1h_3^\prim = S_1h_0 + \cdots + S_1h_3 = 
\tH^0(\sci_{\{x\}}(2))$, it follows that $\tH^0(\phi) : \tH^0(E^\prim) 
\ra \tH^0(\sco_\piv(2))$ is surjective which implies that $\phi$  
\emph{is an epimorphism}. 

\vskip2mm 

\noindent
{\bf Claim 6.6.}\quad \emph{Under the hypothesis of Claim 6.5}, $\Ker \phi$ 
\emph{is globally generated}. 

\vskip2mm

\noindent
\emph{Indeed}, 
let $(C_p^\prim,\, \delta_p^\prim)_{p\geq 0}$ be the Koszul complex associated to 
the morphism $\delta_1^\prim : 4\sco_\piv(-1) \oplus \sco_\piv(-2) \ra 
\sco_\piv$ defined by $h_0^\prim,\ldots ,h_3^\prim,h^2$. As in the case of the  
Koszul complex $C_\bullet$ considered above, 
if $(K_p^\prim,\kappa_p^\prim)_{p \geq 0}$ denotes the Koszul complex 
associated to the morphism $\kappa_1^\prim : 4\sco_\piv(-1) \ra \sco_\piv$ 
defined by $h_0^\prim, \ldots ,h_3^\prim$ then the differentials of the Koszul 
complex $C_\bullet^\prim$ are described by the following matrices$\, :$  
\[
\delta_p^\prim = 
\begin{pmatrix}
\kappa_p^\prim & (-1)^{p-1}(h^2\cdot -)\\
0 & \kappa_{p-1}^\prim(-2)
\end{pmatrix}
\, .
\]  
The components of ${\widehat g}\circ \kappa_2(2)$ belong to 
$S_1h_0 + \cdots + S_1h_3 = S_1h_0^\prim + \cdots + S_1h_3^\prim$ hence there is a 
morphism ${\widetilde g} : 6\sco_\piv \ra 4\sco_\piv(1)$ such that 
the diagram$\, :$ 
\[
\begin{CD}
6\sco_\piv @>{\widetilde g}>> 4\sco_\piv(1)\\
@V{\kappa_2(2)}VV @VVfV\\
4\sco_\piv(1) @>{\widehat g}>> \sco_\piv(2)
\end{CD}
\]   
commutes. Now, the commutative diagram$\, :$ 
\[
\begin{CD}
4\sco_\piv(1) \oplus 6\sco_\piv @>>> E^\prim\\
@V{{\left(\begin{smallmatrix} \text{id} & {\widetilde g}\\ 
0 & -{\widehat f}(-2) \end{smallmatrix}\right)} =:\, \phi_1}VV @VV{\phi}V\\
4\sco_\piv(1) \oplus \sco_\piv @>{(f\, ,\, h^2)}>> \sco_\piv(2)
\end{CD}
\]
can be extended to a morphism of complexes$\, :$ 
\[
\begin{CD}
\sco(-2) @>{\delta_5(4)}>> \sco \oplus 4\sco(-1) 
@>{\delta_4(4)}>> 4\sco(1) \oplus 6\sco @>>> E^\prim\\
@VV{\phi_3}V @VV{\phi_2}V @VV{\phi_1}V @VV{\phi}V\\
4\sco(-1) \oplus 6\sco(-2) @>{\delta_3^\prim(2)}>> 
6\sco \oplus 4\sco(-1) @>{\delta_2^\prim(2)}>> 
4\sco(1) \oplus \sco @>{\delta_1^\prim(2)}>> \sco(2)  
\end{CD}
\] 
 
\noindent
{\bf Claim 6.6.1.}\quad \emph{The component} $\alpha : \sco \ra 6\sco$ 
\emph{of} $\phi_2$ \emph{is nonzero}. 

\vskip2mm 

\noindent
\emph{Indeed}, one has$\, :$ 
\[
\phi_1 \circ \delta_4(4) = 
\begin{pmatrix}
\kappa_4(4) & -(h^2\cdot -) + {\widetilde g}\circ \kappa_3(2)\\
0 & -{\widehat f}(-2) \circ \kappa_3(2)
\end{pmatrix}\, .
\]  
It follows that the component $\sco \ra 4\sco(1)$ of $\phi_1 \circ 
\delta_4(4)$ is $\kappa_4(4)$, hence $\tH^0(\phi_1 \circ \delta_4(4))$ is 
injective. One deduces that $\tH^0(\phi_2)$ is injective, whence the claim. 

\vskip2mm 

\noindent
{\bf Claim 6.6.2.}\quad \emph{The component} $\beta : 4\sco(-1) \ra 
4\sco(-1)$ \emph{of} $\phi_2$ \emph{is the identity}. 

\vskip2mm

\noindent
\emph{Indeed}, let $x\in \piv \setminus \Pi$ be the point (considered above) 
where $h_0, \ldots , h_3$ (hence, also, $h_0^\prim, \ldots , h_3^\prim$) vanish 
simultaneously. One has$\, :$ 
\[
(\phi_1 \circ \delta_4(4))(x) = 
\begin{pmatrix} 0 & -(h(x)^2 \cdot -)\\ 0 & 0 \end{pmatrix}\, ,\  \  
\delta_2^\prim(2)(x) = 
\begin{pmatrix} 0 & -(h(x)^2 \cdot -)\\ 0 & 0 \end{pmatrix}\, .  
\]  
One deduces that $\beta(x)$ is the identity, whence the claim. 

\vskip2mm 

\noindent
{\bf Claim 6.6.3.}\quad \emph{The component} $\gamma : \sco(-2) \ra 
6\sco(-2)$ \emph{of} $\phi_3$ \emph{equals} $\alpha(-2)$. 

\vskip2mm

\noindent
\emph{Indeed}, one has$\, :$ 
\[
(\phi_2 \circ \delta_5(4))(x) = 
\begin{pmatrix} \alpha(x)h(x)^2\\ 0\end{pmatrix}\, ,\  \  
(\delta_3^\prim(2) \circ \phi_3)(x) = 
\begin{pmatrix} h(x)^2\gamma(x)\\ 0 \end{pmatrix}\, .
\]
One deduces that $\gamma(x) = \alpha(-2)(x)$, whence the claim. 

\vskip2mm 

Now, the mapping cone of the morphism of complexes$\, :$ 
\[
\begin{CD}
0 @>>> \sco(\text{--} 2) @>{\delta_5(4)}>> \sco \oplus 4\sco(\text{--} 1)  
@>{\delta_4(4)}>> 4\sco(1) \oplus 6\sco @>>> 0\\
@.  @VV{\phi_3}V @VV{\phi_2}V @VV{\phi_1}V\\
\cdots  @>>> 4\sco(\text{--} 1) \oplus 6\sco(\text{--} 2) 
@>{\delta_3^\prim(2)}>> 
6\sco \oplus 4\sco(\text{--} 1) @>{\delta_2^\prim(2)}>> 
4\sco(1) \oplus \sco @>>> 0  
\end{CD}
\]
has the homology concentrated at the term$\, :$ 
\[
(4\sco(1) \oplus 6\sco) \oplus 
(6\sco \oplus 4\sco(-1)) 
\] 
and there it is isomorphic to $E$. Using the last three claims and the 
concrete form of the matrix of $\phi_1$, one can cancel a number of direct 
summands of some terms of this complex and one gets a resolution of $E$ of 
the form$\, :$ 
\[
0 \ra \sco_\piv(-4) \ra \sco_\piv(-2) \oplus 4\sco_\piv(-3)  
\ra 4\sco_\piv(-1) \oplus 5\sco_\piv(-2)  
\ra 10\sco_\piv \ra E \ra 0\, .
\] 
In particular, $E$ \emph{is globally generated} and Claim 6.6 is proven. 

\vskip2mm

Finally, since one knows 
the cohomology groups $\tH^i(E(l))$, $0 \leq i \leq 4$, $-5 \leq l \leq -1$, 
one gets that the Beilinson monad of $E(-1)$ has the form$\, :$ 
\[
0 \lra \Omega_\piv^3(3) \lra \Omega_\piv^2(2) \oplus \Omega_\piv^1(1) 
\lra \sco_\piv \lra 0 
\]
hence $E(-1)$ is isomorphic to the bundle $\mathcal G$ from the statement of 
the main result of Abo, Decker and Sasakura \cite{ads}.  

\vskip2mm

\noindent
{\bf Case 7.}\quad $c_2 = 8$, $c_3 = 8$, $n = 4$ \emph{and} $F$ 
\emph{as in Prop.}~\ref{P:c1=4c2=8n3}(vi). 

\vskip2mm 

\noindent
We will show that, in this case too, $E$ is as in item (viii) of the 
statement. Let, again, $h = 0$ be an equation of the hyperplane $\Pi \subset 
\piv$. From Prop.~\ref{P:c1=4c2=8n3}(vi), one has exact sequences$\, :$  
\begin{gather*}
0 \lra K \lra 2\sco_\Pi(2) \oplus 2\sco_\Pi(1)  
\overset{p}{\lra} \sco_\Pi(3) \lra 0\, ,\\
0 \lra \sco_\Pi(-1) \overset{\sigma}{\lra} K \oplus 4\sco_\Pi  
\lra F \lra 0\, ,
\end{gather*}
where $p$ is defined by $h_0,\, h_1,\, h_2^2,\, h_3^2$ for some $k$-basis 
$h_0, \ldots , h_3$ of $\tH^0(\sco_\Pi(1))$ and where the component 
$u : \sco_\Pi(-1) \ra 4\sco_\Pi$ of $\sigma$ is defined by $h_0, \ldots , 
h_3$. It follows that$\, :$ 
\[
\tH^1_\ast(F) \simeq \left((S/hS)/(h_0,\, h_1,\, h_2^2,\, h_3^2)\right)(3)\  
\text{and}\  
\tH^1_\ast(F^\ast) \simeq k(1)\  \text{hence}\  \tH^2_\ast(F)\simeq k(3)\, .
\]
Our idea for the treatment of Case 7 is to reduce it to Case 6. More 
precisely, one has an exact sequence$\, :$ 
\[
0 \lra (r - 1)\sco_\piv \lra E \lra \sci_Y(4) \lra 0
\]
where $Y \subset \piv$ is a nonsingular surface of degree 8, sectional 
genus $\pi = 5$ (because $c_3 = 2\pi - 2$) and with $\omega_Y(1)$ globally 
generated. By Lemma~\ref{L:d=8pi=5}, the geometric genus and the irregularity 
of $Y$ satisfy $p_g = 0$ and $q\leq 1$. 
For any hyperplane $\Pi^\prim \subset \piv$, $E_{\Pi^\prim}$ is a globally 
generated vector bundle on $\Pi^\prim \simeq \piii$, with $c_1(E_{\Pi^\prim}) 
= 4$, $c_2(E_{\Pi^\prim}) = 8$ and $c_3(E_{\Pi^\prim}) = 8$. It follows from 
Prop.~\ref{P:c1=4c2=8n3} and from Remark~\ref{R:s=qt=pg} that one has an exact 
sequence$\, :$
\[
0 \lra q\sco_{\Pi^\prim} \lra F^\prim \lra E_{\Pi^\prim} \lra 0
\] 
where $F^\prim$ is as in Prop.~\ref{P:c1=4c2=8n3}(vi) or $F^\prim \simeq 
2\Omega_{\Pi^\prim}(2)$. In the former case one has $\h^0(F^\prim(-1)) = 1$ 
(hence $\h^0(E_{\Pi^\prim}(-1)) = 1$), while in the latter case  
$\h^0(F^\prim(-1)) = 0$ (hence $\h^0(E_{\Pi^\prim}(-1)) = 0$). Our aim is to show 
that for a general hyperplane $\Pi^\prim \subset \piv$ one has 
$\tH^0(E_{\Pi^\prim}(-1)) = 0$. This will imply, according to Case 6, that $E$ is 
as in item (viii) of the statement of our proposition. 

\vskip2mm 

\noindent
{\bf Claim 7.1.}\quad $\tH^1(E(-3)) = 0$. 

\vskip2mm

\noindent
\emph{Indeed}, this follows from Severi's theorem because $\tH^1(E(-3)) 
\simeq \tH^1(\sci_Y(1)) = 0$. We shall give, however, a direct argument. 
By Lemma~\ref{L:hieh=0}(a), $\tH^1(E(l)) = 0$ for $l \leq -4$, hence 
$\tH^1(E(-3))$ injects into $\tH^1(E_\Pi(-3)) \simeq k$. If $\tH^1(E(-3)) 
\neq 0$ then $\h^1(E(-3)) = 1$ and, using the exact sequence$\, :$ 
\[
\tH^1(E(-3)) \overset{h}{\lra} \tH^1(E(-2)) \lra \tH^1(E_\Pi(-2))\, ,
\] 
one gets that $\h^1(E(-2)) \leq 3$. It follows that there exists $0 \neq 
h^\prim \in S_1$ such that the multiplication map $h^\prim \cdot - : 
\tH^1(E(-3)) \ra \tH^1(E(-2))$ is 0. If $\Pi^\prim \subset \piv$ is the 
hyperplane 
of equation $h^\prim = 0$ then $\tH^0(E_{\Pi^\prim}(-2)) \neq 0$, which 
\emph{contradicts} the hypothesis of the proposition. 

\vskip2mm 

\noindent
{\bf Claim 7.2.}\quad $q \neq 0$ (\emph{hence} $q = 1$). 

\vskip2mm

\noindent
\emph{Indeed}, by relation \eqref{E:pi-d+3} (from the proof of 
Lemma~\ref{L:lanok}), one has $\h^2(E(-3)) = \h^2(\sci_Y(1)) = 
\h^1(\sco_Y(1)) = q$. 
Moreover, $\h^3(E(-4)) = \h^3(\sci_Y) = \h^2(\sco_Y) = p_g = 0$. Using the 
exact sequence$\, :$ 
\[
\tH^2(E(-3)) \lra \tH^2(E_\Pi(-3)) \lra \tH^3(E(-4)) = 0
\] 
and the fact that $\h^2(E_\Pi(-3)) = 1$ one deduces that $\tH^2(E(-3)) \neq 0$, 
hence $q \neq 0$. 

\vskip2mm

\noindent
{\bf Claim 7.3.}\quad $\h^1(E(-2)) = 1$. 

\vskip2mm 

\noindent
\emph{Indeed}, consider the exact sequence$\, :$ 
\[
0 = \tH^1(E(-3)) \lra \tH^1(E(-2)) \lra \tH^1(E_\Pi(-2)) \lra \tH^2(E(-3))\, .
\]
Since, as we saw in the proof of Claim 7.2, $\h^2(E(-3)) = 1$ and since 
$\h^1(E_\Pi(-2)) = 2$ it follows that $\h^1(E(-2)) \in \{1,\, 2\}$. If 
$\h^1(E(-2)) = 2$ then the morphism $\tH^1(E(-2)) \ra \tH^1(E_\Pi(-2))$ is 
surjective. Since the $S/hS$-module $\tH^1_\ast(E_\Pi)$ is generated by 
$\tH^1(E_\Pi(-3))$, it follows that $\tH^1(E(l)) \ra \tH^1(E_\Pi(l))$ is 
surjective for $l \geq -2$. But, by Lemma~\ref{L:hieh=0}(b), this would imply 
that $\tH^2(E(-3)) = 0$, a \emph{contradiction}. 

\vskip2mm 

It follows, from the above claims, that the image of  
$\tH^1(E(-2)) \ra \tH^1(E_\Pi(-2))$ is generated (as a $k$-vector space) 
by a non-zero element $\xi \in \tH^1(E_\Pi(-2))$. Since$\, :$ 
\[
\tH^0(\sco_\Pi(l)) \cdot \xi = \tH^1(E_\Pi(-2+l)), \  \forall \, l \geq 1\, ,
\] 
one deduces that the map $\tH^1(E(l)) \ra \tH^1(E_\Pi(l))$ \emph{is surjective 
for} $l \geq -1$ and that $\tH^1(E(-2))$ \emph{generates the graded $S$-module} 
$\tH^1_\ast(E)$. Moreover, one gets, from Lemma \ref{L:hieh=0}(b), that 
$\tH^2(E(l)) = 0$ for $l \geq -2$.      

\vskip2mm

\noindent
{\bf Claim 7.4.}\quad $\tH^0(E(-1)) = 0$. 

\vskip2mm 

\noindent
\emph{Indeed}, since $\h^0(E_\Pi(-1)) = 1$ it follows that $\h^0(E(-1)) \leq 1$. 
If $\h^0(E(-1)) = 1$ then, using the exact sequence$\, :$
\begin{equation}\label{E:heepi}  
0 \ra \tH^0(E(-1)) \ra \tH^0(E_\Pi(-1)) \ra \tH^1(E(-2)) \ra \tH^1(E(-1)) 
\ra \tH^1(E_\Pi(-1)) \ra 0 
\end{equation}
one deduces that $\h^1(E(-1)) = 2$. It follows that there exists $0 \neq 
h^\prim \in S_1$ such that the multiplication map $h^\prim \cdot - : 
\tH^1(E(-2)) \ra \tH^1(E(-1))$ is 0. Let $\Pi^\prim \subset \piv$ be the 
hyperplane of equation $h^\prim = 0$. Using the exact sequence$\, :$ 
\[
0 \lra \tH^0(E(-1)) \lra \tH^0(E_{\Pi^\prim}(-1)) \lra \tH^1(E(-2)) 
\overset{h^\prim}{\lra} \tH^1(E(-1))
\]
one gets that $\h^0(E_{\Pi^\prim}(-1)) = 2$ which, as we noticed at the beginning 
of the proof of Case 7, \emph{is not possible}. 

\vskip2mm 

\noindent
{\bf Claim 7.5.}\quad $\tH^0(E_{\Pi^\prim}(-1)) = 0$ \emph{for a general 
hyperplane} $\Pi^\prim \subset \piv$. 

\vskip2mm

\noindent
\emph{Indeed}, the exact sequence \eqref{E:heepi} implies that 
$\h^1(E(-1)) = 1$. Since $\tH^1(E(-2))$ generates the graded $S$-module 
$\tH^1_\ast(E)$, it follows that, for a general $0 \neq h^\prim \in S_1$, 
the multiplication map $h^\prim \cdot - : \tH^1(E(-2)) \ra \tH^1(E(-1))$ is 
non-zero, hence injective. If $\Pi^\prim \subset \piv$ is the hyperplane of 
equation $h^\prim = 0$ then, using the exact sequence$\, :$ 
\[
0 = \tH^0(E(-1)) \lra \tH^0(E_{\Pi^\prim}(-1)) \lra \tH^1(E(-2)) 
\overset{h^\prim}{\lra} \tH^1(E(-1))\, ,
\]
one gets that $\tH^0(E_{\Pi^\prim}(-1)) = 0$.    

\vskip2mm 

\noindent
{\bf Case 8.} $c_2 = 8$, $c_3 = 8$ \emph{and} $n \geq 5$. 

\vskip2mm 

\noindent
We will show that this case \emph{cannot occur}. 
Indeed, for this purpose, one can assume that $n = 5$. Let $H \subset \pv$ be 
a fixed hyperplane. Lemma~\ref{L:h0h1} implies that $E_H \simeq 
G \oplus t\sco_H$, with $G$ defined by an exact sequence$\, :$ 
\[
0 \lra s\sco_H \lra F \lra G \lra 0\, ,
\] 
where $F$ is a globally generated vector bundle on $H \simeq \piv$ such that 
$\tH^i(F^\ast) = 0$, $i = 0,\, 1$, $t = \h^0(E_H^\ast)$ and $s = \h^1(E_H^\ast)$. 
It follows, from the Cases 6 and 7 above, that $F$ is one of the bundles 
described in item (viii) of the statement of the proposition, that is, one has 
exact sequences$\, :$ 
\begin{gather*}
0 \lra \sco_H(-2) \lra \sco_H \oplus 4\sco_H(-1) \lra 4\sco_H(1) \oplus 
6\sco_H \lra F^\prim \lra 0\, ,\\
0 \lra F \lra F^\prim \overset{\phi}{\lra} \sco_H(2) \lra 0\, . 
\end{gather*}    
Since the resolution of $F^\prim$ above is part of a Koszul complex, one also 
has an exact sequence$\, :$ 
\[
0 \lra F^\prim \lra 6\sco_H(2) \oplus 4\sco_H(1) \lra 4\sco_H(3) \oplus 
\sco_H(2) \lra \sco_H(4) \lra 0 
\] 
and the epimorphism $\phi$ above can be represented as a composite 
morphism $F^\prim \ra 6\sco_H(2) \oplus 4\sco_H(1) 
\overset{\widehat \phi}{\lra} \sco_H(2)$. As we saw in the proof of 
Claim 6.5 above, the component $6\sco_H(2) \ra \sco_H(2)$ of 
$\widehat \phi$ is non-zero. 

The intermediate cohomology of $F$ has been determined in Case 6 above. 
In particular, $\tH^1_\ast(F^\ast) = 0$ hence, by Lemma~\ref{L:h1ehast=0}(a), 
$t = 0$. Moreover, by Lemma~\ref{L:h1ehast=0}(b)$\, :$ 
\begin{gather*}
s \leq \dim \Cok(\tH^2(E(-4)) \ra \tH^2(E_H(-4))) +\\
+ \dim \Cok(\tH^2(E(-3)) \ra \tH^2(E_H(-3))) \leq 2\, .
\end{gather*} 

\noindent
{\bf Claim 8.1.}\quad $s \neq 0$. 

\vskip2mm 

\noindent  
\emph{Indeed}, assume that $s = 0$. Then $\tH^1_\ast(E_H^\ast) \simeq 
\tH^1_\ast(F^\ast) = 0$ hence, by Lemma~\ref{L:hieh=0}(a), $\tH^1_\ast(E^\ast) 
= 0$. One deduces, from the discussion before Claim 8.1, that $E_H^\ast$ has a 
resolution of the form$\, :$ 
\[
0 \ra \sco_H(-4) \ra \sco_H(-2) \oplus 4\sco_H(-3) \ra 
4\sco_H(-1) \oplus 5\sco_H(-2) \ra E_H^\ast \ra 0\, . 
\] 
It follows, from Lemma~\ref{L:h1=0}, that $E^\ast$ has a resolution of the 
form$\, :$
\[
0 \ra \sco_\pv(-4) \ra \sco_\pv(-2) \oplus 4\sco_\pv(-3) \ra 
4\sco_\pv(-1) \oplus 5\sco_\pv(-2) \ra E^\ast \ra 0\, . 
\]
Since there is no epimorphism $4\sco_\pv(3) \oplus \sco_\pv(2) \ra 
\sco_\pv(4)$, there is no vector bundle on $\pv$ having such a resolution. 
Consequently, one must have $s \neq 0$. 

\vskip2mm 

At this point one can get a \emph{contradiction} using Kodaira Vanishing. 
\emph{Indeed}, one has an exact sequence$\, :$ 
\[
0 \lra (r - 1)\sco_\pv \lra E \lra \sci_Y(4) \lra 0
\]
where $Y$ is a nonsingular 3-fold in $\pv$. By Serre duality  
and Kodaira Vanishing, 
\[
\tH^2(E^\ast(-1)) \simeq \tH^3(E(-5))^\ast \simeq \tH^3(\sci_Y(-1))^\ast \simeq 
\tH^2(\sco_Y(-1))^\ast = 0\, .
\]
Using the exact sequence $\tH^1(E^\ast) \ra \tH^1(E_H^\ast) \ra 
\tH^2(E^\ast(-1))$ one deduces that $\tH^1(E_H^\ast) = 0$. Since $s = 
\h^1(E_H^\ast)$ we have got the desired \emph{contradiction}. 

We want, however, to include an \emph{elementary argument}. It is lengthy but 
instructive. 

\vskip2mm 

\noindent 
{\bf Claim 8.2.}\quad  $\tH^1_\ast(E) \simeq k(1)$. 

\vskip2mm 

\noindent
\emph{Indeed}, $\tH^1(E(l)) = 0$ for $l \leq -3$ by Lemma~\ref{L:hieh=0}(a) 
hence $\tH^1(E(-2))$ injects into $\tH^1(E_H(-2)) \simeq k$. 

\vskip2mm 

\noindent
$\bullet$\quad If $\tH^1(E(-2)) \neq 0$ then a non-zero element of 
$\tH^1(E(-2))$ defines an extension$\, :$ 
\[
0 \lra E \lra E^\prim \lra \sco_\pv(2) \lra 0 
\]      
with $E_H^\prim$ given by an exact sequence $0 \ra s\sco_H \ra F^\prim 
\ra E_H^\prim \ra 0$ where $F^\prim$ is defined by a Koszul resolution$\, :$ 
\[
0 \lra \sco_H(-2) \lra \sco_H \oplus 4\sco_H(-1) \lra 
4\sco_H(1) \oplus 6\sco_H \lra F^\prim \lra 0\, .
\]
Since there is no locally split monomorphism $\sco_H \ra 4\sco_H(1)$, 
$E_H^\prim$ must have a resolution of the form$\, :$ 
\[
0 \lra \sco_H(-2) \lra \sco_H \oplus 4\sco_H(-1) \lra 
4\sco_H(1) \oplus (6 - s)\sco_H \lra E_H^\prim \lra 0\, .
\] 
Since $\tH^1_\ast(E_H^\prim) = 0$, Lemma~\ref{L:hieh=0}(a) implies that 
$\tH^1_\ast(E^\prim) = 0$ and Lemma~\ref{L:h1=0} implies that $E^\prim$ has a 
resolution of the form$\, :$
\[
0 \lra \sco_\pv(-2) \lra \sco_\pv \oplus 4\sco_\pv(-1) \lra 
4\sco_\pv(1) \oplus (6 - s)\sco_\pv \lra E^\prim \lra 0\, .
\]
Since there is no epimorphism $4\sco_\pv(1) \oplus \sco_\pv \ra 
\sco_\pv(2)$, \emph{this is not possible}. 

It remains that $\tH^1(E(-2)) = 0$ hence $\tH^1(E(-1))$ injects into 
$\tH^1(E_H(-1)) \simeq k$.  

\vskip2mm 

\noindent
$\bullet$\quad If $\tH^1(E(-1)) = 0$ then Lemma~\ref{L:hieh=0}(a) would imply 
that $\tH^1_\ast(E) = 0$. One would deduce, from Lemma~\ref{L:h1=0} (and taking 
into account the resolution of $F$ from Case 6 above), 
that $E$ admits a resolution of the form$\, :$  
\begin{gather*}
0 \ra \sco_\pv(-4) \ra \sco_\pv(-2) \oplus 4\sco_\pv(-3)  
\ra 4\sco_\pv(-1) \oplus 5\sco_\pv(-2) \ra\\  
\ra (10 - s)\sco_\pv \ra E \ra 0\, .
\end{gather*}
Since there is no epimorphism $4\sco_\pv(3) \oplus \sco_\pv(2) \ra 
\sco_\pv(4)$, \emph{this is not possible}. 

It thus remains that $\h^1(E(-1)) = 1$. 

\vskip2mm 

\noindent
$\bullet$\quad Since $\tH^1(E_H) = 0$, $\tH^1(E)$ is a quotient of 
$\tH^1(E(-1))$. If $\tH^1(E) \neq 0$ then, using the exact sequence$\, :$ 
\[
\tH^0(E) \lra \tH^0(E_H) \lra \tH^1(E(-1)) \lra \tH^1(E) \lra \tH^1(E_H) = 0\, ,
\]
one deduces that $\tH^0(E) \ra \tH^0(E_H)$ is surjective. Since the graded 
module $\tH^0_\ast(E_H)$ is generated by $\tH^0(E_H)$ (over the projective 
coordinate ring of $H$) it follows that $\tH^0_\ast(E) \ra \tH^0_\ast(E_H)$ is 
surjective which, according to Lemma~\ref{L:hieh=0}(b), implies that 
$\tH^1_\ast(E) = 0$. But this \emph{contradicts} the fact that $\tH^1(E(-1)) 
\neq 0$. 

It thus remains that $\tH^1(E) = 0$. Lemma~\ref{L:hieh=0}(a) implies, now, 
that $\tH^1(E(l)) = 0$ for $l \geq 0$ and this concludes the proof of 
Claim 8.2. 

\vskip2mm

\noindent
{\bf Claim 8.3.}\quad $\tH^2_\ast(E) \simeq k(3)$ \emph{and} $s = 1$.  

\vskip2mm 

\noindent
\emph{Indeed}, $\tH^2(E(l)) = 0$ for $l \leq -5$ by Lemma~\ref{L:hieh=0}(a). 
One deduces that $\tH^2(E(-4))$ injects into $\tH^2(E_H(-4)) \simeq k$. 
If $\tH^2(E(-4)) \neq 0$ then $\tH^2_\ast(E) \ra \tH^2_\ast(E_H)$ would be 
surjective, which would imply that $s = 0$. But this \emph{contradicts} 
Claim 8.1. 

It remains that $\tH^2(E(-4)) = 0$ hence $\tH^2(E(-3))$ injects into 
$\tH^2(E_H(-3)) \simeq k$. Using the exact sequence$\, :$ 
\[
0 = \tH^1(E(-2)) \lra \tH^1(E_H(-2)) \lra \tH^2(E(-3))
\]
and the fact that $\tH^1(E_H(-2)) \simeq k$ one gets that $\tH^2(E(-3)) \neq 
0$, hence $\h^2(E(-3)) = 1$. 
Since, as we saw in the proof of Claim 8.2, $\tH^1(E(-1)) \ra \tH^1(E_H(-1))$ 
is surjective, Lemma~\ref{L:hieh=0}(b) implies that $\tH^2(E(l)) = 0$ for 
$l \geq -2$, proving the first part of Claim 8.3. The second part follows from 
Lemma~\ref{L:h1ehast=0}(b) and from Claim 8.1.   

\vskip2mm 

\noindent
{\bf Claim 8.4.}\quad $\tH^3_\ast(E) \simeq k(5)$. 

\vskip2mm

\noindent
\emph{Indeed}, as we saw in the proof of Claim 8.3, the map $\tH^2(E(l)) 
\ra \tH^2(E_H(l))$ is surjective for $l \geq -3$, hence, by 
Lemma~\ref{L:hieh=0}(b), $\tH^3(E(l)) = 0$ for $l \geq -4$. Using the exact 
sequence$\, :$ 
\[
0 = \tH^2(E(-4)) \lra \tH^2(E_H(-4)) \lra \tH^3(E(-5)) \lra \tH^3(E(-4)) = 0 
\, ,
\] 
one derives that $\h^3(E(-5)) = 1$. This implies, by Serre duality, that 
$\h^2(E^\ast(-1)) = 1$. Taking into account that $\h^1(E_H^\ast) = s = 1$, one 
deduces, from the exact sequence$\, :$ 
\[
0 = \tH^1(E^\ast) \lra \tH^1(E_H^\ast) \lra \tH^2(E^\ast(-1)) \lra \tH^2(E^\ast) 
\lra \tH^2(E_H^\ast) = 0\, , 
\] 
that $\tH^2(E^\ast) = 0$. One also has$\, :$ 
\begin{gather*}
\tH^3(E^\ast(-1)) \simeq \tH^2(E(-5))^\ast = 0\, ,\\
\tH^4(E^\ast(-2)) \simeq \tH^1(E(-4))^\ast = 0\, ,\\
\tH^5(E^\ast(-3)) \simeq \tH^0(E(-3))^\ast = 0\, .
\end{gather*}
Lemma \ref{L:cm} implies, now, 
that $\tH^2(E^\ast(l)) = 0$ for $l \geq 0$ hence, by Serre duality, 
$\tH^3(E(l)) = 0$ for $l \leq -6$, and this concludes the proof of 
Claim 8.4. 

\vskip2mm 

\noindent
{\bf Claim 8.5.}\quad $\tH^4(E(l)) = 0$ \emph{for} $l \geq -6$. 

\vskip2mm 

\noindent
\emph{Indeed}, since $\tH^1(E_H^\ast(l)) = 0$ for $l \leq -1$, 
Lemma~\ref{L:hieh=0}(a) implies that $\tH^1(E^\ast(l)) = 0$ for $l \leq -1$. 
Moreover, by our hypothesis, $\tH^1(E^\ast) = 0$. The claim follows, now, 
from Serre duality. 

\vskip2mm  

Notice, also, that $\tH^0(E(l)) = 0$ for $l \leq -1$ (because 
$\tH^0(E_H(-1)) = 0$) and that $\tH^5(E(l)) = 0$ for $l \geq -6$ (by Serre 
duality, because, by hypothesis, $\tH^0(E^\ast(l)) = 0$ for $l \leq 0$). 
Applying, now,  Beilinson's theorem~\ref{T:beilinson} to $E(-1)$ one derives 
that $E(-1)$ is the cohomology of a monad of the form$\, :$ 
\[
0 \lra \Omega_\pv^4(4) \overset{\alpha}{\lra} \Omega_\pv^2(2) 
\overset{\beta}{\lra} \sco_\pv \lra 0\, .
\]

\noindent
{\bf Claim 8.6.}\quad \emph{Such a monad cannot exist}. 

\vskip2mm 

\noindent
\emph{Indeed}, according to Remark~\ref{R:contraction}, the morphism 
$\alpha$ is 
defined by an element $\omega \in \overset{2}{\bigwedge} V$, where $V = 
k^6$. Let $\psi : V^\ast \ra V$ be the skew-symmetric map associated to 
$\omega$. We show, firstly, that $\alpha$ is a locally split monomorphism 
if and only if $\psi$ is an isomorphism. 

\vskip2mm

\noindent
$\bullet$\quad If $\text{rk}\, \psi = 2$ then $\omega = v_0 \wedge v_1$, for 
some linearly independent vectors $v_0,\, v_1 \in V$. Consider any vector 
$v \in V\setminus (kv_0 + kv_1)$, put $V^\prim := kv_0 + kv_1 + kv$ and consider 
a decomposition $V = V^\prim \oplus V^\secund$. By Remark~\ref{R:contraction}(c)
one has$\, :$
\begin{gather*}
\Omega_\pv^4(4) \vb \p(V^\prim) \simeq \Omega_{\p(V^\prim)}^2(2)\otimes 
\overset{2}{\textstyle \bigwedge} V^{\secund \ast} \oplus \Omega_{\p(V^\prim)}^1(1)
\otimes \overset{3}{\textstyle \bigwedge} V^{\secund \ast}\, ,\\
\Omega_\pv^2(2) \vb \p(V^\prim) \simeq \Omega_{\p(V^\prim)}^2(2) \oplus 
\Omega_{\p(V^\prim)}^1(1)\otimes V^{\secund \ast} \oplus 
\sco_{\p(V^\prim)}\otimes \overset{2}{\textstyle \bigwedge} V^{\secund \ast}\, .
\end{gather*} 
Since $\omega \in \overset{2}{\bigwedge} V^\prim$, $\alpha \vb \p(V^\prim)$ maps 
$\Omega_{\p(V^\prim)}^1(1)\otimes \overset{3}{\bigwedge} V^{\secund \ast}$ into 0, 
hence it drops rank by at least 2 at every point of $\p(V^\prim)$. 
Since $v$ was an arbitrary vector in $V\setminus (kv_0 + kv_1)$,  
$\alpha$ drops rank by at least 2 at every point of $\pv$. 

\vskip2mm

\noindent
$\bullet$\quad If $\text{rk}\, \psi = 4$ then $\omega = v_0 \wedge v_1 + 
v_2 \wedge v_3$ for some linearly independent vectors $v_0, \ldots ,v_3 \in 
V$. Put $V^\prim := kv_0 + \cdots + kv_3$ and consider a decomposition 
$V = V^\prim \oplus V^\secund$. By Remark~\ref{R:contraction}(c) one has$\, :$ 
\begin{gather*}
\Omega_\pv^4(4) \vb \p(V^\prim) \simeq 
\Omega_{\p(V^\prim)}^3(3)\otimes V^{\secund \ast} \oplus 
\Omega_{\p(V^\prim)}^2(2)\otimes \overset{2}{\textstyle \bigwedge} 
V^{\secund \ast}\, ,\\
\Omega_\pv^2(2) \vb \p(V^\prim) \simeq \Omega_{\p(V^\prim)}^2(2) \oplus 
\Omega_{\p(V^\prim)}^1(1)\otimes V^{\secund \ast} \oplus 
\sco_{\p(V^\prim)}\otimes \overset{2}{\textstyle \bigwedge} V^{\secund \ast}\, .
\end{gather*} 
Since $\omega \in \overset{2}{\bigwedge} V^\prim$, $\alpha \vb \p(V^\prim)$ maps 
$\Omega_{\p(V^\prim)}^2(2) \otimes \overset{2}{\bigwedge} V^{\secund \ast}$ into 
$\sco_{\p(V^\prim)} \otimes \overset{2}{\bigwedge} V^{\secund \ast}$. It follows 
that $\alpha$ drops rank by 2 at every point of $\p(V^\prim)$. 

\vskip2mm

\noindent 
$\bullet$\quad Assume, now, that $\text{rk}\, \psi = 6$. Then there exists a 
$k$-basis $v_0, \ldots ,v_5$ of $V$ such that $\omega = v_0 \wedge v_1 + 
v_2 \wedge v_3 + v_4 \wedge v_5$. By Remark~\ref{R:contraction}, the map 
$\tH^0(\alpha^\ast) : \tH^0(\Omega^2(2)^\ast) \ra \tH^0(\Omega^4(4)^\ast)$ can be 
identified with the map $\omega \wedge - : \overset{2}{\bigwedge} V \ra 
\overset{4}{\bigwedge} V$. Using the bases of $\overset{2}{\bigwedge} V$ and 
of $\overset{4}{\bigwedge} V$ deduced from the basis $v_0, \ldots , v_5$ of 
$V$ one shows easily that $\tH^0(\alpha^\ast)$ is an isomorphism. Denoting 
the cokernel of $\alpha$ by $E^\prim$, one deduces that $\tH^0(E^{\prim \ast}) 
= 0$, hence there is no non-zero morphism $\beta : \Omega_\pv^2(2) \ra 
\sco_\pv$ such that $\beta \circ \alpha = 0$.  
Consequently, \emph{the above monad cannot exist}. 

The bundle $E^\prim$ appears in Horrocks' construction \cite{ho2} of a 
rank 3 vector bundle on $\pv$ (see also, for example,  
\cite[Example~6.10]{ct} where this bundle is denoted by $\mathcal{E}_2$). One 
might have hoped to show that there is no epimorphism 
$E^\prim \ra \sco_\pv$ by showing that $c_5(E^\prim) \neq 0$. Unfortunately, 
$c_t(E^\prim) = 1 + 2t^2 - 3t^4$.  
\end{proof}

\newpage

\appendix 
\section{The case $c_1 = 4$, $c_2 = 8$, $c_3 = 2$ on 
$\piii$}\label{A:case3}

In this appendix we complete the argument for the Case 3 of the proof of 
Prop.~\ref{P:c1=4c2=8n3}. Actually, we have to complete the characterization 
of the epimorphisms $\phi : 3\sco_\piii(1) \oplus 3\sco_\piii \ra 
\Omega_\piii(3)$ with $\Ker \phi(2)$ globally generated. 
This characterization was reduced there to the characterization of the 
epimorphisms 
\[
\e : 3\sco_\Sigma \lra \sci_{Z,\Sigma}(3) 
\]   
with $\Ker \e(2)$ is globally generated, in one of the following three 
cases$\, :$ (III) $\Sigma$ is a nonsigular quadric 
surface in $\piii$ and $Z$ is the union of two disjoint lines contained in 
$\Sigma$$\, ;$ (IV) $\Sigma = H_1 \cup H_2$ and $Z = L_1 \cup L_2$, where 
$L_1$ and $L_2$ are disjoint lines contained in the planes $H_1$ and $H_2$, 
respectively$\, ;$ (V) $Z$ is the divisor $2L$ on a nonsigular quadric surface 
$Q \subset \piii$, $L$ being a line, and $\Sigma$ is the first infinitesimal 
neighbourhood $H^{(1)}$ of a plane $H \supset L$. Case (III) was settled in 
Claim 3.3 of the proof of Prop.~\ref{P:c1=4c2=8n3}. In this appendix we settle 
the remaining two cases. 

\vskip2mm 

Before doing that, we need a number of lemmata. The first one, 
which is a matter of elementary linear algebra, was used in Case 3 of the 
proof of Prop.~\ref{P:c1=4c2=8n3} in order to get the reduction recalled 
above.     

\begin{lemma}\label{L:wsubsetwedge2v} 
Let $V$ be a $4$-dimensional $k$-vector space and $W$ a $3$-dimensional 
$k$-vector subspace of $\overset{2}{\bigwedge} V$. Let $\Delta$ be the image of 
the wedge product $V \times V \ra \overset{2}{\bigwedge} V$. $\Delta$ is the 
affine cone over the Pl\"{u}cker hyperquadric $\mathbb{G} \subset 
\p(\overset{2}{\bigwedge} V) \simeq \pv$. 

\emph{(a)} If $\p(W) \cap \mathbb{G}$ is a nonsingular conic then there exists 
a basis $v_0, \ldots , v_3$ of $V$ such that $v_0 \wedge v_1$, $v_2 \wedge 
v_3$, $(v_0 + v_2) \wedge (v_1 + v_3)$ is a basis of $W$. 

\emph{(b)} If $\p(W) \cap \mathbb{G}$ is the union of two intersecting lines 
then there exists a basis $v_0, \ldots , v_3$ of $V$ such that $v_0 \wedge 
v_1$, $v_1 \wedge v_2$, $v_2 \wedge v_3$ is a basis of $W$. 

\emph{(c)} If $\p(W) \cap \mathbb{G}$ is a double line then there exists a  
basis $v_0, \ldots , v_3$ of $V$ such that $v_0 \wedge v_1$, $v_0 \wedge v_2$, 
$v_0 \wedge v_3 + v_1 \wedge v_2$ is a basis of $W$.  
\end{lemma} 

\begin{proof} 
(a) Choose three distinct points $P$, $P^\prim$, $P^\secund$ on $\p(W) \cap 
\mathbb{G}$. They correspond to three $2$-dimensional vector subspaces 
$U$, $U^\prim$, $U^\secund$ of $V$. Since none of the lines 
$\overline{PP^\prim}$, $\overline{P^\prim P^\secund}$, $\overline{PP^\secund}$ is 
contained in $\mathbb{G}$ it follows that $U \cap U^\prim = (0)$, 
$U^\prim \cap U^\secund = (0)$, $U \cap U^\secund = (0)$. In particular, 
$V = U \oplus U^\prim$. $U^\secund$ has a basis of the form $v_0 + v_2$, 
$v_1 + v_3$ with $v_0,\, v_1 \in U$ and $v_2,\, v_3 \in U^\prim$. Now, 
$U^\secund \cap U = (0)$ implies that $v_2$ and $v_3$ are linearly independent 
and $U^\secund \cap U^\prim = (0)$ implies that $v_0$ and $v_1$ are linearly 
independent. 

\vskip2mm 

(b) Let $P^\secund$ the intersection point of the two lines, $P$ another point 
on one of the lines and $P^\prim$ another point on the other line. Using the 
notation from the proof of (a) one has $V = U \oplus U^\prim$ and $U^\secund \cap 
U$ and $U^\secund \cap U^\prim$ are 1-dimensional. The assertion follows easily. 

\vskip2mm 

(c) Choose two different points $P$ and $P^\prim$ on the double line and let 
$U$ and $U^\prim$ be the corresponding 2-dimensional subspaces of $V$. There 
exists three linearly independent vectors $v_0$, $v_1$, $v_2$ such that 
$v_0$, $v_1$ is a basis of $U$ and $v_0$, $v_2$ is a basis of $U^\prim$. $\p(W)$ 
is contained in the intersection of the tangent spaces to $\mathbb{G}$ at 
$P$ and $P^\prim$. But $\text{T}_P\mathbb{G} = \p(U \wedge V)$, 
$\text{T}_{P^\prim}\mathbb{G} = \p(U^\prim \wedge V)$ and 
\[
(U \wedge V) \cap (U^\prim \wedge V) = (v_0 \wedge V) + k\cdot v_1 \wedge v_2 
\, .
\]  
It follows and one can find $v_3 \in V$ such that $v_0, \ldots , v_3$ is a 
basis of $V$ and $v_0 \wedge v_1$, $v_0 \wedge v_2$, $v_0 \wedge v_3 + 
v_1 \wedge v_2$ is a basis of $W$. 
\end{proof}

\begin{lemma}\label{L:quasibdl} 
Let $A$ be a vector bundle on $\pii$ admitting a linear presentation$\, :$  
\[
0 \lra s\sco_\pii(-1) \lra t\sco_\pii \lra A \lra 0\, .
\]
Let $x$ be a point of $\pii$ and let $\sca$ be the kernel of an epimorphism 
$\pi : A \ra \sco_{\{x\}}$. Then $\sca$ fails to be globally generated 
precisely when there exists a line $L \ni x$ such that $\pi$ factorizes 
through $\sco_L$.   
\end{lemma}

\begin{proof}
Since $P(A) \simeq s\sco_\pii(1)$, Lemma~\ref{L:ggquasib}(a) implies that if 
$\sca$ is not globally generated then there exists a line $L \ni x$ such that 
$\pi$ factorizes through $\sco_L$. 

Conversely, assume that there exists a line $L \ni x$ such that $\pi$ 
factorizes through $\sco_L$. Since $\tH^0(\sco_L) \simeq k$, $A \ra \sco_L$ 
must be an epimorphism. 
Let $\sca^\prim$ be the kernel of this epimorphism. Then $\sca^\prim \subseteq 
\sca$ and $\sca/\sca^\prim \simeq \Ker(\sco_L \ra \sco_{\{x\}}) = 
\sci_{\{x\},L} \simeq \sco_L(-1)$ hence $\sca$ is not globally generated.   
\end{proof} 

The next lemma is classically known and easy. 

\begin{lemma}\label{L:singpencil}
Let $q_0,\, q_1 \in {\fam0 H}^0(\sco_\pii(2))$ be two linearly independent 
quadratic forms on $\pii$. If the linear system $\vert \, kq_0 + kq_1\, \vert$ 
contains only singular divisors then either $q_0$ and $q_1$ have a common 
linear factor or the divisors $q_0 = 0$ and $q_1 = 0$ have a common singular 
point. 
\end{lemma} 

\begin{proof}
Assume that $q_0$ and $q_1$ have no common linear factor. In this case, the 
subscheme $X$ of $\pii$ of equations $q_0 = q_1 = 0$ is a complete intersection 
of type $(2,2)$. By Bertini, a general member of the linear system 
$\vert \, kq_0 + kq_1 \, \vert$ is nonsingular outside $X$. The divisor 
$q_i = 0$ must have the form $L_i + L_i^\prim$, with $L_i$ and $L_i^\prim$ lines 
in $\pii$, $i = 0, 1$. We want to show that $L_0 \cap L_0^\prim \cap L_1 \cap 
L_1^\prim \neq \emptyset$. Assume the contrary. Then, for every point $x \in X$, 
at least one of the divisors $L_i + L_i^\prim$, $i = 0, 1$, is nonsingular at 
$x$. It follows that, for every point $x \in X$, the general member of the 
linear system $\vert \, kq_0 + kq_1 \, \vert$ is nonsingular at $x$. 
Consequently, a general member of this linear system is nonsingular, a 
contradiction.  
\end{proof}

\begin{lemma}\label{L:ponceletbdl} 
Let $x$ be a point of $\pii$, consider an epimorphism $3\sco_\pii \ra 
\sci_{\{x\}}(2)$ and let $G$ be the kernel of this epimorphism. The epimorphism 
is defined by a $3$-dimensional vector subspace $\Lambda$ of 
${\fam0 H}^0(\sci_{\{x\}}(2)) \subset {\fam0 H}^0(\sco_\pii(2))$. Then$\, :$  

\emph{(a)} There exists a unique element of the linear system $\vert \, 
\Lambda \, \vert$ having a singular point at $x$. 

\emph{(b)} Let $L$ be a line in $\pii$. If $x \in L$ then there exists a 
unique line $L^\prim$ such that the divisor $L + L^\prim$ belongs to 
$\vert \, \Lambda \, \vert$. If $x \notin L$ then there exists at most one 
such line $($which, of course, must contain $x$$)$.    

\emph{(c)} $\vert \, \Lambda \, \vert$ contains no pencil consisting 
entirely of singular divisors. 

\emph{(d)} $G(2)$ admits a linear presentation of the form$\, :$  
\[
0 \lra 2\sco_\pii(-1) \lra 4\sco_\pii \lra G(2) \lra 0\, .
\]

\emph{(e)} Let $y \in \pii \setminus \{x\}$, let $q$ be a non-zero element of 
$G(y) \simeq \{f \in \Lambda \vb f(y) = 0\}$, and let $\scg$ be the kernel 
of the epimorphism $G \ra (G(y)/kq)\otimes_k\sco_{\{y\}}$. Then $\scg(2)$ 
fails to be globally generated precisely when the divisor $q = 0$ is of the 
form $L + L^\prim$ with $L$ and $L^\prim$ lines such that $y \in L$ and 
$x \in L^\prim$. 
\end{lemma} 

\begin{proof}
(a) Since the image of the evaluation morphism $\Lambda \otimes_k
\sco_\pii \ra \sco_\pii(2)$ is $\sci_{\{x\}}(2)$, the map $\Lambda \ra 
(\sci_{\{x\}}/\sci_{\{x\}}^2)(2)$ must be surjective, hence its kernel has 
dimension 1. 

\vskip2mm 

(b) If $x \in L$ then one has an epimorphism $\Lambda \otimes_k\sco_L \ra 
\sci_{\{x\},L}(2) \simeq \sco_L(1)$. It follows that the map $\Lambda \ra 
\tH^0(\sci_{\{x\},L}(2))$ must be surjective, hence its kernel has dimension 1. 

If $x \notin L$ then one has an epimorphism $ \Lambda \otimes_k\sco_L \ra 
\sco_L(2)$. It follows that the image of the map $\Lambda \ra 
\tH^0(\sco_L(2))$ has dimension at least 2, hence its kernel has dimension 
$\leq 1$.  

\vskip2mm 

(c) One uses Lemma~\ref{L:singpencil} and the assertions (a) and (b), just 
proven. 

\vskip2mm 

(d) From the exact sequences$\, :$ 
\begin{gather*}
0 \lra G \lra 3\sco_\pii \lra \sci_{\{x\}}(2) \lra 0\, ,\\
0 \lra \sco_\pii \lra 2\sco_\pii(1) \lra \sci_{\{x\}}(2) \lra 0\, , 
\end{gather*} 
one derives an exact sequence 
$0 \ra G \ra 4\sco_\pii \ra 2\sco_\pii(1) \ra 0$. 
Since $\rk G = 2$ and $c_1(G) = -2$ it follows that $G^\ast \simeq G(2)$. The 
exact sequence from the statement can be now obtained by dualizing the last 
exact sequence. 

\vskip2mm 

(e) Applying Lemma~\ref{L:quasibdl} to $A = G(2)$ one gets that $\scg(2)$ 
fails to be globally generated precisely when there exists a line $L \ni y$ 
such that $G(2) \ra (G(y)/kq)\otimes_k\sco_{\{y\}}(2)$ factorizes through 
$\sco_L$. 

Since $G(2)$ is globally generated and $c_1(G) = -2$, it follows that, for 
every line $L \subset \pii$, $G_L \simeq 2\sco_L(-1)$ or $G_L \simeq 
\sco_L \oplus \sco_L(-2)$. (In the latter case, one says that $L$ is a 
\emph{jumping line} for $G$.) 
One deduces that, for every line $L \ni y$, 
$G(2) \ra (G(y)/kq)\otimes_k\sco_{\{y\}}(2)$ factorizes through $\sco_L$ if and 
only if there exists a non-zero morphism $\sco_L \ra G_L$ such that the 
composite morphism $\sco_L \ra G_L \ra (G(y)/kq)\otimes_k\sco_{\{y\}}$ is 0, 
i.e., if and only if $q \in \text{Im}(\tH^0(G_L) \ra G(y))$.

Let, now, $L \subset \pii$ be an arbitrary line, of equation $\ell = 0$, 
$\ell \in \tH^0(\sco_\pii(1))$.  If $x \notin L$ then one has an exact 
sequence$\, :$ 
\[
0 \lra G_L \lra 3\sco_L \lra \sco_L(2) \lra 0
\]  
hence $\tH^0(G_L) \simeq \{f \in \Lambda \vb f \  \text{vanishes on}\  L\}$. 

On the other hand, if $x \in L$ then one has an exact sequence$\, :$  
\[
0 \lra G_L \lra 3\sco_L \lra \sci_{\{x\}}(2)\otimes_{\sco_\p}\sco_L \lra 
0\, .
\]
Applying the Snake Lemma to the commutative diagram$\, :$  
\[
\begin{CD}
0 @>>> \sci_{\{x\}}(-1) @>>> \sco_\pii(-1) @>>> \sco_{\{x\}}(-1) @>>> 0\\
@. @VV{\ell}V @VV{\ell}V @VV{\ell}V\\
0 @>>> \sci_{\{x\}} @>>> \sco_\pii @>>> \sco_{\{x\}} @>>> 0
\end{CD}
\]
one gets an exact sequence$\, :$ 
\[
0 \lra \sco_{\{x\}}(-1) \lra \sci_{\{x\}}\otimes_{\sco_\p}\sco_L \lra 
\sci_{\{x\},L} \lra 0
\]
hence $\sci_{\{x\}}\otimes_{\sco_\p}\sco_L \simeq \sci_{\{x\},L} \oplus 
\sco_{\{x\}}(-1) \simeq \sco_L(-1) \oplus \sco_{\{x\}}(-1)$. One deduces that if 
$s \in \tH^0(\sci_{\{x\}}(2)\otimes_{\sco_\p}\sco_L)$ then $s = 0$ if and only if 
$s(y) = 0$, $\forall \, y \in L$. But$\, :$ 
\[
(\sci_{\{x\}}\otimes_{\sco_\p}\sco_L)\otimes_{\sco_L}\sco_{\{x\}} \simeq 
\sci_{\{x\}}\otimes_{\sco_\p}\sco_{\{x\}} \simeq \sci_{\{x\}}/\sci_{\{x\}}^2 
\]
hence if $x \in L$ then$\, :$ 
\begin{gather*}
\tH^0(G_L) \simeq \{f \in \Lambda \vb f \  \text{vanishes on}\  L \ 
\text{and} \  \{f = 0\} \  \text{is singular at}\  x\} =\\
\{f \in \Lambda \vb \{f = 0\} = L + L^\prim \  \text{with} \  x \in L^\prim\}
\, .  
\qedhere
\end{gather*} 
\end{proof} 

\begin{prop}\label{P:case3IV} 
Let $L_1$, $L_2$ be two disjoint lines in $\piii$ and let $H_1$, $H_2$ be two 
planes such that $L_i \subset H_i$, $i = 1,\, 2$. Let $L$ be the line 
$H_1 \cap H_2$ and $x_i$ the point where $L_i$ intersects $L$, $i = 1,\, 2$. 
Consider an epimorphism$\, :$ 
\[
\e : 3\sco_{H_1 \cup H_2} \ra \sci_{L_1 \cup L_2, H_1 \cup H_2}(3)
\]
defined by a $3$-dimensional subspace 
$\Lambda \subset {\fam0 H}^0(\sci_{L_1 \cup L_2, H_1 \cup H_2}(3))$. Then 
${\fam0 Ker}\, \e(2)$ is globally generated if and only if the linear system 
$\vb \Lambda \vb \subset \vb \sco_{H_1 \cup H_2}(3) \vb$ contains no divisor 
supported on a union of lines passing, each of them, through $x_1$ or $x_2$.  
\end{prop}

\begin{proof}
Let $h_i = 0$ be an equation of $H_i$ in $\piii$. 
Put $\Sigma = H_1 \cup H_2$, $Z = L_1 \cup L_2$ and $Z_i := Z \cap H_i = L_i 
\cup \{x_{3-i}\}$, $i = 1, 2$. 
Since $Z$, $Z_0$, $Z_1$ are reduced schemes, the exact sequence$\, :$  
\[
0 \lra \sco_\Sigma \lra \sco_{H_1}\oplus \sco_{H_2} \lra \sco_L \lra 0 
\] 
induces an exact sequence$\, :$ 
\[
0 \lra \sci_{Z,\Sigma} \lra \sci_{Z_1,H_1} \oplus \sci_{Z_2,H_2} \lra 
\sci_{\{x_1,x_2\},L} \lra 0\, ,
\]
and, for $i \in \{1,\, 2\}$, the exact sequence$\, :$ 
\[
0 \lra \sco_{H_{3-i}}(-1) \overset{h_i}{\lra} \sco_\Sigma \lra \sco_{H_i} 
\lra 0 
\]
induces an exact sequence$\, :$ 
\[
0 \lra \sci_{L_{3-i},H_{3-i}}(-1) \overset{h_i}{\lra} \sci_{Z,\Sigma} \lra 
\sci_{Z_i,H_i} \lra 0\, . 
\]
Let $\ell_i \in \tH^0(\sco_{H_i}(1))$ be a linear form vanishing on $L_i$, 
$i = 1,\, 2$. Multiplication by $\ell_i$ induces an isomorphism$\, :$ 
\[
\sci_{\{x_{3-i}\},H_i}(-1) \Izo \sci_{Z_i,H_i}\, ,\  i= 1,\, 2\, . 
\]
$\e \vb H_i$ can be, consequently, written as a composite morphism$\, :$ 
\[
3\sco_{H_i} \overset{\e_i}{\lra} \sci_{\{x_{3-i}\},H_i}(2) \overset{\ell_i}{\lra} 
\sci_{Z_i,H_i}(3) 
\]
with $\e_i$ epimorphism. Put $\Lambda := \text{Im}\, \tH^0(\e) \subset 
\tH^0(\sci_{Z,\Sigma}(3)) \subset \tH^0(\sco_\Sigma(3))$  and$\, :$  
\[
\Lambda_i := \text{Im}\, \tH^0(\e_i) 
\subset \tH^0(\sci_{\{x_{3-i}\},H_i}(2)) \subset \tH^0(\sco_{H_i}(2))\, ,\  
i = 1,\, 2\, . 
\] 
Then $\Lambda \vb H_i = \ell_i\Lambda_i$, $i = 1,\, 2$. Since there is no 
epimorphism $2\sco_{H_i} \ra \sci_{\{x_{3-i}\},H_i}(2)$, it follows that 
$\dim_k\Lambda_i = 3$, $i = 1,\, 2$. Put, also$\, :$ 
\[
\scg := \Ker \e \  \  \text{and}\  \  G_i:= \Ker \e_i\, ,\  i = 1,\, 2\, . 
\]
Applying the Snake Lemma to the diagram$\, :$
\[
\begin{CD}
0 @>>> 3\sco_{H_{3-i}}(-1) @>{h_i}>> 3\sco_\Sigma @>>> 3\sco_{H_i}  
@>>> 0\\
@. @V{\ell_{3-i}\e_{3-i}(-1)}VV @VV{\e}V @VV{\ell_i\e_i}V\\ 
0 @>>> \sci_{L_{3-i},H_{3-i}}(2) @>{h_i}>> \sci_{Z,\Sigma}(3) @>>> 
\sci_{Z_i,H_i}(3) @>>> 0
\end{CD}
\]
one gets an exact sequence$\, :$  
\begin{equation}\label{E:g3-iscggi} 
0 \lra G_{3-i}(-1) \lra \scg \lra G_i \lra \sco_{\{x_i\}}(1) \lra 0\, . 
\end{equation}
We describe, now, the (1-dimensional) image of the map $\scg(x_i) \ra 
G_i(x_i)$. Using the exact sequences$\, :$ 
\begin{gather*} 
0 \lra \scg \lra \Lambda \otimes_k \sco_\Sigma \overset{\e}{\lra} 
\sci_{Z,\Sigma}(3) \lra 0\, ,\\ 
0 \lra G_i \lra \Lambda_i \otimes_k \sco_{H_i} \overset{\e_i}{\lra} 
\sci_{\{x_{3-i}\},H_i}(2) \lra 0\, ,  
\end{gather*}
one deduces that$\, :$ 
\begin{gather*}
\text{Im}(\scg(x_i) \ra \Lambda) = \Ker(\Lambda \xra{\e(x_i)} 
\sci_{Z,\Sigma}(3)\otimes_{\sco_\Sigma}\sco_{\{x_i\}})\\ 
G_i(x_i) = \Ker(\Lambda_i \xra{\e_i(x_i)} 
\sci_{\{x_{3-i}\},H_i}(2) \otimes_{\sco_{H_i}} \sco_{\{x_i\}}) = 
\{f \in \Lambda_i \vb f(x_i) = 0\}\, . 
\end{gather*}
Consider the diagram$\, :$ 
\[  
\begin{CD}
@. \sci_{Z,\Sigma}(3)\otimes_{\sco_\Sigma}\sco_{\{x_i\}}\\
@. @VVV\\
\sci_{\{x_i\},H_{3-i}}(2)\otimes_{\sco_{H_{3-i}}}\sco_{\{x_i\}} @>\sim>> 
\sci_{Z_{3-i},H_{3-i}}(3)\otimes_{\sco_{H_{3-i}}}\sco_{\{x_i\}} 
\end{CD}
\] 
where the horizontal isomorphism is defined by multiplication by 
$\ell_{3-i}(x_i)$. The vertical map is surjective. Since $Z$ is a locally 
complete intersection of codimension 2 in $\piii$ it follows that 
$\sci_{Z,\Sigma}(3)\otimes_{\sco_\Sigma}\sco_{\{x_i\}}$ is a $k$-vector space of 
dimension $\leq 2$. One deduces that the vertical map is an isomorphism. 

It follows, from what has been said, that $\text{Im}(\scg(x_i) \ra \Lambda)$ 
is generated by the unique (up to multiplication by a scalar) non-zero 
element $f_i$ of $\Lambda$ such that $f_i \vb H_{3-i} = \ell_{3-i}f_{i,3-i}$, 
where $f_{i,3-i}$ is the unique non-zero element of $\Lambda_{3-i}$ such that the 
divisor $\{ f_{i,3-i} = 0\}$ on $H_{3-i}$ has a singular point at $x_i$ 
(see Lemma~\ref{L:ponceletbdl}(a)). Then $f_i \vb H_i = \ell_if_{ii}$, with 
$f_{ii} \in \Lambda_i$ vanishing at $x_i$ (and at $x_{3-i}$ because it belongs 
to $\Lambda_i$) and$\, :$ 
\[
\text{Im}(\scg(x_i) \ra G_i(x_i)) = kf_{ii}\, . 
\]  

Consider, now, the quasi-bundle $\scg_i := \Ker(G_i \ra (G_i(x_i)/kf_{ii}) 
\otimes_k \sco_{\{x_i\}})$ on $H_i$, $i = 1,\, 2$. One gets, from the exact 
sequence \eqref{E:g3-iscggi} an exact sequence$\, :$ 
\begin{equation}\label{E:g3-iscgscgi} 
0 \lra G_{3-i}(-1) \lra \scg \lra \scg_i \lra 0\, . 
\end{equation}

\noindent
{\bf Claim.}\quad $\scg(2)$ \emph{is globally generated if and only if} 
$\scg_1(2)$ \emph{and} $\scg_2(2)$ \emph{are globally generated}. 

\vskip2mm 

\noindent 
\emph{Indeed}, only the ``if'' part needs an argument. Assume that 
$\scg_1(2)$ and $\scg_2(2)$ are globally generated. Since $\tH^p(G_j(1)) = 0$, 
$p = 0,\, 1$, $j = 1,\, 2$, (see Lemma~\ref{L:ponceletbdl}(d)) one gets, from 
the exact sequence \eqref{E:g3-iscgscgi}, that $\tH^0(\scg(2)) \izo 
\tH^0(\scg_i(2))$, $i = 1,\, 2$. Since $Z$ is an effective Cartier divisor on 
$(H_1 \cup H_2)\setminus \{x_1,x_2\}$ it follows that $\dim_k \scg(x) = 2$, 
$\forall \, x \in (H_1 \cup H_2)\setminus \{x_1,x_2\}$. Moreover, since $\scg$ 
is the quotient of a rank 3 vector bundle on $\piii$, one has 
$\dim_k \scg(x_i) \leq 3$, $i = 1,\, 2$. On the other hand, $\dim_k\scg_i(x) 
= 2$, $\forall \, x \in H_i \setminus \{x_i\}$, and $\dim_k\scg_i(x_i) = 3$ 
($\scg_i$ is isomorphic to $\sco_{H_i} \oplus \sci_{\{x_i\},H_i}$ in a 
neighbourhood of $x_i$ on $H_i$). One deduces that $\scg(x) \izo \scg_i(x)$, 
$\forall \, x \in H_i \setminus \{x_{3-i}\}$, $i = 1,\, 2$. Since 
$H_1 \cup H_2 = (H_1 \setminus \{x_2\}) \cup (H_2 \setminus \{x_1\})$ it follows 
that $\scg$ is globally generated and the claim is proven. 

\vskip2mm 

Finally, in order to conclude the proof of the proposition, it suffices to 
notice that, for $i = 1,\, 2$, Lemma~\ref{L:ponceletbdl}(e) implies that 
$\scg_i(2)$ fails to be globally generated 
precisely when the divisor $f_{ii} = 0$ on $H_i$ is of the form 
$L_i^\prim + L_i^\secund$ with $x_1 \in L_i^\prim$ and $x_2 \in L_i^\secund$. 
\end{proof}

\begin{remark}\label{R:case3IV}
When none of the distinguished elements $f_1$ and $f_2$ of $\Lambda$,  
appearing in the above proof of Prop.~\ref{P:case3IV}, vanishes 
on $L$ then it turns out that $\Ker \e(2)$ is 
globally generated if and only if $\{f_{ii} = 0\} \subset H_i$ is a nonsingular  
conic, $i = 1, 2$. One constructs such epimorphisms $\e : 
3\sco_{H_1 \cup H_2} \ra \sci_{L_1 \cup L_2,H_1 \cup H_2}(3)$ in the following way: 
choose $f_{1i}$, $f_{2i}$, $f_{3i} \in \tH^0(\sco_{H_i}(2))$, $i = 1, 2$, 
such that  
\begin{itemize}
\item $\{f_{12} = 0\} = L_{12} + L_{12}^\prim$ with $L_{12} \neq L \neq 
L_{12}^\prim$, $x_1 \in L_{12} \cap L_{12}^\prim$; 
\item $\{f_{11} = 0\} = C_1$ nonsingular conic containing $x_1$ and $x_2$; 
\item $\{f_{21} = 0\} = L_{21} + L_{21}^\prim$ with $L_{21} \neq L \neq 
L_{21}^\prim$, $x_2 \in L_{21} \cap L_{21}^\prim$; 
\item $\{f_{22} = 0\} = C_2$ nonsingular conic containing $x_1$ and $x_2$; 
\item $\{f_{31} = 0\} = L + L_{31}$, $L_{31} \cap C_1 \cap (L_{21}\cup L_{21}^\prim) 
= \emptyset$; 
\item $\{f_{32} = 0\} = L + L_{32}$, $L_{32} \cap C_2 \cap (L_{12} \cup 
L_{12}^\prim) = \emptyset$.
\end{itemize}
Multiplying $f_{11}$, $f_{12}$, $f_{21}$, $f_{22}$ with convenient non-zero 
scalars one may assume that $(\ell_1f_{11}) \vb L = (\ell_2f_{12}) \vb L$ and 
$(\ell_1f_{21}) \vb L = (\ell_2f_{22}) \vb L$. There exists unique elements 
$f_1$, $f_2$, $f_3 \in \tH^0(\sci_{L_1 \cup L_2,H_1 \cup H_2}(3))$ such that 
$f_j \vb H_i = \ell_if_{ji}$, $j = 1, 2, 3$, $i = 1, 2$. They define an 
epimorphism $\e : 3\sco_{H_1 \cup H_2} \ra \sci_{L_1 \cup L_2,H_1 \cup H_2}(3)$ for 
which $\Ker \e(2)$ is globally generated. 

The case where one of the distinguished elements $f_1$ and $f_2$ of $\Lambda$ 
vanishes on $L$ can be analysed similarly. 
\end{remark}

Before proving the final result of this appendix, namely 
Prop.~\ref{P:case3V} below, we need one more lemma. 

\begin{lemma}\label{L:ponceletbdl2} 
Under the hypothesis of Lemma~\ref{L:ponceletbdl}, $\Lambda$ contains, by 
assertion (a) of that lemma, a distinguished element $f_0$ such that 
$f_0 = \ell_0\ell_0^\prim$ with $\ell_0$, $\ell_0^\prim$ linear forms vanishing 
at $x$. Choose another two linear forms $\ell_1^\prim$ and $\ell_2^\prim$ 
vanishing at $x$ such that they are linearly independent and none of them 
belongs to $k\ell_0$ or to $k\ell_0^\prim$. By Lemma~\ref{L:ponceletbdl}(b), 
there exist elements $f_1$, $f_2$ of $\Lambda$ such that 
$f_1 = \ell_1\ell_1^\prim$, $f_2 = \ell_2\ell_2^\prim$, with $\ell_1$ and $\ell_2$ 
linear forms. Let $L_i$ $($resp., $L_i^\prim$$)$ be the line of equation 
$\ell_i = 0$ $($resp., $\ell_i^\prim = 0$$)$, $i = 0, 1 ,2$. Then$\, :$  

\emph{(a)} $f_0$, $f_1$, $f_2$ is a $k$-basis of $\Lambda$. 

\emph{(b)} The subscheme $\{f_1 = f_2 = 0\}$ of $\pii$ is a complete 
intersection of type $(2,2)$ consisting of either four simple points or of two 
simple points $($one of them being $x$$)$ and of a point of length $2$. If the 
point of length $2$ occurs, then it is a subscheme of $L_1$ or of $L_2$.   

\emph{(c)} According to the last part of the proof of 
Lemma~\ref{L:ponceletbdl}(e), there exist canonical epimorphisms 
$G(2) \ra \sco_{L_0}$ and $G(2) \ra \sco_{L_0^\prim}$. By the assertion (b) above, 
there exist linear forms $h_1^\prim$ and $h_2^\prim$ uniquely determined by the 
condition$\, :$ 
\[
h_1^\prim f_1 + h_2^\prim f_2 = \ell_1\ell_2\ell_0^\prim\, .
\]
Then the following three elements of ${\fam0 H}^0(3\sco_\pii(2))$$\, :$  
\[
(f_1,-f_0,0),\  (f_2,0,-f_0),\  (\ell_1\ell_2,-\ell_0h_1^\prim ,-\ell_0h_2^\prim ) 
\]
form a $k$-basis of the kernel of ${\fam0 H}^0(G(2)) \ra 
{\fam0 H}^0(\sco_{L_0})$. The kernel of ${\fam0 H}^0(G(2)) \ra 
{\fam0 H}^0(\sco_{L_0^\prim})$ has an analogous description.  
\end{lemma} 

\begin{proof}
By the uniqueness assertion of Lemma~\ref{L:ponceletbdl}(a), $L_1$ and $L_2$ do 
not contain $x$. Since $L_1^\prim \neq L_2^\prim$, the second part of 
Lemma~\ref{L:ponceletbdl}(b) implies that $L_1 \neq L_2$. By the last part of 
the proof of Lemma~\ref{L:ponceletbdl}(e), $L_0$, $L_0^\prim$, $L_1$ and $L_2$ 
are jumping lines for $G$. It follows, from Lemma~\ref{L:ponceletbdl}(d), 
that $G(1)$ is a stable rank 2 vector bundle on $\pii$ with Chern classes 
$c_1(G(1)) = 0$ and $c_2(G(1)) = 2$. It is well known that, in this case, 
the set of jumping lines for $G(1)$ form a nonsingular conic in $\p^{2\vee}$. 
In particular, any point of $\pii$ is contained in at most two jumping lines 
for $G$. One deduces that $L_0 \cap L_1 \cap L_2 = \emptyset$ and that 
$L_0^\prim \cap L_1 \cap L_2 = \emptyset$. 

\vskip2mm 

(a) Consider the points $x_{01}$ and $x_{02}$ defined by $\{x_{0i}\} = L_0 \cap 
L_i$, $i = 1, 2$. By what has been said, the points $x$, $x_{01}$ and $x_{02}$ 
are distinct. The zero divisor of $f_i \vb L_0$ is $x + x_{0i}$, $i = 1, 2$, 
hence $f_1 \vb L_0$ and $f_2 \vb L_0$ are linearly independent, hence $f_0$, 
$f_1$, $f_2$ are linearly independent. 

\vskip2mm

(b) The set $\{f_1 = f_2 = 0\}$ consists of the following (at most) four  
points$\, :$ $\{x\} = L_1^\prim \cap L_2^\prim$, $\{x_1\} := L_1 \cap L_2^\prim$, 
$\{x_2\} := L_1^\prim \cap L_2$, and $\{x_{12}\} :=L_1 \cap L_2$. One has 
$x \notin \{x_1,x_2,x_{12}\}$ and $x_1 \neq x_2$. If $x_1 = x_{12}$ then the 
scheme $\{f_1 = f_2 = 0\}$ coincides, in a neighbourhood of $x_1$, with the 
scheme $L_1 \cap \{f_2 = 0\}$ (because $x_1 \notin L_1^\prim$ since $L_1^\prim 
\cap L_2^\prim = \{x\}$ and $x \notin L_1$). Similarly, if $x_2 = x_{12}$ then 
the scheme $\{ f_1 = f_2 = 0\}$ coincides, in a neighbourhood of $x_2$, with 
the scheme $\{f_1 = 0\} \cap L_2$. 

\vskip2mm 

(c) The three elements from the statement are obviously relations among 
$f_0$, $f_1$, and $f_2$ hence they belong to $\tH^0(G(2))$. Let us denote 
them by $\rho_1$, $\rho_2$, and $\rho_3$. One has $G_{L_0}(2) \simeq 
\sco_{L_0}(2) \oplus \sco_{L_0}$. $\rho_1 \vb L_0$ vanishes at $x$ and $x_{01}$, 
$\rho_2 \vb L_0$ vanishes at $x$ and $x_{02}$ and $\rho_3 \vb L_0$ vanishes at 
$x_{01}$ and $x_{02}$. Since, as we noticed in the proof of assertion (a), the 
points $x$, $x_{01}$, and $x_{02}$ are distinct, it follows that $\rho_i \vb 
L_0$, $i = 1, 2, 3$, is a $k$-basis of $\tH^0(\sco_{L_0}(2)) \subset 
\tH^0(G_{L_0}(2))$. 
\end{proof} 

Finally, let us investigate the problem of the global generation of 
$\Ker \e(2)$ for epimorphisms  
$
\e : 3\sco_{H^{(1)}} \lra \sci_{Z,H^{(1)}}(3) 
$,  
where $H^{(1)}$ is the first infinitesimal neighbourhood of a plane $H \subset 
\piii$ and $Z$ is the divisor $2L$ on a nonsingular quadric surface $Q \subset 
\piii$, $L$ being a line contained in $H \cap Q$.     

After a linear change of coordinates one may assume that 
$H = \{X_3 = 0\}$ (hence $H^{(1)} = \{X_3^2 = 0\}$), $L = \{X_2 = X_3 = 0\}$ and 
$Q = \{X_1X_2 - X_0X_3 = 0\}$. 
The homogeneous ideal $I(Z)$ of $Z$ in the coordinate ring $S = k[X_0,\ldots 
,X_3]$ of $\piii$ is $I(Z) = (X_2^2,X_2X_3,X_3^2,X_1X_2 - X_0X_3)$. Let 
${\widehat X}_3$ be the image of $X_3 \in \tH^0(\sco_\piii(1))$ into 
$\tH^0(\sco_{H^{(1)}}(1))$. Consider the scheme $Z_0 := Z \cap H$. One has 
$\sci_{Z_0,H} = X_2\sci_{\{x_0\},H}(-1)$, where $x_0 = (1:0:0:0)$. In other words, 
the subscheme $Z_0$ of $H$ is the line $L$ with an embedded point in $x_0$. 
The exact sequence$\, :$ 
\[
0 \lra \sco_H(-1) \overset{{\widehat X}_3}{\lra} \sco_{H^{(1)}} \lra \sco_H 
\lra 0
\] 
induces an exact sequence$\, :$  
\[
0 \lra \sci_{L,H}(-1) \overset{{\widehat X}_3}{\lra} \sci_{Z,H^{(1)}} \lra 
\sci_{Z_0,H} \lra 0 
\]
which can be written in the form$\, :$  
\begin{equation}
\label{E:izh(1)}
0 \lra \sco_H(-2) \xra{X_2{\widehat X}_3} \sci_{Z,H^{(1)}} 
\lra X_2\sci_{\{x_0\},H}(-1) \lra 0\, .
\end{equation}
The restriction to $H$ of the epimorphism $\e : 3\sco_{H^{(1)}} \lra 
\sci_{Z,H^{(1)}}(3)$ can be, consequently, written as a composite morphism$\, :$ 
\[
3\sco_H \overset{\e_0}{\lra} \sci_{\{x_0\},H}(2) \overset{X_2}{\lra} 
\sci_{Z_0,H}(3) 
\] 
with $\e_0$ an epimorphism. Put $\Lambda := \text{Im}\, \tH^0(\e) \subset 
\tH^0(\sci_{Z,H^{(1)}}(3)) \subset \tH^0(\sco_{H^{(1)}}(3))$  and   
$\Lambda_0 := \text{Im}\, \tH^0(\e_0) 
\subset \tH^0(\sci_{\{x_0\},H}(2)) \subset \tH^0(\sco_H(2))$. One has, of 
course, $\Lambda \vb H = X_2\Lambda_0$. 

\begin{prop}\label{P:case3V} 
Under the above hypotheses and notation, the linear system $\vb \Lambda_0 
\vb$ contains, according to Lemma~\ref{L:ponceletbdl}(a), a unique divisor  
having a singular point at $x_0$. In other words, there exists $f_0 \in 
\Lambda_0$ such that $f_0 = \ell_0\ell^\prim_0$ with $\ell_0,\, \ell_0^\prim \in 
k[X_1,X_2]_1$. $X_2f_0$ can be lifted to a unique element ${\widetilde f}_0$ 
of $\Lambda$. Taking into account the exact sequence \eqref{E:izh(1)}, it 
follows that if $\ell_0^\prim = a_0^\prim X_1 + b_0^\prim X_2$ then$\, :$ 
\[
{\widetilde f}_0 = a_0^\prim\ell_0(X_1X_2 - X_0{\widehat X}_3) + 
b_0^\prim \ell_0X_2^2 + \ell_0^\secund X_2{\widehat X}_3 
\]  
with $\ell_0^\secund \in {\fam0 H}^0(\sco_H(1))$ uniquely determined. Then 
${\fam0 Ker}\, \e(2)$ is globally generated if and only if $\ell_0^\secund(x_0) 
\neq 0$.   
\end{prop} 

\begin{proof} 
Put $\scg = \Ker \e$ and $G_0 = \Ker \e_0$. Applying the Snake Lemma to the 
diagram$\, :$ 
\[
\begin{CD}
0 @>>> 3\sco_H(-1) @>{{\widehat X}_3}>> 3\sco_{H^{(1)}} @>>> 
3\sco_H @>>> 0\\
@. @V{\e_0(-1)}VV @VV{\e}V @VV{X_2\e_0}V\\
0 @>>> \sco_H(1) @>{X_2{\widehat X}_3}>> \sci_{Z,H^{(1)}}(3) @>>> 
\sci_{Z_0,H}(3) @>>> 0  
\end{CD}
\] 
one gets an exact sequence$\, :$  
\begin{equation} 
\label{E:g0scgg0} 
0 \lra G_0(-1) \lra \scg \lra G_0 \lra \sco_{\{x_0\}}(1) \lra 0\, .
\end{equation}
Let $\scg_0$ be the kernel of $G_0 \ra \sco_{\{x_0\}}(1)$. One can show, as in 
proof of Prop.~\ref{P:case3IV}, that $\tH^0(\scg(2)) \izo \tH^0(\scg_0(2))$ 
and that $\scg(x) \izo \scg_0(x)$, $\forall \, x \in H$.  
It follows that $\scg(2)$ \emph{is globally generated if and only if}  
$\scg_0(2)$ \emph{is globally generated}. 

Now, according to Lemma~\ref{L:quasibdl}, $\scg_0(2)$ fails to be globally 
generated precisely when there exists a line $L \ni x_0$ such that 
$G_0(2) \lra \sco_{\{x_0\}}(3)$ factorizes through $\sco_L$. Since $G_0(2)$ is 
globally generated, this is equivalent to the fact that there exists a line 
$L \ni x$ and an epimorphism $G_0(2) \ra \sco_L$ such that the composite 
map $\tH^0(\scg_0(2)) \ra \tH^0(G_0(2)) \ra \tH^0(\sco_L)$ is 0 (because for 
any epimorphism $\sco_L \ra \sco_{\{x_0\}}(3)$ one has $\tH^0(\sco_L) \izo 
\tH^0(\sco_{\{x_0\}}(3))$). 

But, by the proof of Lemma~\ref{L:ponceletbdl}(e), there are only two lines 
$L \ni x$ for which there exists an epimorphism $G_0(2) \ra \sco_L$, namely 
the line $L_0$ of equation $\ell_0 = 0$ and the line $L_0^\prim$ of equation 
$\ell_0^\prim = 0$. Since $\tH^0(\scg(2)) \izo \tH^0(\scg_0(2))$, one deduces 
that $\scg(2)$ is globally generated if and only if none of the composite 
maps$\, :$ 
\[
\tH^0(\scg(2)) \lra \tH^0(G_0(2)) \lra \tH^0(\sco_{L_0}),\  
\tH^0(\scg(2)) \lra \tH^0(G_0(2)) \lra \tH^0(\sco_{L_0^\prim})
\] 
is 0. Now, complete $f_0$ to a $k$-basis $f_0$, $f_1 = \ell_1\ell_1^\prim$, 
$f_2 = \ell_2\ell_2^\prim$ of $\Lambda_0$ as in Lemma~\ref{L:ponceletbdl2}. 
Let's say that $\ell_i^\prim = a_i^\prim X_1 + b_i^\prim X_2$, $i = 1, 2$. 
$f_1$ and $f_2$ lift to elements ${\widetilde f}_1$ and ${\widetilde f}_2$ of 
$\Lambda$ which must have the form$\, :$  
\[
{\widetilde f}_i = a_i^\prim\ell_i(X_1X_2 - X_0{\widehat X}_3) + 
b_i^\prim \ell_iX_2^2 + \ell_i^\secund X_2{\widehat X}_3 
\]  
with $\ell_i^\secund \in \tH^0(\sco_H(1))$ uniquely determined, $i = 1, 2$. 
By Lemma~\ref{L:ponceletbdl2}(b), there exist linear forms $h_1$, $h_2$, 
$h_1^\prim$, $h_2^\prim \in \tH^0(\sco_H(1))$ uniquely determined by the 
relations$\, :$ 
\[
h_1f_1 + h_2f_2 = \ell_1\ell_2\ell_0,\  h_1^\prim f_1 + h_2^\prim f_2 = 
\ell_1\ell_2\ell_0^\prim\, . 
\]
Lemma~\ref{L:ponceletbdl2}(c) implies that the kernel of $\tH^0(G_0(2)) \ra 
\tH^0(\sco_{L_0})$ is generated by the following elements$\, :$  
\[
(f_1,-f_0,0),\  (f_2,0,-f_0),\  (\ell_1\ell_2,-\ell_0h_1^\prim ,  
-\ell_0h_2^\prim) 
\] 
and the kernel of $\tH^0(G_0(2)) \ra \tH^0(\sco_{L_0^\prim})$ is generated 
by$\, :$  
\[
(f_1,-f_0,0),\  (f_2,0,-f_0),\  (\ell_1\ell_2,-\ell_0^\prim h_1,
-\ell_0^\prim h_2)\, . 
\]
Since the elements $(f_1,-f_0,0)$ and $(f_2,0,-f_0)$ of $\tH^0(G_0(2))$ lift 
to the elements $({\widetilde f}_1,-{\widetilde f}_0,0)$ and 
$({\widetilde f}_2,0,-{\widetilde f}_0)$ of $\tH^0(\scg(2))$, it follows that 
$\scg(2)$ is globally generated if and only if none of the elements$\, :$  
\[
(\ell_1\ell_2,-\ell_0h_1^\prim ,-\ell_0h_2^\prim),\  
(\ell_1\ell_2,-\ell_0^\prim h_1,-\ell_0^\prim h_2)
\]
can be lifted to $\tH^0(\scg(2))$. Let us investigate under what conditions 
an element $(g_0,g_1,g_2)$ of $\tH^0(G_0(2))$ can be lifted to an element 
$({\widetilde g}_0,{\widetilde g}_1,{\widetilde g}_2)$ of $\tH^0(\scg(2))$. 
One must have ${\widetilde g}_i = g_i + \gamma_i{\widehat X}_3$, with 
$\gamma_i \in \tH^0(\sco_H(1))$, $i = 0, 1, 2$. Writing ${\widetilde f}_i$ 
in the form$\, :$ 
\[
{\widetilde f}_i = f_iX_2 - a_i^\prim \ell_iX_0{\widehat X}_3 + 
\ell_i^\secund X_2{\widehat X}_3 
\]
one sees that the condition $\sum_{i=0}^2{\widetilde g}_i{\widetilde f}_i = 0$ 
is equivalent to the condition$\, :$  
\[
{\textstyle \sum_{i=0}^2}a_i^\prim g_i\ell_iX_0 - 
{\textstyle \sum_{i=0}^2}g_i\ell_i^\secund X_2 = 
{\textstyle \sum_{i=0}^2}\gamma_if_iX_2\, .
\] 
Assume, now, that $(g_0,g_1,g_2) = (\ell_1\ell_2,-\ell_0h_1^\prim ,
-\ell_0h_2^\prim)$. In this case$\, :$ 
\[
{\textstyle \sum_{i=0}^2}a_i^\prim g_i\ell_i = 
\ell_0(a_0^\prim \ell_1\ell_2 - a_1^\prim h_1^\prim \ell_1 - a_2^\prim h_2^\prim 
\ell_2)\, .
\]
The relation which defines $h_1^\prim$ and $h_2^\prim$, namely$\, :$  
\[
\ell_1\ell_2\ell_0^\prim - h_1^\prim f_1 - h_2^\prim f_2 = 0 
\]
can be written in the form$\, :$ 
\[
(a_0^\prim \ell_1\ell_2 - a_1^\prim h_1^\prim \ell_1 - a_2^\prim h_2^\prim 
\ell_2)X_1 + (b_0^\prim \ell_1\ell_2 - b_1^\prim h_1^\prim \ell_1 - 
b_2^\prim h_2^\prim \ell_2)X_2 = 0
\] 
hence $a_0^\prim \ell_1\ell_2 - a_1^\prim h_1^\prim \ell_1 - a_2^\prim h_2^\prim 
\ell_2$ is divisible by $X_2$. One gets that, when $(g_0,g_1,g_2) = 
(\ell_1\ell_2,-\ell_0h_1^\prim ,-\ell_0h_2^\prim)$, the lifting condition 
$\sum_{i=0}^2{\widetilde g}_i{\widetilde f}_i = 0$ is equivalent to$\, :$  
\[
\frac{a_0^\prim \ell_1\ell_2 - a_1^\prim h_1^\prim \ell_1 - a_2^\prim h_2^\prim 
\ell_2}{X_2}\ell_0X_0 - {\textstyle \sum_{i=0}^2}g_i\ell_i^\secund  = 
{\textstyle \sum_{i=0}^2}\gamma_if_i\, .
\]
Since $f_0$, $f_1$, $f_2$ generate $\sci_{\{x_0\}}(2)$ and since $\tH^1(G_0(1)) 
= 0$ it follows that the existence of elements $\gamma_0$, $\gamma_1$, 
$\gamma_2 \in \tH^0(\sco_H(1))$ such that the above relation is valid is 
equivalent to the fact that the left hand side of the relation vanishes 
at $x_0$. Since $\ell_0(x_0) = 0$, $g_1(x_0) = 0$, $g_2(x_0) = 0$, but 
$g_0(x_0) \neq 0$ one concludes that $(\ell_1\ell_2,-\ell_0h_1^\prim ,
-\ell_0h_2^\prim)$ can be lifted to $\tH^0(\scg(2))$ if and only if 
$\ell_0^\secund (x_0) = 0$. 

Similarly, $(\ell_1\ell_2,-\ell_0^\prim h_1,-\ell_0^\prim h_2)$ can be lifted 
to $\tH^0(\scg(2))$ if and only if $\ell_0^\secund (x_0) = 0$.  
\end{proof}

\section{The case $c_1 = 4$, $c_2 = 8$, $c_3 = 4$ on 
$\piii$}\label{A:case4}

In this appendix we complete the argument for the Case 4 of the proof of 
Prop.~\ref{P:c1=4c2=8n3}. Recall that we considered in that case a vector 
bundle $E$ on $\piii$ defined by an exact sequence$\, :$  
\[
0 \lra E(-2) \lra 4\sco_\piii \oplus 2\sco_\piii(-1)  
\overset{\rho}{\lra} 2\sco_\piii(1) \lra 0
\]
such that the degeneracy scheme $Z$ of the component $\phi : 
4\sco_\piii \ra 2\sco_\piii(1)$ of $\rho$ is 0-dimensional. We  
recall in the next remark, in a form which is more convenient for our 
purposes, a classification, due to Mir\'{o}-Roig and 
Trautmann~\cite{mrt}, of the $2\times 4$ matrices of linear forms in 
four indeterminates. The result of Mir\'{o}-Roig and Trautmann refines 
Lemma~2.14 in the paper of Chang~\cite{ch3}. 

\begin{remark}\label{R:2x4} 
Let $S = k[X_0,X_1,X_2,X_3]$ be the projective coordinate ring of $\piii$. 
Consider on $\piii$ a morphism   
$\phi : 4\sco_\piii \ra 2\sco_\piii(1)$. 
$\phi$ is defined by a $k$-linear map $A : k^4 \ra k^2\otimes S_1$ or, 
equivalently, by a $2\times 4$ matrix of linear forms$\, :$  
\[
\begin{pmatrix}
h_{00} & h_{01} & h_{02} & h_{03}\\
h_{10} & h_{11} & h_{12} & h_{13} 
\end{pmatrix}
\] 
The linear map $A$ is called \emph{stable} if it is injective and if, for 
every non-zero $k$-linear function $k^2 \ra k$, the composite map 
$k^4 \overset{A}{\lra} k^2\otimes S_1 \lra S_1$ has corank $\leq 1$ (see 
\cite[Lemma~1.1.1]{mrt} for a justification of this definition). Notice that 
if $A$ is injective but not stable then, up to the action of 
$\text{GL}(2) \times \text{GL}(4)$, $A$ is represented by a matrix of the 
form$\, :$ 
\[
\begin{pmatrix}
h_{00} & h_{01} & h_{02} & h_{03}\\
h_{10} & h_{11} & 0 & 0 
\end{pmatrix}
\]  
hence $\phi$ degenerates along the line $L$ of equations $h_{10} = h_{11} = 0$ 
(if $h_{10}$ and $h_{11}$ are linearly independent) or along a plane or 
everywhere. 

\vskip2mm 

We want, now, to classify the stable maps $A$ up to the action of 
$\text{GL}(2) \times \text{GL}(S_1) \times \text{GL}(4)$. For this purpose, 
we use the fact that one can also associate to $A$ a morphism on 
$\p^1$$\, :$  
\[
\psi : 4\sco_{\p^1} \lra \sco_{\p^1}(1)\otimes_kS_1\, .
\] 
The stability of $A$ is equivalent to the fact that $\psi$ has rank $\geq 3$ 
at every point of $\p^1$. We consider, firstly, the case where $\psi$ has, 
generically, rank 4. In this case, the degeneracy locus $D(\psi)$ of 
$\psi$ is an effective divisor of degree 4 on $\p^1$. One can assume, up to 
the action of $\text{GL}(2)$ (or $\text{PGL}(1)$), that the components of 
$D(\psi)$ are some special points of $\p^1$. 

\vskip2mm

\noindent 
{\bf Case 1.}\quad $\psi$ \emph{has, generically, rank $4$ and $D(\psi)$ 
consists of $4$ simple points}. 

\vskip2mm 

\noindent 
Let's say, that the four points are  
$(-a_0,1)$, $(-a_1,1)$, $(-1,1)$, $(0,1)$. Then 
$\psi$ is defined (up to the action of $\text{GL}(S_1)\times \text{GL}(4)$)  
by the matrix$\, :$ 
\[
\begin{pmatrix}
T_0+a_0T_1 & 0 & 0 & 0\\
0 & T_0+a_1T_1 & 0 & 0\\
0 & 0 & T_0+T_1 & 0\\
0 & 0 & 0 & T_0
\end{pmatrix}
\]
hence $\phi$ is defined by the matrix$\, :$  
\[
\begin{pmatrix}
X_0 & X_1 & X_2 & X_3\\
a_0X_0 & a_1X_1 & X_2 & 0
\end{pmatrix}
\] 
The $2\times 2$ minors of this matrix generate the ideal $(X_iX_j \vb 
0 \leq i < j \leq 3)$ of $S$ hence the degeneracy scheme $D(\phi)$ of $\phi$ 
consists of the four simple points $(1,0,0,0),\ldots , (0,0,0,1)$. 
Applying the Snake Lemma to the diagram$\, :$ 
\[
\begin{CD} 
0 @>>> 0 @>>> 4\sco_\piii @>>> 4\sco_\piii @>>> 0\\
@. @VVV @VV{\phi}V @VVV\\
0 @>>> \sco_\piii(1) @>>> 2\sco_\piii(1) @>{\text{pr}_1}>> 
\sco_\piii(1) @>>> 0
\end{CD}
\] 
one derives an exact sequence$\, :$ 
\[
\Omega_\piii(1) \lra \sco_\piii(1) \lra \Cok \phi \lra 0\, . 
\]
It follows, in particular, that $\Cok \phi \simeq \sco_{D(\phi)}(1)$ (by the 
defining property of Fitting ideals).   

\vskip2mm 

\noindent 
{\bf Case 2.}\quad 
$\psi$ \emph{has, generically, rank $4$ and}  
$D(\psi) = (-a_0,1) + (-1,1) + 2(0,1)$.   

\vskip2mm

\noindent
In this case, $\psi$ is defined by the matrix$\, :$  
\[
\begin{pmatrix}
T_0+a_0T_1 & 0 & 0 & 0\\
0 & T_0+T_1 & 0 & 0\\
0 & 0 & T_0 & 0\\
0 & 0 & T_1 & T_0
\end{pmatrix}
\] 
hence $\phi$ is defined by the matrix$\, :$ 
\[
\begin{pmatrix}
X_0 & X_1 & X_2 & X_3\\
a_0X_0 & X_1 & X_3 & 0
\end{pmatrix}
\, .
\] 
The $2\times 2$ minors of this matrix generate the ideal$\, :$ 
\[
(X_0X_3,\, X_1X_3,\, X_3^2,\, X_0X_1,\, X_1X_2,\, X_0X_2)
\]
of $S$. Notice that$\, :$ 
\[
(X_0X_2,\, X_0X_3,\, X_1X_2,\, X_1X_3) = (X_0,\, X_1) \cap (X_2,\, X_3)
\]
is the homogeneous ideal of $L \cup L^\prim$, where $L$ (resp., $L^\prim$) is 
the line of equations $X_2 = X_3 = 0$ (resp., $X_0 = X_1 = 0$). One deduces 
that the degeneracy scheme $D(\phi)$ of $\phi$ consists of two simple points 
$(1,0,0,0)$ and $(0,1,0,0)$ situated on $L$ and of the double point 
$2(0,0,1,0)$ on the line $L^\prim$. Moreover, as above, $\Cok \phi \simeq 
\sco_{D(\phi)}(1)$. 

\vskip2mm 

\noindent
{\bf Case 3.}\quad  $\psi$ \emph{has, generically, rank $4$ and}  
$D(\psi) = 2(-1,1) + 2(0,1)$. 

\vskip2mm 

\noindent 
In this case, $\psi$ is defined by the matrix$\, :$  
\[
\begin{pmatrix}
T_0+T_1 & 0 & 0 & 0\\
T_1 & T_0+T_1 & 0 & 0\\
0 & 0 & T_0 & 0\\
0 & 0 & T_1 & T_0
\end{pmatrix}
\] 
hence $\phi$ is defined by the matrix$\, :$  
\[
\begin{pmatrix}
X_0 & X_1 & X_2 & X_3\\
X_0+X_1 & X_1 & X_3 & 0
\end{pmatrix} 
\, .
\]  
The $2\times 2$ minors of this matrix generate the ideal$\, :$  
\[
(X_1X_3,\, X_0X_3,\, X_3^2,\, X_1X_2,\, X_0X_2,\, X_1^2)
\]
of $S$. One deduces that $D(\phi)$ consists of the 
double point $2(1,0,0,0)$ on the line $L$ of equations $X_2 = X_3 = 0$ and of 
the double point $2(0,0,1,0)$ on the line $L^\prim$ of equations 
$X_0 = X_1 = 0$. Again, $\Cok \phi \simeq \sco_{D(\phi)}(1)$. 

\vskip2mm 

\noindent
{\bf Case 4.}\quad  
$\psi$ \emph{has, generically, rank $4$ and}  
$D(\psi) = (-1,1) + 3(0,1)$ 

\vskip2mm

\noindent 
In this case, $\psi$ is defined by the matrix$\, :$ 
\[
\begin{pmatrix}
T_0+T_1 & 0 & 0 & 0\\
0 & T_0 & 0 & 0\\
0 & T_1 & T_0 & 0\\
0 & 0 & T_1 & T_0
\end{pmatrix}
\]   
hence $\phi$ is defined by the matrix$\, :$ 
\[
\begin{pmatrix}
X_0 & X_1 & X_2 & X_3\\
X_0 & X_2 & X_3 & 0
\end{pmatrix}
\, .
\] 
The $2\times 2$ minors of this matrix generate the ideal$\, :$  
\[
(X_0X_3,\, X_2X_3,\, X_3^2,\, X_0X_2,\, X_0X_1,\, X_1X_3 - X_2^2) 
\]
of $S$. The ideal$\, :$ 
\[
(X_0X_1,\, X_0X_2,\, X_0X_3) = (X_1,\, X_2,\, X_3) \cap (X_0)
\]
is the homogeneous ideal of the union of the point $(1,0,0,0)$ and of the 
plane $X_0 = 0$, hence the ideal$\, :$  
\[
(X_0X_1,\, X_0X_2,\, X_0X_3,\, X_1X_3 - X_2^2)
\]
defines the subscheme of $\piii$ 
consisting of the (simple) point $(1,0,0,0)$ and of the 
conic $C$ of equations $X_0 = 0$, $X_1X_3 - X_2^2 = 0$. The conic $C$ has a 
parametrization $\alpha : \p^1 \ra \piii$ given by $\alpha(u_0,u_1) = 
(0,\, u_0^2,\, u_0u_1,\, u_1^2)$. Since $\alpha^\ast(X_2X_3) = u_0u_1^3$ and 
$\alpha^\ast(X_3^2) = u_1^4$ one deduces that $D(\phi)$ consists of the point 
$(1,0,0,0)$ and of the triple point $3(0,1,0,0)$ on the conic $C$. Again, 
$\Cok \phi \simeq \sco_{D(\phi)}(1)$. 

\vskip2mm 

\noindent 
{\bf Case 5.}\quad $\psi$ \emph{has, generically, rank $4$ and} 
$D(\psi) = 4(0,1)$.  

\vskip2mm

\noindent
In this case, $\psi$ is defined by the matrix$\, :$  
\[
\begin{pmatrix}
T_0 & 0 & 0 & 0\\
T_1 & T_0 & 0 & 0\\
0 & T_1 & T_0 & 0\\
0 & 0 & T_1 & T_0
\end{pmatrix}
\]
hence $\phi$ is defined by the matrix$\, :$ 
\[
\begin{pmatrix}
X_0 & X_1 & X_2 & X_3\\
X_1 & X_2 & X_3 & 0
\end{pmatrix}
\, .
\] 
The $2\times 2$ minors of this matrix generate the ideal$\, :$  
\[
(X_1X_3,\, X_2X_3,\, X_3^2,\, X_2^2,\, X_0X_3 - X_1X_2,\, X_0X_2 - X_1^2)
\]
of $S$. The complete intersection of equations $X_0X_3 - X_1X_2 = 0$, 
$X_0X_2 - X_1^2 = 0$ consists of the line $L^\prim$ of equations $X_0 = X_1 = 0$ 
and of the twisted cubic curve $\Gamma \subset \piii$ parametrized by the 
map$\, :$ 
\[
\beta : \p^1 \lra \piii,\  (u_0,u_1) \mapsto 
(u_0^3,\, u_0^2u_1,\, u_0u_1^2,\, u_1^3)\, .
\]
The line $L$ of equations $X_2 = X_3 = 0$ cuts this complete intersection 
in the point $(1,0,0,0)$ which is situated on $\Gamma \setminus L^\prim$. 
Since$\, :$ 
\[
\beta^\ast(X_1X_3) = u_0^2u_1^4,\  \beta^\ast(X_2X_3) = u_0u_1^5,\  
\beta^\ast(X_3^2) = u_1^6,\  \beta^\ast(X_2^2) = u_0^2u_1^4, 
\]
one gets that the degeneracy scheme $D(\phi)$ consists of the 
quadruple point $4(1,0,0,0)$ on the twisted cubic curve $\Gamma$. Again, 
$\Cok \phi \simeq \sco_{D(\phi)}(1)$. 

\vskip2mm 

We consider, now, the case where $\psi$ has rank 3 everywhere. In this case 
$\Cok \psi \simeq \sco_{\p^1}(m)$, for some $m \geq 1$. It follows that 
$\Ker \psi \simeq \sco_{\p^1}(m-4)$. Since $A$ is injective, one deduces that 
$m \leq 3$. 

\vskip2mm 

\noindent 
{\bf Case 6.}\quad $\psi$ \emph{has rank $3$ and} $m = 1$.   

\vskip2mm

\noindent
In this case, up to the action of $\text{GL}(S_1)\times \text{GL}(4)$, 
$\psi$ is defined by the matrix$\, :$  
\[
\begin{pmatrix}
T_0 & T_1 & 0 & 0\\
0 & T_0 & T_1 & 0\\
0 & 0 & T_0 & T_1\\
0 & 0 & 0 & 0
\end{pmatrix}
\] 
hence $\phi$ is defined by the matrix$\, :$  
\[
\begin{pmatrix}
X_0 & X_1 & X_2 & 0\\
0 & X_0 & X_1 & X_2
\end{pmatrix}
\, .
\]
One deduces that the degeneracy scheme $D(\phi)$ is the ``fat point'' 
defined by the ideal $\sci_{\{x\}}^2$, where $x = (0,0,0,1)$. 

\vskip2mm 

\noindent
{\bf Case 7.}\quad $\psi$ \emph{has rank $3$ and} $m = 2$.   

\vskip2mm 

\noindent
In this case, $\psi$ is defined by the matrix$\, :$  
\[
\begin{pmatrix}
T_0 & T_1 & 0 & 0\\
0 & T_0 & T_1 & 0\\
0 & 0 & 0 & T_0\\
0 & 0 & 0 & T_1
\end{pmatrix}
\] 
hence $\phi$ is defined by the matrix$\, :$ 
\[
\begin{pmatrix}
X_0 & X_1 & 0   & X_2\\ 
0   & X_0 & X_1 & X_3
\end{pmatrix}
\, .
\]
The $2\times 2$ minors of this matrix generate the ideal$\, :$  
\[
(X_0^2,\, X_0X_1,\, X_1^2,\, X_1X_2,\, X_0X_3,\, X_0X_2 - X_1X_3)
\] 
of $S$. Let $L$ (resp., $L^\prim$) be the line of equation $X_2 = X_3 = 0$ 
(resp., $X_0 = X_1 = 0$). The ideal $(X_0^2,\, X_0X_1,\, X_1^2)$ defines the 
first infinitesimal neighbourhood of $L^\prim$ in $\piii$ with ideal sheaf 
$\sci_{L^\prim}^2$. Since $X_0X_2 - X_1X_3 = 0$ is a smooth quadric surface in 
$\piii$, one sees easily that the ideal $(X_1X_2,\, X_0X_3,\, X_0X_2 - X_1X_3)$ 
defines the subscheme $L \cup L^\prim$ of $\piii$. Consequently, the 
degeneracy scheme $D(\phi)$ is $L^\prim$. In particular, $\Cok \phi$ is an 
$\sco_{L^\prim}$-module. Restricting $\phi$ to $L^\prim$ one deduces, now, 
that $\Cok \phi \simeq \sco_{L^\prim}(2)$. 

\vskip2mm

\noindent 
{\bf Case 8.}\quad $\psi$ \emph{has rank $3$ and} $m = 3$. 

\vskip2mm 

\noindent
In this case, $\psi$ is defined by the matrix$\, :$ 
\[
\begin{pmatrix}
T_0 & T_1 & 0  & 0\\
0   & 0  & T_0 & 0\\
0   & 0  & T_1 & T_0\\
0   & 0  & 0   & T_1
\end{pmatrix}
\] 
hence $\phi$ is defined by the matrix$\, :$ 
\[
\begin{pmatrix}
X_0 & 0   & X_1 & X_2\\
0   & X_0 & X_2 & X_3
\end{pmatrix}
\, .
\]
The $2\times 2$ minors of this matrix define the ideal$\, :$  
\[
(X_0^2,\, X_0X_2,\, X_0X_3,\, X_0X_1,\, X_1X_3 - X_2^2)
\] 
of $S$. The ideal $(X_0^2,\, X_0X_2,\, X_0X_3,\, X_0X_1)$ defines the plane 
$X_0 = 0$, hence the degeneracy scheme $D(\phi)$ is the conic $C$ of equations 
$X_0 = 0$, $X_1X_3 - X_2^2 = 0$. Consider the parametrization$\, :$  
\[
\alpha : \p^1 \Izo C, \  (u_0,u_1) \mapsto (0,\, u_0^2,\, u_0u_1,\, u_1^2)\, .
\]
$\alpha^\ast(\phi)$ is defined by the matrix$\, :$ 
\[
\begin{pmatrix}
0 & 0 & u_0^2 & u_0u_1\\
0 & 0 & u_0u_1 & u_1^2
\end{pmatrix}
\, .
\]
Using the matrix identity$\, :$ 
\[
\begin{pmatrix}
u_0^2 & u_0u_1\\
u_0u_1 & u_1^2
\end{pmatrix}
 = 
\begin{pmatrix}
u_0\\
u_1
\end{pmatrix}
\begin{pmatrix}
u_0 & u_1
\end{pmatrix} 
\] 
one derives that $\Cok \phi \simeq \sco_{\p^1}(3)$. 
\end{remark}

\begin{lemma}\label{L:ilcupz2}
Let $Z$ be a $0$-dimensional subscheme of $\piii$ which is the zero scheme 
of a global section of ${\fam0 T}_\piii$ and let $L \subset \piii$ be a line 
not intersecting $Z$. If ${\fam0 length}(H \cap Z) \leq 1$, for every plane 
$H \supset L$, then $\sci_{L\cup Z}(2)$ is globally generated. 
\end{lemma} 

\begin{proof}
The schemes $Z$ occuring in the statement of the lemma were classified in 
Remark~\ref{R:2x4}, Cases 1--5. Actually, recalling the exact sequence 
$0 \ra \Omega_\piii \ra 4\sco_\piii(-1) \ra \sco_\piii \ra 0$, any morphism 
$\tau : \Omega_\piii \ra \sco_\piii$ is induced by a morphism $\theta : 
4\sco_\piii(-1) \ra \sco_\piii$. $\theta$ is defined by four linear forms 
$h_0, \ldots , h_3$. If $\phi : 4\sco_\piii \ra 2\sco_\piii(1)$ is the morphism 
defined by the matrix$\, :$ 
\[
\begin{pmatrix} 
X_0 & X_1 & X_2 & X_3\\
h_0 & h_1 & h_2 & h_3 
\end{pmatrix}
\]  
then the arguments from the last part of Case 1 in Remark~\ref{R:2x4} show 
that $\text{Im}\, \tau = \sci_{D(\phi)}$. 

\vskip2mm

\noindent
{\bf Claim 1.}\quad \emph{For every plane} $H \supset L$, \emph{the map}$\, :$  
\[
\tH^0(\sci_{L \cup Z}(2)) \lra \tH^0(\sci_{L \cup (H \cap Z),H}(2))
\]
\emph{is surjective}. 

\vskip2mm
 
\noindent 
\emph{Indeed}, if $H \cap Z = \emptyset$ then the above map is injective, 
hence an isomorphism (for dimensional reasons). 

If $H \cap Z \neq \emptyset$ then$\, :$ 
\[
\sci_{H \cap Z,H} = \frac{\sci_Z + \sci_H}{\sci_H} \simeq 
\frac{\sci_Z}{\sci_Z \cap \sci_H}\, .
\] 
One has $\sci_Z \cap \sci_H = h\sci_W(-1)$ for a certain closed 
subscheme $W$ of $Z$, where $h = 0$ is the equation of $H$. 
From the fact that$\, :$ 
\[
\frac{\sci_{H \cap Z}}{\sci_Z} = \frac{\sci_H + \sci_Z}{\sci_Z} \simeq 
\frac{\sci_H}{\sci_Z \cap \sci_H} 
\]
one derives an exact sequence$\, :$  
\[
0 \lra \sco_W(-1) \lra \sco_Z \lra \sco_{H \cap Z} \lra 0  
\] 
from which one deduces that $\text{length}(W) = 3$. If $L^\prim \subset \piii$ 
is any line then $L^\prim \cap Z$ is the scheme of zeroes of a section of 
$\text{T}_\piii \vb L^\prim \simeq 2\sco_L(1) \oplus \sco_L(2)$. It 
follows that $\text{length}(L^\prim \cap Z) \leq 2$. Consequently, 
$\text{h}^0(\sci_W(1)) = 1$. Using the exact sequence$\, :$  
\[
0 \lra \sci_W(-1) \lra \sci_Z \lra \sco_H 
\]
one gets that the kernel of the restriction map $\tH^0(\sci_Z(2)) \ra 
\tH^0(\sco_H(2))$ has dimension 1. For dimensional reasons, it follows that 
the map$\, :$ 
\[
\tH^0(\sci_{L \cup Z}(2)) \lra \tH^0(\sci_{L \cup (H \cap Z),H}(2))
\]
is surjective.  

\vskip2mm 

\noindent
{\bf Claim 2.}\quad \emph{The composite map}$\, :$  
\[
\tH^0(\sci_{L \cup Z}(2))\otimes_k\sco_\piii \lra \sci_L(2) \lra 
(\sci_L/\sci_L^2)(2) 
\] 
\emph{is an epimorphism}. 

\vskip2mm
 
\noindent
\emph{Indeed}, $\sci_L/\sci_L^2 \simeq 2\sco_L(-1)$, 
$\tH^0(\sci_L(1)) \izo \tH^0((\sci_L/\sci_L^2)(1))$ and if $H \supset L$ is a 
plane of equation $h = 0$ then, denoting by $\overline{h}$ the image of 
$h$ into $\tH^0((\sci_L/\sci_L^2)(1))$, one has an exact sequence$\, :$  
\[
0 \lra \sco_L \overset{\overline{h}}{\lra} (\sci_L/\sci_L^2)(1) \lra 
(\sci_{L,H}/\sci_{L,H}^2)(1) \lra 0\, .
\]
Since, for every plane $H \supset L$, the maps$\, :$ 
\begin{gather*}
\tH^0(\sci_{L \cup Z}(2)) \lra \tH^0(\sci_{L \cup (H \cap Z),H}(2))\, ,\\  
\tH^0(\sci_{L \cup (H \cap Z),H}(2))\otimes_k\sco_H \lra \sci_{L,H}(2) \lra 
(\sci_{L,H}/\sci_{L,H}^2)(2)  
\end{gather*}
are epimorphisms (the first one from Claim 1 and the second one because 
$\text{length}(H \cap Z) \leq 1$) one deduces that the composite map from 
Claim~2 is an epimorphism (a morphism of vector bundles $E \ra 
W\otimes_k \sco_L$ is epi iff, $\forall \, \lambda : W \ra k$ non-zero linear 
function, the composite morphism $E \ra W\otimes_k \sco_L 
\overset{\lambda}{\lra} \sco_L$ is epi). 

\vskip2mm
 
Claim 2 implies that if $f$ is a general element of 
$\tH^0(\sci_{L \cup Z}(2))$ then its image into $\tH^0((\sci_L/\sci_L^2)(2))$ 
vanishes at no point of $L$ from which one gets that the surface $Q \subset 
\piii$ of equation $f = 0$ is nonsingular along the line $L$. But this 
implies that $Q$ is, actually, a smooth quadric surface. One has an exact 
sequence$\, :$ 
\[
0 \lra \sco_\piii \overset{f}{\lra} \sci_{L \cup Z}(2) \lra 
\sci_{L \cup Z,Q}(2) \lra 0 
\]
hence $\sci_{L \cup Z}(2)$ \emph{is globally generated iff} $\sci_{L \cup Z,Q}(2)$ 
\emph{is globally generated}. 

Fix an isomorphism $\p^1\times \p^1 \izo Q$ and assume that $L$ belongs to the 
linear system $\vb \sco_Q(1,0) \vb$ on $Q$. In this case, 
$\sci_{L \cup Z,Q}(2) \simeq \sci_{Z,Q}(1,2)$. Moreover, since 
$\text{h}^0(\sci_{L \cup Z}(2)) = 3$, one has $\text{h}^0(\sci_{Z,Q}(1,2)) = 2$. 
Taking into account that the selfintersection of $\sco_Q(1,2)$ on $Q$ is 4 
which equals $\text{length}(Z)$, it follows that in order to show  
that $\sci_{L \cup Z,Q}(2)$ is globally generated it suffices to show that the 
linear system $\vb \sci_{Z,Q}(1,2) \vb$ \emph{has no base components} (on $Q$). 

So, let $D$ be the (effective) base divisor of the linear system 
$\vb \sci_{Z,Q}(1,2) \vb$. 

\vskip2mm 

\noindent
$\bullet$\quad If $D \in \vb \sco_Q(0,1) \vb$ then there exists a 
complete intersection $Z^\prim$ of type $((1,1),(1,1))$ on $Q$  
(i.e., the intersection of $Q$ with a line in $\piii$) such that 
$\sci_{D,Q}\sci_{Z^\prim,Q} \subseteq \sci_{Z,Q}$. If $W$ is the closed subscheme 
of $Q$ defined by the ideal $(\sci_{Z,Q}:\sci_{D,Q})$ it follows that $W 
\subseteq Z^\prim$ hence $\text{length}(W) \leq 2$. Consider the exact 
sequence$\, :$ 
\[
0 \lra \sco_W[-D] \lra \sco_Z \lra \sco_{D \cap Z} \lra 0\, .
\]
Since $D \cup L$ is the intersection of $Q$ with a plane $H \supset L$ one 
has $\text{length}(D \cap Z) \leq 1$. Using the above exact sequence and the 
fact that $\text{length}(Z) = 4$, one gets a \emph{contradiction}. 

\vskip2mm 

\noindent
$\bullet$\quad If $D \in \vb \sco_Q(0,2) \vb$ then, taking into account the 
fact that 
$\text{h}^0(\sco_Q(1,0)) = 2$, one must have $Z \subset D$. But $D = L^\prim 
+ L^\secund$ with $L^\prim$ and $L^\secund$ two lines belonging to the linear 
system $\vb \sco_Q(0,1) \vb$. $L \cup L^\prim$ (resp., $L \cup L^\secund$) is 
the intersection of $Q$ with a plane $H^\prim \supset Q$ (resp., $H^\secund 
\supset Q$) hence $\text{length}(L^\prim \cap Z) \leq 1$ and 
$\text{length}(L^\secund \cap Z) \leq 1$. One deduces that $\text{length}(Z) 
\leq 2$, a \emph{contradiction}. (Actually, if $L^\prim = L^\secund$ one needs a 
little argument of commutative algebra: let $A$ be a regular local ring of 
dimension 2, $I \subset A$ an ideal and $x,\, y$ a regular system of 
parameters. Assume that $y^2 \in I$ and that $I + Ay = Ax + Ay$. Then there 
exists $x^\prim \in I$ and $a \in A$ such that $x^\prim + ay = x$. Then $x^\prim, 
\, y$ is a regular system of parameters of $A$, hence 
$\text{length}(A/Ax^\prim + Ay^2) = 2$, hence $\text{length}(A/I) \leq 2$.) 

\vskip2mm

\noindent 
$\bullet$\quad If $D \in \vb \sco_Q(1,1) \vb$ then, taking into account that 
$\text{h}^0(\sco_Q(0,1)) = 2$, one must have $Z \subset D$. $D$ is the 
intersection of $Q$ with a plane $H \subset \piii$. Since $\tH^0(\sci_Z(1)) 
= 0$, one gets a \emph{contradiction}. 

\vskip2mm

\noindent
$\bullet$\quad If $D \in \vb \sco_Q(1,0) \vb$ then, taking into account the 
fact that a 1-dimensional linear subsystem of $\vb \sco_Q(0,2) \vb$ which has 
no base component has, actually, no base point, one gets that $Z \subset D$. 
But $D$ is a line in $\piii$, hence, as we saw above, $\text{length}(D \cap Z) 
\leq 2$, a \emph{contradiction}. 

It remains that $D = 0$, hence $\sci_{L \cup Z,Q}(2,2)$ is globally generated, 
hence $\sci_{L \cup Z}(2)$ is globally generated.    
\end{proof} 

\begin{lemma}\label{L:z=fatpoint}
Using the notation from the beginning of this appendix, if $Z$ is a fat 
point, i.e., it is the subscheme of $\piii$ defined by the ideal sheaf 
$\sci_{\{x\}}^2$ for some point $x \in \piii$, then $E$ is globally 
generated.  
\end{lemma}

\begin{proof}
According to Remark~\ref{R:2x4}, up to a linear change of coordinates in 
$\piii$, one may assume that $x = (0:0:0:1)$ and that 
$\phi : 4\sco_\piii \ra 2\sco_\piii(1)$ is defined by the matrix: 
\[
A = 
\begin{pmatrix}
X_0 & X_1 & X_2 & 0\\
0 & X_0 & X_1 & X_2
\end{pmatrix}
\, .
\]  
As we noticed in Case 4 of the proof of Prop.~\ref{P:c1=4c2=8n3} (right 
before Subcase 4.6) we may assume that the component 
$\psi : 2\sco_\piii(-1) \ra 2\sco_\piii(1)$ of $\rho$ is of the 
form$\, :$ 
\[
2\sco_\piii(-1) \overset{{\widehat \psi}}{\lra} 2\sco_\piii  
\overset{X_3}{\lra} 2\sco_\piii(1)   
\] 
where $\widehat \psi$ has the property that the composite morphism 
(denoted $\widetilde \psi$)$\, :$ 
\[
2\sco_\piii(-1) \overset{{\widehat \psi}}{\lra} 2\sco_\piii  
\lra \Cok \phi(-1) 
\]
is an epimorphism. Since the morphism$\, :$ 
\[
(2\sco_\piii)\otimes_{\sco_\piii}\sco_{\{x\}} \lra \Cok \phi(-1)
\otimes_{\sco_\piii}\sco_{\{x\}} 
\]
is an isomorphism, it follows that ${\widehat \psi}(x) := 
{\widehat \psi}\otimes_{\sco_\piii}\sco_{\{x\}}$ is an isomorphism. 
One deduces that, up to an automorphism of 
$2\sco_\piii(-1)$, one may assume that $\widehat \psi$ coincides  
with  
\[
2\sco_\piii(-1) \overset{X_3}{\lra} 2\sco_\piii  
\]
after applying $-\otimes_{\sco_\piii}\sco_{\{x\}}$. In other words, one may 
assume that $\widehat \psi$ is defined by a matrix of the form$\, :$  
\[
\begin{pmatrix}
X_3 + \lambda_{00} & \lambda_{01}\\
\lambda_{10} & X_3 + \lambda_{11}
\end{pmatrix}
\]
with $\lambda_{ij} \in k[X_0,X_1,X_2]_1$. In this case, $\rho$ is defined by 
the matrix$\, :$ 
\[
\begin{pmatrix}
X_0 & X_1 & X_2 & 0 & X_3^2 + X_3\lambda_{00} & X_3\lambda_{01}\\ 
0 & X_0 & X_1 & X_2 & X_3\lambda_{10} & X_3^2 + X_3\lambda_{11}
\end{pmatrix}
\, .
\] 
Moreover, performing some elementary transformations on the columns of this 
matrix (i.e., applying a convenient automorphism of $4\sco_\piii \oplus 
2\sco_\piii(-1)$) one may assume that $\lambda_{01} = 0 = \lambda_{10}$, 
$\lambda_{00} \in k[X_1,X_2]_1$ and $\lambda_{11} \in k[X_0,X_1]_1$. 

Now, we assert that there exists a $4\times 4$ matrix $B$ of linear forms (in 
the indeterminates $X_0,\ldots ,X_3$) such that the diagram$\, :$  
\[
\begin{CD}
4\sco_\piii(-2) @>A>> 2\sco_\piii(-1)\\
@VBVV @VV{{\widehat \psi}}V\\
4\sco_\piii(-1) @>A>> 2\sco_\piii  
\end{CD}
\]
commutes. \emph{Indeed}, let $\text{I}_n$ denote the $n\times n$ unit matrix. 
Since $(X_3\text{I}_2)A = A(X_3\text{I}_4)$, it suffices to find a  
$4\times 4$ matrix $B^\prim$ of linear forms in the indeterminates 
$X_0,X_1,X_2$ such that$\, :$ 
\[
AB^\prim = (\lambda_{ij})A \, .
\]
$A$ defines on $\pii$ an epimorphism $4\sco_\pii(-1) \ra 
2\sco_\pii$ whose kernel is of the form $M(-2)$ where $M$ is a rank 2 
vector bundle with $c_1(M) = 0$. In particular, $M^\ast \simeq M$. One 
derives an exact sequence$\, :$  
\[
0 \lra 2\sco_\pii(-2) \overset{A^{\text{t}}}{\lra} 4\sco_\pii(-1)  
\lra M \lra 0
\] 
from which one deduces that $\tH^i(M) = 0$, $i = 0,\, 1$. Consequently$\, :$  
\[
\tH^0(4\sco_\pii(1)) \overset{A}{\lra} \tH^0(2\sco_\pii(2))
\]
is an isomorphism. It follows that there exists a morphism $\beta^\prim: 
4\sco_\pii(-2) \ra 4\sco_\pii(-1)$ making commutative the diagram$\, :$  
\[
\begin{CD}
4\sco_\pii(-2) @>A>> 2\sco_\pii(-1)\\
@V{\beta^\prim}VV @VV{(\lambda_{ij})}V\\
4\sco_\pii(-1) @>A>> 2\sco_\pii  
\end{CD}
\, .
\]
One can take $B^\prim = \tH^0(\beta^\prim(2))$. 

Finally, since ${\widehat \psi}A = AB$, ${\widehat \psi} : 2\sco_\piii(-1) \ra 
2\sco_\piii$ induces an epimorphism $\Cok \phi(-2) \ra \Cok \phi(-1)$. 
Since both sheaves have length 4, this epimorphism must be an isomorphism. 
Consequently, the kernel $\sck$ of the composite morphism $\widetilde \psi$
coincides with the kernel of $2\sco_\piii(-1) \ra \Cok \phi(-2)$ hence 
with the image of $\phi(-2) : 4\sco_\piii(-2) \ra 2\sco_\piii(-1)$. 
In particular, $\sck(2)$ is globally generated. 
The exact sequence \eqref{E:g(-1)e(-2)o(-1)} (from Case 4 of the proof of 
Prop.~\ref{P:c1=4c2=8n3}) induces an exact sequence$\, :$ 
\[
0 \lra 2\sco_\piii(-3) \xra{\phi^\ast(-2)} 4\sco_\piii(-2) 
\lra E(-2) \lra \sck \lra 0  
\]
hence $E$ is globally generated. 
\end{proof}

\end{document}